\documentclass[12pt]{amsart}
\usepackage{xcolor}
\usepackage{cancel}
\usepackage{amsmath, amsthm, amssymb,esint,amscd}
\usepackage{fullpage}
\usepackage{enumitem}
\usepackage{algorithm}
\usepackage{bm}
\usepackage[bookmarks]{hyperref}
\usepackage{graphicx,amscd,amsfonts,amsopn,verbatim}
\usepackage{fullpage}
\usepackage[utf8]{inputenc}
\usepackage{yfonts}

\DeclareFontFamily{U}{mathb}{\hyphenchar\font45}
\DeclareFontShape{U}{mathb}{m}{n}{
<-6> mathb5 <6-7> mathb6 <7-8> mathb7
<8-9> mathb8 <9-10> mathb9
<10-12> mathb10 <12-> mathb12
}{}
\DeclareSymbolFont{mathb}{U}{mathb}{m}{n}
\DeclareMathSymbol{\llcurly}{\mathrel}{mathb}{"CE}
\DeclareMathSymbol{\ggcurly}{\mathrel}{mathb}{"CF}

\newtheorem{theorem}{Theorem}[section]
\newtheorem{lemma}[theorem]{Lemma}

\newtheorem{proposition}[theorem]{Proposition}
\newtheorem{definition}[theorem]{Definition}
\newtheorem{corollary}[theorem]{Corollary}

\theoremstyle{definition}
\newtheorem{remark}[theorem]{Remark}

\numberwithin{equation}{section}

\newcommand{\tu}{\tilde u}

\newcommand{\tv}{\tilde v}
\newcommand{\ta}{\tilde a}
\newcommand{\tp}{\tilde p}
\newcommand{\tX}{\tilde X}
\newcommand{\tY}{\tilde Y}
\newcommand{\tP}{\tilde P}
\newcommand{\hP}{\hat P}
\newcommand{\zX}{X^0}
\newcommand{\rA}{\mathring{A}}

\newcommand{\bu}{{\mathbf u}}
\renewcommand{\H}{\mathcal H}
\newcommand{\M}{{\mathfrak M}}
\newcommand{\tM}{\tilde{\mathfrak M}}

\newcommand{\CC}{\mathfrak{C}}

\newcommand{\DCC}{\mathfrak{DC}}

\newcommand{\D}{\partial}

\newcommand{\g}{{\bf \gamma}}

\renewcommand{\u}{{\bf u}}

\newcommand{\R}{{\mathbb R}}

\newcommand{\LL}{{\mathcal L}}

\newcommand{\tg}{{\tilde g}}
\newcommand{\hg}{{\hat g}} 
\newcommand{\tA}{{\tilde A}}
\newcommand{\tB}{{\tilde B}}
\newcommand{\tR}{{\tilde R}}
\renewcommand{\AA}{{\mathcal A}}
\newcommand{\BB}{{\mathcal B}}
\newcommand{\bB}{{\mathbf B}}
\newcommand{\PP}{{\mathfrak P}}
\newcommand{\DPP}{{\mathfrak{DP}}}

\newcommand{\AAs}{{\mathcal A^\sharp}}
\newcommand{\hc}{{\mathring c}}
\newcommand{\tc}{{\tilde c}}
\newcommand{\tb}{{\tilde b}}

\newcommand{\tq}{{\tilde q}}
\newcommand{\htc}{\hat{\tilde c}}
\newcommand{\hq}{{\mathring q}}
\newcommand{\ff}{\mathbf{f}}
\newcommand{\bv}{\mathbf{v}}
\newcommand{\bE}{\mathbf{E}}
\newcommand{\bw}{\mathbf{w}}

\newcommand{\half}{ \frac{1}{2}}
\newcommand{\myvec}[1]{\ensuremath{\begin{pmatrix}#1\end{pmatrix}}}
\newcommand{\bD}{\langle D_x \rangle}
\newcommand{\bapprox}{\overset{bal}{\approx}}
\newcommand{\iintT}{\int_0^T \!\!\!\! \int_{\R^n}}
\newcommand{\ag}{{\hat P}}

\begin{document}

\title{The time-like minimal surface equation in Minkowski space: low regularity solutions}

\author{Albert Ai}
\address{Department of Mathematics, University of Wisconsin, Madison}
%\thanks{}
\email{aai@math.wisc.edu}
\author{Mihaela Ifrim}
 \address{Department of Mathematics, University of Wisconsin, Madison}
\email{ifrim@wisconsin.edu}
%\thanks{The second author was supported by }
\author{ Daniel Tataru}
\address{Department of Mathematics, University of California at Berkeley}
 %\thanks{The third author was partially supported by }
\email{tataru@math.berkeley.edu}

\begin{abstract}
  It has long been conjectured that for nonlinear wave equations which satisfy 
  a nonlinear form of the null condition, the low regularity well-posedness theory
  can be significantly improved compared to the sharp results of Smith-Tataru 
  for the generic case. The aim of this article is to prove the first result in this direction, namely for the time-like minimal surface equation in the Minkowski space-time. Further, our improvement is substantial, namely by $3/8$ derivatives in two space dimensions and by $1/4$ derivatives in higher dimensions.

\end{abstract}

\setcounter{tocdepth}{1}
\maketitle

\tableofcontents

\section{Introduction}

The question of local well-posedness for nonlinear wave equations
with rough initial data is a fundamental question in the the study 
of nonlinear waves, and which has received a lot of attention over the years.
The result of Smith and Tataru \cite{ST}, proved almost 20 years ago, 
provides the sharp regularity threshold for generic nonlinear wave equations in view of Lindblad's counterexample \cite{L-blow}. On the other hand, it has also been conjectured \cite{T:ICM} that for nonlinear wave equations which satisfy a suitable nonlinear null condition, the result of \cite{ST} can be improved, and the well-posedness threshold can be lowered. In this paper we provide the first result which proves the validity of this conjecture, for a representative equation in this class, namely the hyperbolic minimal surface equation.
Further, our improvement turns out to be substantial; precisely, we gain $3/8$ derivatives in two space dimensions and $1/4$ derivatives in higher
dimension. At this regularity level, the Lorentzian metric $g$ in our problem 
is no better that $C_{x,t}^{\frac14+} \cap L^2_t C_x^{\frac12+}$, 
($C_{x,t}^{\frac38+} \cap L^4_t C_x^{\frac12+}$ in $2d$)
far below anything studied before.

Most of the ideas introduced in this paper will likely extend to other nonlinear wave models, and open the way toward further progress in the study of low regularity 
solutions.

\subsection{The Minimal Surface Equation in Minkowski Space}

Let $n \geq 2$, and $\M^{n+2}$ be the $n+2$ dimensional Minkowski space-time.
A codimension one time-like submanifold $\Sigma \subset \M^{n+2}$ is called a minimal 
surface if it is locally a critical point for the area functional 
\[
\mathcal L  = \int_{\Sigma} \, dA,
\]
where the area element is measured relative to the Minkowski metric.
A standard way to think of this equation is by representing $\Sigma$ as a graph
over $\M^{n+1}$,
\[
\Sigma = \{ x_{n+1} = u (t,x_1,\cdot,x_n)\},
\]
where  $u$ is a real valued function 
\[
u: D \subset \M^{n+1} \rightarrow \R ,
\]
which satisfies the constraint
\begin{equation}\label{time-like}
u_t^2 < 1+ |\nabla_x u|^2,
\end{equation}
expressing the condition that its graph is a time-like surface in $\M^{n+2}$.

Then the surface area functional takes the form 
\begin{equation}\label{L}
\mathcal L(u) = \int \sqrt{1-u_t^2 +|\nabla_x u|^2}\ dx.
\end{equation}
Interpreting this as a Lagrangian, the minimal surface equation can be thought of as the  associated Euler-Lagrange equation, which takes the form
\begin{equation}
\label{minimal surface eq}
- \frac{\D}{\D t} \left( \frac{u_t}{\sqrt{1 - u_t^2 + |\nabla_x u|^2}}\right) + \sum_{i = 1}^n \frac{\D}{\D x_i}  \left( \frac{u_{x_i}}{\sqrt{1 - u_t^2 + |\nabla_x u|^2}}\right) = 0.
\end{equation}
Under the condition \eqref{time-like}, the above equation is a quasilinear wave equation. 

The left hand side of the last equation can be also interpreted as 
the mean curvature of the hypersurface $\Sigma$, and as such the minimal surface equation is alternatively described as the \emph{zero mean curvature flow}. 

In addition to the above geometric interpretation,
the minimal surface equation for time-like surfaces in the Minkowski space is also known as the Born-Infeld model in nonlinear electromagnetism \cite{Zwiebach}, as well as a model for evolution of branes in string theory \cite{Hoppe}.

On the mathematical side, the question of global existence for small, smooth and  localized initial data was considered in work of Lindblad~\cite{L-small}, Brendle~\cite{B-small}, Stefanov~\cite{S-small} and Wong~\cite{MR3730736}. The stability 
of a nonflat steady solution, called the catenoid, was studied in \cite{KL,DKS}. Some blow-up scenarios due to failure of immersivity
were investigated by Wong~\cite{MR3786768}.
Minimal surfaces have also been studied as singular limits of certain semilinear wave equations by Jerrard~\cite{Jerrard}. The local well-posedness question fits into the similar theory for the broader class of quasilinear 
wave equations, but there is also one  result which is specific to minimal surfaces, due to Ettinger~\cite{boris}; this is discussed later in the paper.

In our study of the minimal surface equation, the above way of representing it is less useful, and instead it is better to think of it in geometric terms. 
In particular the fact that the above Lagrangian \eqref{L} and the equation 
\eqref{minimal surface eq} are formulated relative to a background Minkowski
metric is absolutely non-essential; one may instead use any flat Lorentzian metric. This is no surprise since any two such metrics are equivalent 
via a linear transformation. Perhaps less obvious is the fact that 
the equations may be actually written in an identical fashion, independent of the background metric; see Remark~\ref{r:all-Lorentz} in Section~\ref{s:equations}.

For full details on the structure of the equation we refer the reader to Section~\ref{s:equations} of the paper, but here we review the most important facts.

The main geometric object is the metric  $g$ which is the trace of the Minkowski metric in $\M^{n+2}$ on $\Sigma$, and which, expressed in the $(t=x_0,x_1,\cdots, x_n)$ coordinates, has the form
\begin{equation}\label{metric-g}
g_{\alpha \beta} := m_{\alpha \beta} + \D_{\alpha} u \D_{\beta} u,
\end{equation}
where $m_{\alpha \beta}$ denotes the Minkowski metric with signature $(-1, 1, ..., 1)$ in $\M^{n+1}$. Since $\Sigma$ is time-like, this is also a Lorentzian metric.
This has determinant 
\begin{equation}\label{def-g}
g: = |\det (g^{\alpha \beta})| =1+ m^{\alpha \beta} \partial_\alpha u \,  \partial_\beta u   ,   
\end{equation}
and the dual metric is 
\begin{equation}\label{dual-metric-intro}
g^{\alpha \beta}  := m^{\alpha \beta} - 
\frac{m^{\alpha \gamma} m^{\beta\delta} \D_{\gamma} u \, \D_\delta u}{1 + m^{\mu\nu} \D_\mu u \, \D_\nu u}.
\end{equation}
Here, and later in the paper, we carefully avoid raising indices with respect to the  Minkowski metric. Instead, all raised indices in this paper will be with respect to the metric $g$.

Relative  to this metric, the equation \eqref{minimal surface eq} can be expressed in the form
\begin{equation}\label{msf}
\Box_g u = 0,    
\end{equation}
where $\Box_g$ is the covariant d'Alembertian, and which in this problem will 
be shown to have the simple expression
\begin{equation}
 \Box_g = g^{\alpha \beta} \partial_\alpha \partial_\beta.  
\end{equation}
An important role will also be played by the associated linearized equation,
which, as it turns out, may be easily expressed in divergence form as 
\begin{equation}
\partial_\alpha \hg^{\alpha \beta} \partial_\beta v = 0, \qquad   \hg^{\alpha \beta}:= g^{-\frac12} g^{\alpha \beta}.
\end{equation}

Our objective in this paper will be to study the local well-posedness of the 
associated Cauchy problem with initial data at $t = 0$, 
\begin{equation}\label{cp-msf}
\left\{ 
\begin{aligned}
& \Box_g u = 0,\\
& u(t=0) = u_0, \\
& u_t(t=0) = u_1,
\end{aligned}
\right.
\end{equation}
where the initial data $(u_0,u_1)$ is taken in classical Sobolev spaces,
\begin{equation}\label{data}
  u[0]:= (u_0,u_1) \in \H^s := H^s \times H^{s-1},
\end{equation}
and is subject to the constraint 
\begin{equation}\label{time-like-data}
u_1^2 - |\partial_x u_0|^2 < 1.
\end{equation}

 Here we use the following notation for the Cauchy data in \eqref{minimal surface eq} at time $t$,
\[
u[t]:= (u(t, \cdot), u_t(t, \cdot)).
\]
We aim to investigate the range of exponents $s$ for which local well-posedness holds, and significantly improve the lower bound for this range.

\subsection{Nonlinear wave equations}

The hyperbolic minimal surface equation \eqref{minimal surface eq} can be seen as a special case of more general scalar quasilinear wave equations,
which have the form
\begin{equation}\label{NLW-gen}
g^{\alpha \beta}(\partial u) \partial_\alpha \partial_\beta u = N(u,\partial u) ,
\end{equation}
where, again, $g^{\alpha\beta}$ is assumed to be Lorentzian, but without any further structural properties, and where $u$ may be a vector valued function. This generic equation will serve as a reference. 

As a starting point, we note that the equation \eqref{minimal surface eq} (and also \eqref{NLW-gen} if $N=0$) admits the scaling law 
\[
u(t,x) \to \lambda^{-1} u(\lambda t,\lambda x).
\]
This allows us to identify the  critical Sobolev exponent  as 
\[
s_c=\frac{n+2}{2}.
\]
Heuristically, $s_c$ serves as a universal threshold for local well-posedness,
i.e. we have to have $s > s_c$. Taking a naive view, one might think of trying
to reach the scaling exponent $s_c$. However, this is a quasilinear wave equation, and getting to $s_c$ has so far proved impossible in any problem of this type.

As a good threshold from above, one might start with the classical well-posedness result, due to Hughes, Kato, and Marsden~\cite{HKM}, and which asserts that local well-posedness holds for $s > s_c+1$. This applies to all
equations of the form \eqref{NLW-gen}, and can be proved solely by using energy estimates. These have  the form
\begin{equation}
\| u[t]\|_{\H^s} \lesssim e^{\int_0^t \| \partial^2 u(s)\|_{L^\infty} ds} \| u[0] \|_{\H^s}. 
\end{equation}

They may also be restated in terms of quasilinear energy functionals 
$E^s$ which have the following two properties:

\begin{enumerate}[label=(\alph*)]
    \item Coercivity,
\[
 E^s(u[t]) \approx \| u[t]\|_{\H^s}^2.
\]
\item Energy growth,
\begin{equation}\label{ee-clasic-diff}
\frac{d}{dt} E^s(u) \lesssim \| \partial^2 u\|_{L^\infty} \cdot  E^s(u).
\end{equation}
\end{enumerate}
To close the energy estimates, it then suffices to use Sobolev embeddings,
which allow one to bound the above $L^\infty$ norm, which we will refer to as a \emph{control parameter}, in terms of the $\H^s$
Sobolev norm provided that $s > \frac{n}{2}+2$, which is one derivative above scaling.

The reason a derivative is lost in the above analysis is that one 
would only need to bound $\|\partial^2 u\|_{L^1 L^\infty}$, whereas
the norm that is actually controlled is $\|\partial^2 u\|_{L^\infty L^\infty}$; this exactly accounts for the one derivative difference in scaling. It also suggests that the natural way to improve the classical result is to control the $L^p L^\infty$ norm directly. This is indeed 
possible in the context of the Strichartz estimates, which in dimension three
and higher give the bound
\[
\|\partial^2 u\|_{L^2 L^\infty} \lesssim \| u[0]\|_{\H^{\frac{n+3}{2}}}, 
\]
with another $\epsilon$ derivatives loss in three space dimensions. When true, such a bound yields well-posedness for $s > \frac{n+3}{2}$, which is $1/2$ derivatives above scaling. The numerology changes slightly in two space dimensions, where the best possible Strichartz estimate has the form 
\[
\|\partial^2 u\|_{L^4 L^\infty} \lesssim \|u[0]\|_{\H^{\frac{n}2+\frac{7}4}},
\]
 which is $3/4$ derivatives above scaling. 

The difficulty in using Strichartz estimates is that, while these are well known in the constant coefficient case \cite{GV,Keel-Tao} and even for smooth variable coefficients \cite{Kap,MSS}, that is not as simple in the case of rough coefficients. Indeed, as it turned out, the full Strichartz estimates 
are true for $C^2$ metrics, see \cite{Sm} ($n = 2,3$), \cite{T:nlw2} (all $n$), but not, in general, for $C^\sigma$ metrics when $\sigma <2$, see the counterexamples of \cite{SS-counter,ST-counter}. This difficulty was resolved in two stages:

\begin{enumerate}[label=(\roman*)]
\item \emph{Semiclassical time scales and Strichartz estimates with loss of derivatives.}   The idea here, 
which applies even for $C^\sigma$ metrics with $\sigma <2$, is that, associated to each dyadic frequency scale $2^k$,
there is a corresponding `` semiclassical " time scale $T_k = 2^{-\alpha k}$, with $\alpha$ dependent on $\sigma$, so that full Strichartz estimates hold at frequency $2^k$ on the  scale $T_k$. Strichartz estimates with loss of derivatives   are then obtained by summing up the short time 
estimates with respect to the time intervals, separately at each frequency.
This idea was independently introduced in \cite{BC} and \cite{T:nlw1},
and further refined in \cite{BC1} and \cite{T:nlw3}.

\item \emph{Wave packet coherence and parametrices.} The observation here 
is that in the study of nonlinear wave equations such as \eqref{NLW-gen},
in addition to Sobolev-type regularity for the metric, we have an additional piece of information, namely that the metric itself can be seen as a solution to a nonlinear wave equation. This idea was first introduced and partially exploited 
in \cite{KR}, but was brought to full fruition in \cite{ST}, where it was shown that almost loss-less Strichartz estimates hold for the solutions 
to \eqref{NLW-gen} at exactly the correct regularity level.
\end{enumerate}

The result in \cite{ST} represents the starting point of the present work,
and is concisely stated as follows\footnote{The primary result in \cite{ST} is
for the case when $g=g(u)$, but it directly carries over to equations 
of the form \eqref{NLW-gen}.  The result as stated below applies equally 
to both cases, but if $g=g(u)$ then $s_c$ is one unit lower.}:

\begin{theorem}[Smith-Tataru \cite{ST}]\label{t:ST-intro}
\eqref{NLW-gen} is locally well-posed in $\H^s$ provided that 
\begin{equation}
s> s_c+\frac34,   \qquad n = 2  ,
\end{equation}
respectively
\begin{equation}
s> s_c+\frac12,   \qquad n \geq 3.
\end{equation}
\end{theorem}
As part of this result, almost loss-less Strichartz estimates were obtained both directly for the solution $u$, and more generally for the associated linearized evolution.
We will return to these estimates in Section~\ref{s:ST} for a more detailed statement and an in-depth discussion.

The optimality of this result, at least in dimension three, follows from work of Lindblad \cite{L-blow}, see also the  more recent two dimensional result in \cite{2d-counter}. However, this counterexample should only apply to ``generic'' models, and the local well-posedness threshold might possibly be improved in problems with additional structure, i.e. some form of null condition.

Moving forward, we recall that in \cite{T:ICM}, a null condition 
was formulated for quasilinear equations of the form \eqref{NLW-gen}. 

\begin{definition}[\cite{T:ICM}]\label{d:null}
The nonlinear wave equation \eqref{NLW-gen} satisfies the \emph{nonlinear null condition}
if
\begin{equation}\label{null}
   \frac{\partial g^{\alpha\beta}(u,p)}{\partial p_\gamma} \xi_\alpha \xi_\beta \xi_\gamma = 0 \qquad \text{in} \qquad 
g^{\alpha\beta}(u,p) \xi_\alpha \xi_\beta = 0.
\end{equation}
\end{definition}
Here we use the terminology ``nonlinear null condition" in order to distinguish it 
from the classical null condition, which is relative to the Minkowski metric, and was heavily used in the study of global well-posedness for problems with small localized data, see \cite{null} as well as the books \cite{Sogge,H-book}. In geometric terms, this null condition may be seen as a cancellation condition for the self interactions
of wave packets traveling along null geodesics. In Section~\ref{s:equations} we verify that the minimal surface equation indeed satisfies the nonlinear null condition.

Further, it was conjectured in \cite{T:ICM} that, for problems satisfying \eqref{null},
the local well-posedness threshold can be lowered below the one in \cite{ST}. This conjecture  
has remained fully open until now, though one should mention two results in \cite{KRS} and \cite{boris} for the Einstein equation, respectively the minimal surface equation, where the endpoint in Theorem~\ref{t:ST-intro} is reached but not crossed.

The present work provides the first positive result in this direction, specifically for the minimal surface equation. Indeed, not only are we able to lower the local well-posedness threshold in 
Theorem~\ref{t:ST-intro}, but in effect we obtain a substantial improvement, namely by $3/8$ derivatives in two space dimensions and by $1/4$ derivatives in higher dimension.

\subsection{The main result}

Our main result, stated in a succinct form,  is as follows:

\begin{theorem} \label{t:main}
The Cauchy problem for the minimal surface equation \eqref{cp-msf} is locally well-posed for initial data $u[0]$ in $\mathcal{H}^s$ which satisfies the constraint \eqref{time-like-data}, where
\begin{equation}\label{s-AIT2}
s> s_c+\frac38,   \qquad n = 2 , 
\end{equation}
respectively
\begin{equation}\label{s-AIT3}
s> s_c+\frac14,   \qquad n \geq 3.
\end{equation}
\end{theorem}

The result is valid regardless of the $\H^s$ size of the initial data. 
Here we interpret local well-posedness in a strong Hadamard sense, including:

\begin{itemize}
\item \emph{existence of solutions} in the class $u([\cdot]) \in C[0,T;\H^s]$,
with $T$ depending only on the $\H^s$ size of the initial data.
\item \emph{uniqueness of solutions}, in the sense that they are the  unique limits of smooth solutions.
\item \emph{higher regularity}, i.e. if in addition the initial data $u[0] \in \H^m$ with $m > s$, then the solution satisfies $u([\cdot]) \in C(0,T;\H^m)$, with a bound depending only on the $\H^m$ size of the data,
\[
\| u([\cdot])\|_{C(0,T;\H^m)} \lesssim \| u[0]\|_{\H^m}.
\]

\item \emph{continuous dependence} in $\H^s$, i.e. continuity of the 
the data  to solution map
\[
\H^s \ni u[0] \to u ([ \cdot ]) \in C[0,T;\H^s].
\]

\item \emph{weak Lipschitz dependence}, i.e. for two $\H^s$ solutions $u$ and $v$ we have the difference bound
\[
\| u([\cdot])-v([\cdot])\|_{C(0,T;\H^\frac12)} \lesssim 
\| u[0]-v[0]\|_{\H^\frac12}.
\]
\end{itemize}

In addition to the above components of the local well-posedness result,
a key intermediate role in the proof of the above theorem is played by 
the Strichartz estimates, not only for the solution $u$, but also, more
importantly, for the linearized problem 
\begin{equation}\label{cp-lin}
\left\{ 
\begin{aligned}
& \partial_\alpha \hg^{\alpha\beta} \partial_\beta v = 0,\\
& v(t=0) = v_0, \\
& v_t(t=0) = v_1,
\end{aligned}
\right.
\end{equation}
as well as its paradifferential counterpart
\begin{equation}\label{cp-Tlin}
\left\{ 
\begin{aligned}
& \partial_\alpha T_{\hg^{\alpha\beta}} \partial_\beta v  = 0,\\
& v(t=0) = v_0, \\
& v_t(t=0) = v_1,
\end{aligned}
\right.
\end{equation}
 Here the paraproducts are defined using the Weyl quantization, see Section~\ref{s:para} for more details.
For later reference, we state the Strichartz estimates in a separate theorem:

\begin{theorem}
Then there exists some $\delta_0 > 0$, depending on $s$ in \eqref{s-AIT2}, \eqref{s-AIT2} so that the following properties hold for every solution $u$ as in Theorem~\ref{t:main}: 

a) The solution $u$ in Theorem~\ref{t:main} satisfies the Strichartz estimates 
\begin{equation}
\begin{aligned}
\| \bD^{\frac12+\delta_0} \partial u \|_{L^4 L^\infty} \lesssim 1, \qquad n = 2,
\\
\| \bD^{\frac12+\delta_0} \partial u \|_{L^2 L^\infty} \lesssim 1, \qquad n \geq 3.
\end{aligned}
\end{equation}

b) Both the linearized equation \eqref{cp-lin} and its paradifferential version \eqref{cp-Tlin} are well-posed in $\H^{\frac58}$ for $(n=2)$
respectively $\H^\frac12$ for $n \geq 3$, and the following 
Strichartz estimates hold for each\footnote{Of course with an implicit constant which may depend on $\delta$.} $\delta > 0$: 
\begin{equation}\label{v-long2-intro}
\begin{aligned}
\|v\|_{L^\infty \H^{\frac58}}+
\| \bD^{-\frac{n}2-\frac14-\delta} \partial v \|_{L^4(0,1; L^\infty)}  \lesssim    & \ \| v[0]\|_{\H^\frac58}
%+ \| f\|_{L^1 H^{-\frac12} + L^{p'} W^{-\frac12+\sigma+\delta,q'}} 
\qquad n =2, 
% \\ 
% \frac{2}{p} + \frac{1}{q} = \frac12, & \qquad \sigma =  \frac{7}{2p}, \qquad p \in [4,\infty],
\end{aligned}
\end{equation}
respectively
\begin{equation}\label{v-long3-intro}
\begin{aligned}
\|v\|_{L^\infty \H^{\frac12}} + \| \bD^{-\frac{n}{2}-\frac14-\delta} \partial v \|_{L^2(0,1; L^\infty)}  \lesssim   & \  \| v[0]\|_{\H^\frac12}
%+ \| f\|_{L^1 H^{-\frac12} + L^{p'} W^{-\frac12+\sigma+\delta,q'}}
 \qquad n \geq 3, 
% \\
% \frac{1}{p} + \frac{1}{q} = \frac12, & \qquad \sigma =  \frac{2n-1}{2p}, \qquad p \in [2,\infty].
\end{aligned}
\end{equation}
\end{theorem}
We note that the Strichartz estimates in both parts (a) and (b) 
have derivative losses, namely $1/8$ derivatives in the $L^4 L^\infty$ bound in two dimensions, respectively $1/4$ derivatives in higher dimensions. These estimates only represent the tip of the iceberg.
One may also consider the inhomogeneous problem, 
allow source terms in dual Strichartz spaces, etc. These and other 
variations which play a role in this paper are discussed in Section~\ref{s:Strichartz}.

\bigskip
To understand the new ideas in the proof of our main theorem, we recall the 
two key elements of the proof of the result in \cite{ST}, namely (i)
the classical energy estimates \eqref{ee-clasic-diff} and (ii) the nearly lossless Strichartz estimates; at the time, the chief difficulty was 
to prove the Strichartz estimates.

In this paper we completely turn the tables, taking part (ii) above for 
granted, and instead work to improve the energy estimates. Let us begin with a simple observation, which is that the minimal surface equation \eqref{msf}  has a cubic nonlinearity, which allows one to 
replace \eqref{ee-clasic-diff} with 
\begin{equation}\label{ee-clasic-cubic}
\frac{d}{dt} E^s(u) \lesssim \| \partial u\|_{L^\infty} \| \partial^2 u\|_{L^\infty} \cdot  E^s(u).
\end{equation}
This is what one calls a \emph{cubic energy estimate}, which is useful in the study of long time solutions but does not yet help with the low regularity well-posedness question. The key to progress
lies in developing a much stronger form of this bound, which roughly has the form\footnote{See Section~\ref{s:not} for our Besov norm notations.}
\begin{equation}\label{ee-clasic-balanced}
\frac{d}{dt} E^s(u) \lesssim \| \partial u\|_{B^{\frac12}_{\infty,2}}^2  \cdot  E^s(u),
\end{equation}
where the two control norms on the right are now balanced, and only require 
$1/2$ derivative less than \eqref{ee-clasic-cubic}. 
This is what we call a \emph{balanced energy estimate}, which may only hold for a very carefully chosen energy functional $E^s$.

This is an idea which originates in our recent work on 2D water waves (see \cite{ai2019dimensional}), where balanced energy estimates are also used 
in order to substantially lower the low regularity well-posedness threshold.
Going back further, this has its roots in earlier work of the last two authors \cite{BH}, \cite{HIT}, in the context of trying to apply normal form methods in order to obtain long time well-posedness results in quasilinear
problems.  There we have introduced what we called the \emph{modified energy method}, which in a nutshell asserts that in quasilinear problems 
it is far better to modify the energies in a normal form fashion, rather than to transform the equation. It was the cubic energy estimates of \cite{HIT}  
which were later refined in \cite{ai2019dimensional} to balanced energy estimates. Along the way, we have also borrowed and adapted another idea from
Alazard and Delort \cite{AD,AD1}, which is to prepare the problem with a partial normal form transformation, and is a part of their broader concept of paradiagonalization; that same idea is also used here.

There are several major difficulties in the way of proving 
energy estimates such as \eqref{ee-clasic-balanced}:  
\begin{itemize}
\item The normal form structure is somewhat weaker in the case of the minimal surface equation, compared to water waves. As a consequence, we have to carefully understand which components of the equation can be improved 
with a normal form analysis and which cannot, and thus have to be estimated directly.

\item Not only are the energy functionals $E^s$ not explicit, they have 
to be constructed in a very delicate way, following a procedure which is reminiscent of Tao's renormalization idea in the context of wave-maps \cite{Tao-wm2d}, as well as the subsequent work \cite{T:wm2} of the third author  on the same problem. 

\item Keeping track of symbol regularities in our energy functionals
and in the proof of the energy estimates is also a difficult task. To succeed, here we adapt and refine a suitable notion of paracontrolled 
distributions, an idea which has already been used successfully in the realm of stochastic pde's \cite{GIP,KO}.

\item The balanced energy estimates need to be proved not only for the full equation, but also for the associated linear paradifferential equation, as a
key intermediate step, as well as for the full linearized flow. In 
particular, when  linearizing,  some of the favourable normal form structure (or null structure, to use the nonlinear wave equations language) is lost,
and the proofs become considerably more complex.

\end{itemize}

Finally, the Strichartz estimates of \cite{ST} cannot be used directly here.
Instead, we are able to reformulate them in a paradifferential fashion, and to apply them on appropriate semiclassical time scales. After interval summation, this leads to Strichartz estimates on the unit time scale but with derivative losses. Precisely, in our main Strichartz estimates, whose aim is to bound the control parameters in \eqref{ee-clasic-balanced}, we end up losing essentially $1/8$ derivatives in two space dimensions, and  $1/4$ derivatives in higher dimension. These losses eventually determine the regularity thresholds in our main result in Theorem~\ref{t:main}.

One consequence of these energy estimates is the following continuation result for the solutions:

\begin{theorem}\label{t:continuation}
The $\H^s$ solution $u$ given by Theorem~\ref{t:main} can be continued 
for as long as the following integral remains finite:
\begin{equation}
\int_0^T   \| \partial u(t)\|_{B^{\frac12}_{\infty,2}}^2 dt < \infty  .
\end{equation}
\end{theorem}

\subsection{An outline of the paper} 

\subsubsection*{Paraproducts and paradifferential calculus}
The bulk of the paper is written in the language of paradifferential calculus. The 
notations and some of the basic product and paracommutator bounds are introduced in Section~\ref{s:not}. Importantly, we use the Weyl quantization throughout; this plays a substantial role 
as differences between quantizations are not always perturbative in our analysis. Also of note,
we emphasize the difference between balanced and unbalanced bounds, so some of our 
$\Psi$DO product or commutator expansions have the form
\[
\text{commutator} = \text{principal part} + \text{unbalanced lower order} + \text{balanced error}. 
\]

\subsubsection*{The geometric form of the minimal surface equation}
While the flat d'Alembertian may naively appear to play a role in the expansion  \eqref{minimal surface eq} of the minimal surface equation, this is not at all useful, and instead we need to adopt a geometric viewpoint. As a starting point, in Section~\ref{s:not} we consider several equivalent formulations of the minimal surface equation, leading to its geometric form in \eqref{msf}. This is based on the metric $g$ associated to the solution $u$ by \eqref{metric-g}, whose dual we also compute. Two other conformally equivalent metrics will also play a role. In the same section we derive the linearized equation, and also introduce the associated linear paradifferential flow.

\subsubsection*{Strichartz estimates}
As explained earlier, Strichartz estimates play a major role in our analysis. These are applied to several equations, namely the full 
evolution, the linear paradifferential evolution and finally the linearized equation; in the present paper, we view the bounds for the 
paradifferential equation as the core ones, and the other bounds
as derived bounds, though not necessarily in a directly perturbative fashion. The Strichartz estimates  admit a number of formulations:
in direct form for the homogeneous flow, in dual form for the inhomogeneous one, or in the full form. The aim of Section~\ref{s:Strichartz}
is to introduce all these forms of the Strichartz estimates, as well as to describe the relations between them, in the context of this paper.
A new idea here is to allow source terms which are time derivatives of
distributions in appropriate spaces; this is achieved by reinterpreting the wave equation as a system.

\subsubsection*{Control parameters in energy estimates}
We begin Section~\ref{s:control} by defining the control parameters
$\AA$ and $\BB$, which will play a fundamental role in our energy estimate. Here $\AA$
is a scale invariant norm, at the level of $\| \partial u\|_{L^\infty}$, which will remain small uniformly in time. $\BB$, on the other hand, is time dependent and at the level of $\||D_x|^\frac12 \partial u\|_{L^\infty}$, and will control the energy growth. Typically, our \emph{balanced cubic  energy estimates} will have the form 
\[
\frac{\partial E}{\partial t} \lesssim_{\AA} \BB^2 E.
\]
To propagate energy bounds we will need to know that $\BB \in L^2_t$.
Also in the same section we prove a number of core bounds for our solutions in terms of the control parameters.

\

\subsubsection*{The multiplier method and paracontrolled distributions}
Both the construction of our energies and the proof of the energy estimates are based on a paradifferential implementation of the multiplier method, which leads to space-time identities of the form
\[
\iintT \Box_g u \cdot X u \, dx  dt = \left. E_X(u) \right|_{0}^T + \iintT R(u) \, dx  dt
\]
in a paradifferential format, where the vector field $X$ is our multiplier and $E_X$ is its associated energy, while $R(u)$ is the energy flux term which will have to be estimated perturbatively. A fundamental difficulty is that the multiplier $X$, which should heuristically be at the regularity level of $\partial u$
cannot be chosen algebraically, and instead has to be constructed in an inductive manner relative to the dyadic frequency scales. In order to accurately quantify the regularity of $X$, in Section~\ref{s:paracontrol} we use and refine the notion of paracontrolled distributions; in a nutshell, 
while $X$ may not be chosen to be a function of $\partial u$, it will still have to be paracontrolled by $\partial u$, which we denote by $X \llcurly u$.

\subsubsection*{Energy estimates for the paradifferential equation}
The construction of the energy functionals is carried out in Section~\ref{s:para}, primarily at the level of the linear paradifferential equation, first in $\H$
and then in $\H^s$. In both cases there are two steps: first the construction of the symbol of the multiplier $X$, as a paracontrolled distribution, and then the proof of the energy estimates. 
The difference between the two cases is that $X$ is a vector field in the first case,
but a full pseudodifferential operator in the second case; because of this, we prefer to present the two arguments separately.

\subsubsection*{Energy estimates for the full equation}
The aim of Section~\ref{s:ee-full-eqn} is to prove that balanced cubic energy estimates hold for the full equation in all $\H^s$ spaces with $s \geq 1$. We do this by thinking about the full equation in a paradifferential form, i.e. as a linear paradifferential equation with a nonlinear source term, and then by applying a  normal form transformation to the unbalanced part of the source term.

\subsubsection*{Well-posedness  for the linearized equation}
The goal of Section~\ref{s:linearized} is to establish both energy and Strichartz estimates for $\H^\frac12$ solutions ($\H^{\frac58}$ in dimension two) to the linearized equation.
This is achieved under the assumption that both energy and Strichartz estimates for $\H^\frac12$ solutions ($\H^{\frac58}$ in dimension two) for the linear paradifferential equation hold. We remark that, while the energy estimates for the linear paradifferential equation have already been established by this point in the paper,
the corresponding Strichartz estimates have yet to be proved.

\subsubsection*{Short time Strichartz estimates for the full equation}
The local well-posedness result of Smith and Tataru~\cite{ST} yields well-posedness and nearly sharp Strichartz estimates on the unit time scale for initial data which is small in the appropriate Sobolev space. Our objective in Section~\ref{s:ST} is to recast this result as a short time result 
for a corresponding large data problem. This is a somewhat standard scaling/finite speed of propagation argument, though with an interesting twist due to the need to use homogeneous Sobolev norms.

\subsubsection*{Small vs. large $\H^s$ data}  In our main well-posedness
prof, in order to avoid more cumbersome notations and estimates,
it is convenient to work with initial data which is small in $\H^s$. 
This is not a major problem, as this is a nonlinear wave equation which exhibits finite speed of propagation. This allows us to reduce the large data problem 
to the small data problem by appropriate localizations. This argument is carried out at the beginning of Section~\ref{s:final}.

\subsubsection*{Rough solutions as limits of smooth solutions}
Our sequence of modules discussed so far comes together in 
Section~\ref{s:final}, where we finally obtain  our rough solutions $u$ as a limit of smooth solutions $u^h$ 
with initial data frequency localized below frequency $2^h$. 
The bulk of the proof is organized as a bootstrap argument, where the 
bootstrap quantities are uniform energy type bounds for both 
$u^h$ and for their increments $v^h = \dfrac{d}{dh} u^h$, which solve the corresponding linearized equation. The main steps are as follows:

\begin{itemize}
    \item we use the short time Strichartz estimates derived from \cite{ST} for $u^h$ and $v^h$ in order to obtain long time Strichartz estimates for $u^h$, which in turn implies energy estimates for both the full equation and the paradifferential equation, and closes one half of the bootstrap.

\item  we combine the short time Strichartz estimates and the long time 
energy estimates for the paradifferential equation in $\H^\frac12$
($\H^\frac58$ if $n=2$) to obtain long time Strichartz estimates for the same paradifferential equation.
    
\item we use the energy and Strichartz estimates for the paradifferential equation to obtain similar bounds for the linearized equation. This in turn implies long time energy estimates for $v^h$, closing the second half of the bootstrap loop.    
\end{itemize}

\subsubsection*{The well-posedness argument} Once we have a complete collection of energy estimates and Strichartz estimates for both the full equation 
and the linearized equation, we are able to use frequency envelopes in order 
to prove the remaining part of the well-posedness results, namely the strong convergence of the smooth solutions,  the continuous dependence, and the associated uniqueness property. In this we follow the strategy outlined
in the last two authors' expository paper \cite{IT-primer}.

\subsection{Acknowledgements}
The first author was supported by the Henry Luce Foundation. The second author was supported by a Luce Associate Professorship, by the Sloan Foundation, and by an NSF CAREER grant DMS-1845037. The third author was supported by the NSF grant DMS-2054975 as well as by a Simons Investigator grant from the Simons Foundation.  

This material is also based upon work supported by the National Science Foundation under Grant No. DMS-1928930 while all three authors participated in the program \textit{Mathematical problems in fluid dynamics} hosted by the Mathematical Sciences Research Institute in Berkeley, California, during the Spring 2021 semester.

\section{Notations, paraproducts and some commutator type bounds}
 \label{s:not}
 
 We begin with some standard notations and conventions:
 \begin{itemize}
\item The greek indices $\alpha,\beta,\gamma,\delta$ etc. in expressions range from $0$
to $n$, where $0$ stands for time. Roman indices $i, j$ are limited to the range from $1$ to $n$,
and are associated only to spatial coordinates.

\item The differentiation operators with respect to all coordinates are $\partial_\alpha$,
$\alpha = 0,..., n$. By $\partial$ without any index we denote the full space-time gradient.
To separate only spatial derivatives we use the notation $\partial_x$.

\item We consistently use the Einstein summation convention, where repeated indices 
are summed over, unless explicitly stated otherwise.   

\item The inequality sign $x \lesssim y$ means $x \leq Cy$ with a universal implicit constant $C$. If instead the implicit constant $C$ depends on some parameter $A$ then we write
instead $x \lesssim_A y$. 

\end{itemize}

\subsection{Littlewood-Paley decompositions and  Sobolev spaces} 

We denote the Fourier variables by $\xi_\alpha$ with $\alpha = 0,...,n$. To separate 
the spatial Fourier variables we use the notation $\xi'$.

\subsubsection{Littlewood-Paley decompositions}
For distributions in $\R^n$ we will use the standard inhomogeneous Littlewood-Paley decomposition
\[
u = \sum_{k=0}^\infty P_k u, 
\]
where $P_k = P_k (D_x)$ are multipliers with smooth symbols $p_k(\xi')$, localized in the dyadic frequency region  $\{|\xi| \approx 2^k\}$ (unless $k=0$, where we capture the entire unit ball). We emphasize that no such decompositions are used in the paper with respect to the time variable. We will also use the notations $P_{<k}$, $P_{>k}$ with the standard meaning.
Often we will use shorthand for the Littlewood-Paley pieces of $u$, such as
$u_k :=P_k u$ or $u_{<k}:= P_{<k} u$.

\subsubsection{Function spaces}
For our main evolution we will use inhomogeneous Sobolev space $H^s$,
often combined as product spaces $\H^s = H^s \times H^{s-1}$ for the
position/velocity components of our evolution. In the next to last section of the paper only we will have an auxiliary use for the corresponding homogeneous spaces $\dot H^s$, in connection with scaling analysis.

For our estimates we will use $L^\infty$ based control norms.
In addition to the standard $L^\infty$ norms, in many estimates we will use the standard inhomogeneous $BMO$ norm, as well as its close relatives $BMO^s$, with norm defined as 
\[
\| f\|_{BMO^s} = \| \bD^s f\|_{BMO}.
\]
We will also need several related $L^\infty$ based Besov
norms $B^{s}_{\infty,q}$, defined as 
\[
\|u\|_{B^{s}_{\infty,q}}^q = \sum_k 2^{pks} \|P_k u\|_{L^\infty}^p
\]
with the obvious changes if $p = \infty$. In particular 
the spaces $ B^{0}_{\infty,1}$ and $B^{\frac12}_{\infty,2}$
will be used for our control norms $\AA$ and $\BB$.

\subsubsection{Frequency envelopes} 

Throughout the paper we will use the notion of \emph{frequency envelopes}, introduced by Tao (see for example \cite{Tao-wm2d}), which is  a very useful device that tracks the evolution of the energy of solutions between dyadic energy shells.
\begin{definition}
We say that $\{c_k\}_{k\geq 0} \in \ell^2$ is a frequency envelope for a function $u$ in $H^s$ if we have the following two properties:

a) Energy bound:
\begin{equation}
\|P_k u\|_{H^s} \leq c_k, 
\end{equation}

b) Slowly varying
\begin{equation}
\frac{c_k}{c_j} \lesssim 2^{c|j-k|} , \quad j,k\in \mathbb{N}.
\end{equation}
\end{definition}
Here  $c$ is a positive constant, which is taken small enough in order to account for energy leakage between nearby frequencies.

One can also limit from above the size of a frequency envelope, 
by requiring that 
\[
\| u\|_{H^s}^2 \approx \sum c_k^2.
\]
Such frequency envelopes always exist, for instance one can define 
\[
c_k = \sup_j 2^{-\delta|j-k|} \|P_j u\|_{H^s}.
\]
The same notion can be applied to any Besov norms. In particular 
we will use it jointly for the Besov norms which define our control parameters 
$\AA$ and $\BB$.

 \subsection{Paraproducts and paradifferential operators} \label{s:para} \
  For multilinear analysis, we will consistently use paradifferential calculus, for which we refer the reader to \cite{Bony,Metivier}.
 
 We begin with the simplest bilinear expressions, namely products,
 for which we will use the Littlewood-Paley trichotomy 
 \[
 f\cdot g = T_f g + \Pi(f,g) + T_g f, 
 \]
 where the three terms capture the \emph{low$\times$high} frequency interactions,
 the \emph{high$\times$high} frequency interactions and the \emph{low$\times$high} frequency interactions. The paraproduct $T_f g$ might be heuristically thought of as the dyadic sum
 \[
 T_{f} g = \sum_{k} f_{<k-\kappa} g_k
 \]
 where the frequency gap $\kappa$ can be simply chosen as a universal parameter, say $k = 4$,
 or on occasion may be increased and used as a smallness parameter in a large data context.
 However, in our context a definition such as the above one is too imprecise, and 
the difference between usually equivalent choices is nonperturbative. Also, the symmetry properties of $T_f$ as an operator in $L^2$ are important in our energy estimates.
For this reason, we choose to work with the Weyl quantization, and we define 
\[
\mathcal F (T_f g)(\zeta) = \int_{\xi+\eta = \zeta}   \hat f(\eta) \chi\left( \frac{|\eta|}{\langle \xi+\frac12\eta\rangle}  \right)  \hat g(\xi) d\xi.
\]
Here $\chi$ is a smooth function supported in a small ball and which equals $1$ near the origin. With this convention, if $f$ is real then $T_f$ is an $L^2$
self-adjoint operator.

For paraproducts we have  a number of standard bounds which we list below, and we will refer to as Coifman-Meyer estimates:
\begin{equation}
\| T_f g\|_{L^p} \lesssim \| f\|_{L^\infty} \|g\|_{L^p},
\end{equation} 

 \begin{equation}
\| T_f g\|_{L^p} \lesssim \| f\|_{L^p} \|g\|_{BMO},  \end{equation}
 
\begin{equation}
\|\Pi(f,g)\|_{L^p} \lesssim \| f \|_{L^p} \| g\|_{BMO}  .  
\end{equation}
These hold for $1 < p < \infty$, but there are also endpoint results 
available roughly corresponding to $p = 1$ and $p = \infty$.
 
 Paraproducts may also be thought of as 
 belonging to the larger class of translation invariant bilinear operators.
 Such operators 
 \[
 f,g \to B(f,g)
 \]
 may be described by their symbols $b(\eta,\xi)$ in the Fourier space, by
 \[
 \mathcal F B(u,v)(\zeta) = \int_{\xi+\eta = \zeta} b(\eta,\xi) \hat f(\eta) \hat g(\xi) d \xi.
 \]
 A special class of such operators, which we denote by $L_{lh}$, will play 
 an important role later in the paper:
 
 \begin{definition}\label{d:Llh}
a) By $L_{lh}$ we denote translation invariant bilinear forms whose symbol
$\ell_{lh}(\eta,\xi)$ is supported in $\{|\eta| \ll |\xi|+1\}$ and satisfies 
bounds of the form
\[
|\partial^i_\eta \partial^j_\xi \ell_{lh}(\eta,\xi) | \lesssim \langle \xi \rangle^{-i-j}.
\]
\end{definition}

We remark that in particular the bilinear form $B(f,g) = T_f g$
is an operator of type $L_{lh}$, with symbol
\[
b(\eta,\xi) = \chi\left( \frac{|\eta|}{\langle \xi+\frac12\eta\rangle}  \right).
\]
Here the factor in the denominator $\xi+\eta/2$ is the average of the $g$ input frequency 
and the output frequency, and corresponds exactly to our use of the Weyl calculus. 
The $L^p$ bounds and the commutator estimates for such bilinear form mirror exactly the similar bounds for paraproducts.

 \subsection{Commutator and other paraproduct bounds} 
Here we collect a number of general paraproduct estimates, 
which are relatively standard. See for instance Appendix B of \cite{HIT} and Section 2 of \cite{ai2019dimensional} for proofs of the following estimates as well as further references.

We begin with the following standard commutator estimate:

\begin{lemma}[$P_k$ commutators]\label{l:para-com-pk}
 We have
\begin{equation}
\| [T_f,P_k] \|_{\dot H^{s} \to \dot H^{s}} \lesssim 2^{-k} \| \partial f\|_{L^\infty},
\end{equation}
also in $L^\infty$.
\end{lemma}

\

The following commutator-type estimates are either exact reproductions of, or closely follow, statements from Section 2 of \cite{ai2019dimensional}:

\begin{lemma}[Para-commutators]\label{l:para-com}
 Assume that $\gamma_1, \gamma_2 < 1$. Then we have
\begin{equation}
\| T_f T_g - T_g T_f \|_{\dot H^{s} \to \dot H^{s+\gamma_1+\gamma_2}} \lesssim 
\||D|^{\gamma_1}f \|_{BMO}\||D|^{\gamma_2}g\|_{BMO}.
\end{equation}
\end{lemma}

\begin{lemma}[Para-associativity]\label{l:para-assoc}
For $s + \gamma_2 \geq 0, s + \gamma_1 + \gamma_2  \geq 0$, and $\gamma_1 < 1$
we have
\begin{equation}
\| T_f \Pi(v, u) - \Pi(v, T_f u)\|_{\dot H^{s + \gamma_1+\gamma_2}} \lesssim 
\||D|^{\gamma_1}f \|_{BMO}\||D|^{\gamma_2}v\|_{BMO} \|u\|_{\dot H^{s}}.
\end{equation}
\end{lemma}

\begin{lemma}[Para-Leibniz rule]\label{l:para-leibniz}
For the balanced Leibniz rule error 
\[
E^{\pi}_L (u,v) = T_{f}\D_\alpha \Pi(u, v) - \Pi(T_{f}\D_\alpha u, v) - \Pi(u, T_{f}\D_\alpha v)
\]
we have the bound
\begin{equation}
\|  E^{\pi}_L (u,v) \|_{H^s} \lesssim \|f\|_{BMO^\frac12} \| u\|_{BMO^{-\frac12 - \sigma}} \| v\|_{H^{s+\sigma}}, \qquad 
\sigma \in \R.
\end{equation}
\end{lemma}

\

Next, we state paraproduct estimates which also may be found in \cite{ai2019dimensional}:

\begin{lemma}[Para-products]\label{l:para-prod}
Assume that $\gamma_1, \gamma_2 < 1$, $\gamma_1+\gamma_2 \geq 0$. Then
\begin{equation}
\| T_f T_g - T_{fg} \|_{\dot H^{s} \to \dot H^{s+\gamma_1+\gamma_2}} \lesssim 
\||D|^{\gamma_1}f \|_{BMO}\||D|^{\gamma_2}g\|_{BMO}.
\end{equation}
\end{lemma}

\begin{lemma}[Low-high para-products]\label{l:para-prod2}
Assume that $\gamma_1, \gamma_2 < 1$, $\gamma_1+\gamma_2 \geq 0$. Then
\begin{equation}
\| T_f T_g - T_{T_fg} \|_{\dot H^{s} \to \dot H^{s+\gamma_1+\gamma_2}} \lesssim 
\||D|^{\gamma_1}f \|_{BMO}\||D|^{\gamma_2}g\|_{BMO}.
\end{equation}
\end{lemma}
These are stated here in the more elegant homogeneous setting, but 
there are obvious modifications which also apply in the inhomogeneous case.
We end with the following Moser-type result:

\begin{lemma}\label{l:R}
Let $F$ be smooth with $F(0)=0$, and $w \in H^s$.
Set
\[
R(w) = F(w) - T_{F'(w)} w.
\]
Then we have the estimate
\begin{equation}\label{est-R}
\|R(w)\|_{H^{s+\frac12}} \lesssim C(\|w\|_{L^\infty})\| D^\frac12 w\|_{BMO} \| w\|_{H^s}.
\end{equation}
\end{lemma}

 \subsection{Paradifferential operators}
 
As a generalization of paraproducts, we will also work with paradifferential operators.
Precisely, given a symbol $a(x,\xi)$ in $\R^n$, we define its paradifferential Weyl 
quantization $T_a$ as the operator
\[
\mathcal F (T_f g)(\zeta) = \int_{\xi+\eta = \zeta}   \hat f(\eta) \chi\left( \frac{|\eta|}{\langle \xi+\frac12\eta\rangle}  \right) \hat{a}(\eta,\xi)  \hat g(\xi) d\xi ,  
\]
 where 
 \[
 \hat a(\eta,\xi) = \mathcal F_x a(x,\xi).
 \]
 The simplest class of symbols one can work with is $L^\infty S^m$, which contains symbols
 $a$ for which 
\begin{equation}
| \partial_\xi^\alpha a(x,\xi) |\leq c_\alpha \langle \xi \rangle^{m-|\alpha|}     
\end{equation}
for all multi-indices $\alpha$. For such symbols, the Calderon-Vaillancourt theorem insures
appropriate boundedness in Sobolev spaces,
\[
T_a: H^s \to H^{s-m}.
\]

More generally, given a translation invariant space of distributions $X$, 
we can define an associated symbol class $X S^m$ of symbols with the property that 
\begin{equation}
\| \partial_\xi^\alpha a(x,\xi) \|_{X} \leq c_\alpha \langle \xi \rangle^{m-|\alpha|}     
\end{equation}
for each $\xi \in \R^n$.  Later in the paper, we will use several choices of symbols
of this type, using function spaces which we will associate to our problem.

\section{A complete set of equations}\label{s:equations}

Here we aim to further describe the minimal surface equation and the underlying geometry, and, in particular, its null structure. We also derive the linearized equation, and introduce the paralinearization of both the main  equation and its linearization.

\subsection{The Lorentzian geometry of the minimal surface}
Starting from the expression of the metric $g$ in \eqref{metric-g}, the dual metric is easily computed to be
\begin{equation}\label{dual-metric}
g^{\alpha \beta}  := m^{\alpha \beta} - 
\frac{m^{\alpha \gamma} m^{\beta\delta} \D_{\gamma} u \D_\delta u}{1 + m^{\mu\nu} \D_\mu u \D_\nu u}.
\end{equation}
Also associated to the metric $g$ is its determinant
\[
g = \det (g_{\alpha\beta}) = \det(g^{\alpha\beta})^{-1},
\]
and the associated volume form
\[
dV = \sqrt{g}\,  dx.
\]
This can be easily computed as
\[
g = 1+m^{\mu\nu} \D_\mu u \, \D_\nu u.
\]

In the sequel, we will always raise indices with respect to the metric $g$, never with respect to Minkowski. In particular we will
use the standard notation
\begin{equation}\label{d-up}
\D^\alpha = g^{\alpha \beta} \partial_\beta.
\end{equation}
We remark that, when applied to the function $u$, this operator
has nearly the same effect as the corresponding Minkowski 
operator,
\begin{equation}\label{g-vs-m}
\partial^\alpha u = \frac{1}g m^{\alpha \beta}\D_\beta u.
\end{equation}

\subsection{The minimal surface equation}
Here we rewrite the minimal surface equation in covariant form.
Using the $g$ notation above and the 
Minkowski metric, we rewrite \eqref{minimal surface eq} as 
\[
m^{\alpha\beta} \partial_\alpha ( g^{-\frac12} \partial_\beta u) = 0,
\]
or equivalently 
\[
m^{\alpha\beta} (\partial_\alpha \partial_\beta u 
- \frac1{2g} \partial_\alpha g \partial_\beta u)= 0 .
\]
Expanding the $g$ derivative, we have 
\begin{equation}\label{first-dg}
\partial_\alpha g = 2 m^{\mu \nu} \partial_\mu u \, \partial_\alpha \partial_\nu u.
\end{equation}
Then in the previous equation we recognize the expression for the 
dual metric, and the minimal surface equation  becomes
\begin{equation}\label{msf-short}
g^{\alpha \beta} \D_\alpha \D_\beta u = 0.
\end{equation}

Using the notation \eqref{d-up}, this is written in an even shorter form,
\begin{equation}\label{msf-short-plus}
\D^\alpha \D_\alpha u = 0.
\end{equation}
Similarly, using also \eqref{g-vs-m}, the 
relation \eqref{first-dg} becomes
\begin{equation}\label{second-dg}
\frac{1}{2g} \partial_\alpha g = \partial^\nu u \, \partial_\alpha \partial_\nu u.
\end{equation}

\subsection{The covariant d'Alembertian}
The covariant d'Alembertian associated to the metric $g$ has the form \[
\Box_g = \frac{1}{\sqrt{g}} \partial_\alpha \sqrt{g} g^{\alpha \beta}
\partial_\beta,
\]
which we can rewrite as 
\[
\begin{aligned}
\Box_g &
=  \partial_\alpha g^{\alpha \beta} \partial_\beta
+ \frac1{2 g} (\partial_\alpha g) g^{\alpha\beta} \partial_\beta
\\
&=  g^{\alpha \beta}\partial_\alpha  \partial_\beta+ \left(\partial_\alpha g^{\alpha \beta}\right) \partial_\beta
+ \frac1{2 g} (\partial_\alpha g) g^{\alpha\beta} \partial_\beta.
\end{aligned}
\]
Next we need to compute the two coefficients in round brackets. 
The second one is given by \eqref{second-dg}. For the first one,
for later use, we perform a slightly more general 
computation where we differentiate $g^{\alpha \beta}(\partial_{\gamma}u)$ as a function of its arguments $p_{\gamma}:=\partial_\gamma u$,
\begin{equation}
\label{dg-dp}
\frac{\partial g^{\alpha \beta}}{\partial p_{\gamma}} =-\partial^{\alpha}u\, g^{\beta \gamma} -\partial^{\beta}u\, g^{\alpha \gamma}.
\end{equation}
This formula follows by directly  differentiating \eqref{dual-metric} and from  \eqref{g-vs-m}, 
\[
\frac{\partial g^{\alpha \beta}}{\partial p_{\gamma}} =-m^{\alpha \gamma} \partial^{\beta} u \,  - m^{\beta \gamma}\partial^{\alpha}u \,+2g\partial^{\alpha}\, u\partial^{\beta}u\,  \partial^{\nu} u  .
\]
We  use \eqref{dual-metric} once again  to get \eqref{dg-dp}
\[
\begin{aligned}
\frac{\partial g^{\alpha \beta}}{\partial p_{\gamma}} &=-[g^{\alpha \gamma }+g\partial^{\alpha}u \partial^{\gamma}u] \partial^{\beta} u -[g^{\beta \gamma }+g\partial^{\gamma}u \partial^{\beta}u]\partial^{\alpha}u+2\partial^{\alpha} u\partial^{\beta}u g\partial^{\nu} u,\\
&= -g^{\alpha \gamma}\partial^{\beta}u - g^{\beta \gamma }\partial^{\alpha}u.
\end{aligned}
\]

From \eqref{dg-dp} and chain rule, we arrive at
\begin{equation}\label{d-gab}
  \D_\gamma g^{\alpha\beta} =  - \D^\alpha u \, g^{\beta\delta}
\D_\gamma \D_\delta u  - \D^\beta u\,  g^{\alpha \sigma} \D_\gamma \D_\sigma u.
\end{equation}

Setting $\gamma=\alpha$ and using the minimal surface equation in the \eqref{msf-short} formulation, we get
\begin{equation}\label{div-g}
 \D_\alpha g^{\alpha\beta} =  - \D^\alpha u \, g^{\beta\delta}
\D_\alpha \D_\delta u.  
\end{equation}
Comparing this with \eqref{second-dg}, we see that the last two terms 
in the $\Box_g$ expression above cancel, and we obtain
the following simplified form for the covariant d'Alembertian: 
\begin{equation} \label{box-g-cov}
\Box_g = g^{\alpha \beta} \partial_\alpha \partial_\beta.
\end{equation}
In particular, we get the covariant form of the minimal surface equation for $u$:
\begin{equation}\label{u-cov}
\Box_g u = 0.
\end{equation}

For later use, we introduce the notation
\begin{equation}\label{def-A}
A^\alpha = - \D_\beta g^{\alpha\beta} = \frac{1}{2g} \D^\alpha g = 
\D^\beta u \, g^{\alpha\delta}
\D_\beta \D_\delta u.
\end{equation}

An interesting observation is that from here on, the Minkowski metric
plays absolutely no role:

\begin{remark}\label{r:all-Lorentz}
In order to introduce the minimal surface equations we have started from the 
Minkowski metric $m^{\alpha\beta}$. However, the formulation \eqref{msf-short} of the equations together with the relations \eqref{dg-dp} provide a complete 
description of the equations without any reference to the Minkowski metric, 
and which is in effect valid for any other Lorentzian metric. Indeed, the 
equation \eqref{msf-short} together with the fact that the metric components
$g^{\alpha \beta}$ are smooth functions of $\partial u$ satisfying \eqref{dg-dp}
are all that is used for the rest of the paper. Thus, our results apply equally 
for any other Lorentzian metric in $\R^{n+2}$.
\end{remark}

\subsection{The linearized equations}
Our objective now is to derive the linearized minimal surface equations. We will denote by $v$ the linearized variable. 
Then, by \eqref{dg-dp},  the linearization of the dual metric $g^{\alpha\beta} = g^{\alpha\beta}(u)$ takes the form 
\[
 \delta g^{\alpha\beta} =  - \D^\alpha u \, g^{\beta\nu}
 \D_\nu v  - \D^\beta u\,  g^{\alpha \sigma} \D_\sigma v.
\]
Then the linearized equation is directly computed, using the symmetry in $\alpha$ and $\beta$, as
\[
g^{\alpha\beta} \D_\alpha \D_\beta v - 2 \D^\alpha u \, g^{\beta\gamma}
 \D_\gamma v   = 0.
\]
Using the expression of $A$ in \eqref{def-A}, 
the linearized equations take the form
\begin{equation}\label{linearizedeqn}
(g^{\alpha \beta} \D_{\alpha} \D_{\beta} - 2 A^\gamma\D_\gamma) v = 0.
\end{equation}
Alternatively this may also be written in a divergence form,
\begin{equation}\label{divlinearizedeqn}
(\D_{\alpha} g^{\alpha \beta} \D_{\beta} - A^\gamma \D_\gamma )v = 0.
\end{equation}
or in covariant form,
\begin{equation} \label{v-cov}
\begin{aligned}
 \Box_g v &= 2 A^{\beta} \partial_\beta v. 
\end{aligned}
\end{equation}

\subsection{ Null forms and the nonlinear null condition} \ 
The primary null form which plays a role in this article is 
$Q_0$, defined by 
\begin{equation}\label{def-Q0}
    Q_0 (v,w) :=g^{\alpha \beta} \D_{\alpha}v
    \D_{\beta} w  = \D^\alpha v \D_\alpha w.
\end{equation}
Now, we verify that the nonlinear null condition \eqref{null} holds; for this we use \eqref{dg-dp}  to compute
\[
\frac{\partial g^{\alpha \beta}}{\partial_{g_{\gamma}}} \xi_{\alpha}\xi_{\beta}\xi_{\gamma} =\left( -g^{\alpha \beta}\partial^{\beta}u -g^{ \beta \gamma}\partial^{\alpha}u\right)\xi_{\alpha}\xi_{\beta}\xi_{\gamma} ,
\]
which vanishes on the null cone $  g^{\alpha\beta} \xi_\alpha \xi_\beta = 0$.
\medskip

In addition we  would like the contribution of $A$ to the linearized equation to be a null form. We get 
\[
A^\beta \partial_\beta v ={\D^{\alpha}u}\,  Q_0(\D_{\alpha}u, v).
\]

\subsection{ Two conformally equivalent metrics}
While the metric $g$ is the primary metric we use in this paper, for technical reasons we will also introduce 
two additional, conformally equivalent metrics, as follows:

\bigskip

(i) The metric $\tg$ is defined by 
\begin{equation}\label{def-tg}
\tg^{\alpha \beta}   := (g^{00})^{-1} g^{\alpha \beta}.
\end{equation}
Then the minimal surface equation can be written as 
\begin{equation}
\label{ms-tg}
  \tg^{\alpha \beta}  \partial_\alpha \partial_\beta u = 0
\end{equation}
while the linearized equation, written in divergence form is
\begin{equation}
\label{ms-tg-v}
(  \tg^{\alpha \beta}  \partial_\alpha \partial_\beta 
  - \tA^\alpha \partial_\alpha) v= 0,
\end{equation}
where, still raising indices only with respect to $g$, 
\begin{equation}\label{ta-def}
\tA^\alpha = (g^{00})^{-1} A^\alpha -  \tg^{\alpha \beta} \partial_\alpha (\ln g^{00})  = 
\partial^\beta u \tg^{\alpha \delta} \partial_\beta \partial_\delta u
+ 2 \partial^0 u \tg^{0\delta} \tg^{\alpha \beta} \partial_\beta \partial_\delta u
.    
\end{equation}

The main feature of $\tilde g$ is that $\tilde g^{00}=1$.
Because of this, it will be useful in the study of the linear paradifferential flow, in order to prevent 
a nontrivial paracoefficient in front of $\partial_0^2 v$
in the equations.

\bigskip

(ii) The metric $\hg$ is defined by 
\begin{equation}\label{def-hg}
\hg^{\alpha \beta}   = g^{-\frac12} g^{\alpha \beta}.
\end{equation}
Then the minimal surface equation can be written as 
\begin{equation}
  \hg^{\alpha \beta}  \partial_\alpha \partial_\beta u = 0,
\end{equation}
which is not so useful. Instead, the advantage of using this metric is that, using \eqref{def-A},  the linearized equation can now be written in divergence form,
\begin{equation}\label{hg-lin}
\partial_\alpha \hg^{\alpha \beta}  \partial_\beta  v
  = 0.
  \end{equation}
This will be very useful when we study the linearized equation in $\H^\frac12$ (respectively $\H^\frac58$ in two dimensions).

\subsection{Paralinearization and the linear paradifferential flow}
A key element in our study of the minimal surface equation is the 
associated linear paradifferential flow, which is derived from the 
linearized flow \eqref{linearizedeqn}. In inhomogeneous form, 
this is
\begin{equation}\label{paralin-inhom}
(\D_{\alpha} T_{g^{\alpha \beta}} \D_{\beta} -  T_{A^\gamma}\D_\gamma) w  = f.
\end{equation} 
Similarly we can write the paradifferential equations associated to $\tg$,
namely
\begin{equation}\label{paralin-inhom-tg}
(\D_{\alpha} T_{\tg^{\alpha \beta}}  \D_{\beta} - T_{\tA^\gamma}\D_\gamma) w  = f.
\end{equation} 
as well as $\hg$, which can be written in divergence form:
\begin{equation}\label{paralin-inhom-hg}
\D_{\alpha} T_{\hg^{\alpha \beta}}  \D_{\beta} w  = f.
\end{equation} 
These are all equivalent up to perturbative errors. Accordingly, we introduce the notation 
 \begin{equation}
T_{P} =  \partial_\alpha T_{g^{\alpha \beta}} \partial_\beta 
\end{equation}
for the paradifferential wave operator
 as well as its counterparts $T_{\tP}$ and $T_{\hP}$ with the metric $g$ replaced by  $\tg$, respectively $\hg$.

We will first use the paradifferential equation in the study of the 
minimal surface problem \eqref{msf-short}, which we rewrite in the form 
\begin{equation}\label{paralin-nonlin}
(T_{g^{\alpha \beta}} \D_{\alpha} \D_{\beta} - 2 T_{A^\gamma}\D_\gamma) u  = N(u).
\end{equation} 

A key contention of our paper is that the nonlinearity $N$ plays a perturbative 
role. However, this has to be interpreted in a more subtle way, 
in the sense that $N$ becomes perturbative only after a well chosen 
partial, variable coefficient normal form transformation.

Secondly, we will use it in the study of the linearized minimal surface
equation, which we can write in the form
\begin{equation}\label{para-linearized}
\D_{\alpha}T_{\hg^{\alpha \beta}} \D_{\beta} v  = N_{lin}(u) v.
\end{equation} 
Here the nonlinearity $N_{lin}$ will also play a perturbative 
role, in the same fashion as above. We caution the reader that this is \emph{not}
the linearization of $N$.

%%%%%%%%%%%%%%%%%%%%%%%%%%%%%%%%%%%%%%%%%%%
%%%%%%%%%%%%%%%%%%%%%%%%%%%%%%%%%%%%%%%%%%%
%%%%%%%%%%%%%%%%%%%%%%%%%%%%%%%%%%%%%%%%%%%

\section{Energy and Strichartz estimates}
\label{s:Strichartz}

Both energy and Strichartz estimates play an essential role in this paper,
in various forms and combinations. These are primarily applied first 
to the linear paradifferential flow, and then to the linearized flow 
associated to solutions to our main equation \eqref{msf}.  Our goal here 
is to provide a brief overview of these estimates.

Importantly, in this section we do not prove any energy or Strichartz estimates.
Instead, we simply provide definitions and context for what will be proved later in the paper,
and prove a good number of equivalences between various well-posedness statements and 
estimates.  We do this under absolutely minimal assumptions (e.g. boundedness) on the metric $g$,
in order to be able to apply these properties easily later on. In particular there are no 
commutator bounds needed or used in this section. The structure
of the minimal surface equations also plays no role here.

\subsection{The equations}
For context, here we consider a pseudo-Riemannian metric $g$ in $I \times \R^n$, where
$I=[0,T]$ is a time interval of unspecified length. We will make some minimal universal assumptions on the metric $g$:
\begin{itemize}
    \item  both $g$ and its inverse are uniformly bounded,
    \item the time slices are uniformly space-like.
\end{itemize}
Associated to this metric $g$, we will consider several equations:
\begin{description}
\item[ The linear paradifferential flow in divergence form]
\begin{equation}\label{Tbox-g-div}
\partial_\alpha  T_{g^{\alpha \beta}} \partial_\beta v = f, \qquad 
 v[0] = (v_0,v_1) .
\end{equation} 
\item[The linear paradifferential flow in non-divergence form]
\begin{equation}\label{Tbox-g-nodiv}
  T_{g^{\alpha \beta}} \partial_\alpha \partial_\beta v = f, \qquad 
 v[0] = (v_0,v_1) .
 \end{equation}

\item[ The linear flow in divergence form]
\begin{equation}\label{box-g-div}
\partial_\alpha  g^{\alpha \beta} \partial_\beta v = f, \qquad 
 v[0] = (v_0,v_1) .
\end{equation} 
\item[ The linear flow in non-divergence form]
\begin{equation}\label{box-g-nodiv}
  g^{\alpha \beta} \partial_\alpha \partial_\beta v = f \qquad
 v[0] = (v_0,v_1) .
 \end{equation}
 \end{description}

Several comments are in order:

\begin{itemize}
    \item As written, the above evolutions are inhomogeneous. If $f = 0$
    then we will refer to them as the \emph{homogeneous} flows.
    
    \item In the context of this paper, we are primarily interested in the metric $\hg$,
    in which case the equation \eqref{box-g-div} represents our main linearized flow,
    and \eqref{Tbox-g-div} represents our main linear paradifferential flow. The metric $g$ and the nondivergence form of the equations will be used in order to connect our results with the result of Smith-Tataru, which will be used in our proofs. 
    
    \item One may also add a gradient potential in the equations above; with the gradient potential added there is no difference between the divergence and the non-divergence form of the equations. We omit it in this section, as it plays no role.
\end{itemize}

We will consider these evolutions in the inhomogeneous Sobolev spaces $\H^s$. In order 
to do this uniformly, we will assume that $|I| \leq 1$; else using homogeneous spaces 
would be more appropriate. The exponent $s$ will be an arbitrary real number in the case 
of the paradifferential flows, but will have a restricted range otherwise.

\subsection{ Energy estimates and well-posedness for the homogeneous problem}
Here we review some relatively standard definitions and facts about local well-posedness.

\begin{definition}
For any of the above flows in the homogeneous form, we say that they are (forward) well-posed in $\H^s$ in the time interval $I=[0,T]$ if for each initial data $u[0] \in \H^s$ there exists a unique solution $u$ with the property that 
\[
u([\cdot]) \in C(I;\H^s).
\]
\end{definition}
This corresponds to a linear estimate of the form
\begin{equation}
\| v[\cdot] \|_{L^\infty(I;\H^s)} \lesssim \|v[0]\|_{\H^s}.
\end{equation}
Sometimes one establishes additional bounds for the solution (e.g. Strichartz estimates)
and these are then added in to the class of solutions for which uniqueness is established.
We will comment on this where needed. If no such assumption is used, we call this \emph{unconditional uniqueness}. 

For completeness and reference, we now state without proof a classical well-posedness result:

\begin{theorem} \label{t:easy-wp}
Assume that $\D g \in L^1(I;L^\infty)$. Then 

a) The paradiffererential flows \eqref{Tbox-g-div} and \eqref{Tbox-g-nodiv} are wellposed 
in $\H^s$ for all real $s$.

b) The divergence form evolution \eqref{box-g-div} is well-posed in $\H^s$ for $s \in [0,1]$, and the non-divergence form evolution \eqref{box-g-nodiv} is well-posed in $\H^s$ for $s \in [1,2]$.
\end{theorem}

We remark that the metrics $g$ associated with the solutions of Smith-Tataru satisfy the above hypothesis, but the solutions in our paper do not.

A slightly stronger form of well-posedness is to assert the existence of a suitable 
(time dependent) energy functional $E^s$ in $\H^s$:

\begin{definition}\label{d:eest}
An energy functional for either of the above problems in $\H^s$ is a bounded quadratic form in $\H^s$ which has the following two properties:

\begin{enumerate}[label = \alph*)]
    \item Coercivity,
\begin{equation}\label{ee-equiv}
E^s(v[t])  \approx \| v[t] \|_{\H^s}^2  .
\end{equation}

\item Bounded growth for solutions $v$ to the homogeneous equation,
\begin{equation}\label{ee-hom}
\frac{d}{dt}E^s(v[t])  \lesssim B(t) \| v[t] \|_{\H^s}^2,  
\end{equation}
where $B \in L^1$ depends only on $g$.
\end{enumerate}
\end{definition}
Later we will also interpret $E^s$ as a symmetric bilinear form in $\H^s$. Such an interpretation is unique.

We remark that, in the context of Theorem~\ref{t:easy-wp}, where $\partial g \in L^1 L^\infty$, 
an energy functional $E^1$ corresponding to $s = 1$ is classically obtained by multiplying the equation with a suitable smooth time-like vector field  and integrating by parts; we refer the reader to Section~\ref{s:multiplier} where this procedure is described in greater detail.
Then for $s \neq 1$ one simply defines 
\[
E^s(v[0]) = E^1(\bD^{s-1} v[0]),
\]
and the corresponding control parameter $B$ may be taken as 
\[
B(t) = \| \partial g(t)\|_{L^\infty}.
\]

\subsection{ The wave equation as a system and the inhomogeneous problem}
\label{s:Duhamel}

Switching now to the associated inhomogeneous flows, the classical set-up is to take 
a source term $f \in L^1 H^{s-1}$, and then look for solutions $v$ in $C(I;\H^s)$ as above. This is commonly done using the \emph{Duhamel principle}, which is most readily applied 
by rewriting the wave equation as a system. We next describe this process.

A common choice is to write the system 
for the pair of variables $(v,\partial_t v)$. However, for us it will be more convenient
to to make a slightly different linear transformation, and use instead the pair
\begin{equation}\label{bv}
\bv[t] := \myvec{v(t)\\g^{0\alpha}\partial_\alpha v(t)}:= Q\myvec{v \\ \partial_t v}, \qquad
Q = \myvec{1 &  0 \\  g^{0j} \partial_j & g^{00}  }
\end{equation}
for \eqref{box-g-div} and \eqref{box-g-nodiv}, with products replaced by paraproducts in the case of the equation \eqref{Tbox-g-div} or \eqref{Tbox-g-nodiv}. For later use, we record the inverse
of $Q$; this is either 
\begin{equation}\label{Q-inverse}
 Q^{-1} = \myvec{ 1 & 0 \\ -(g^{00})^{-1} g^{0j} \partial_j  & (g^{00})^{-1}}   ,
\end{equation}
or its version with products replaced by paraproducts, as needed.

The system for $\bv$ will have the form 
\begin{equation}\label{bv-syst}
\frac{d}{dt} \bv[t] = \LL \bv[t],
\end{equation}
with the appropriate choice for the matrix operator $\LL$.
For instance in the case of the homogeneous equation \eqref{box-g-div} we have
\begin{equation}\label{calA}
 \LL = \myvec{ - (g^{00})^{-1} g^{0j} \partial_j &  (g^{00})^{-1} \\  - (g^{00})^{-1} \partial_i g^{ij}  \partial_j + \partial_i g^{i0}   (g^{00})^{-1} g^{0j} \partial_j   & -\partial_j (g^{00})^{-1} g^{0j}},
\end{equation}
which has the antisymmetry property (only for the principal part, in the non-divergence case)
\begin{equation}\label{A-sym}
\LL^* = - J \LL J^{-1}, \qquad J = \myvec{0 & 1 \\ -1 & 0}.
\end{equation}

We will always work in settings where $Q$ is bounded and invertible in $\H^s$. This is 
nearly automatic in the paradifferential case; there we only need to make sure that the operator 
$T_{g^{00}}$ is invertible. In the differential case we will have to ask that multiplication by $g$
and by $(g^{00})^{-1}$ are bounded in $H^{s-1}$. In such settings, $\H^s$ well-posedness 
for our original wave equation and for the associated system are equivalent. If a good energy functional $E^s$ exists for the wave equation, then we may define an associated energy functional
for the system by setting 
\begin{equation}\label{bE}
  \bE^s(\bv[t]) = E^s(Q^{-1} \bv[t]).
\end{equation}
Then the properties \eqref{ee-equiv} and \eqref{ee-hom} directly transfer to the homogeneous system
\eqref{bv-syst}.

If our system is (forward) well-posed in $\H^s$, then solving it generates a (forward) evolution operator $S(t,s)$ which is bounded in $\H^s$ and maps the data at time $s$ to the solution at time $t$,
\[
S(t,s) \bv[s] = \bv[t].
\]
For the system it is easy to consider the inhomogeneous version
\begin{equation}\label{bv-syst-inhom}
\frac{d}{dt} \bv[t] = \LL \bv[t] + \ff[t]. 
\end{equation}
If $f \in L^1 \H^s$ then the solution to \eqref{bv-syst-inhom} is given by 
Duhamel's formula,
\begin{equation}
\bv[t] = S(t,0) \bv[0] + \int_{0}^t S(t,s) \ff[s] \, ds,    
\end{equation}
and satisfies the bound
\begin{equation}\label{eb-bf}
\| \bv \|_{L^\infty \H^s} \lesssim \| \bv[0]\|_{H^s} + \| \ff\|_{L^1 \H^s}.    
\end{equation}
If we have a good energy $\bE^s$ for the homogeneous system, then Duhamel's formula
easily allows us to obtain the corresponding energy estimate for the inhomogeneous one,
namely
\begin{equation}\label{ee-bf}
\frac{d}{dt}\bE^s(\bv[t])  \lesssim \bE^s(\bv[t], \ff[t] )+  B(t) \| \bv[t] \|_{\H^s}^2.     \end{equation}

\bigskip

 Now we are ready to return to our original set of equations, add the source term $f$ and reinterpret the above consequences of Duhamel's formula there. As in the homogeneous case, we 
 define $\bv[t] = Q v[t]$. Then adding the source term $f$ in the original equation is equivalent to adding a source term $\ff$ in the  above system. Indeed, it is readily seen that for all our four equations, $\ff$ is given by
\begin{equation}\label{first-ff}
\ff[t] = \myvec{0 \\  f(t)}.   
\end{equation}
To complete the correspondence, we note that for such $\ff$ we have 
\[
Q^{-1} \ff = ({g^{00}})^{-1} \ff[t].
\]
Then we immediately arrive at the following result:

\begin{theorem}\label{t:lin-wp-inhom}
a) Assume that either of the homogeneous paradifferential flows \eqref{Tbox-g-div} or  \eqref{Tbox-g-nodiv} are well-posed in $\H^s$. Then the associated inhomogeneous 
flows are well-posed in $\H^s$ for $f \in L^1 H^{s-1}$, and the following estimate holds
\begin{equation}
\| v[\cdot] \|_{L^\infty(I;\H^s)} \lesssim \|v[0]\|_{\H^s} + \| f \|_{L^1 H^{s-1}}.
\end{equation}
In addition, if an energy functional $E^s$ in $\H^s$ exists, then 
\begin{equation}\label{ee-f}
\frac{d}{dt}E^s(v[t])  \lesssim E^s(v[t], (T_{g^{00}})^{-1} \ff[t] )+  B(t) \| u[t] \|_{\H^s}^2.
\end{equation}
%where on the right we interpret $E^s$ as a symmetric bilinear form.

b) The same holds for the  flows \eqref{box-g-div} or  \eqref{box-g-nodiv} under the additional assumption that multiplication by $g$ and $({g^{00}})^{-1}$ is bounded in $H^{s-1}$, with 
the paraproduct replaced by the corresponding product. 
\end{theorem}

For our purposes in this paper, we  will also need to allow for a larger class of source terms
of the form 
\begin{equation}\label{f+}
 f = \partial_t f_1 +f_2.   
\end{equation} 
To understand why this is natural, it is instructive 
to start from the inhomogeneous system \eqref{bv-syst-inhom} and argue backward. 

Above, we have used the inhomogeneous system in the case where the first component of $\ff$ was zero. Now we will allow for both terms to be nonzero in $\ff$, and derive the corresponding wave equation. For clarity we do this in the context of the equation \eqref{box-g-div}, for which we have computed the corresponding operator $\LL$ in \eqref{calA}; however, a similar computation will apply in all four cases. 

We begin by defining 
\begin{equation}\label{new-v}
 v(t) = \bv_1(t)
\end{equation}
as our candidate for the wave equation solution. Then the first equation of the system reads
\begin{equation}\label{new-vt}
\partial_t v = - (g^{00})^{-1} g^{0j} \partial_j v + (g^{00})^{-1} \bv_2 + \ff_1,
\end{equation}
or equivalently
\[
g^{0\alpha} \partial_\alpha v = \bv_2 + g^{00} \ff_1.
\]
Differentiating this with respect to time we obtain
\[
\begin{aligned}
\partial_\beta g^{\beta\alpha} \partial_\alpha v = & \ \partial_j g^{j\alpha} \partial_\alpha v
+ \partial_t \bv_2 + \partial_t g^{00} \ff_1
\\
= & \ \partial_j g^{jk} \partial_k v + \partial_j g^{j0} \partial_t v 
+ \partial_t \bv_2 + \partial_t g^{00} \ff_1.
\end{aligned}
\]
Finally we substitute $\partial_t v$ from \eqref{new-vt} and $\partial_t \bv_2$ from the second equation of the system. We already know the right hand side should vanish if $\ff=0$, so it 
suffices to track the $\ff$ terms. Then we easily obtain the desired equation for $v$:
\begin{equation}
\partial_\beta g^{\beta\alpha} \partial_\alpha v = \partial_\alpha g^{\alpha 0} \ff_1 + \ff_2. 
\end{equation}
Comparing this with \eqref{f+}, we obtain the correspondence between the source terms for the 
wave equation and the system:
\begin{equation}\label{f-vs-ff}
f_1 = g^{00} \ff_1, \qquad f_2 =  \partial_k g^{k0} \ff_1 + \ff_2.  
\end{equation}
We also record here the correspondence between the solutions, in the form
\begin{equation}\label{v-vs-bv}
\bv_1 = v, \qquad \bv_2 = g^{0\alpha} \partial_\alpha v - g^{00}\ff_1,    
\end{equation}
noting that this is no longer homogeneous, as in \eqref{bv}.

The last step in our analysis is to reinterpret the bounds \eqref{eb-bf} and \eqref{ee-bf}
in terms of $v$ and $f$. To do this we make the assumption that multiplication by $g$ and $(g^{00})^{-1}$ is bounded in both $H^s$ and $H^{s-1}$. 
Then from \eqref{eb-bf} we get
\[
\begin{aligned}
\|v\|_{L^\infty \H^s} \lesssim & \ \| \bv\|_{L^\infty \H^s} + \| \ff_1\|_{L^\infty \H^s}
\\
\lesssim & \ \|\bv[0]\|_{\H^s} + \| \ff\|_{L^1 \H^s} + \| \ff_1\|_{L^\infty \H^s}
\\
\lesssim & \ \|v[0]\|_{\H^s} + \|f_1\|_{L^1 H^{s-1}} + \| f_2\|_{L^1 H^s \cap L^\infty \H^s}.
\end{aligned}
\]

Similarly, from \eqref{ee-bf} and \eqref{bE} we obtain the energy bound
\begin{equation}
  \frac{d}{dt}E^s(Q^{-1}\bv[t])  \lesssim \bE^s(Q^{-1}\bv[t], Q^{-1}\ff[t] )+  B(t) \| v[t] \|_{\H^s}^2.
\end{equation}
Here we use \eqref{Q-inverse} and \eqref{v-vs-bv} to compute 
\[
Q^{-1}\bv[t] = v[t] - \myvec{0 \\ \ff_1 } = v[t] - \myvec{0 \\ (g^{00})^{-1} f_1 } := v[t] - \tilde v[t], 
\]
respectively, using also \eqref{f-vs-ff},
\begin{equation}\label{new-source}
\begin{aligned}
 \tilde f[t]: = Q^{-1}\ff[t] = & \  Q^{-1} \myvec{ (g^{00})^{-1} f_1 \\ f_2 - \partial_k g^{k0} (g^{00})^{-1} f_1}
\\
 = & \ \myvec{ (g^{00})^{-1} f_1 \\ (g^{00})^{-1}( f_2 - \partial_k g^{k0} (g^{00})^{-1} f_1 -  g^{0k} \partial_k (g^{00})^{-1} f_1) }.
\end{aligned}
\end{equation}
Thus we obtain the following natural extension of  Theorem~\ref{t:lin-wp-inhom} above:

\begin{theorem}\label{t:lin-wp-inhom+}
a) Assume that the homogeneous evolution \eqref{box-g-nodiv} or \eqref{box-g-div} is well-posed in $\H^s$, and that  multiplication by $g$ and  $(g^{00})^{-1}$ is bounded in $H^s$ and in $H^{s-1}$.
Consider the evolution \eqref{box-g-nodiv} with a source term $f$ of the form
\[
f = \partial_t f_1+ f_2,  \qquad f_1 \in L^1 H^s \cap C H^{s-1}, \qquad f_2 \in L^1 H^{s-1}.
\]
Then a unique solution $u \in C(I, \H^s)$ exists. If in addition the homogeneous problem admits an energy functional $E^s$ as in Definition~\ref{d:eest} then we have the energy estimate
\begin{equation}\label{ee-corect}
\frac{d}{dt}E^s(v[t]-\tilde v[t])  \lesssim E^s(v[t]-\tilde v[t], \tilde f[t] )+  B(t) \| v[t]-\tilde v[t] \|_{\H^s}^2  
\end{equation}
with $\tilde v$ and $\tilde f$ defined above and $B$ as in \eqref{ee-hom}.

b) The same result applies for the paradifferential equations \eqref{Tbox-g-div}, respectively 
\eqref{Tbox-g-nodiv}, where all instances of $g$ above are replaced by the corresponding paraproducts $T_g$.
\end{theorem}
We remark that in the situations where we apply this result, the mapping properties 
for $g$ and $(g^{00})^{-1}$
will be fairly straightforward to verify. In the paradifferential case, for instance,
the continuity of $g$ will suffice.

\subsection{A duality argument}
Duality plays an important role in many estimates for evolution equations. We will also 
use duality considerations in this paper for several arguments. We restrict the discussion below
to the problems written in divergence form, as this is what we will use later in the paper.
However, similar versions may be formulated in the nondivergence case.

At heart, this is based on the following identity, which in the context of the 
operator $\partial_\alpha g^{\alpha \beta} \partial_\beta$ is written as follows:
\begin{equation}\label{dual}
\iintT   \partial_\alpha g^{\alpha \beta}\partial_\beta  v \cdot w -  v \cdot\partial_\alpha g^{\alpha \beta} \partial_\beta w\, dx dt
= \left.\int g^{0j} \partial_j v \cdot w - v \cdot g^{0j} \partial_j w\,  dx \right|_0^T.
\end{equation}
This holds for any test functions $v$ and $w$. The integral on the right can be viewed as 
a duality relation between $u[t]$ and $v[t]$,
\[
B(u[t],v[t]) = \int g^{0j} \partial_j u \cdot v - u \cdot g^{0j} \partial_j v\,  dx.
\]
Precisely, assuming that $g:H^{s-1} \to H^{s-1}$ and that $g^{00}$ is invertible, this expression 
has the following two properties
\begin{enumerate}
    \item Boundedness,
\[
B: \H^{s} \times \H^{1-s} \to \R.
\]

\item Coercivity, 
\[
\sup_{\| v\|_{\H^{1-s}} \leq 1} B(u,v) \approx \| u \|_{\H^s}  .
\]
\end{enumerate}

A standard consequence of this relation is the following property:

\begin{proposition}
The evolutions \eqref{box-g-div}, respectively \eqref{Tbox-g-div} are forward well-posed in 
$\H^s$ iff they are backward well-posed in $\H^{1-s}$.
\end{proposition}
We remark that in the context of this paper forward and backward well-posedness are 
almost identical, so for us this property says that well-posedness in $\H^s$ and $\H^{1-s}$ are equivalent. 

The above proposition may be equivalently reformulated as the corresponding result for the 
system \eqref{bv-syst}. It will be more convenient to view it in this context. 
To do this, we  reinterpret the  above duality, in terms of the associated system \eqref{bv-syst-inhom}. In view of the  symmetry property \eqref{A-sym}, we have the relation
\begin{equation}\label{dual-system}
\iintT   (\partial_t - \LL) \bv \cdot J \bw -  J \bv \cdot (\partial_t - \LL) \bw\, dx dt
= \left.\int  \bv \cdot J \bw  \, dx\  \right|_0^T,
\end{equation}
where the corresponding duality relation is 
\begin{equation}
\bB(\bu,\bv) =   \int  \bv \cdot J \bw  \, dx,  
\end{equation}
which provides the duality between $\H^s$ and $\H^{1-s}$.
Incidentally, a consequence of \eqref{dual-system} is the duality relation
\[
S(t,s) = S(s,t)^\ast ,
\]
where the duality between $\H^s$ and $\H^{1-s}$ is the one given by the bilinear form $\bB$ above.
This can be used to construct the backward evolution in $\H^{1-s}$ given the forward evolution in $\H^s$, and vice-versa. The full equivalence argument is standard, and is omitted.

\subsection{Strichartz estimates}
Here we discuss several versions of Strichartz estimates, as well as the connection between them.

\subsubsection{Estimates for homogeneous equations} In the context of this paper,
these have the form 
\begin{equation}\label{Str-hom}
\| v \|_{S^r} + \| \partial_t v\|_{S^{r-1}} \lesssim \| v[0]\|_{H^r},    
\end{equation}
where for the Strichartz space $S$ we will consider two different choices:
\begin{enumerate}[label=\roman*)]
    \item Almost lossless estimates, akin to those established in Smith-Tataru~\cite{ST}.
The corresponding Strichartz norms, denoted by $S=S_{ST}$ are defined as 
\begin{equation}
\begin{aligned}
\| v\|_{S_{ST}^r} = \| v \|_{L^\infty H^r} + \|\bD^{r-\frac{3}4-\delta} v \|_{L^4 L^\infty}, \qquad n = 2,
\\
\| v\|_{S_{ST}^r} =\| v \|_{L^\infty H^r} + \|\bD^{r-\frac{n-1}2-\delta}  v \|_{L^2 L^\infty},
\qquad n \geq 3.
\end{aligned}    
\end{equation}
Here the loss of derivatives is measured by $\delta > 0$, which is an arbitrarily small parameter.

 \item Estimates with  derivative losses, precisely the type that will be established in this paper. The corresponding Strichartz norms, denoted by $S=S_{AIT}$ are defined as 
\begin{equation}
\begin{aligned}
\| v\|_{S_{AIT}^r} = \| v \|_{L^\infty H^r} + \|\bD^{r-\frac{3}4 - \frac18-\delta}  v \|_{L^4 L^\infty}, \qquad n = 2,
\\
\| v\|_{S_{AIT}^r} =\| v \|_{L^\infty H^r} + \|\bD^{r-\frac{n-1}2-\frac14-\delta}  v \|_{L^2 L^\infty},
\qquad n \geq 3.
\end{aligned}    
\end{equation}
Here $\delta > 0$ is again an arbitrarily small parameter, but we allow for an additional loss
of derivatives in the endpoint (Pecher) estimate, namely $1/8$ derivatives in two space dimensions and $1/4$ in higher dimensions.
\end{enumerate}

These estimates can be applied to any of the four equations discussed in this section.
There are also  appropriate counterparts for the corresponding system \eqref{bv-syst}, 
which have the form
\begin{equation}\label{Str-hom-syst}
\| \bv_1 \|_{S^r} + \| \bv_2\|_{S^{r-1}} \lesssim \| \bv[0]\|_{\H^r},  \qquad S \in \{ S_{ST},S_{AIT}\}  .
\end{equation}
Under very mild assumptions on $g$, these are equivalent to the ones for the corresponding wave equation:

\begin{proposition}\label{p:v-vs-bv-hom}
The Strichartz estimates \eqref{Str-hom} for the homogeneous wave equation are equivalent to the Strichartz estimates \eqref{Str-hom-syst} for the associated system.
\end{proposition}

We also remark on a very mild extension of the estimate \eqref{Str-hom}
to the inhomogeneous case. Precisely, if \eqref{Str-hom} holds then we also have the inhomogeneous bound
\begin{equation}\label{Str-hom-ext}
\| v \|_{S^r} + \| \partial_t v\|_{S^{r-1}} \lesssim \| v[0]\|_{H^r} + \|f\|_{L^1 H^{r-1}}.
\end{equation}
This follows in a straightforward manner by the Duhamel formula, see the discussion in Section~\ref{s:Duhamel}.

We conclude the discussion of the Strichartz estimates for the homogeneous equation 
with a simple but important case, which will be useful for us in the sequel, and applies in particular to the solutions in \cite{ST}.

\begin{proposition}\label{p:Str-allr}
Assume that $\partial g \in L^1L^\infty$ and that the Strichartz estimates for the homogeneous equation \eqref{box-g-nodiv} hold in $\H^1$. Then the Strichartz estimates for the homogeneous equation hold in $\H^r$ for all $r \in \R$ for both paradifferential flows \eqref{Tbox-g-div}
and \eqref{Tbox-g-nodiv}.
\end{proposition}
We remark that the implicit constant in these Strichartz estimates depends on the implicit constant in the Strichartz estimate in the hypothesis and on the bound for $\|\partial g\|_{L^1 L^\infty}$.
Later when we apply this result we will have uniform control over both, so we obtain 
uniform control over the $\H^r$ Strichartz norm.

\begin{proof}
It will be easier to work with the inhomogeneous bound \eqref{Str-hom-ext}, as it is more stable with respect to perturbations. We divide the proof into  several steps, all of which are relatively standard.

\emph{Step 1:}  We start with the case $r=1$ with the additional assumption $g^{00}= -1$.
Then the second equation in \eqref{Tbox-g-nodiv} can be seen as a perturbation of \eqref{box-g-nodiv} with an $L^1 L^2$ source term. Hence the bound \eqref{Str-hom-ext}
for \eqref{box-g-nodiv} implies the same bound for \eqref{Tbox-g-nodiv}.

\medskip
\emph{Step 2:} Next, assuming still that $g^{00}= -1$, we extend the 
bound \eqref{Str-hom-ext} for \eqref{Tbox-g-nodiv} to all integers $r$
by conjugating by $\bD^{\sigma}$ with $\sigma = r-1$, where we can estimate perturbatively the commutator 
\begin{equation}
\|  [T_{g^{\alpha \beta}},\bD^{\sigma}]  \partial_\alpha  \partial_\beta \bD^{-\sigma} v\|_{L^1 L^2} \lesssim \| \partial g\|_{L^1 L^\infty} \| \partial v\|_{L^\infty L^2}.
\end{equation}

\medskip
\emph{Step 3:} We multiply by $(T_{g^{00}})^{-1}$ to reduce the problem with nonconstant 
$g^{00}$ to the case when $g^{00}= -1$. At the conclusion of this step, we have 
the bound \eqref{Str-hom-ext} for \eqref{Tbox-g-nodiv} for all $r$.

\medskip
\emph{Step 4:} We commute the paracoefficients $T_{g^{\alpha\beta}}$ inside $\partial_\alpha$
perturbatively, in order to obtain the bound \eqref{Str-hom-ext} for \eqref{Tbox-g-div} for all $r$.
\end{proof}

\subsubsection{Dual Strichartz estimates}

Here one considers the corresponding inhomogeneous problems, with source terms in dual Strichartz spaces. The estimates have the form
\begin{equation}\label{Str-dual}
\| v \|_{L^\infty \H^r} \lesssim \| v[0]\|_{\H^r} + \| f \|_{(S^{1-r})'},  \qquad S \in \{ S_{ST},S_{AIT}\}   .
\end{equation}
Classically, these are obtained by duality from the homogeneous estimates, as follows:

\begin{proposition}\label{p:dual}
If the homogeneous estimates \eqref{Str-hom} hold in $\H^{r}$ for the forward (backward) evolution then the dual estimates \eqref{Str-dual} hold in $\H^{1-r}$ for the backward (forward) evolution.
\end{proposition}

However, one can do better than this by going instead through the system form of the equations
\eqref{bv-syst-inhom}. The dual estimates for \eqref{bv-syst-inhom} have the form
\begin{equation}\label{Str-dual-syst}
\| \bv \|_{L^\infty \H^r} \lesssim \| \bv[0]\|_{\H^r} + \| \ff_1 \|_{(S^{-r})'}
+ \|\ff_2 \|_{(S^{-r+1})'} ,  \qquad S \in \{ S_{ST},S_{AIT}\}  . 
\end{equation}
These are directly obtained from the homogeneous estimates for the system \eqref{bv-syst}
via the duality \eqref{dual-system}:

\begin{proposition}
If the homogeneous estimates hold in $\H^{r}$ for the forward (backward) evolution \eqref{bv-syst} then the dual estimates hold in $\H^{1-r}$ for the backward (forward) evolution \eqref{bv-syst-inhom}.
\end{proposition}

One can  now further return to the original inhomogeneous equation with a source term 
as in \eqref{f+}, and use the correspondence \eqref{f-vs-ff} and \eqref{v-vs-bv}, 
in order to transfer the dual bounds back. These dual estimates, which represent a generalization
of \eqref{Str-dual}, have the form 
\begin{equation}\label{Str-dual+}
\| v \|_{L^\infty \H^r} \lesssim \| v[0]\|_{\H^r} + \| f_1 \|_{L^\infty H^{r-1} \cap (S^{-r})'}
+ \|f_2\|_{(S^{1-r})'},  \qquad S \in \{ S_{ST},S_{AIT}\}   .
\end{equation}
We obtain the following strengthening of Proposition~\ref{p:dual}: 
\begin{proposition}\label{p:dual+}
If the homogeneous estimates \eqref{Str-hom} hold in $\H^{r}$ for the forward (backward) evolution then the dual estimates \eqref{Str-dual+} hold in $\H^{1-r}$ for the backward (forward) evolution.
\end{proposition}

\subsubsection{ Full (retarded) Strichartz estimates}
Here we combine the homogeneous and dual Strichartz estimates in a single bound for the 
inhomogeneous problem. The classical form is 
\begin{equation}\label{Str-full}
\| v \|_{S^r} + \| \partial_t v\|_{S^{r-1}} \lesssim \| v[0]\|_{\H^r} 
+ \|f\|_{(S^{1-r})'},  \qquad S \in \{ S_{ST},S_{AIT}\}   .
\end{equation}
However, here we need to take the extra step where we allow source terms 
of the form $f = \partial_t f_1 + f_2$, and then the estimates have the form
\begin{equation}\label{Str-full+}
\| v \|_{S^r} + \| \partial_t v\|_{S^{r-1}} \lesssim \| v[0]\|_{\H^r} 
+ \|f_1\|_{S^{r-1}\cap (S^{-r})'} + \|f_2\|_{(S^{1-r})'},  \qquad S \in \{ S_{ST},S_{AIT}\}  . 
\end{equation}
As we will see, this is closely related to the corresponding bound for the 
associated inhomogeneous system \eqref{bv-syst-inhom}:
\begin{equation}\label{Str-full-syst}
\| \bv_1 \|_{S^r} + \|\bv_2\|_{S^{r-1}} \lesssim \| \bv[0]\|_{\H^r} + \| \ff_1 \|_{(S^{-r})'}
+ \|\ff_2 \|_{(S^{-r+1})'} ,  \qquad S \in \{ S_{ST},S_{AIT}\}  . 
\end{equation}

Our main result here is as follows:

\begin{theorem}\label{t:Str-move-around}
Consider either the equation \eqref{box-g-div} or \eqref{Tbox-g-div}. If the homogeneous 
problem is well-posed forward in $\H^r$ and backward in $\H^{1-r}$ and satisfies the homogeneous Strichartz estimates \eqref{Str-full} in both cases, then the solutions to the associated forward inhomogeneous problem with source term $f = \partial_t f_1 + f_2$ satisfy the bounds \eqref{Str-full+}.
\end{theorem}

\begin{proof}
The proof consists of a number of steps:

Step 1: If the  homogeneous problem is well-posed forward in $\H^r$ and satisfies the homogeneous Strichartz estimates \eqref{Str-hom-syst}, then so does the corresponding system, see Proposition~\ref{p:v-vs-bv-hom}.

Step 2: If the  homogeneous problem is well-posed backward in $\H^{1-r}$ and satisfies the homogeneous Strichartz estimates, then so does the corresponding system. By duality, 
the inhomogeneous system is well-posed forward in $H^r$ and satisfies the dual Strichartz bounds
\eqref{Str-dual-syst}.

Step 3: We represent the forward $\H^r$ solution by the Duhamel formula
\[
\bv[t] = S(t,0) \bv[0] + \int_{0}^t S(t,s) \ff(s) \, ds.
\]
The first term represents the solution to the homogeneous equation, and is estimated by 
\eqref{Str-hom-syst}. For the second term we have two bounds at our disposal: 
the dual bound where we fix $t$ and estimate the output in $\H^s$ in terms of the input 
in the dual Strichartz space, and the homogeneous bound where we fix $s$, set $\ff(s) \in \H^r$
and estimate the output as a function of $t$ in the Strichartz space. Concatenating the two,
we get the restricted bound 
\begin{equation}\label{Str-full-syst-part}
\| \bv_1 \|_{S^r(J)} + \|\bv_2\|_{S^{r-1}(J)} \lesssim  \| \ff_1 \|_{(S^{-r})'(I)}
+ \|\ff_2 \|_{(S^{-r+1})'(I)} ,  \qquad S \in \{ S_{ST},S_{AIT}\},   \end{equation}
where the source $\ff$ is supported in an interval $I$ and the output $\bv$ is measured in an interval $J$ so that $I$ precedes $J$.  In two dimensions we can now apply the Christ-Kiselev lemma
\cite{CK} (or the $U^p$-$V^p$ spaces, see \cite{UV}) to get the full estimate.  In three and higher dimensions we have a slight problem which is that neither method applies for bounds from $L^2_t$
to $L^2_t$. However in our case this is not an issue, because our estimates allow for at least a loss of $\delta$ derivatives. Then we can afford to interpolate the two endpoints and use 
the Christ-Kiselev lemma for bounds from $L^{2-}_t$ to $L^{2+}_t$ and then return to the endpoint setting by Bernstein's inequality in space and Holder's inequality in time, all at the expense of
an arbitrarily small increase in the size $\delta$ of the loss.

Step 4. We transfer the estimate \eqref{Str-full-syst} back to the original system 
via the correspondence \eqref{f-vs-ff}, \eqref{v-vs-bv}, in order to obtain \eqref{Str-full+}.
\end{proof}

We conclude with a corollary of Theorem~\ref{t:Str-move-around}, which will be used later in the paper and follows by combining this result with Proposition~\ref{p:Str-allr}:

\begin{corollary}
Assume that $\partial g \in L^1L^\infty$ and that the Strichartz estimates for the homogeneous equation \eqref{box-g-nodiv} hold in $\H^1$. Then the full Strichartz estimates 
\eqref{Str-full+} hold in $\H^r$ for all $r \in \R$ for both paradifferential flows \eqref{Tbox-g-div}
and \eqref{Tbox-g-nodiv}. 
\end{corollary}

 \section{Control parameters and related bounds}
\label{s:control}
 
 %%%%%%%%%%%%%%%%%%%%%%%%%%%%%%%%%%%%%%%%%%%
%%%%%%%%%%%%%%%%%%%%%%%%%%%%%%%%%%%%%%%%%%%
%%%%%%%%%%%%%%%%%%%%%%%%%%%%%%%%%%%%%%%%%%%
\subsection{ Control parameters } 
  Here we introduce our main control parameters associated to a solution $u$
  to the minimal surface equation, which serve to bound the growth of energy
 for both solutions to the minimal surface flow and for its linearization. 
 We will use two such primary quantities, $\AA$ and $\BB$, which are defined as $L^\infty$ based Besov norms of the solution $u$, as follows:
  \begin{equation}\label{def-AA}
 \AA = \sup_{t \in [0,T]} \sum_k \| P_k \partial u\|_{L^\infty},
 \end{equation}
respectively
 \begin{equation} \label{def-BB}
\BB(t) = \left( \sum_k 2^k \| P_k \partial u\|_{L^\infty}^2 \right)^\frac12      .
 \end{equation}
In connection with $\AA$, we will also need the slightly stronger variant
$\AAs \gtrsim \AA$,
\begin{equation}\label{def-AAs}
 \AAs = \sup_{t \in [0,T]} \sum_k 2^{\frac{k}2} \| P_k \partial u\|_{L^{2n}}, \qquad 2 < p < \infty.
\end{equation}
Here the choice of the exponent $2n$ is in no way essential, though 
it does provide some minor simplifications in one or two places.

In a nutshell, the energy functionals we construct later in the paper will be shown to satisfy \emph{cubic balanced}
bounds of the form 
\begin{equation}\label{ee-model}
\frac{dE}{dt} \lesssim_{\AAs} \BB^2 E,
\end{equation}
which guarantee that energy bounds can be propagated for as long as $\AAs$ remains finite
and $\BB$ remains in $L^2_t$. One should compare these bounds with the 
classical energy estimates, which have the form
\begin{equation}\label{ee-classical}
\frac{dE}{dt} \lesssim_{\AA} \| \partial^2 u\|_{L^\infty} E,
\end{equation}
and which require an extra half derivative in the control parameter.
\bigskip

We continue with a few comments concerning  our choice of control parameters:

\begin{itemize}
    \item Here $\AA$ and $\AAs$ are critical norms for $u$, which may be 
    described using the Besov notation as capturing the norm 
    \[
    \AA = \| \partial u\|_{L^\infty_t B^{0}_{\infty,1}}, \qquad \AAs = \| \partial u\|_{L^\infty_t B^\frac{1}{2}_{2n,1}}.
    \]
    In a first approximation, the reader should think of $\AA$ as simply capturing the 
    $L^\infty$ norm of $\partial u$; the slightly stronger Besov norm above is needed
    for minor technical reasons, and allows us to work with scale invariant bounds. Often we will simply rely on the simpler $L^\infty$-bound, since
      \begin{equation}
 \| \partial u\|_{L^\infty} \lesssim \AA.
    \end{equation}
    
    \item The control norm $\BB$, taken at fixed time, is $1/2$ derivative above scaling,
    and may also be described using the Besov notation as 
    \[
    \BB(t) = \| \partial u(t)\|_{B^\frac12_{\infty,2}}.
    \]
    Again, in a first approximation one should simply think of it as $\|\partial u\|_{BMO^\frac12}$,
    which in effect suffices for most of the analysis. Indeed, we have
     \begin{equation}
 \| \partial u\|_{BMO^{\frac12}} \lesssim \BB.
    \end{equation}

\item Given the choice of these control parameters, it is not difficult to see 
that our energy estimates of the form \eqref{ee-model} are invariant with respect to scaling. 
This by itself does not mean much; even the classical energy estimates, of the form 
\eqref{ee-classical}, are scale invariant, but much less useful for low regularity 
well-posedness. What is important here is that our energy estimates are \emph{cubic} and 
\emph{balanced}.

\item The fact that our control norms are based on uniform, rather than $L^2$-bounds, particularly at the level of $\BB$, is also critical.
This is what allows us to use Strichartz estimates to further improve the low regularity  well-posedness threshold in our results.
\end{itemize}

For bookkeeping reasons we will use a joint frequency envelope $\{c_k\}_k$ for the dyadic components of both $\AAs$ and $\BB$, so that 

(i) $\{c_k\}_k$ is normalized in $\ell^2$ and slowly varying,
\[
\sum c_k^2 = 1;
\]

(ii) We have control of dyadic Littlewood-Paley pieces as follows for $\partial u$:
\begin{equation}\label{fe-control}
 2^{-\frac{k}2} \| P_k \partial u\|_{L^{2n}} \lesssim \AAs c_k^2, \qquad 2^{\frac{k}2} \| P_k  \partial u\|_{L^\infty} \lesssim \BB c_k. 
\end{equation}
A-priori, these frequency envelopes depend on time.
However, at the conclusion of the paper, we will see that for our rough solutions they can be taken to be independent of time, essentially equal to appropriate $L^2$-type frequency envelopes
for the initial data.

\subsection{Related bounds}
We will frequently need to use  bounds which are similar to  \eqref{fe-control} in nonlinear expressions, so it is convenient to have a notation for the corresponding space:

\begin{definition}\label{d:CC}
The space $\CC_0$ is the Banach space of all distributions $v$ which satisfy the bounds
\begin{equation}\label{fe-control-C}
 \|v\|_{L^\infty} \leq C, \qquad 
 2^{\frac{k}2} \| P_k v\|_{L^{2n}} \leq C \AAs c_k^2, \qquad 2^{\frac{k}2} \| P_k v\|_{L^\infty} \leq C \BB c_k
\end{equation}
with the norm given by the best constant $C$
in the above inequalities.
\end{definition}

For this space we have the following
algebra and Moser-type result:

\begin{lemma}\label{l:Moser-control0}
a) The space $\CC_0$ is closed with respect to multiplication and para-multiplication. In particular $\CC_0$ is an algebra.

b) Let $F$ be a smooth function, and $v \in \CC_0$. Then $F(v) \in \CC_0$. In particular if $\|v\|_{\CC_0} \lesssim 1$ then $F(v)$
satisfies
\begin{equation}\label{fe-control-v0}
\| F(v) \|_{\CC_0} \lesssim_{\AA} \|v\|_{\CC_0}.
\end{equation}

\end{lemma}
In particular the above result applies to the metrics $g$, $\tg$ and $\hg$,
all of which are smooth functions of $\partial u$, and thus belong to $\CC_0$. 
\begin{proof}
a) We first estimate the $\CC_0$ norm for the paraproduct $T_f g$ for $f,g \in \CC_0$.
This is straightforward, using the $L^\infty$ bound for $f$, for both the second and the third norms in \eqref{fe-control-C}. It remains to obtain a pointwise bound,
for which we change the summation order in the Littlewood-Paley expansion to obtain
\[
\| T_f g\|_{L^\infty} \lesssim \sum_k \|P_k f \|_{L^\infty} \| P_{> k} g \|_{L^\infty}
\lesssim \| f\|_{\CC_0} \|g\|_{L^\infty}.
\]

It now remains to estimate $\Pi(f,g)$ in $\CC_0$. The uniform bound is almost identical to the one above. For the $\AAs$ norm we use Bernstein's inequality
\[
2^{\frac{k}2} \| P_k \Pi(f,g) \|_{L^{2n}} 
\lesssim \sum_{j \geq k} 2^{k} \|  f_j g_j \|_{L^n}
\lesssim \sum_{j \geq k} 2^{k} \|  f_j \|_{L^p} \|g_j \|_{L^p}
\lesssim  \sum_{j \geq k} 2^{k-j} c_j^2 \AAs^2 \|f\|_{\CC_0} \|g\|_{\CC_0} ,
\]
and now the $j$ summation is straightforward.

For the $\BB$ norm, on the other hand, we estimate
\[
 \| P_k \Pi(f,g) \|_{L^\infty} 
\lesssim \sum_{j \geq k} \|  f_j g_j \|_{L^\infty}
\lesssim  \sum_{j \geq k} \|  f_j \|_{L^\infty} \|g_j \|_{L^\infty}
\lesssim  \sum_{j \geq k} 2^{\frac{j}2} c_j \AAs \BB \|f\|_{\CC_0} \|g\|_{\CC_0} ,
\]
and again the $j$ summation is straightforward.

\medskip

b) To prove the Moser inequality we use a continuous Littlewood-Paley decomposition,
which leads to the expansion
\[
F(v) = F(v_0) + \int_0^\infty F'(v_{<j}) v_j \, dj. 
\]
To estimate $P_k F(v)$ we consider several cases:

\medskip

i) $j = k+ O(1)$. Then $c_j \approx c_k$, $F'(v_{<j})$ is directly bounded in $L^\infty$ and our bounds are straightforward.

\medskip

ii) $j < k-4$. Then we can insert an additional localization,
\[
 P_k(F'(v_{<j}) v_j) = P_k(\tilde P_k F'(v_{<j}) v_j),
\]
where we gain from the frequency difference
\begin{equation}\label{Fprime-inf}
\|  P_k F'(v_{<j}) \|_{L^\infty} \lesssim 2^{-N(k-j)},
\end{equation}
which more than compensates for the difference (ratio) between $c_j$ and $c_k$.

\medskip

iii) $j > k+4$. In this case we reexpand $F'(v_{<j})$ and write
\[
 F'(v_{<j}) v_j = F'(v_{0}) v_j + \int_0^\infty F''(v_{<l}) v_l v_j \, dl .
\]
We further separate into two cases:

\medskip

(iii.1) $l = j + O(1)$. Then we simply bound $F''(v_{<l})$ in $L^\infty$,
and estimate first for the $\AAs$ bound using Bernstein's inequality
\[
2^{\frac{k}2} \| P_k (F''(v_{<l}) v_l v_j) \|_{L^{2n}} 
\lesssim 2^{k} \|  v_l v_j \|_{L^n}
\lesssim 2^{k} \|  v_l \|_{L^{2n}} \|v_j \|_{L^{2n}}
\lesssim  2^{k-j} c_j^2 \AAs^2,
\]
where the $j$ and $l$ integrations are trivial. Next we estimate for the $\BB$-bound
\[
\| P_k (F''(v_{<l}) v_l v_j) \|_{L^\infty} \lesssim \|v_l\|_{L^\infty} \| v_j\|_{L^\infty} \lesssim 2^{-\frac{j}2} \AA \BB c_j,
\]
again with easy $j$ and $l$ integrations.

\medskip

(iii.2) $l < j-4$. Then we can insert another frequency localization,
\[
 P_k (F''(v_{<l}) v_l v_j) = P_k (\tilde P_j F''(v_{<l}) v_l v_j),
\]
and repeat the computation in (b.ii)  but using \eqref{Fprime-inf} to account for the 
difference between $l$ and $j$.

\medskip
\end{proof}

 In order to avoid tampering with causality, the Littlewood-Paley projections we use in this paper
are purely spatial. This is more of a choice between different evils than a necessity; see for instance the alternate  choice made in \cite{ST}. A substantial
but worthwhile price to pay is that on occasion we will need 
to separately estimate double time derivatives, in a somewhat imperfect but sufficient fashion.

A good starting point in this direction is to think of bounds for second derivatives of our solution $u$. If at least one of the derivatives 
is spatial, then this is straightforward:
\begin{equation}\label{dxdu}
 \| P_{<k} \partial_x \partial u \|_{L^\infty} \lesssim  2^{\frac{k}{2}} \BB c_k.
 \end{equation} 
However, matters become more complex if instead we look at the second time derivative of $u$. The natural idea is to use the main equation \eqref{msf-short}
to estimate $\partial_t^2 u$, by writing
of spatial derivatives, 
\[
\partial_t^2 u = - \sum_{(\alpha,\beta) \neq (0,0)} \tg^{\alpha\beta} \partial_\alpha \partial_\beta u.
\]
If one takes this view, the main difficulties we face are 
with the high-high interactions in this expression. But  these high-high interactions have the redeeming feature that they are balanced, so they will often play a perturbative role. This leads us to define a corrected expression as follows:

\begin{definition} \label{d:dt2u}
 We denote
 \begin{equation}\label{hat-dt2}
\hat \partial_t^2 u = \partial_t^2 u +  \sum_{(\alpha,\beta) \neq (0,0)} \Pi(\tg^{\alpha\beta}, \partial_\alpha \partial_\beta u).
\end{equation}
More generally, by 
$\widehat{\partial_\alpha \partial_\beta} u$
we denote ${\partial_\alpha \partial_\beta} u$
if $(\alpha,\beta) \neq (0,0)$.
\end{definition}
 With this notation, we have

\begin{lemma}\label{l:utt}
Assume that $u$ solves the equation \eqref{msf-short}. Then for its second time derivative we have the decomposition 
\begin{equation}
 \partial_t^2 u =  \hat \partial_t^2 u + \pi_2(u),
\end{equation}
where the two components satisfy the uniform bounds
\begin{equation}\label{fe-dt2u}
\| P_{<k}  \hat \partial_t^2 u \|_{L^\infty} \lesssim  2^{\frac{k}{2}} \BB c_k,  \qquad \| P_{<k}  \hat \partial_t^2 u \|_{L^\infty} \lesssim  2^{k} \AA c_k^2,\end{equation}
respectively 
\begin{equation}\label{fe-dt2u-err}
\| \pi_2(u) \|_{L^\infty} \lesssim_{\AA} \BB^2, \qquad    \| \pi_2(u) \|_{L^n} \lesssim \AAs^2.    
\end{equation}
\end{lemma}
 One should compare this with the easier direct bound \eqref{dxdu} for spatial derivatives; the good part $\hat \D_t^2 u$ satisfies a similar bound,
 but the error $\pi_2(u)$ does not. Later, when such expressions are involved,
 we will systematically peel off perturbatively the error, and always avoid differentiating it further.
 
 \begin{proof}
The main ingredient here is the Littlewood-Paley decomposition. For expository simplicity we prove \eqref{fe-dt2u} at fixed frequency $k$. Using the notation in \eqref{hat-dt2} we can rewrite equation \eqref{ms-tg} as
\begin{equation}
\label{wave decomp}
\hat \partial_t^2 u + \sum_{(\alpha,\beta) \neq (0,0)} T_{\tg^{\alpha \beta}}\partial_{\alpha}\partial_{\beta}u+T_{\partial_{\alpha}\partial_{\beta}u}\tg^{\alpha \beta} =0.
\end{equation}
To finish the proof we consider the expression above localized at frequency $k$, and evaluated in the $L^{\infty}$-norm
\begin{equation}
\label{decomp}
\begin{aligned}
\Vert P_{k}\hat \partial_t^2 u \Vert_{L^{\infty}} &\leq
\Vert P_{<k} (\tg^{\alpha \beta}) P_k( \partial_{\alpha}\partial_{\beta}u )\Vert_{L^{\infty}}+ \Vert P_{<k}(\partial_{\alpha}\partial_{\beta}u) P_k(\tg^{\alpha \beta})\Vert_{L^{\infty}}.\\
\end{aligned}
\end{equation}
We bound each of the terms separately. For the second we use that $\tg$ is bounded in $L^\infty$, and   get 
\[
\Vert P_{<k}(\partial_{\alpha}\partial_{\beta}u) P_k(\tg^{\alpha \beta})\Vert_{L^{\infty}} \leq 2^{\frac{k}{2}}\BB c_k.
\]
For the first term we rely on the same procedure, and hence finish the proof of the first bound in \eqref{fe-dt2u}. The second bound in \eqref{fe-dt2u} has as a starting point the decomposition in \eqref{decomp}, only that this time we want to bound the RHS terms using the control norm $\AA$. Here we use Lemma~\ref{l:Moser-control0}, part b), and the algebra property of $L^{\infty}$ so that
\[
\Vert P_{<k}(\partial_{\alpha}\partial_{\beta}u) P_k(\tg^{\alpha \beta})\Vert_{L^{\infty}} \leq \Vert  P_{<k}(\partial_{\alpha}\partial_{\beta}u) \Vert_{L^{\infty}} \Vert P_k(\tg^{\alpha \beta})\Vert_{L^{\infty}}\leq 2^k\AA c_k^2,
\]
where for both   factors we used \eqref{fe-control-v0} in order to arrive at the result.  

The last bound to prove is \eqref{fe-dt2u-err}, where  because of the balanced frequencies we can easily even out the derivatives balance and  estimate each of the factors using the $\BB$ norm. Explicitly, $\tg^{\alpha\beta}$ is in $\CC_0 $ by Lemma~\ref{l:Moser-control0} and hence, we get that for $(\alpha, \beta)\neq (0, 0)$:
\[
\begin{aligned}
\Vert  \Pi(\tg^{\alpha\beta}, \partial_\alpha \partial_\beta u)\Vert_{L^{\infty}} &\leq \sum_k  \Vert  P_k(\tg^{\alpha\beta}) P_k(\partial_\alpha \partial_\beta u) \Vert_{L^{\infty}}\\
&\leq \sum_k 2^{-\frac{k}{2}}\Vert 2^{\frac{k}{2}}P_k(\tg^{\alpha\beta})\Vert_{L^{\infty}}2^{\frac{k}{2}}\Vert 2^{-\frac{k}{2}}P_k(\partial_\alpha \partial_\beta u)\Vert_{L^{\infty}}\\
&\lesssim_{\AA} \BB^2.
\end{aligned}
\]
which is the first bound in \eqref{fe-dt2u-err}. The second bound in \eqref{fe-dt2u-err} is similar, but replacing the $L^\infty$ norms with  $L^{2n}$
norms.
\end{proof}

The above lemma motivates narrowing the space $\CC_0$, in order to also include information 
about $\partial_t v$. For later use, we also 
define two additional closely related spaces.

\begin{definition}
a) The space $\CC$ is the space of distributions $v$ which satisfy \eqref{fe-control-C} and, in addition, $\partial_t v$ admits a decomposition
$\partial_t v = w_1+w_2$ so that 
\begin{equation}\label{fe-dt2u-defCC}
\| P_{k}  w_1 \|_{L^\infty} \leq C  2^{\frac{k}{2}} \BB c_k,  \qquad 
\| w_2 \|_{L^\infty} \leq  C\BB^2,    
\end{equation}
endowed with the norm defined as the best possible constant $C$ in  \eqref{fe-control-C} and in the above inequality relative to all such possible decompositions.

b) The space $\DCC$ consists of all functions $f$ which admit a decomposition 
$f = f_1+f_2$ so that 
\begin{equation}\label{DC-decomp}
\| P_k f_1\|_{L^\infty} \leq C 2^{\frac{k}2} \BB c_k, \qquad  \| f_2\|_{L^\infty} \leq C \BB^2,    
\end{equation}
endowed with the norm defined as the best possible constant $C$ in the above inequality relative to all such possible decompositions.

c) The space $\partial_x \DCC$ consists of functions $f$ which admit a decomposition 
$f = f_1+f_2$ so that 
\begin{equation}
\label{dxdc-norm}
\| P_k f_1\|_{L^\infty} \leq C 2^{\frac{3k}2} \BB c_k,   \qquad \| P_k f_2\|_{L^\infty} \leq C 2^k \BB^2,
\end{equation}
endowed also  with the corresponding  norm. 
\end{definition}
 
 We remark that, by definition, 
 we have the simple inclusions
 \begin{equation}
 \label{inclus}
\CC\subset \CC_0, \qquad \partial:\CC \to \DCC,
\qquad \partial_x: \DCC \to \partial_x \DCC.
 \end{equation}
 Based on what we have so far,  we begin by identifying some 
elements of these spaces:

\begin{lemma}\label{l:du-in-dcc}
We have 
\begin{equation}\label{fe-utt}
\|\partial u\|_{\CC} \lesssim 1, \qquad \| \partial^2 u \|_{\DCC} \lesssim 1. 
\end{equation}
\end{lemma}

\begin{proof}
The  bounds in \eqref{fe-utt} are trivial unless both derivatives are time derivatives, in which case it follows directly from the previous Lemma~\ref{l:utt}. 
\end{proof}

The Moser estimates of Lemma~\ref{l:Moser-control0} may be extended to this setting to include all smooth functions of $\partial u$:

\begin{lemma}\label{l:Moser-control}
a) We have the bilinear multiplicative relations
\begin{equation}\label{C-DC}
\CC_0 \cdot \DCC \to \DCC, \qquad     T_{\CC_0} \cdot \DCC \to \DCC, 
\qquad  T_{\DCC} \CC_0 \to \AA \DCC,
\end{equation}
as well as 
\begin{equation}\label{TDCC-CC}
T_{\DCC} \CC_0 \to \BB^2 L^\infty, \qquad \Pi(\DCC,\CC_0) \to \BB^2 L^\infty.
\end{equation}

b) The space $\CC$ is closed under multiplication and para-multiplication; in particular it is an algebra. 

c) Let $F$ be a smooth function, and $v \in \CC$. Then $F(v) \in \CC$. In particular if $\|v\|_{\CC} \lesssim 1$ then $F(v)$
satisfies
\begin{equation}\label{fe-control-v}
\| F(v) \|_{\CC} \lesssim_{\AA} \|v\|_{\CC}.
\end{equation}

d) In addition we also have the paralinearization error bound
\begin{equation}
\| \partial R(v)\|_{L^\infty}    \lesssim \BB^2, \qquad R(v) = F(v) - T_{F'(v)} v. \end{equation}
\end{lemma}

Here part (a) is the main part, after which parts (b) and (c) become immediate improvements of Lemma~\ref{l:Moser-control0}. But the new interesting bound is the one in part (d),
where, notably, we also bound the time derivative of $R(v)$.

\begin{proof}
a) Let $z \in \CC_0$ and $w \in \DCC$ with the decomposition $w=w_1+w_2$ as in \eqref{DC-decomp}. We skip the first bound in \eqref{C-DC}, as it is a consequence of the rest
of the estimates in \eqref{C-DC} and \eqref{TDCC-CC}, and first consider the paraproduct $T_z w$.
We will bound the contributions of $w_1$ and $w_2$ in the 
same norms as $w_1$, respectively $w_2$. Precisely, 
we have 
\[
\| P_{k} T_z w_1 \|_{L^\infty} \lesssim \| z\|_{L^\infty} \| \tilde P_k w_1\|_{L^\infty} \lesssim  \| z\|_{\CC_0} 2^{\frac{k}2} \BB c_k \| w \|_{\DCC} ,
\]
respectively 
\[
\| T_z w_2 \|_{L^\infty} \lesssim \sum_k \|z_k\|_{L^\infty} \| P_{>k} w_2\|_{L^\infty}
\lesssim \| z\|_{\CC_0} \|w_1\|_{L^\infty}.
\]

Next we consider $T_w z$, where we have two choices. The first choice is to use 
only the $\AA$ component of the $\CC_0$ norm of $z$, and prove the last bound in \eqref{C-DC}.
Precisely, we have 
\[
\| P_{k} T_{w_1} z \|_{L^\infty} \lesssim \| w_{1,<k}\|_{L^\infty} \| \tilde P_k z\|_{L^\infty} \lesssim   2^{\frac{k}2} \BB c_k \| w \|_{\DCC} \cdot \AA \|z\|_{\CC_0}, 
\]
respectively
\[
\| T_{w_2} z \|_{L^\infty} \lesssim \sum_k \| w_{2,<k}\|_{L^\infty} \| P_{k} z\|_{L^\infty}
\lesssim \| w_2\|_{L^\infty} \cdot \AA \|z\|_{\CC_0}.
\]

Alternatively, we can use the $\BB$ component of the $\CC_0$ norm of $z$ in the bound for the $w_1$ component,
\[
\| T_{w_1} z \|_{L^\infty} \lesssim \sum_k \| w_{1,<k}\|_{L^\infty} \| z_k\|_{L^\infty} \lesssim  \sum_k 2^{\frac{k}2} \BB c_k \| w \|_{\DCC} \cdot  2^{-\frac{k}2} \BB c_k\|z\|_{\CC_0} \lesssim \BB^2  \| w \|_{\DCC}\|z\|_{\CC_0},
\]
which leads to the first bound in \eqref{TDCC-CC}. 

It remains to consider the second bound in \eqref{TDCC-CC}, where we have
\[
\| \Pi(w_1,z) \|_{L^\infty} \lesssim \sum_k \| w_{1,k}\|_{L^\infty} \|  z_k \|_{L^\infty} \lesssim  \sum_k 2^{\frac{k}2} \BB c_k \| w \|_{\DCC} \cdot  2^{-\frac{k}2} \BB c_k\|z\|_{\CC_0} \lesssim \BB^2  \| w \|_{\DCC}\|z\|_{\CC_0},
\]
respectively 
\[
\| \Pi(w_2,z) \|_{L^\infty} \lesssim \sum_k \| w_{2,k}\|_{L^\infty} \| z_k\|_{L^\infty}
\lesssim \| w_2\|_{L^\infty} \cdot \AA \|z\|_{\CC_0} \lesssim A \BB^2\| w\|_{\DCC} \cdot \AA \|z\|_{\CC_0}.
\]

\bigskip

b) Compared with part (a) of Lemma~\eqref{l:Moser-control0}, it remains to estimate
the time derivative of products and paraproducts. Using Leibniz's rule, 
this reduces directly to the multiplicative bounds in (a).
\bigskip

c) Compared with part (b) of Lemma~\eqref{l:Moser-control0} it remains to estimate
\[
\partial_0 F(v) = F'(v) \partial_0 v 
\]
in $\DCC$.
By Lemma~\eqref{l:Moser-control0} we have $F'(v) \in \CC_0$, while $\partial_0 v \in \DCC$. 
Then we can bound the product in $\DCC$ by part (a) of this Lemma.

\bigskip

d) We have 
\[
\partial R = (F'(v)-T_{F'(v)}) \partial v - T_{\partial F'(v)} v
= \Pi(F'(v), \partial v) + T_{\partial v} F'(v) - T_{\partial F'(v)} v.
\]
Now it suffices to use \eqref{fe-control-v} for both $v$ and $F'(v)$. 
\end{proof}

Applying the above lemma shows  that $F(\partial u) \in \CC$,
and in particular all components of the metrics $g$, $\tg$ and $\hg$ are in $\CC$.
We also have $F(\partial u) \partial^2 u \in \DCC$, which in particular shows that the gradient potentials $A$ and $\tA$ belong to $\DCC$.

 We will use part (d) when $w=\partial u$ and $F=g$, in which case
\eqref{est-R} reads
\begin{equation}\label{est-R-for-du-inf}
\|\partial R(\partial u)\|_{L^\infty} \lesssim_\AA \BB^2,
\qquad R = g^{\alpha \beta} + T_{\partial^{\alpha} u g^{\beta \gamma}}
\partial_\gamma u + T_{\partial^{\beta} u g^{\alpha \gamma}}
\partial_\gamma u .
\end{equation}
We remark that a similar $H^s$ type bound for the same $R$ is provided 
by \eqref{est-R}, namely
\begin{equation}\label{est-R-for-du}
\|R(\partial u)\|_{H^{s-\frac12}} \lesssim_\AA \BB \| \partial u\|_{H^{s-1}}.
\end{equation}

The next lemma provides us with the 
primary example of elements of the space $\partial_x \DCC$:

\begin{lemma}\label{l:d3u-in-d2cc}
We have 
 \begin{equation}\label{fe-uttt}
 \|  \partial_\alpha \widehat{\partial_\beta \partial_\gamma} u \|_{\partial_x \DCC} \lesssim 1.
 \end{equation}
\end{lemma}

\begin{proof}
The  bound in \eqref{fe-uttt} is trivial if at least two derivatives are spatial, and follows from \eqref{fe-utt}  unless all indices are zero. It remains to consider the case
$\alpha = \beta = \gamma = 0$. Here we rely on the earlier decomposition  \eqref{wave decomp}  to which we further apply a $\partial_t$:
\begin{equation}
\label{dhatdd}
\partial_t \hat{\partial}_t^2u  =- \sum_{(\alpha, \beta)\neq (0, 0)}\left(  T_{\partial_t \tilde{g}^{\alpha \beta}} \partial_{\alpha}\partial_{\beta}u + T_{ \tilde{g}^{\alpha \beta}} \partial_t\partial_{\alpha}\partial_{\beta}u+ T_{\partial_t \partial_{\alpha \beta}u} \tg^{\alpha \beta}+T_{ \partial_{\alpha}\partial_{ \beta} u} \partial_t\tg^{\alpha \beta}\right) .
\end{equation}
We now investigate each term separately. We begin with the first term, which needs to be bounded in the $\partial_x \DCC$ norm given in \eqref{dxdc-norm}. We have
\[
\Vert P_k T_{\partial_t \tilde{g}^{\alpha \beta}} \partial_{\alpha}\partial_{\beta}u \Vert_{L^{\infty}}  \leq \Vert P_{<k} (\partial_t \tilde{g}^{\alpha \beta})\Vert_{L^{\infty}} \Vert P_k (\partial_{\alpha}\partial_{\beta}u) \Vert_{L^{\infty}}.
\]
The term that contains the time derivative falling onto the metric will be bounded using the Moser estimate Lemma~\ref{l:Moser-control}. Explicitly, we have $\tg^{\alpha \beta}\in \CC$, and  due to Lemma~\ref{l:Moser-control}, part $c)$, we get $\partial_t\tg^{\alpha \beta}\in \DCC $ which allows us to decompose it as 
in \eqref{DC-decomp}, $\partial_t\tg^{\alpha \beta}=\tg_1^{\alpha \beta} + \tg_2^{\alpha \beta}$ where
\[
\Vert P_{<k}\partial_t \tg^{\alpha \beta}\Vert_{L^{\infty}} \leq \Vert  \tg_1^{\alpha \beta}\Vert_{L^{\infty}} +\Vert P_{<k} \tg_2^{\alpha \beta}\Vert_{L^{\infty}} \leq C\BB^2 +C2^{\frac{k}{2}}\BB c_k.
\]
We now turn to the last bound which we can estimate in two ways
\[
  \Vert P_k (\partial_{\alpha}\partial_{\beta}u) \Vert_{L^{\infty}} \leq 2^{\frac{k}{2}} \Vert  2^{-\frac{k}{2}}P_k(\partial_\alpha \partial_\beta u)\Vert_{L^{\infty}}\leq 2^{\frac{k}{2}}\BB c_k,
\]
or
\[
  \Vert P_k (\partial_{\alpha}\partial_{\beta}u) \Vert_{L^{\infty}} \leq 2^{k }\Vert  2^{-k}P_k(\partial_\alpha \partial_\beta u)\Vert_{L^{\infty}}\leq 2^{k}\AA  c_k^2.
\]

Putting together the bounds we have, leads to
\[
\Vert P_{k}T_{\partial_t \tg^{\alpha \beta}} \partial_{\alpha}\partial_{\beta}u\Vert_{L^{\infty}}\leq C2^k\AA \BB^2 c_k^2 + C2^k \BB^2 c_k^2.
\]

We now bound the second term in \eqref{dhatdd}
\[
\Vert T_{\tg^{\alpha \beta}}\partial_t \partial_{\alpha}\partial_{\beta} u\Vert_{L^{\infty}} \leq \Vert  \tg^{\alpha \beta}\Vert_{L^{\infty}}\Vert  \partial_t\partial_{\alpha}\partial_{\beta} u\Vert_{L^{\infty}}.
\]
For the last term in the above estimate we know that $(\alpha, \beta) \neq (0,0)$ hence there are two cases to consider: (i) we have either   $\alpha =0$ or $\beta=0$, but not both zero, which overalls means we need to bound $\partial_t ^2 \partial_x u$, or (ii) we have  both $\alpha, \beta \neq 0$, in which case we need a pointwise bound for $\partial_t\partial^2_x  u$. However, both cases can be handled  in the same way if we observe that $\partial_x (\partial_x\partial_t u)$ and $\partial_x (\partial^2_t u)$ are elements in $\partial_x\DCC$; this is a direct consequence of $\partial^2 u \in \DCC $ as shown in \eqref{fe-utt}, followed by the inclusion in \eqref{inclus}. 

Finally, the third and fourth terms in \eqref{dhatdd} can be treated in the same way the first term in \eqref{dhatdd} was shown to be  bounded.
\end{proof}
 
 We continue with another, slightly more subtle balanced bound:
 
 \begin{lemma}\label{l:para-p+}
 For $g, h \in \CC$ define
 \[
 r =( T_g T_h - T_{gh} ) \partial^2 u.
 \]
 Then we have the balanced bound
\begin{equation}\label{para-p-r}
\| r\|_{L^\infty} \lesssim \BB^2. 
\end{equation}
\end{lemma}

\begin{proof}
For $\partial^2 u$ we use the $\DCC$ decomposition as in \eqref{DC-decomp},
\[
 \partial^2 u = f_1 + f_2.
\]
We begin with the contribution $r_1$ of $f_1$, which we expand as
\[
r_1 = \sum_{k} ( T_g T_h - T_{gh} ) f_{1,k}.
\]
This vanishes unless the frequencies $k_1$, $k_2$ of $g$ and $h$ are either 

(i) $k_1, k_2 \leq k$ and $\max\{k_1,k_2\} = k+O(1)$, or

(ii) $k_1 = k_2 > k + O(1) $.

Then we use the $\AA$ component of the $\CC_0$ norm for the lower
frequency  and and the $\BB$ component for the lower frequency
to estimate
\[
\begin{aligned}
\|r_1\|_{L^\infty} \lesssim & \ \sum_k \| f_{1,k}\|_{L^\infty}
\left( \sum_{k_1 < k+O(1)} \| g_{k_1}\|_{L^\infty}    \sum_{k_2=k + O(1)}
 \| h_{k_2}\|_{L^\infty} \right.\\
&\,   +
\left. \sum_{k_1 = k+O(1)} \| g_{k_1}\|_{L^\infty}    \sum_{k_2< k + O(1)}
 \| h_{k_2}\|_{L^\infty} +
\sum_{k_1 = k_2 \geq k+O(1)} \| g_{k_1}\|_{L^\infty}
 \| h_{k_2}\|_{L^\infty}  \right)
\\
\lesssim & \ 2^{\frac{k}2} \BB c_k ( \AA \cdot 2^{-\frac{k}2} \BB c_k
+ 2^{-\frac{k}2} \BB c_k \cdot \AA + \AA \cdot 2^{-k} \BB c_k) \\
\lesssim  & \ \AA \BB^2,
\end{aligned}
\]
as needed.
\end{proof}
 
 As already discussed in the introduction, the 
 paradifferential wave operator
 \begin{equation}
T_{P} =  \partial_\alpha T_{g^{\alpha \beta}} \partial_\beta,  \end{equation}
 as well as its counterparts $T_{\tP}$ and $T_{\hP}$ with the metric $g$ replaced by  $\tg$, respectively $\hg$, play an important role in our context.

 Throughout the paper, we will interpret various objects related to $u$ as approximate solutions for the $T_{P}$ equation.
 We provide several results of this type, where we use our control parameters $\AA,\BB$ in order to estimate the source term in the 
 paradifferential equation for both $u$ and for its 
 derivatives.
 
 \begin{lemma}\label{l:tp-u}
 We have 
 \begin{equation}
 \label{5.33}
\| T_{P} u \|_{L^\infty} \lesssim \BB^2,
 \end{equation}
 as well as the similar bounds for $T_{\tP}$ and $T_{\hP}$.
 \end{lemma}
 \begin{proof}
We first  prove the bound \eqref{5.33}, and  for this we begin with the paradifferential equation associated  to the minimal surface equation  \eqref{msf-short}
 \[
 T_{g^{\alpha \beta}}\partial_{\alpha}\partial_{\beta}u +T_{\partial_{\alpha}\partial_{\beta}u}g^{\alpha \beta}+\Pi (g^{\alpha,\beta}, \partial_{\alpha}\partial_{\beta}u)=0,
 \]
 and further isolate the part we are interested in estimating
 \[
 \partial_{\beta}T_{g^{\alpha \beta}}\partial_{\alpha}u - T_{\partial_{\beta} g^{\alpha \beta}}\partial_{\alpha}u +T_{\partial_{\alpha}\partial_{\beta}u}g^{\alpha \beta}+\Pi (g^{\alpha\beta}, \partial_{\alpha}\partial_{\beta}u)=0.
 \]
 The bound we want relies on getting bounds for the following terms
 \[
 \Vert T_Pu\Vert _{L^\infty} =\Vert \partial_{\beta}T_{g^{\alpha \beta}}\partial_{\alpha}u   \Vert_{L^{\infty}}\leq 
 \Vert T_{\partial_{\beta} g^{\alpha \beta}}\partial_{\alpha}u \Vert_{L^{\infty}} +  \Vert T_{\partial_{\alpha}\partial_{\beta}u}g^{\alpha \beta} \Vert_{L^{\infty}}+ \Vert\Pi (g^{\alpha\beta}, \partial_{\alpha}\partial_{\beta}u) \Vert_{L^{\infty}.}
 \]
 However, the bounds  for all of these terms  rely on the use  of the fact that $g^{\alpha \beta} \in \CC$, $\partial g^{\alpha \beta}, \partial_{\alpha}\partial_{\beta}u \in \DCC$ (consequence of Lemma~\ref{l:Moser-control}), as well as on the  bound given by Lemma~\ref{l:utt}. Precisely the estimate \eqref{TDCC-CC} implies that
 \[
 \Vert T_{\partial_{\beta} g^{\alpha \beta}}\partial_{\alpha}u \Vert_{L^{\infty}} +  \Vert T_{\partial_{\alpha}\partial_{\beta}u}g^{\alpha \beta} \Vert_{L^{\infty}}+ \Vert\Pi (g^{\alpha\beta}, \partial_{\alpha}\partial_{\beta}u) \Vert_{L^{\infty}}\lesssim \BB^2.
 \]
 Similar bounds will be obtained for $T_{\hat{P}}$ and  $T_{\tilde{P}}$ by the use of the same results mentioned in the proof of bound \eqref{5.33}.
 
 \end{proof}
 
 We next consider similar bounds for derivatives of $u$.
 Here we will differentiate between space and time derivatives.
 We begin with spatial derivatives:
 
 \begin{lemma}\label{l:tp-dxu}
 We have 
 \begin{equation}
\| P_{<k} T_{P} \partial_x u \|_{L^\infty} \lesssim 2^k \BB^2,
 \end{equation}
 as well as the similar bounds for $T_{\tP}$ and $T_{\hP}$.
 \end{lemma}
 \begin{proof}
 For this proof we rely on the previous Lemma~\ref{l:tp-u} and on Lemma~\ref{l:utt} . This becomes obvious after we commute the $\partial_x$  across the $T_P$ operator
 \[
 P_{k} T_{P} \partial_x u  = \partial_x P_{k} \partial_\alpha T_{g^{\alpha \beta}} \partial_\beta u-  P_{k} \partial_\alpha T_{ \partial_xg^{\alpha \beta}} \partial_\beta u
\]
The first term on the RHS of the identity above is bounded using \eqref{5.33} as follows
\[
\Vert   \partial_x P_{k} \partial_\alpha T_{g^{\alpha \beta}} \partial_\beta u\Vert_{L^{\infty}}\lesssim 2^k \Vert \partial_\alpha T_{g^{\alpha \beta}} \partial_\beta u\Vert_{L^{\infty}} \lesssim 2^k\BB^2.
\]
Here we took advantage of the $\partial_x$ accompanied by the frequency projector $P_k$. A similar advantage will not present itself for the last term, where we need to distribute the $\alpha$ derivative 
\[
P_{k} \partial_\alpha T_{ \partial_xg^{\alpha \beta}} \partial_\beta u =P_{k}  T_{ \partial_\alpha\partial_xg^{\alpha \beta}} \partial_\beta u +  P_{k}  T_{ \partial_xg^{\alpha \beta}}\partial_{\alpha} \partial_\beta u :=e_1+e_2.
\]
We bound  $e_1$ using Lemma~\ref{l:Moser-control}, by placing $\partial_x \partial g^{\alpha\beta} \in \partial_x\DCC$ which means it will admit a decomposition as follows 
\[
\partial_x \partial g^{\alpha\beta} =f_1+f_2,
\]
where 
\[
\Vert P_kf_1\Vert_{L^{\infty}}\lesssim 2^{\frac{3k}{2}}\BB c_k, \quad \Vert P_kf_2\Vert_{L^{\infty}}\lesssim 2^k\BB^2.
\]
Thus, we get
\[
\Vert e_1\Vert_{L^{\infty}} \lesssim \left( \Vert P_kf_1\Vert_{L^{\infty}}+\Vert P_kf_2\Vert_{L^{\infty}} \right)\Vert \partial_{\beta}u\Vert_{L^{\infty}}\lesssim \left( 2^k\BB^2 + 2^{\frac{3k}{2}}\BB c_k\right)  \Vert \partial_{\beta}u\Vert_{L^{\infty}},
\]
which leads to the desired bound once we estimate the last term accordingly. The bounds can be one of the following
\[
\Vert \partial_\beta u \Vert_{L^{\infty}} \lesssim  2^{-\frac{k}{2}}\Vert 2^{\frac{k}{2}}\partial_{\beta}u\Vert_{L^{\infty}}\lesssim  2^{-\frac{k}{2}}\BB \qquad \mbox{ or }  \qquad  \Vert \partial_\beta u \Vert_{L^{\infty}} \lesssim \AA. 
\]
For the first term in the bracket we  use the control norm $\AA$, and for the second term  we use the $\BB$ norm bound.

For $e_2$ we use  the decomposition in Lemma~\ref{l:utt}  for $\partial_{\alpha}\partial_{\beta}u$ and for  $ g$ we use the fact that  $g\in \CC_0$ where we can use either the $\AA$ bound or the $\BB$ bound. The computations are similar to the case of $e_1$. 

The bounds for $T_{\tP}$ and $T_{\hP}$ follow from the exact argument as the one used above in the $T_P$ case.

  \end{proof}

 \begin{lemma}\label{l:tp-du}
 a) We have 
 \begin{equation}
T_{P} \partial u \in \partial (\BB^2 L^\infty),   
 \end{equation}
 i.e. there exists a representation
 \begin{equation}
T_{P} \partial u = \partial_\alpha f^\alpha, \qquad  |f| \lesssim \BB^2.
 \end{equation}
b) We also have
\begin{equation}
\| P_{k} \partial_\alpha T_{g^{\alpha\beta}}  \widehat{ \partial_\beta \partial_\gamma} u\|_{L^\infty}
+ \| P_{k} \partial_\gamma T_{g^{\alpha\beta}}  \widehat{ \partial_\beta \partial_\alpha} u\|_{L^\infty}
\lesssim 2^k \BB^2.
\end{equation}
Similar results hold with $g$ replaced by $\tg$ or $\hg$.

 \end{lemma}
 
 \begin{proof}
a)  We write
 \[
 \begin{aligned}
 \partial_\alpha T_{g^{\alpha\beta}} \partial_\beta
 \partial u = & \
 \partial  \partial_\alpha T_{g^{\alpha\beta}} \partial_\beta u 
 - \partial_\alpha T_{\partial g^{\alpha\beta}} \partial_\beta u .
 \end{aligned}
 \]
 Here by \eqref{TDCC-CC} we have 
 \[
 \| T_{\partial g} \partial u \|_{L^\infty}
\lesssim \BB^2,
 \]
so we get
\[
 \begin{aligned}
 \partial_\alpha T_{g^{\alpha\beta}} \partial_\beta
 \partial u = & \
 \partial  T_{g^{\alpha\beta}} \partial_\alpha \partial_\beta u + \partial ( \BB^2 L^\infty)
 = \partial  \Pi( g^{\alpha\beta}, \partial_\alpha \partial_\beta u) + \partial T_{\partial_\alpha \partial_\beta u} g^{\alpha\beta}
 + \partial ( \BB^2 L^\infty),
 \end{aligned}
\]
where we can use again the previous lemma.

b) The first step here is to reduce to the case of the metric $\tg$. Each of the other two metrics may be written in the form $h \tg$, with $h= h(\partial u)$.
Then we can write 
\begin{equation}
\label{expresion}
\partial_\alpha T_{h \tg^{\alpha\beta}} \widehat{ \partial_\beta \partial_\gamma} u
= T_h \partial_\alpha T_{\tg^{\alpha\beta}}  \widehat{ \partial_\beta \partial_\gamma} u -
T_{\partial_\alpha h}  T_{\tg^{\alpha\beta}}  \widehat{ \partial_\beta \partial_\gamma} u
+ \partial_\alpha (T_{h \tg^{\alpha\beta}}
- T_{h} T_{\tg^{\alpha\beta}}) \widehat{ \partial_\beta \partial_\gamma} u.
\end{equation}
The first term corresponds to our reduction,  and the remaining terms 
need to be estimated perturbatively. This is straightforward unless $\alpha = 0$,
so we focus now on this case.

For the middle term in \eqref{expresion} we can use the last part of \eqref{fe-control-v} in Lemma~\ref{l:Moser-control} for $\partial_0 h$. Using a $\DCC$ decomposition for 
it, $\partial_0 h = h_1+h_2$, we can match the two terms 
with the two pointwise bounds for $ \widehat{ \partial_\beta \partial_\gamma} u$,
namely
\[
\|P_k \widehat{ \partial_\beta \partial_\gamma} u\|_{L^\infty} \lesssim 2^k \AA c_k^2,
\qquad \| P_k \widehat{ \partial_\beta \partial_\gamma} u\|_{L^\infty} \lesssim 2^{\frac{k}2} \BB c_k.
\]
This yields
\[
\begin{aligned}
\| P_{k} T_{\partial_\alpha h}  T_{\tg^{\alpha\beta}}  \widehat{ \partial_\beta \partial_\gamma} u\|_{L^\infty} \lesssim & \ \| h_1\|_{L^\infty} \cdot \| \tilde P_k  \widehat{ \partial_\beta \partial_\gamma} u\|_{L^\infty} +
\| h_2 \|_{L^\infty} \cdot \| \tilde P_k  \widehat{ \partial_\beta \partial_\gamma} u\|_{L^\infty} 
\\
\lesssim & \  \BB^2 \cdot 2^k \AA c_k^2 + 2^{\frac{k}2} \BB c_k \cdot 
 2^{\frac{k}2} \BB c_k ,
\end{aligned}
\]
where the frequency envelope provides the summation with respect to $k$. 

For the last expression in \eqref{expresion} we distribute $\partial_0$. If it falls on the main term,
then we can combine the bound \eqref{fe-uttt} with the para-composition bound in Lemma~\ref{l:para-com} exactly as in the proof of Lemma~\ref{l:para-p+}. If it falls on $h$ (or similarly on $\tg$) then we use the same decomposition as above for $\partial_0 h$.
For the $h_1$ contribution we have a direct bound without using any cancellations, while for $h_2$ we use again
Lemma~\ref{l:para-com}.

We continue with the second\footnote{Note that these two steps are interchangeable.} step, which is to switch $\partial_\alpha$ and $\partial_\gamma$. For fixed $\alpha$ and $\gamma$,
we write
\begin{equation}
\partial_\alpha T_{\tg^{\alpha\beta}}  \widehat{ \partial_\beta \partial_\gamma} u
= \partial_\gamma T_{\tg^{\alpha\beta}}  \widehat{ \partial_\beta \partial_\alpha} u
+ f,
\end{equation}
where $f$ satisfies 
\begin{equation}\label{good-f}
\| P_{<k} f \|_{L^\infty}  \lesssim 2^k \BB^2.
\end{equation}

This is trivial if  $\alpha = \gamma = 0$. If both are nonzero, or if one of them is zero but $\beta \neq 0$, then there is no hat correction and this is a straightforward commutator bound. 
It remains to discuss the case when $\beta = 0$ and exactly one of $\alpha$ and $\gamma$ are zero, say $\gamma = 0$. Then we need to consider the difference
\[
\begin{aligned}
f = & \ \partial_\alpha T_{\tg^{\alpha 0 }}   \widehat{\partial_0 \partial_0} u
-  \partial_0 T_{\tg^{\alpha 0 }}   \partial_0 \partial_\alpha u
\\
= & \ \partial_\alpha T_{\tg^{\alpha 0 }}   \Pi(u, \partial_x \partial u) 
+ T_{\partial_\alpha\tg^{\alpha 0 }}   \partial_0 \partial_0 u
-  T_{\partial_0 \tg^{ \alpha 0 }}   \partial_0 \partial_\alpha u,
\end{aligned}
\]
which can be estimated as in \eqref{good-f}  using the fact that $\alpha \neq 0$
and the bound \eqref{fe-dt2u} for the second time derivative of $u$, respectively the 
similar bound \eqref{fe-control-v} (third estimate) for $\partial_0 \tg$.

It remains to examine the expression 
\[
g_\gamma = \partial_\gamma T_{\tg^{\alpha\beta}}  \widehat{ \partial_\beta \partial_\alpha} u,
\]
where, unlike above, we take advantage of the summation with respect to $\alpha$ and $\beta$. Then, using the $u$ equation, we have
\[
g_\gamma = \partial_\gamma T_{\partial_\alpha \partial_\beta u} \tg^{\alpha \beta}.
\]
The term where  both $\alpha$ and $\beta$ are zero vanishes, so this can be estimated directly as in \eqref{good-f} if $\gamma \neq 0$, and using either \eqref{fe-dt2u} 
or \eqref{fe-control-v} (third estimate) for $\partial_0 \tg$, otherwise.

\end{proof}

 \section{Paracontrolled distributions}
 \label{s:paracontrol}
To motivate this section, we start from the classical energy estimates for the wave equation, which are obtained using 
the multiplier method. Precisely, one multiplies the equation
$\Box_g u = f$ by $Xu$ and simply integrates by parts. Here 
$X$ is any regular time-like vector field. In the next section,
we prove energy estimates for the paradifferential equation 
\eqref{paralin-inhom}, by emulating this strategy at the paradifferential level. The challenge is then to uncover 
a suitable vector field $X$. Unlike the classical case, 
here not every time-like vector field $X$ will suffice.
Instead $X$ must be carefully chosen, and in particular it will inherently have a limited regularity. 

Since the metric $g$ is a function of $\nabla u$, scaling considerations indicate that the vector field $X$ should be at the same regularity level. Naively, one might hope to have an explicit expression $X = X(\nabla u)$ for our vector field. Unfortunately,
seeking such an $X$ eventually leads to an overdetermined system.
At the other extreme, one might enlarge the class of $X$ to 
all distributions which satisfy the same $H^s$ and Besov norms
as $\partial u$, which is essentially the class of functions 
which satisfy \eqref{fe-control-v}. While this class will turn out  to contain the 
correct choice for $X$, it is nevertheless too large to allow
for a clean implementation of the multiplier method.

Instead, there is a more subtle alternative, namely to have 
the vector $X$ to be \emph{paracontrolled} by $\partial u$.
This terminology was originally introduced by Gubinelli, Imkeller and 
Perkowski~\cite{GIP} in connection to Bony's calculus,  in order to study stochastic pde problems, see also \cite{KO}. However, similar constructions have been carried earlier in the renormalization arguments e.g. for wave maps,
in work of Tao~\cite{Tao-wm2d}, Tataru~\cite{T:wm2} and Sterbenz-Tataru~\cite{ST-wm1}; the last 
reference used the name \emph{renormalizable} for the corresponding class of distributions. 

In the standard usage, this is more of a principle than an exact 
notion, which needs to be properly adapted to one's purposes.
For our own objective here, we provide a very precise definition
of this notion, which is exactly tailored to the problem at hand.

\subsection{ Definitions and key properties}

\begin{definition}\label{d:paracontrol}
We say that a function $z$ is \textbf{paracontrolled} by 
$\partial u$ in a time interval $I$
if it admits a representation\footnote{Such a representation might not be unique
in general, though later in the paper we often identify specific choices for the paracoefficients.} of the form
\begin{equation} \label{parac-rep}
z = T_a \partial u + r   , 
\end{equation}
where the vector field $a$ and the error $r$ have the 
following properties:

(i) bounded para-coefficients $a$:
\begin{equation}\label{parac-1}
\| a \|_{\CC} \leq  C. 
\end{equation}

(ii) balanced error $r$:
\begin{equation}\label{parac-2}
\|r\|_{L^\infty} \leq C \AA^2, \qquad \|\partial r\|_{L^\infty} \leq C \BB^2   . 
\end{equation}

\end{definition}
It is convenient to think of the space of distributions $z$ paracontrolled
by $\partial u$ as a Banach space, which we denote by $\PP(\partial u)$,
or simply $\PP$. The norm in this Banach space is defined to be the 
largest implicit constant in \eqref{parac-1} and \eqref{parac-2}, minimized over  all representations of the form \eqref{parac-rep}. If $\|z\|_{\PP} \lesssim_{\AAs} 1$ then
we will simply write 
\[
 z \llcurly \partial u.
\]

While for the most part this definition can be applied separately at each time $t$, in our context we will think of both $u$ and $z$ as functions of time, 
and think of these bounds as uniform in $t$.
Precisely, above we think of $\AA$ as a global, time independent parameter,
whereas $\BB$ is allowed to be a possibly unbounded function of $t$.

To better understand the space  $\PP$ of paracontrolled distributions,
it is useful to relate it to the objects we have already discussed
in the previous section:

\begin{lemma}
a) If $F$ is a smooth function, then $F(\partial u) \in \PP$.

b) We have the inclusion $\PP \subset \CC$. 
\end{lemma}

\begin{proof}
Part (a) is a direct consequence of parts (c), (d) of Lemma~\ref{l:Moser-control}. Part (b) follows instead from part (a) 
of the same Lemma.
\end{proof}
Thus one may think of the class $\PP$ of paracontrolled distributions
as an intermediate stage between the class of smooth functions of $\partial u$, which is too narrow for our purposes, and the larger
class $\CC$, which does not carry sufficient structure.

Next we consider nonlinear properties:

\begin{lemma}\label{l:Moser}
a) [Algebra property] The space $\PP(\partial u)$ is an algebra. Further,
if $z_1, z_2 \in \PP$ have paracoefficients $a_1$, respectively $a_2$,
then the paracoefficients of $z_1 z_2$ can be taken to be $ z_1 a_2+z_2 a_1$.

b) [Moser inequality] If $F$ is a smooth function with $F(0)=0$ and $z \llcurly \partial u$, then $F(z) \llcurly \partial u$.  
\begin{equation}
 \| F(z) \|_{\PP} \lesssim_{\AA,\|z\|_{L^\infty}} \|z\|_{\PP}   .
\end{equation}
Further, if $z \in \PP$ has paracoefficients $a$,
then the paracoefficients of $F(z)$ can be taken to be $F'(z) a$.
\end{lemma}

\begin{proof}
a) We consider the algebra property. 
Let 
\[
z_1 = T_{a_1} \partial u + r_1, 
\qquad z_2 = T_{a_2} \partial_u + r_2
\]
and expand $z_1z_2$. 

We first observe that we can place 
$\Pi(z_1,z_2)$ into the error term. For this it suffices to use the $\CC$ norm
and apply the second bound in \eqref{TDCC-CC}.

We next consider $T_{z_1} z_2$ where for $z_1$ we again use only the $\CC$ norm. We begin with $T_{z_1} r_2$, which we also estimate as an error term. Here we estimate again the more difficult time derivative.
If it falls on the first term then we can bound the output 
exactly as the balanced case above, see \eqref{TDCC-CC}.
Else, it suffices to use the uniform bound on $z_1$.

Finally, we consider the expression 
\[
T_{z_1} T_{a_2^\gamma} \partial_\gamma u = 
T_{z_1 a_2^\gamma} \partial_\gamma u + (T_{z_1} T_{a_2^\gamma}- 
T_{z_1 a_2^\gamma})
\partial_\gamma u .
\]
where the first term has a $\CC$ coefficient by the $\CC$ algebra property, and the second may be estimated perturbatively.
Here if the time derivative goes on the first factor then we are 
back to the previous case and no cancellation is needed. 
Else for $\partial_t \partial_\gamma u$ we use the decomposition in Definition~\ref{d:CC}(a) (or simply Lemma~\ref{l:utt}), combined with Lemma~\ref{l:para-prod}.

\medskip

b) To prove the Moser inequality, our starting point 
is Lemma~\ref{l:Moser-control}(d), which allows us to reduce the problem to estimating $T_{F'(z)} z$, using only the $\CC$ norm of $z$.
But here we can bound $F'(z)$ in $\CC$ using the Moser bound in $C$,
which allows us to conclude as in part (a).
\end{proof}

In addition to the above lemmas, functions in $\PP$
essentially solve a paradifferential $\Box_{\tP}$ equation.
This will be used later to estimate lower order terms 
in the proof of Theorem~\ref{t:para-wp}, and for \eqref{def-J-Y}:

\begin{lemma}\label{l:pcbounds-extra}
Let $h \in \PP$. Then there exist functions $f^\alpha$ so that 
we have the representation
\begin{equation*}
    \partial_\alpha T_{g^{\alpha\beta}}\partial_\beta h = \partial_\alpha f^\alpha ,
\end{equation*}
which satisfy the following bounds:
\begin{equation}\label{f-a}
|f_\alpha| \lesssim \BB^2    ,
\end{equation}
respectively
\begin{equation}\label{f-b}
\| P_{<k} (T_{g^{00}} \partial_0 h - f^0) \|_{L^\infty} \lesssim 2^k.    
\end{equation}
The same result holds for the metrics $\tg$, $\hg$.
\end{lemma}

\begin{proof}
We use the representation \eqref{parac-rep} for $h$.
The property in the lemma holds trivially for the $r$ component
of $h$, with 
\[
f^\alpha = T_{\tg^{\alpha \beta}} \partial_\beta r.
\]
Precisely, the bound \eqref{f-a} holds  due to the second part of 
\eqref{parac-2}, while for the bound \eqref{f-b}, the $\partial_0 r$ component cancels and then we can use the first part of \eqref{parac-2}.

It remains to consider $h$ of the form $h = T_{a^\gamma} \partial_\gamma u$. We write 
\[
\partial_\alpha T_{\tg^{\alpha\beta}}\partial_\beta h := \partial_\alpha h^\alpha, 
\]
noting that the expression on the left hand side of \eqref{f-b}
is exactly $h^0-f^0$.
We begin by refining the expression for $h^\alpha$, noting that corrections of size $\BB^2$ may be directly included into $f_\alpha$ without harming \eqref{f-b}. For this we write
\[
\begin{aligned}
h^\alpha 
= & \ \partial_\gamma T_{a^\gamma} T_{g^{\alpha \beta}} \partial_\beta   u +  T_{\tg^{\alpha \beta}} 
T_{\partial_\beta a^\gamma} \partial_\gamma u
+ [T_{\tg^{\alpha \beta}}, 
T_{a^\gamma}] \partial_\gamma \partial_\beta u
-  T_{a^\gamma} T_{\partial_\gamma \tg^{\alpha\beta}} \partial_\beta u
%\\ & \ 
- \partial_\alpha T_{\partial_\gamma a^\gamma} T_{ \tg^{\alpha\beta}} \partial_\beta u,
\end{aligned}
\]
where the first term on the right is the leading term, while 
the remaining terms can be estimated by $\BB^2$ as follows:

\begin{itemize}
\item The second, fourth and fifth terms are estimated directly using \eqref{parac-1} for $\partial_\beta a^\gamma$, 
$\partial_\gamma g^{\alpha\beta}$
respectively $\partial_\gamma a^\gamma$. 
\item The third term is estimated  using the commutator bound in Lemma~\ref{l:para-com}, as well as Lemma~\ref{l:utt} if both $\beta$ and $\gamma$ are zero.
\end{itemize}

We have reduced the problem to the case when 
\[
h^\alpha =  \partial_\gamma T_{a^\gamma} T_{g^{\alpha \beta}} \partial_\beta   u.
\]
At this point we rewrite
\[
\partial_\alpha h^\alpha = \partial_\gamma \tilde h^\gamma, \qquad 
\tilde h^\gamma = \partial_\alpha T_{a^\gamma} T_{g^{\alpha \beta}} \partial_\beta   u,
\]
noting that 
\[
\| P_{<k} (h^0 - \tilde h^0) \|_{L^\infty} \lesssim 2^{k}, 
\]
which allows us to switch $\alpha$ and $\gamma$ also in \eqref{f-b}.
Then we are allowed to correct $\tilde h^\gamma$, by writing
\[
\tilde h^\gamma = T_{\partial_\alpha a^\gamma} T_{g^{\alpha \beta}} \partial_\beta   u + T_{a^\gamma} \partial_\alpha  T_{g^{\alpha \beta}} \partial_\beta   u.
\]
Now both terms on the right can be estimated by $\BB^2$ as follows:

\begin{itemize}
\item The first term is estimated directly using \eqref{parac-1} for $\partial_\alpha a^\gamma$.

\item The second term is estimated using Lemma~\ref{l:tp-u}.
\end{itemize}
Hence the proof of the lemma is concluded.

\end{proof}

In addition to the class of paracontrolled distributions $\PP$ we also define a secondary class of distributions, which
roughly speaking corresponds to derivatives of $\PP$ functions.

\begin{definition}
The space $\DPP$ of distributions consists of functions $y$ which admit a representation 
\begin{equation}
y = \partial_\alpha z^\alpha + r   , 
\end{equation}
where
\begin{equation}
\| z^\alpha \|_{\PP} \lesssim 1, \qquad \| r \|_{L^\infty} \lesssim \BB^2  .  
\end{equation}
\end{definition}

Due to the inclusion $\BB \subset \CC$,
we can directly relate it to the class $\DCC$ introduced earlier.

\begin{lemma}\label{l:DP-in-Q}
We have $\DPP \subset \DCC$. 
\end{lemma}

% \begin{proof}
% By the previous Lemma, it suffices to show that, for $z \in \PP$,
% we have $\partial_0 z \in \DCC$. It suffices to consider $z$ of the form
% \[
% z = T_{a^{\gamma}} \partial_\gamma u
% \]
% Then expand have
% \[
% \begin{aligned}
% \partial_0 z = T_{\partial_0 a^\gamma} \partial_\gamma u +
% T_{a^{\gamma}} \widehat{\partial_0 \partial_\gamma} u  
% + T_{a^{\gamma}} (\partial_0 \partial_\gamma u - \widehat{\partial_0 \partial_\gamma} u)
% \end{aligned}
% \]
% Now we use the last bound in \eqref{parac-1} for the first term,
% the bound \eqref{fe-dt2u} for the second term and the bound \eqref{fe-dt2u-err}
% for the third.
% \end{proof}

Next we verify that it is stable under multiplication by 
$\PP$ functions.

\begin{lemma}\label{l:ppxDPP}
We have the bilinear bound
\begin{equation}
\PP \times \DPP \to \DPP    .
\end{equation}
\end{lemma}
As a corollary of this lemma, it follows that our gradient potentials $A^\gamma$
and $\tA^\gamma$ are in $\DPP$. 

\begin{proof}
For $h,z \in \PP$ we consider the expansion
\[
\begin{aligned}
q =& \  h \partial_\alpha z
\\
= & \ \partial_\alpha T_h z - T_{\partial_\alpha h} z + \pi(h,\partial_\alpha z) + T_{\partial_\alpha z} h .
\end{aligned}
\]
The first term is in $\DPP$ by Lemma~\ref{l:Moser}(a).
The three remaining terms can be perturbatively estimated by $\BB^2$, using the bounds in \eqref{TDCC-CC}.
\end{proof}

Finally, we consider decompositions for $\DPP$ functions which are akin
to Lemma~\ref{l:utt}. We will do this in two different ways, one 
which is shorter but loses some structure, and another which is 
more involved but retains full structure.

\begin{lemma}\label{l:switch-dt}
Let $ w  \in \DPP$. Then $w$ admits a representation of the form
\[
w = \partial_x w_1 + r,
\]
where 
\begin{equation}
\| w_1\|_{\PP} \lesssim \|w\|_{\DPP}, \qquad \|r\|_{L^\infty} \lesssim \BB^2 \|w\|_{\PP}.    
\end{equation}

%\red{The same result carries over to symbols $a \in \DPP S^0$.}
\end{lemma}

\begin{proof}
It suffices to consider $w$ of the form $w = \partial_0 z$ where
$z \in \PP$, with a representation as in \eqref{parac-rep},
\[
z = T_{a^\beta} \partial_\beta u + r ,
\]
with $a^\beta,r$ as in \eqref{parac-1},\eqref{parac-2}. The bound \eqref{parac-2}
allows us to discard the contribution of $r$ to \eqref{choose-d}. It remains to produce an appropriate modification $\partial_x z_1$,
with $z_1 \in \partial_x \PP$, for the expression 
\[
q =  \partial_0 T_{a^\beta} \partial_\beta u.
\]
We successively peel off perturbative $O(\BB^2)$ layers from $q$.
First we use \eqref{parac-1} to write
\[
q
= 
 T_{a^\beta} \partial_0\partial_\beta u + g^{00} T_{\partial_0 a^\beta} \partial_\beta u
= g^{00}  T_{a^\beta} \partial_0 \partial_\beta u + O(\BB^2).
\]

At this point we have two cases to consider:

(i) $\beta \neq 0$.
Then we write
\[
q = \partial_\beta T_{g^{00} a^\beta}\partial_0 u + O(\BB^2),
\]
and the remaining expression is in $\partial_x \PP$.

(ii) $\beta = 0$.  Here we use the equation for $u$ to write
\[
\partial_t^2 u = - \sum_{(\alpha,\beta) \neq (0,0)} T_{\tg^{\alpha\beta}} \partial_\alpha \partial_\beta u + \Pi( \tg^{\alpha\beta}, \partial_\alpha \partial_\beta u)
+ T_{\partial_\alpha \partial_\beta u} \tg^{\alpha\beta}.
\]
Here the first term on the right involves at least one spatial derivative and is treated as before, in the case $\gamma \neq 0$, while the contributions of the last two terms are perturbative, and can be bounded  by $\BB^2$.

\end{proof}

Our second representation provides a more explicit recipe 
to obtain the corrected version not only of $\DPP$ functions, but also of $\PP \times \DPP$ functions:

\begin{lemma}\label{l:circle-app}
Let $w = z_1^\alpha \partial_\alpha z$, where $z_1, z_2 \in \PP$,
and $z_2$ has the $\PP$ representation 
\[
z = T_{a^\gamma} \partial_\gamma u + r.
\]
Define
\[
\mathring{w} = T_{z_1^\alpha a^\gamma}\widehat{\partial_\alpha\partial_\gamma} u.
\]
Then we have
\begin{equation}\label{ring-error}
    \|w-\mathring w\|_{L^\infty} \lesssim \BB^2,
\end{equation}
while
\begin{equation}\label{ring-bd}
    \| P_k \mathring w\|_{L^\infty} \lesssim 2^{\frac{k}2} \BB c_k.
\end{equation}
\end{lemma}

\begin{proof}
The contribution of $r$ is directly perturbative so we discard it.
Furthermore, the bounds in \eqref{TDCC-CC} allow us to replace perturbatively $w$ by
\[
w = T_{z_1^\alpha} \partial_\alpha z + O(\BB^2) 
= T_{z_1^\alpha} T_{a^\gamma} \partial_\alpha \partial_\gamma u + O(\BB^2).
\]
Using also Lemma~\ref{l:utt} we obtain
\[
w = T_{z_1^\alpha} T_{a^\gamma} \widehat{\partial_\alpha \partial_\gamma} u + O(\BB^2).
\]
Finally, we use Lemma~\ref{l:para-prod} to combine the two 
paraproducts, arriving at 
\[
w = \mathring w + O(\BB^2),
\]
as needed. Finally, the bound \eqref{ring-bd} is also a consequence of Lemma~\ref{l:utt}.

\end{proof}

The last lemma helps us uncover a more subtle, hidden $\Box_g$ structure which appears if we compute the double divergence of the metric $\tg$.

\begin{lemma} \label{l:ddiv-g}
We have 
\begin{equation}\label{ddiv-g}
   \| P_{<k} \partial_\alpha \partial_\beta \tg^{\alpha\beta} \|_{L^\infty} \lesssim 2^{k} \BB^2 .
\end{equation}
\end{lemma}

\begin{proof}
For fixed $\beta$ we expand $\partial_\beta \tg^{\alpha\beta}$ using the relations \eqref{d-gab} and \eqref{div-g} to obtain
\[
\partial_\beta \tg^{\alpha\beta} 
= - \partial^\beta u \, \tg^{\alpha \delta} \partial_\beta \partial_\delta u 
- \partial^\alpha u \, \tg^{\beta \delta} \partial_\beta \partial_\delta u 
+ 2\partial^0 u \, \tg^{0\delta} \tg^{\alpha\beta}  \partial_\beta \partial_\delta u.
\]
For this expression we define a corresponding ring correction
\[
\partial_\beta \mathring{\tg}^{\alpha\beta} := - T_{\partial^\beta u} \, T_{\tg^{\alpha \delta}} \widehat{\partial_\beta \partial_\delta} u 
- T_{\partial^\alpha u} \, T_{\tg^{\beta \delta}} \widehat{\partial_\beta \partial_\delta} u 
+ 2T_{\partial^0 u} \, T_{\tg^{0\gamma}} T_{\tg^{\alpha\beta}} \widehat{\partial_\gamma \partial_\beta} u,
\]
which is also chosen to vanish if $(\alpha,\beta)=(0,0)$. We claim that the difference
is perturbative for fixed $\alpha$ and $\beta$,
\[
|P_{<k} \partial_{\alpha}( \partial_\beta \tg^{\alpha\beta} -  \partial_\beta \mathring{\tg}^{\alpha\beta}) |\lesssim 2^k \BB^2.
\]
Indeed, if $\alpha \neq 0$ then this follows directly from Lemma~\ref{l:circle-app}.
On the other hand if $\alpha=0$ then $\beta \neq 0$ in which case the hat correction 
can be discarded and we may distribute the time derivative, using the fact that 
$\partial u, \tg \in \CC$, see Lemma~\ref{l:Moser-control}.

It remains to estimate the expression $\partial_\alpha(\partial_\beta \mathring{\tg}^{\alpha\beta}) $, where we return to the standard summation convention
and take the sum with respect to all $(\alpha,\beta)$. Here we separate the 
three terms in $\partial_\beta \mathring{\tg}^{\alpha\beta}$, in particular forfeiting 
the cancellation when $(\alpha,\beta)=(0,0)$. By Lemma~\ref{l:Moser-control} all paracoefficients are in $\CC$, which allows us to perturbatively commute $\partial_\alpha$ with them   as needed. Then it suffices to estimate the expression
\[
- T_{\partial^\beta u} \, \partial_\alpha T_{\tg^{\alpha \delta}} \widehat{\partial_\beta \partial_\delta} u 
- T_{\partial^\alpha u} \, \partial_\alpha T_{\tg^{\beta \delta}} \widehat{\partial_\beta \partial_\delta} u 
+ 2T_{\partial^0 u} \, T_{\tg^{0\gamma}} \partial_\alpha T_{\tg^{\alpha\beta}} \widehat{\partial_\gamma \partial_\beta} u.
\]
For all terms here we may directly use Lemma~\ref{l:tp-du}(b) directly. Hence the proof of the lemma is concluded.

\end{proof}

  \

\subsection{Symbol classes and the $\PP$DO calculus}
In a similar fashion to the  $L^\infty S^m$ classes of symbols, our analysis will involve paradifferential operators with symbols which on the physical side are at either the $\PP$ or the $\DPP$ level. Precisely, we will work with both the symbol 
classes $\PP S^m$ and with the classes $\DPP S^m$.

For comparison purposes, we recall that for just paraproducts 
with $\PP$ functions $f,g \llcurly \partial u$
we have the uniform in time product bounds 
\begin{equation}\label{pdo-ref1}
\|T_f T_g - T_{fg}\|_{H^s \to H^{s}} \lesssim \AA^2,    
\end{equation}
as well as the time dependent bounds 
\begin{equation}\label{pdo-ref2}
\|T_f T_g - T_{fg}\|_{H^s \to H^{s-1}} \lesssim \BB^2,    
\end{equation}
and the corresponding commutator estimates. We also have, for $h \in \DPP$,
\begin{equation}\label{pdo-ref3}
\|T_f T_h - T_{T_f h}\|_{H^s \to H^{s}} \lesssim \BB^2.    
\end{equation}

Our objective in what follows is to expand these kinds of bounds
to the $\Psi$DO setting. We will see that things become more complex there.
Fortunately, in the present paper we will only need such results primarily 
when one of the operators is a paraproduct, so we only prove our results in this case
and merely make some comments about the general case. 

We begin with the uniform in time bounds, i.e. the counterpart of 
\eqref{pdo-ref1}, where not much changes:
\begin{lemma}\label{l:para-pdo-A}
 Let $f \in \PP S^j$, $g \in \PP S^k$. Then
\begin{equation}
\|T_f T_g - T_{fg}\|_{H^s \to H^{s-j-k}} \lesssim \AA.    
\end{equation}
\end{lemma}

\begin{proof}
By definition we have $f = f_1 + f_2$ where $f_1$ is an $S^j$ multiplier 
and $f_2 \in A L^\infty S^j$, and similarly for $g$. Since $T_{f_1} = f_1(D)$
and $T_{g_1} = g_1(D)$, the leading parts cancel and we are left only with $O(\AA)$
terms, which can be estimated directly without using any cancellation.

\end{proof}

Our next result is concerned with the counterpart of \eqref{pdo-ref3}, where 
again the result is similar:

\begin{lemma}\label{l:para-pdo-BDB}
 For $g \in \PP$ and  $h \in \DPP S^m$ we have
\begin{equation*}
\|T_g T_h -    T_{T_g h}\|_{H^s \to H^{s-m}} \lesssim \BB^2.
\end{equation*}
\end{lemma}
Here, by a slight abuse of notation, by $T_g h$ we mean the symbol paraproduct, where
the Fourier variable is viewed as a parameter.
\begin{proof}
All operators in the lemma preserve dyadic frequency localization, so it suffices to 
fix a dyadic frequency size $k$ and then show that we have
\[
\| (T_g T_h -    T_{T_g h}) P_k u \|_{L^2} \lesssim 2^{mk} \|u\|_{L^2}.
\]
Here we can include the $2^{mk}$ factor in $h$ and reduce the problem to the case 
when $m=0$. 

In the first term we can also harmlessly replace $g$ by $g_{<k}$
and $T_g$ by multiplication by $g_{<k}$, as 
\begin{equation}\label{trick1}
\| (T_g - g_{<k}) P_k u \|_{L^2 \to L^2} \lesssim 2^{-\frac{k}2} \BB ,
\end{equation}
while $h_{<k} \in 2^\frac{k}{2} \BB L^\infty S^0$ therefore
\[
\| T_{h} P_k \|_{L^2 \to L^2} \lesssim 2^\frac{k}{2}\BB.
\]
Similarly, in the second term we can replace $T_g h$ by $g_{<k} h$, 
as 
\begin{equation}\label{trick2}
\| P_{<k} (T_g h - g_{<k} h)\|_{L^\infty S^0} \lesssim \BB^2,  
\end{equation}
akin to Lemma~\ref{l:ppxDPP}.

Thus it remains to bound in $L^2$ the simpler operator 
\[
R = (g_{<k} T_h - T_{g_{<k} h}) P_k.
\]
Our last simplification here is to separate variables in $h$, and reduce 
to the case where $h$ has a product form at frequency $2^k$, 
namely 
\[
h(x,\xi) = f(x) a(\xi), \qquad |\xi| \approx 2^k,
\]
where $ f \in \DPP$ and $a \in S^0$. In this case we may represent 
the operator $T_{h}$ in the form 
\[
T_h P_k u = L_{lh}(f, P_k u),
\]
where the symbol of the bilinear form $L_{lh}$ depends linearly (and explicitly)
on $a$. In this case we may rewrite the operator $R$ in the form 
\[
R u = g_{<k} L_{lh}(f, P_k u) - L_{lh}(g_{<k} f, P_k u),
\]

At this point we can apply one last time the method of separation of variables
to the symbol of $L_{lh}$ to reduce the problem to the case when the bilinear form $L_{lh}$ is of product type,
\[
L_{lh}(f, P_k u) = b_{<k}(D) f c(D) P_k u, 
\]
where the symbols for both symbols $b_{<k}$ and $c P_k$ are bounded and smooth on the $2^k$ scale. After this final reduction the  operator $R$ has a commutator structure,
\[
R u =  [g_{<k}, b_{<k}(D)] f c(D) P_k u.
\]
Here $|P_{<k} f| \lesssim 2^{\frac{k}2} \BB$, while the commutator can be bounded by 
\[
\| [g_{<k}, b_{<k}(D)]\|_{L^2 \to L^2} \lesssim 2^{-k} \| \partial_x g_{<k} \|_{L^\infty}
\lesssim 2^{-\frac{k}2} \BB.
\]
Hence we obtain 
\[
\|R\|_{L^2 \to L^2} \lesssim \BB^2, 
\]
and the proof of the lemma is concluded.

\end{proof}

In very limited circumstances, we will also need a more precise
commutator expansion, which arises in the context where 
we commute one paradifferential operator with symbol $h \in \PP S^m$
with a function $g \in \PP$. This will be applied when 
$g = \tg^{\alpha\beta}$, but the result holds more generally.
The novelty in the commutator expansion below is that 
we do not simply expand
\[
\text{commutator} = \text{principal part} + \text{error}
\]
but instead we seek to better understand the structure of the error,
\[
\text{commutator} = \text{principal part} + \text{unbalanced subprincipal part}
+ \text{balanced error}
\]
The principal part corresponds exactly with the Lie bracket of the two symbols,
interpreted paradifferentially. For possible use later, we define this more 
generally for two symbols:

\begin{definition}
The para-Lie bracket of two symbols $f \in \PP S^j$, $g \in \PP S^k$ is defined as 
\begin{equation}
   \{ f,g\}_p = T_{\partial_\xi f} \partial_x g   -   T_{\partial_\xi g} \partial_x f .
\end{equation}
This belongs to $\DPP S^{j+k-1}$.
\end{definition}
We remark that if $f$ is merely a function, then the first term on the right drops.

While the principal part of the commutator can be described using a paradifferential
operator with an appropriate symbol, the unbalanced subprincipal part has a more complex structure which would be described best using a variable coefficient 
bilinear form. In order to be able to describe this structure, we need a slight expansion of the class $L_{lh}$ of bilinear operators in Definition~\ref{d:Llh}:

\begin{definition}
 By $\PP S^m L_{hl}$ we denote any bilinear operator which is a linear combination
of operators of the form
\[
T_{h} L_{lh}, \qquad h \in \PP S^m,
\]
which is either finite, or infinite but rapidly convergent.
\end{definition}

With this notation, we have the following commutator result:

\begin{proposition}\label{p:com-pdo}
For $g \in \PP$ and $h \in \PP S^m$ we have the commutator expansion
\begin{equation}\label{multi-com}
[T_g, T_h] =    -i  T_{\{g, h\}_p} + OP\PP S^{m-2} L_{lh}(\partial_x^2 g, \cdot) + R,
\end{equation}
where 
\begin{equation}\label{multi-R}
\| R \|_{H^{s} \to H^{s-m+1}} \lesssim \BB^2.    
\end{equation}
\end{proposition}
\begin{proof}
As in the proof of Lemma~\ref{l:para-pdo-BDB}, we first localize in frequency to 
a dyadic scale $2^k$ for the input/output, and reduce to the case $s = 0$ and $m = 2$.

We consider first the special case when $h$ is a multiplier, $h(x,\xi)=h(\xi)$.
Then 
\[
\{h, g\}_p = h_\xi g_x. 
\]
In this case we claim that we have an exact formula,
\begin{equation}\label{exact-com}
[T_g, T_h] u = -i  T_{\{g, h\}_p} u + C, \qquad Cu = L_{lh}(\partial_x^2 g, u).
\end{equation}
A-priori the last term on the right, $C$, is a $lh$ type translation invariant bilinear 
form in $g,u$; all we need to do is to compute its symbol $R(\eta,\xi)$, and verify 
that it has symbol type regularity and vanishes of second order when $\eta=0$.
The symbol for $T_g u$ as a bilinear form in $g$ and $u$ is
\[
\ell(\eta,\xi) = \chi(\frac{|\eta|}{|\xi+\frac12 \eta|}) .
\]
Then the symbol for the commutator is 
\[
\chi(\frac{|\eta|}{|\xi+\frac12 \eta|}) (h(\xi) - h(\xi+\eta)) .
\]
We expand the last difference as a Taylor series around the middle as 
\[
h(\xi) - h(\xi+\eta) = - \eta \nabla h(\xi+\frac12 \eta) + \eta^2 r(\xi,\eta)
\]
with $r$ a smooth symbol in both $\eta$ and $\xi$ on the $2^k$ scale for 
$|\eta| \ll |\xi| \approx 2^k$. The middle term gives the symbol of the 
Weyl quantization for the Lie bracket $\{h, g\}_p$. The last term yields the  error term $C$, which has the $\eta^2$ factor corresponding to the two derivatives of $g$.

Next we turn our attention to the general case, which we seek to reduce to the special case above. This is achieved by separating variables in $h$, which allows us to assume without any restriction in generality that  the symbol $h$  has the form 
\[
h(x,\xi) = a(x) b(\xi).
\]
Then we have a corresponding decomposition at the operator level,
\begin{equation}\label{separate}
T_h u = T_a B(D) u  + C_0 u, \qquad C_0 u =    L_{lh}(a_x,u).
\end{equation}
Here we can estimate the commutator with $T_g$ as an error term,
\[
[C_0, T_g] = R.
\]
This is most readily seen using another separation of variables,
which allows us to reduce the problem to the case when
\[
C_0 u = C^1_0(D) a_x C_0^2(D) u,
\]
after which we may apply Lemma~\ref{l:para-com}.
The same lemma also shows  that the commutator $[T_a,T_g]$  yields an error term,
so we arrive at
\[
[T_g, T_h] = T_a [B(D), T_g] + R.
\]
For the commutator on the right we apply the formula \eqref{exact-com}, which yields
\[
[T_g, T_h] u = -i T_a ( T_{b_\xi g_x}  + L_{lh}(g_{xx},\cdot)).
\]
It remains to refine the first product,
\[
T_a T_{b_\xi g_x} = T_{T_{a b_\xi} g_x } + R
\]
for which we use Lemma~\ref{l:para-pdo-BDB}.
\end{proof}

Our final result here is a product formula where we also need an expansion akin to \eqref{multi-com}. One should contrast this with Lemma~\ref{l:para-pdo-BDB}, where
such expansion was not necessary.

\begin{proposition}\label{p:prod-pdo}
For $g \in \PP S^m$ and $h \in \DPP$ we have the commutator expansion
\begin{equation}\label{multi-prod}
T_g T_h =  T_{T_g h}  + OP\PP S^{m-1} L_{lh}(\partial_x h, \cdot) + R,
\end{equation}
where 
\begin{equation}\label{multi-pR}
\| R \|_{H^{s} \to H^{s-m}} \lesssim \BB^2.    
\end{equation}
\end{proposition}

\begin{proof}
The proof follows the same outline as the proof of the previous proposition,
so we only outline the main points.

We localize first in frequency to a dyadic frequency region at scale $2^k$, and then separating variables in the first factor. If $g$ is simply a multiplier then 
then \eqref{multi-prod} is an exact identity akin to \eqref{exact-com} above.
If instead
\[
g = a(x) b(\xi), \qquad a \in \PP, 
\]
then we expand $T_g$ as in \eqref{separate}, and then replace $T_a$ by 
multiplication by $a_{<k}$, using \eqref{trick1}, \eqref{trick2}.
After these simplifications, we are left with estimating the difference
\[
R_0 = (g_{<k} T_{bh} - T_{g_{<k} bh})P_k u = g_{<k} L_{lh}(b,P_k u) -  L_{lh}(g_{<k} b,P_k u).
\]
This difference is easily turned into another commutator and estimated as in \eqref{multi-pR}; this is achieved by separating again variables in the symbol of $L_{lh}$ as in the analysis after \eqref{separate}. 

\end{proof}

%%%%%%%%%%%%%%%%%%%%%%%%%%%%%%%%%%%%%%%%%%%%%%%%%%%%%%%%%%%%%%%%%%%%%%%%%%%%%%%%%%%%%%%%%%%%%%%%%%%%%%%%%%%%%%%%%%%%%%%%%%%%%%%%%%%%%%%%%%%%%%%%%%%%%%%%%%%%%%%%%%%%%%%%%%%%%%%%%%%%%%%%%%%%%%%%%%%%%%%%%%%%%%%%%%%%%%%%%%%%%%%%

 %%%%%%%%%%%%%%%%%%%%%%%%%%%%%%%%%%%%%%%%%%%
%%%%%%%%%%%%%%%%%%%%%%%%%%%%%%%%%%%%%%%%%%%
%%%%%%%%%%%%%%%%%%%%%%%%%%%%%%%%%%%%%%%%%%%
\section{Energy estimates for the paradifferential equation}
%%%%%%%%%%%%%%%%%%%%%%%%%%%%%%%%%%%%%%%%%%%
%%%%%%%%%%%%%%%%%%%%%%%%%%%%%%%%%%%%%%%%%%%
%%%%%%%%%%%%%%%%%%%%%%%%%%%%%%%%%%%%%%%%%%%

Our objective in this section is to prove that the linear paradifferential flow
\begin{equation}\label{paralin-inhom-re}
( \D_{\alpha} T_{g^{\alpha \beta}} \D_{\beta} -  T_{A^\gamma}\D_\gamma) v  = f
\end{equation} 
is locally well-posed in a range of Sobolev spaces. Precisely, we will show that

\begin{theorem}\label{t:para-wp}
Let $u$ be a smooth solution for the minimal surface equation \eqref{msf-short}
in a time interval $I=[0,T]$, with associated control parameters $\AA$ and $\BB$ so that 
\begin{equation}
\AA \ll 1, \qquad \BB \in L^2_t.    
\end{equation}
Let $s \in \R$. Then the linear paradifferential flow \eqref{paralin-inhom-re} is locally well-posed in $\H^{s}$ in the time interval $I$. Furthermore, there exists an energy functional $E^s(v) = E^s(v[t])$, depending on $u$, which is smooth in $\H^{s+1}$, with the following two properties:

a) Energy equivalence:
\begin{equation}
E^s(v[t]) \approx \| v[t]\|_{\H^s}^2    .
\end{equation}

b) Energy estimate:
\begin{equation}
\frac{d}{dt} E^s(v[t]) \lesssim \BB^2 E^s(v[t])   + \|f\|_{H^{s-1}}
  E^s(v[t])^\frac12 .
\end{equation}

The same result is also valid for the paradifferential equations \eqref{paralin-inhom-tg}, respectively \eqref{paralin-inhom-tg}
associated to the metrics $\tg$ and $\hg$.
\end{theorem}

We remark on the modular structure of our arguments. Precisely, from this section it is only the conclusion of this theorem which is used later in the paper. We also remark on the smallness condition for $\AA$:

\begin{remark}
The condition that $\AA \ll 1$ in the theorem is a technical convenience rather 
than a necessity. It is only used in the reduction in Proposition~\ref{p:reduction} in order to insure that the operator $T_{g^{00}}$ is invertible. Since $|g^{00}| \gtrsim 1$ this may be alternatively  guaranteed by a more careful choice of the quantization. Another minor advantage is that 
with this assumption we no longer need to track the dependence on $\AA$ of implicit constants in all the estimates.
\end{remark}

It will be easier to prove the result for the 
paradifferential flow associated to the metric $\tg$.
Because of this, our first step will be to reduce the problem to this case.
Then we will prove the result for $\tg$ in two steps. First, we show that the desired result holds for $s = 0$. Then, we use a paraconjugation argument to show that the same result holds for all  real $s$.

\subsection{ Equivalent metrics}
The idea here is that we can replace the metric $g$ with the conformally  equivalent metric $\tg$ given by \eqref{def-tg} in order to simplify the subsequent analysis. 
A similar equivalence holds for the metric $\hg$; the argument is completely identical.

% \begin{equation}
% \tA^{\gamma} = (g^{00})^{-1} A^\gamma.    
% \end{equation}

Then we have the following equivalence:

\begin{proposition}\label{p:reduction}
Assume that $v$ solves \eqref{paralin-inhom-re}. Then it 
also satisfies an equation of the form
\begin{equation} \label{paralin-inhom-renorm}
(\D_{\alpha} T_{\tg^{\alpha \beta}}  \D_{\beta} -  T_{\tA^\gamma}\D_\gamma) v  = E f + \tR v,
\end{equation} 
where $E$ is invertible and elliptic,
\begin{equation}\label{E-ell}
\| Ef \|_{H^s} \approx \|f\|_{H^s},   
\end{equation}
and $\tR$ is balanced,
\begin{equation}
\| \tR v\|_{H^s} \lesssim \BB^2 \| \partial v\|_{H^s}.    
\end{equation}
\end{proposition}

\begin{proof}
We first observe that, since $g^{00}$ is a small, $O(\AA)$ perturbation of a nonzero constant, it follows that $T_{(g^{00})^{-1}}$ is invertible elliptic, with elliptic inverse $E = (T_{(g^{00})^{-1}})^{-1}$, which satisfies \eqref{E-ell} for all real $s$.

Then $v$ solves \eqref{paralin-inhom-renorm} with $\tR$ of the form
\[
\tR = E\partial_{\alpha} ( T_{\tg^{\alpha \beta}} - T_{(g^{00})^{-1}} T_{g^{\alpha \beta}}) \partial_\beta - 
E ( T_{(g^{00})^{-1}} T_{A^{\gamma}} - T_{\tA^\gamma} - T_{\partial_\alpha (g^{00})^{-1}}  T_{g^{\alpha \gamma}}) \partial_{\gamma} .
\]
Here we have the algebraic relations 
\[
\tg^{\alpha \beta} = (g^{00})^{-1} g^{\alpha \beta}, \qquad 
(g^{00})^{-1} A^{\gamma} = \tA^\gamma + \partial_\alpha (g^{00})^{-1}  g^{\alpha \gamma}.
\]
This allows us  to estimate $\tR$ in a balanced fashion using Lemma~\ref{l:Moser-control} and Lemma~\ref{l:para-prod}, as desired.
\end{proof}

As a consequence of this result, we see that it suffices now 
to prove the result in Theorem~\ref{t:para-wp} but with
the equation \eqref{paralin-inhom-renorm} replaced by 
\begin{equation}
   \label{paralin-inhom-new}
(\D_{\alpha} T_{\tg^{\alpha \beta}}  \D_{\beta} - 2 T_{\tA^\gamma}\D_\gamma) v  = f .
\end{equation}

\subsection{ The $H^1 \times L^2$ bound.}
For expository purposes, we first review the multiplier method for proving energy estimates for the wave equation
in a simplified setting. Then we construct a suitable vector field, to be used as our multiplier.
Finally, we reinterpret the energy estimates at the paradifferential level, 
and prove Theorem~\ref{t:para-wp} with $s = 1$.

\subsubsection{Energy estimates via the multiplier method.}
\label{s:multiplier}
Suppose that we have a function $v$ which solves a divergence form wave equation,
\begin{equation}\label{v-eqn}
P v = f, \qquad P = \partial_\alpha g^{\alpha \beta} \partial_\beta - A^\alpha \partial_\alpha .
\end{equation}
Given a vector field $X = X^\alpha \partial_\alpha$, the 
standard strategy is to multiply the equation by $X v$ and integrate by parts. 
For expository purposes we will follow this path here, noting that 
another alternative would be  to interpret the vector field in the Weyl calculus,
and work instead with the skew-adjoint operator
\[
X^w = X^\alpha \partial_\alpha + \frac12 \partial_\alpha X^\alpha .
\]
At this point we only seek to identify the principal part of the energy estimates, 
which will lead us to the choice of the vector field $X$,
so we  do not follow this second path. However, later on, once $X$ is chosen
and we have switched to the paradifferential setting
we will need to also carefully track the lower order terms, and we will
add lower order corrections to our vector field. 

To further place the following computations into context, we remark that 
vector field energy identities for the wave equation 
are often employed in their covariant form, which is derived by contracting the 
divergence free relation for the energy momentum tensor with the vector field $Xu$, and integrating with respect to the measure associated with the metric $g$. Such a strategy would work but would be counterproductive in our setting, where we will reinterpret all these identities in a paradifferential fashion.

Assuming at first that the function $v$ is compactly supported, integrating by parts
several times, in order to essentially commute the second order part of $P$ with $X$,
one arrives at the identity

\begin{equation}\label{IBP-nobdr}
2\iint Pv \cdot X v \, dx dt = \iint c_X(v,v) \, dx dt, \qquad 
\end{equation}
where $c_X$ is a quadratic expression in $\partial v$
of the form
\begin{equation}
c_X (v,v) = c_{X}^{\alpha \beta} \partial_\alpha v \, \partial_\beta v
\end{equation}
with coefficients given by the relation
\begin{equation} \label{cX}
c_X(x,\xi):= c_{X}^{\alpha\beta} \xi_\alpha \xi_\beta = \{ p,X^\gamma \xi_\gamma\}(x,\xi) - 
\partial_\gamma X^\gamma p(x,\xi) + 2 A^\gamma \xi_\gamma X^\delta \xi_\delta,
\end{equation}
where we recall that $p(x,\xi) = g^{\alpha \beta}\xi_\alpha \xi_\beta$.
Removing the compact support assumption on $v$ and introducing boundaries at times 
$t = 0$ and $t = T$, the identity above with the integral taken over $[0,T]\times \R^n$ still holds but with added contributions at these times,
\begin{equation}\label{IBP}
    2\iintT Pv \cdot X v\, dx dt = \iintT c_X(v,v) \ dx dt 
    + \left. \int_{\R^n} e_X(v,v)\, dx \right|_0^T ,
\end{equation}
where the contributions at the initial and final time can be thought of as energies.
Here the energy density $e_X$ is a bilinear expression of the form
\begin{equation}\label{eXvv}
e_X(v,v) = e_{X}^{\alpha\beta} \partial_\alpha v \cdot \partial_\beta v  .
\end{equation}

This can be written  in terms of the energy momentum tensor associated to the $\Box_g$ operator,
\begin{equation}\label{T-def}
T_{\alpha\beta}[v] = \D_\alpha v \cdot \D_\beta v  - \half g_{\alpha \beta} g^{\gamma \delta}\D_{\delta} v \cdot\D_\gamma v .
\end{equation}
Then we have
\begin{equation}\label{eX}
e_{X}^{\alpha\beta} \partial_\alpha v \cdot \partial_\beta v = g^{0\alpha} T_{\alpha\beta} X^\beta = T(\partial_t,X).
\end{equation}
Thus we can define the energy functional associated to the vector field $X$ as
\begin{equation}\label{def-EX}
E_X[v] = \int_{\R^n} e_X(v,v)\, dx   . 
\end{equation}
The key property of the energy density $e_X$ is that it is classically known to be positive definite in a pointwise sense,
\begin{equation}\label{eX-positive}
    e_X(v,v) = T(\partial_t,X) \gtrsim |\partial v|^2,
\end{equation}
provided that the vector fields $\partial_t$ and  $X$ are uniformly forward time-like. Then we obtain the energy coercivity property
\[
 E_X[v(t)] \approx \| \partial v(t)\|_{L^2}^2. 
\]

With these notations, we can rewrite the integral identity \eqref{IBP} as a differential identity
\begin{equation}\label{energy-diff}
\frac{d}{dt} E_X(v) =     2\int_{\R^n} P v \cdot X v \, dx -  \int_{\R^n} c_X(v,v)\, dx.
\end{equation}
In a nutshell, this computation, interpreted paradifferentially, is at the heart
of our proof of the energy estimates. In this context, the choice of the vector field $X$ should naively be governed by the requirement that the energy flux form $c_X$ is 
balanced. We note that one cannot ask for $c_X$ to be zero, as this would produce an overdetermined system for $X$, which in particular implies the condition that $X$ is a conformal Killing field for the metric $g$. Even the requirement that $c_X$ is balanced turns out to be a bit too much, which is why we will need a second step
to the above computation.

Precisely, the second step is based on another interesting observation, 
namely that the contribution of terms  in  $ c_{X}^{\alpha \beta}$
of the form 
\[
I = \iintT q g^{\alpha\beta} \partial_\alpha v \cdot \partial_\beta v \, dxdt 
%\qquad q \in \DPP
\]
has a favourable structure and can be eliminated using a suitable 
Lagrangian type energy correction.

Indeed, for compactly supported $v$, this contribution can be rewritten, integrating by parts, as 
\[
\begin{aligned}
I = & \ - \iint  q \partial_\alpha g^{\alpha\beta} \partial_\beta v \cdot  v \, dx dt + \frac12 \iint \partial_\alpha g^{\alpha\beta} \partial_\beta q\,  v^2 dx
dt\\ 
= & \ - \iint  P v \cdot  q v \, dx + \frac12 \iint  \left(P q
- q \partial_\gamma A^\gamma\right)
v^2 dxdt .
\end{aligned}
\]
The first term can be interpreted as a correction to $X$ in \eqref{IBP-nobdr}. Introducing the notation 
\begin{equation}
\M = 2X+q,    
\end{equation}
it now takes the form
\begin{equation}\label{IBP-nobdr-t}
    \iint  P v \cdot \M v\, dx dt =  \iint c_{X}(v,v)- q g^{\alpha\beta} \partial_\alpha v \cdot \partial_\beta v
    +   d v^2 \, dx dt,
\end{equation}
where the coefficient $ d$ of the additional zero order term is
\begin{equation}\label{tdx}
 d = P q- q \partial_\gamma A^\gamma .
\end{equation}

Finally, adding in boundaries at $t=0,T$ we obtain the integral relation
\begin{equation}\label{IBP-A-t}
    \iintT  P v \cdot \M v\, dx dt = \iintT  c_{X,q} (v,v)+ d v^2
    \ dx dt 
 + \left. \int_{\R^n} \tilde e_{X,q}(v,v)\, dx \right|_0^T,
\end{equation}
where the leading flux density is now
\[
c_{X,q}(v,v) =  c_{X} (v,v) -
    q  g^{\alpha\beta} \partial_\alpha v \cdot \partial_\beta v ,
\]
while the new energy density $ e_{X,q}$ has the form
\[
e_{X,q}(v,v) =  e_{X,A}(v,v) + q g^{0\beta} \partial_\beta v \cdot v - 
\frac12 (g^{0\beta} \partial_\beta q - q A^0) v^2   .  
\]
We can also convert this into a differential relation akin to \eqref{energy-diff}, namely 
\begin{equation}\label{energy-diff-t}
\frac{d}{dt}  E_{X,q}(v) =     \int_{\R^n} P v  \cdot \M
v \, dx -  \int_{\R^n} c_{X,q}(v,v) +  d v^2 \, dx.
\end{equation}
The identity \eqref{energy-diff-t} will be our main tool in establishing the desired energy estimate. The Lagrangian correction weight $q$ will have to be chosen carefully,
so that it satisfies multiple requirements:
\begin{enumerate} 
    \item Comparing  the form of $c_{X,q}$ with the earlier expression for $c_X$, 
a natural choice would seem to be 
\begin{equation}\label{choose-q-first}
q =     \partial_\gamma X^\gamma.
\end{equation}
\item Examining the lower order coefficient $d$ above, we will need to have good control over the function $P q$.
\end{enumerate}
Reconciling these two requirements will play an important role later in in this section.

\bigskip

To complete our discussion here, we need to carry 
out an additional step, namely to investigate what happens if we 
replace $g,A$ by $\tg, \tA$. Observing that 
\[
P v = g^{00} \tP v
\]
it becomes natural to replace the vector field $X$, the Lagrangian weight $q$ and the multiplier 
by 
\[
\tX = g^{00} X, \qquad \tq = g^{00} q, \qquad \tM = 2 \tX +  \tq.
\]
Then the relation \eqref{IBP-A-t} remains essentially unchanged,
\begin{equation}\label{IBP-A-tt}
    \iintT  \tP v \cdot \tM v\, dx dt = \iintT  c_{X,q} (v,v)+ d v^2
    \ dx dt 
 + \left. \int_{\R^n}  e_{X,q}(v,v)\, dx \right|_0^T,
\end{equation}
 and the same applies to the differential form \eqref{energy-diff-t} of the same relation. Here 
 the principal flux symbol can be equivalently expressed in the 
 form
 \begin{equation} \label{tcX}
c_{X,q}(x,\xi):= c_{X}^{\alpha\beta} \xi_\alpha \xi_\beta = \{ \tp,\tX\}(x,\xi) - 
(\partial_\gamma \tX^\gamma+\tq) p(x,\xi) + 2\tA^\gamma \xi_\gamma \tX^\delta \xi_\delta.
\end{equation}

 \bigskip

Our task is now
twofold:

\begin{itemize}
    \item To identify a suitable vector field $X$ so that the energy flux above satisfies
    a balanced energy estimate, and 
    \item To recast the above computation in the paradifferential setting without losing 
    the energy balance; this will also require a careful choice for $q$.
\end{itemize}

\bigskip

\subsubsection{The construction of the vector field $X$} \label{s:def-X}
Our objective here is to construct a vector field $X$ so that the
flux coefficients in $c_{X,q}$ are balanced for $q$ as in \eqref{choose-q-first}.
In essence, at this stage we disregard  any paradifferential frequency localizations, and work as if $v$ has infinite frequency. We also do not distinguish between $g$ and $\tg$, as this does not play a role in the choice of $X$.
Our main result governing the choice of the 
vector field $X$ is where our notion of paracontrolled distributions 
is first needed, and reads as follows:

\begin{lemma}\label{l:X}
There exists a vector field $X$ which is paracontrolled by $\partial u$, 
and so that we have the balanced bound
\begin{equation}\label{x3}
\|c_{X}^{\alpha \beta} + \partial_\gamma X^\gamma g^{\alpha\beta}\|_{L^\infty} \lesssim_{\AA} \BB^2.
\end{equation}
%\blue{Alternatively we may replace all expressions in $c_{X}^{\alpha \beta}$
%with paraproducts and the conclusion is still valid.}
\end{lemma}

We remark that the fact that such a vector field exists is closely 
connected to the fact that our equation satisfies the nonlinear null condition in a strong sense.  One should think of our vector field $X$ as the next best thing to a Killing 
or conformal Killing vector field. Perhaps a good terminology would a \emph{para-Killing vector field}, i.e. whose deformation tensor is balanced, rather than equal to zero or a multiple of the metric.

\begin{proof}

 We compute the  expression in \eqref{x3} as follows:
\[
\begin{aligned}
 (c_{X}^{\alpha \beta}+ \partial_\gamma \tX^\gamma g^{\alpha\beta}) \xi_\alpha \xi_\beta = & \ 2 \xi_\gamma \xi_\alpha g^{\gamma \beta} \partial_\beta X^\alpha
- X^\gamma \xi_\alpha \xi_\beta \partial_{\gamma} g^{\alpha \beta} + 2A^\gamma \xi_\gamma X^\gamma \xi_\gamma
\\
= & \ 2 \xi_\gamma \xi_\alpha g^{\gamma \beta} \partial_\beta X^\alpha
+ 2 X^\gamma \partial^\beta u \xi_\beta
g^{\alpha \delta} \partial_\delta  \partial_\gamma u \xi_\alpha+ 2 \partial^\alpha u \partial_\alpha \partial_\beta u  g^{\beta \delta} \xi_\delta X^\gamma \xi_\gamma
\\
= & \ 2 \xi_\gamma \xi_\alpha g^{\gamma \beta} \partial_\beta X^\alpha
+ 2  \xi_\gamma \xi_\alpha X^\beta \partial^\gamma u 
g^{\alpha \delta} \partial_\delta  \partial_\beta u +   
2 \xi_{\alpha}\xi_\gamma 
\partial^\delta u \partial_\delta \partial_\beta u  g^{\beta \alpha}  X^\gamma 
\\
= & \ 2 \xi_\gamma \xi_\alpha g^{\gamma \beta}(
\partial_\beta X^\alpha
+ 2   X^\delta \partial^\alpha u  \partial_\delta  \partial_\beta u +    2X^\alpha
\partial^\delta u \partial_\delta \partial_\beta u  ).
\end{aligned}
\]
Here one could freely symmetrize the coefficients relative to the pair of indices
$(\alpha, \gamma)$. We have chosen to neglect the symmetrization, but, instead,
we made favourable choices.
The above expression would cancel for instance if 
\begin{equation}\label{overdetermined}
 \partial_\beta X^\alpha = 
- X^\delta \partial^\alpha u 
 \partial_\beta  \partial_\delta u
- \partial^\delta u X^\alpha     \partial_\delta \partial_\beta u .  
\end{equation}
This is an overdetermined system, so we cannot hope for an exact cancellation. 
Even if we symmetrize (raising the $\beta$ index first) and equate the symmetric part of the two sides, it still remains overdetermined.

But we do not need exact cancellation, we only need the difference of the two sides to be balanced.
Assume for the moment that $X$ is at the same regularity level as $\partial u$.
Then, examining the right hand side, the expressions there are unbalanced only in the paraproduct case, where the $\partial^2 u$ term is the high frequency, i.e. for the terms $T_{h(h,\nabla u)} \partial^2 u$. Hence we heuristically arrive at the equivalent requirement 
\[
 \partial_\beta X^\alpha \bapprox 
- T_{X^\delta \partial^\alpha u} 
 \partial_\beta  \partial_\delta u
- T_{\partial^\delta u X^\alpha}     \partial_\delta \partial_\beta u ,
\]
where we introduce the notation "$\bapprox$" to indicate that the difference between the two expressions is balanced, i.e. can be estimated as in \eqref{x3}.
Then, at leading order we may cancel the $\beta$ derivative to obtain
a single paradifferential relation at one regularity level higher, namely
\[
X^\alpha \bapprox
- T_{X^\delta \partial^\alpha u} 
  \partial_\delta u
- T_{\partial^\delta u X^\alpha} \partial_\delta u  .
\]
Modulo balanced terms we may break the paraproducts above in two. This allows us to 
devise an inductive scheme to construct $X$ as a dyadic sum of frequency localized pieces,
by setting
\begin{equation}
X = X_{<0} + \sum_{k = 1}^\infty X_k    
\end{equation}
starting with the initialization 
\[
X_{<0} = \partial_t,
\]
and where the functions $X_k$, localized at frequency $2^k$,
are defined by
\begin{equation}\label{def-Xk}
X^\alpha_{k} = -( T_{X^\delta} T_{ \partial^\alpha u} +   T_{ X^\alpha} T_{\partial^\delta u}) \partial_\delta u_k. 
\end{equation}
It remains to show that, as defined above, the  vector field $X$ has all the properties in the Lemma.
We will achieve this in three stages:

\begin{itemize}
    \item We show that $X$ satisfies the same bounds as $\partial u$, 
    (see  \eqref{fe-control} and \eqref{fe-dt2u}),
\begin{equation}
\| X\|_{\CC} \lesssim 1    .
\end{equation}
    \item We show that $X$ is paracontrolled by $\partial u$.
    \item Finally, we establish the balanced bound \eqref{x3}.
\end{itemize}
To simplify the notations, we will write 
schematically that 
\[
X_k = T_X T_h \partial X 
\]
with coefficients $h$ of the form $h = F(\partial u) \in \CC$.

\bigskip

\emph{ I. Dyadic bounds for $X$.} These are proved at each dyadic frequency $k$ by induction on $k$. We do this in two stages, where we first 
estimate the $\CC_0$ norm of $X$.
Precisely, the first set of statements to be proved by induction for $k > 0$ is as follows:
\begin{equation}\label{induction1}
\| X_k\|_{L^\infty} \leq C \AA c_k^2   , 
\end{equation}
\begin{equation}\label{induction2}
\| X_k \|_{L^\infty} \leq C 2^{-\frac{k}{2}} \BB c_k,
\end{equation}
with a fixed large universal constant $C$. This implies 
that $\|X\|_{\CC_0} \lesssim 1$.
The induction hypothesis yields the bound 
\[
\| X_{<k} -X_0 \|_{L^\infty} \lesssim C \AA .
\]
Then we write
\[
X_k = T_{X_0} T_h  u_k +  T_{X_{<k} - X_0} T_h  u_k,
\]
which yields
\[
\| X_k\|_{L^\infty} \lesssim  (1 + C\AA) \AA c_k^2   ,
\]
respectively
\[
\| X_k \|_{L^\infty} \lesssim (1+ C\AA) 2^{-\frac{k}{2}} \BB c_k.
\]
Thus the induction argument closes if $C$ is a large constant and $\AA \ll 1$.

To finish proving that $X \in \CC$, the
second stage is to prove by induction that 
\begin{equation}\label{induction4}
\|\partial_t X_{\leq k}\|_{\DCC} \leq C, 
\end{equation}
i.e. that  
$\partial_t X_{\leq k}$ admits a decomposition $\partial_t X_{\leq k} = f_{k1}+f_{k2}$,
where 
\[
\| f_{1,\leq k}\|_{L^\infty} \leq C \BB^2 c_k^2, \qquad 
\| P_j f_{2,\leq k}\|_{L^\infty}  \leq  C 2^{\frac{j}{2}}  \BB c_j .     \]
Here again $C$ is a fixed large constant, unrelated to the earlier $C$.

For this we write
\[
\begin{aligned}
 \partial_t X_{\leq k} = & \  \partial_t (T_X T_h \partial u_{\leq k} ) 
 =   (T_{\partial_t X} T_h \partial u_{\leq k} 
+ T_X T_{\partial_t h}  \partial u_{\leq k})
+ T_X T_h \partial_t \partial u_{\leq k} .
\end{aligned}
 \]
Here the $X$ coefficients involve only frequencies below $2^k$, so we may use the induction hypothesis in the first term. For the second and third terms 
it suffices to use the $\CC_0$ bound for $X$, which we already have from the first induction. Hence, repeatedly applying the bounds in \eqref{C-DC}
we obtain 
\[
\begin{aligned}
\|  \partial_t X_{\leq k} \|_{\DCC} \lesssim & \
 \AA \|\partial_t X_{<k}\|_{\DCC}  \| h \|_{L^\infty} \|\partial u\|_{\CC_0} 
+ \AA \| X \|_{\CC_0} \|\partial_t h\|_{\DCC}  \|\partial u\|_{\CC_0}
+ \|X\|_{\CC_0} \|h\|_{\CC_0} \|\partial_t \partial u\|_{\DCC} 
\\ 
\lesssim & \
C \AA + 1,
\end{aligned}
\]
which closes the inductive proof of \eqref{induction4} if $\CC \gg 1$
and $\AA \ll 1$.

\medskip

\emph{ II. $X$ is paracontrolled by $\partial u$.} To prove this, we will establish the representation
\begin{equation}\label{X-paracontrol}
X^\alpha = 
-( T_{X^\delta \partial^\alpha u} 
+ T_{ X^\alpha \partial^\delta u}) \partial_\delta u + r^\alpha .
\end{equation}
This will play the role of \eqref{parac-rep}. The Moser estimates in Lemma~\ref{l:Moser-control} show that the paracoefficients above satisfy the bounds
required of $a$ in \eqref{parac-1}, so it remains to establish that the errors $r_\alpha$ satisfy the bounds \eqref{parac-2}. For this, we write
\[
\begin{aligned}
r^\alpha = & \ \sum_k X^\alpha_k + 
( T_{X^\delta \partial^\alpha u} 
+T_{ X^\alpha \partial^\delta u}) \partial_\delta u_k
\\
= & \ \sum_k  -
[  (T_{X^\delta \partial^\alpha u} - T_{X^\delta} T_{\partial^\alpha u})
+(T_{X^\alpha \partial^\delta u }- T_{X^\alpha} T_{\partial^\delta u})] \partial_\delta u_k
\\ 
:= & \ \sum_k r^\alpha_k .
\end{aligned}
\]
Now we apply Lemma~\ref{l:para-prod} to estimate
\begin{equation}\label{bd-ra}
\|r^\alpha_k\|_{L^\infty} \lesssim c_k^2 \AA^2, \qquad \|r^\alpha_k\|_{L^\infty} \lesssim 2^{-k} c_k^2 \BB^2
\end{equation}
as needed. It remains to bound the time derivative of $r^\alpha$ in $L^\infty$.
For this we distribute the time derivative. If it falls on any of the para-coefficients then we can directly use the bound \eqref{TDCC-CC}. Else, we use 
Lemma~\ref{l:para-p+}.

\medskip

\emph{III. The bound for $c_{X}^{\alpha\beta} - \partial_\gamma X^\gamma g^{\alpha\beta}$.}
Here we recall that 
\[
c_{X}^{\alpha\gamma}- \partial_\gamma X^\gamma g^{\alpha\beta} = 2 g^{\gamma \beta}(
\partial_\beta X^\alpha
+    X^\delta \partial^\alpha u  \partial_\delta  \partial_\beta u +    X^\alpha
\partial^\delta u \partial_\delta \partial_\beta u  ).
\]
To estimate this, our starting point is the relation \eqref{X-paracontrol}, together
with the bounds \eqref{bd-ra} for $r^\alpha$. Denoting
\[
h^{\alpha \delta} =   X^\delta \partial^\alpha u   +    X^\alpha \partial^\delta u 
\]
we write
\[
\begin{aligned}
c_{X}^{\alpha\gamma} - \partial_\gamma X^\gamma g^{\alpha\beta}
= & \  g^{\gamma \beta}( \partial_\beta X^\alpha + h^{\alpha \delta}
\partial_\delta \partial_\beta u)
\\ =  & \  T_{g^{\gamma\beta}} \partial_\beta r^\alpha +
( T_{g^{\gamma\beta} h^{\alpha \delta}} -   T_{g^{\gamma\beta}} T_{h^{\alpha \delta}})
\partial_\alpha \partial_\beta u 
\\
 & \ + T_{\partial_\beta X^\alpha}  g^{\gamma \beta}
+ T_{\partial_\delta \partial_\beta u} [g^{\gamma \beta}h^{\alpha \delta}]
+  \Pi(\partial_\beta X^\alpha,  g^{\gamma \beta})
+ \Pi({\partial_\delta \partial_\beta u},g^{\gamma \beta}h^{\alpha \delta}).
\end{aligned}
\]
For the $r^\alpha$ term we use \eqref{bd-ra}, for the next term we use the earlier bound \eqref{para-p-r}  and the terms on the last line are estimated directly using the algebra property for $\CC_0$ and the bilinear estimate \eqref{TDCC-CC}.
\end{proof}

\subsubsection{Paradifferential energy estimates associated to $X$}
Now we use our vector field $X$ to prove the balanced energy estimates for 
$v$. To do this, we repeat the computations leading to the
key energy relations \eqref{IBP-A-t} and \eqref{energy-diff-t}
at the paradifferential level. 

To fix the notations, we denote by
$T_{\tP}$ the operator in \eqref{paralin-inhom-new}, 
\[
T_{\tP} = \partial_\alpha T_{\tg^{\alpha \beta}} \partial_\beta - T_{\tA^\gamma} \partial_\gamma.
\]
By a slight abuse of notation, this is not exactly the same as the Weyl 
quantized operator with the corresponding symbol, though the difference 
between the two can be seen to be balanced and thus perturbative in our analysis.

For our multiplier, inspired by the energy relation \eqref{IBP-A-tt}, we  will use the paradifferential operator
\begin{equation}\label{def-hX}
 T_{\tM} := 2 T_{\tX^\alpha} \partial_\alpha  + \frac12 T_{\tq}. 
\end{equation}
Here ideally we would like to have 
\[
\tq = - g^{00} \partial_\alpha X^\alpha.
\]
However, such a choice causes some technical difficulties due to the lack of sufficient time regularity of $\tq$. To avoid this, we will forego the  
above explicit expression for $\tq$, and instead ask for $\tq$ to satisfy the
following two properties:
\begin{itemize}
    \item it is close to the ideal setting,
\begin{equation}\label{choose-d}
| \tq -     g^{00} \partial_\alpha X^\alpha | \lesssim \BB^2 .
\end{equation}    
    \item it has the form $\tq = \partial_x q_1$, where $q_1 \in \PP$.
\end{itemize}
We remark that the obvious choice $\tq_0 := - g^{00} \partial_\alpha X^\alpha$
for the first criteria does not satisfy the second criteria, as it contains
expressions involving $\partial_t^2 u$. However, by definition 
we have $\tq_0 \in \DPP$, therefore, a good approximation $\tq$
for $\tq_0$ as above exists by Lemma~\ref{l:switch-dt}.
Note that for this it suffices to use the fact that $X^\alpha \in \PP$ separately for each $\alpha$, rather than the more precise representation in \eqref{X-paracontrol}.
\medskip

\medskip

Now we implement the multiplier method to prove energy estimates in the paradifferential setting. We recall our objective, which is to establish an integral energy identity 
of the form
\begin{equation}\label{en0-int}
\iint T_{\tP} v \cdot T_{\tM} v\, dx dt = \left. E_X(v(t)) \right|_0^T   
+ \int_{0}^T O(\BB^2) \| v(t)\|_{\H}^2 \, dt     
\end{equation}
for a suitable positive definite energy functional $E_X$ in $\H$,
\begin{equation}\label{en0-pos}
    E_X(v(t)) \approx \|v[t]\|_{\H}^2 .
\end{equation}

This may also be interpreted as a differential energy identity,
\begin{equation}\label{en0-diff}
\frac{d}{dt}E_X(v(t)) = \int T_{\tP} v \cdot T_{\tM} v\, dx + O(\BB^2) \| v(t)\|_{\H}^2   .  \end{equation}
\medskip

\textbf{ Notation for errors:} 
There are two types of error/correction terms that appear in our computations:
\begin{itemize}
\item Corrections in the energy functional. Here we will denote by 
$Err(\AA)$ any fixed time expressions which have size $O(\AA) \| v[t]\|_{\H}^2$. A typical example here is a lower order term of the form
\[
\int_{\R^n} \partial v \cdot T_q v \, dx, \qquad q \in \partial_x \PP,
\]
where 
\[
\| P_{<k} q \|_{L^\infty} \lesssim 2^{k} \AA.
\]
This holds for instance if $q \in \partial_x \PP$.

\item Corrections in the energy flux term. These are like the last term on the right in \eqref{en0-int}, respectively \eqref{en0-diff}. For brevity we will denote the admissible errors in the two identities by $Err(\BB^2)$.
\end{itemize}

To establish \eqref{en0-int}, we consider the contributions of the two terms 
in $T_{\tM}$. 
\medskip

\emph{ I. The contribution of $T_{\tX^\alpha} \partial_\alpha$.} 
Integrating by parts and commuting, this is given by 
\[
\begin{aligned}
I_X = & \  \iintT T_{\tP_A} v \cdot  T_{\tX^\gamma} \partial_\gamma  v\, dxdt\\
= & \  \iintT ( \D_\alpha T_{\tg^{\alpha \beta}}
\D_{\beta} v - T_{\tA^\alpha} \partial_\alpha v)   \cdot T_{\tX^\gamma} \D_\gamma v  \, dx dt
\\
= & \  \iintT  (\D_\alpha T_{\tX^\gamma} T_{\tg^{\alpha \beta}}
\D_{\beta} v \cdot  \D_\gamma v -
T_{\D_\alpha X^\gamma} T_{\tg^{\alpha \beta}}
\D_{\beta} v \cdot  \D_\gamma v - T_{\tA^\alpha} \D_\alpha v \cdot   T_{\tX^\gamma} \D_\gamma v
\, dx dt 
\\
= & \  \iintT  \D_\gamma T_{X^\gamma} T_{\tg^{\alpha \beta}}
\D_{\beta} v \cdot  \D_\alpha v -
T_{\D_\alpha \tX^\gamma} T_{\tg^{\alpha \beta}}
\D_{\beta} v \cdot  \D_\gamma v - T_{\tA^\alpha} \D_\alpha v \cdot   T_{\tX^\gamma} \D_\gamma v
\, dx dt 
\\&  + \left.  \int   T_{\tX^\gamma} T_{\tg^{0 \beta}}
\D_{\beta} v \cdot  \D_\gamma v -    T_{\tX^0} T_{\tg^{\alpha \beta}}
\D_{\beta} v \cdot  \D_\alpha v          \, dx \right|_0^T
\\ = & \  \iintT  \frac12 (\D_\gamma T_{\tX^\gamma} T_{\tg^{\alpha \beta}}\! - \! T_{\tg^{\alpha \beta}} T_{\tX^\gamma} \D_\gamma)
\D_{\beta} v \cdot  \D_\alpha v \!-\!
T_{\D_\alpha \tX^\gamma} T_{\tg^{\alpha \beta}}
\D_{\beta} v \cdot  \D_\gamma v\!-\! T_{\tA^\alpha} \D_\alpha v \cdot   T_{\tX^\gamma} \D_\gamma v
\, dx dt 
\\&  + \left.  \int   T_{\tX^\gamma} T_{\tg^{0 \beta}}
\D_{\beta} v \cdot  \D_\gamma v -    T_{\tX^0} T_{\tg^{\alpha \beta}}
\D_{\beta} v \cdot  \D_\alpha v  + \frac12  T_{\tX^0} T_{\tg^{\alpha \beta}}
\D_{\beta} v \cdot \D_\alpha v     \, dx \right|_0^T.
\end{aligned}
\]
For the double integral we peel off some perturbative contributions. The first term has a commutator structure, and we distinguish several cases. If $(\alpha, \beta) = (0,0)$, 
then we simply write
\[
\D_\gamma T_{X^\gamma} T_{\tg^{0 0}} -  T_{\tg^{00}} T_{X^\gamma} \D_\gamma = \tg^{00} T_{\D_\gamma \tX^\gamma}.
\]
If $(\alpha,\beta)= (0,j)$ then we commute the derivative first,
\[
\D_\gamma T_{ \tX^\gamma} T_{\tg^{0 j}} -  T_{\tg^{0j}} T_{ X^\gamma} \D_\gamma
= T_{\D_\gamma \tX^\gamma} T_{\tg^{0j}} + T_{\tX^\gamma} T_{\D_\gamma\tg^{0j}} 
+ [T_{\tX^\gamma},T_{\tg^{0j}}] \partial_\gamma,
\]
where the contribution of the commutator term is estimated using Lemma~\ref{l:para-com},
\[
\| [T_{\tX^\gamma},T_{\tg^{0j}}] \partial_x\|_{L^2 \to L^2} \lesssim \BB^2.
\]
If $(\alpha,\beta)= (j,0)$ then we commute the paraproducts first,
\[
\D_\gamma T_{ \tX^\gamma} T_{\tg^{j0}} -  T_{\tg^{j0}} T_{ \tX^\gamma} \D_\gamma
=  T_{\D_\gamma\tg^{0j}} T_{\tX^\gamma} +  T_{\tg^{j0}} T_{\D_\gamma \tX^\gamma}
+ \partial_\gamma [T_{\tX^\gamma},T_{\tg^{j0}}] ,
\]
where the contribution of the commutator is again perturbative once we integrate by parts
with respect to $x^\gamma$. If $\gamma=0$ then this integration by parts  contributes 
to the energy with the expression 
\[
 \int [T_{\tX^0},T_{\tg^{j0}}] \partial_0 v \cdot \partial_j v \, dx,
\]
which also plays a perturbative role. For the double paraproducts we use Lemma~\ref{l:para-prod} to compound them, as in 
\[
\| T_g T_{\partial h} - T_{g \partial h}\|_{L^2 \to L^2} \lesssim \BB^2.
\]
We arrive at the relation 
\begin{equation}\label{IX-rep}
I_X =   \iint T_{\tP_A} v \cdot  T_{\tX^\gamma} \partial_\gamma  v\, dxdt = \iint T_{c^{\alpha\beta}_{X}} \partial_\alpha v 
  \cdot \partial_\beta v \, dx dt+ \left. E_{X}(v)\right|_0^T + Err(\BB^2),
\end{equation}
where we recall that $c^{\alpha\beta}_{X}$ is given by 
the relation \eqref{tcX},
and the energy functional $E_X$ is given by 
\[
E_X(v) =  \int   T_{\tX^\gamma} T_{\tg^{0 \beta}}
\D_{\beta} v \cdot  \D_\gamma v -    T_{\tX^0} T_{\tg^{\alpha \beta}}
\D_{\beta} v \cdot  \D_\alpha v  + \frac12  T_{\tX^0} T_{\tg^{\alpha \beta}}
\D_{\beta} v \cdot \D_\alpha v   +  [T_{X^0},T_{\tg^{j0}}] \partial_0 v \cdot \partial_j v
\, dx .
\]
Here we may compound all double paraproducts and discard the commutator term, 
at the expense of  $Err(\AA)$ errors.  We arrive at 
\begin{equation}\label{EX-main}
E_X(v) =  \int T_{e^{\alpha \beta}_{X}} \partial_\alpha v \partial_\beta v \, dx
+ Err(\AA),
\end{equation}
with $e^{\alpha \beta}_{X,2} \in \PP$ given by \eqref{eX}. Since $X = \partial_t +O(A)$
is uniformly time-like, it follows that this matrix is positive definite, which implies
the positivity property in \eqref{en0-pos}.

\medskip

\emph{II. The contribution of $T_q$.}
Here we need to consider the integral 
\[
I_q =  \iint (\D_\alpha T_{\tg^{\alpha\beta}} \partial_\beta + \tA^\gamma \partial_\gamma) v \cdot T_\tq v\, dx dt,
\]
where we recall that $\tq = \partial_x q_1$ with $q_1 \in \PP$.
The contribution of $\tA$ is directly perturbative,
as $\tA \in \DPP$. Integrating by parts and using Lemmas~\ref{l:para-prod},
\ref{l:para-prod2},
\[
\begin{aligned}
I_q = & \  \iint T_{\tg^{\alpha\beta}} \partial_\beta v \cdot ( T_q \partial_\alpha + T_{\partial_\alpha \tq}) v \, dx dt + \left. \int   T_{\tg^{0\beta}} \partial_\beta v \cdot T_\tq v\, dx \right|_0^T
\\
= & \  \iint  T_{\tg^{\alpha\beta} \tq} \partial_\beta v \cdot \partial_\alpha v \, dx dt + 
\iint
\partial_\beta v \cdot T_{T_{\tg^{\alpha\beta}} \partial_\alpha \tq} v \, dx dt + 
Err(\BB^2) + \left. \int   T_{\tg^{0\beta}} \partial_\beta v \cdot T_\tq v\, dx \right|_0^T.
\end{aligned}
\]
Here the first term on the right is the one we want and the last term on the right 
yields an energy correction which is perturbative, i.e. of size $Err(\AA)$.
It remains to show that the second term, which we shall denote by $I_q^2$, also yields only perturbative contributions.
Heuristically, this should be relatively simple, in that we can integrate once more by parts,
to obtain 
\[
I_q^2:=\iint
\partial_\beta v \cdot T_{T_{\tg^{\alpha\beta}} \partial_\alpha \tq} v \, dx dt  
= -\frac12 \iint
 v \cdot T_{ \partial_\beta T_{\tg^{\alpha\beta}} \partial_\alpha \tq} v \, dx dt 
+ \frac12 \left. \int
 v \cdot T_{T_{\tg^{\alpha0}} \partial_\alpha \tq} v \, dx \right|_0^T .
\]
Here we could estimate both integrals perturbatively and conclude directly if we knew that 
\[
\| P_{<k} \partial_\beta T_{\tg^{\alpha\beta}} \partial_\alpha \tq\|_{L^\infty} \lesssim 2^{2k}
\BB^2, \qquad \| P_{<k} \partial_\alpha \tq \|_{L^\infty} \lesssim 2^{2k} \AA.
\]
Both of these bounds would be true if $q$ contained no time derivatives of $u$ in its expression.
However, this is too much to hope for, so a more careful argument is needed. The first step 
in this argument has already been carried out earlier, where we saw that we may take 
$\tq$ of the form $\tq = \partial_x q_1$ with $q_1 \in \PP$. This removes one of the two 
potential time derivatives in $q$, but not the second. We can use this property to write
\[
I_q^2 = \iint
\partial_\beta v \cdot T_{\partial_x T_{\tg^{\alpha\beta}} \partial_\alpha q_1} v \, dx dt 
- \iint
\partial_\beta v \cdot T_{T_{\partial_x \tg^{\alpha\beta}} \partial_\alpha q_1} v \, dx dt ,
\]
where the uniform bound
\[
\| P_{<k} T_{\partial_x \tg^{\alpha\beta}} \partial_\alpha q_1 \|_{L^\infty}
\lesssim 2^{k} \BB^2
\]
shows that we can treat the second term perturbatively, to get
\[
I_q^2 = \iint
\partial_\beta v \cdot T_{\partial_x T_{\tg^{\alpha\beta}} \partial_\alpha q_1} v \, dx dt 
+ Err(\BB^2).
\]
At this point, we can use the fact that $q_1 \in \BB$ implies that $q_1$ solves an approximate
paradifferential wave equation. The precise statement we use is the one in Lemma~\ref{l:pcbounds-extra}, which yields the representation
\[
\partial_\alpha T_{\tg^{\alpha\beta}} \partial_\beta q_1 = \partial_\alpha f_\alpha
\]
with 
\begin{equation}\label{fa-bds}
\| f_\alpha\|_{L^\infty} \lesssim \BB^2, \qquad \|P_{<k} (\partial_0 q_1 -f^0)\|_{L^\infty}
\lesssim \AA.
\end{equation}
We use this representation to refine the outcome of the naive integration by parts above,
\[
\begin{aligned}
I_q^2 = & \  -\frac12 \iint
 v \cdot T_{ \partial_x \partial_\beta T_{\tg^{\alpha\beta}} \partial_\alpha q_1} v \, dx dt 
+ \frac12 \left. \int
 v \cdot T_{\partial_x T_{\tg^{\alpha0}} \partial_\alpha q_1} v \, dx \right|_0^T + Err(\BB^2)
 \\
  = & \  -\frac12 \iint
 v \cdot T_{ \partial_x \partial_\alpha f^\alpha} v \, dx dt 
+ \frac12 \left. \int
 v \cdot T_{\partial_x T_{\tg^{\alpha0}} \partial_\alpha q_1} v \, dx \right|_0^T + Err(\BB^2)
  \\
  = & \  \iint
 \partial_\alpha v \cdot T_{ \partial_x  f^\alpha} v \, dx dt 
+ \frac12 \left. \int
 v \cdot T_{\partial_x( T_{\tg^{\alpha0}} \partial_\alpha q_1- f^0)} v \, dx \right|_0^T + Err(\BB^2).
\end{aligned}
\]
By the pointwise bound on $f^\alpha$ in \eqref{fa-bds}, the first term is perturbative, i.e. $Err(\BB^2)$.
By the second bound in \eqref{fa-bds}, the second term can be seen as a perturbative $Err(\AA)$ energy correction. We conclude that for $I_q$ we have
\begin{equation}\label{Id-rep}
I_q = \iint  T_{\tg^{\alpha\beta} q} \partial_\beta v \cdot \partial_\alpha v \, dx dt
+ \left. \int   T_{\tg^{0\beta}} \partial_\beta v \cdot T_q v
+\frac12 v \cdot T_{\partial_x( T_{\tg^{\alpha0}} \partial_\alpha q_1- f^0)} v
\, dx \right|_0^T + Err(\BB^2).
\end{equation}

\medskip

\emph{ III. Conclusion.}
To finish the proof of \eqref{en0-int}, and thus of Theorem~\ref{t:para-wp} for $s=0$,
we combine the relations \eqref{IX-rep} and \eqref{Id-rep} to obtain
\begin{equation}
2\iint T_{\tP} v \cdot T_{\tM} v\, dx dt = 
\iint T_{c^{\alpha\beta}_{X} + \tg^{\alpha\beta} \tq} \partial_\alpha v
\cdot \partial_\beta v \, dx dt
+ \left. E_X(v(t)) \right|_0^T   
+ Err(\BB^2),
\end{equation}
where $E_X$ is redefined as the sum of the two contributions in \eqref{IX-rep} and \eqref{Id-rep}, which still has the leading order term as in \eqref{EX-main} plus an $Err(\AA)$ correction.

It remains to examine the paracoefficient in the integral on the right, and show that it has size $O(\BB^2)$. 
At this point, we simply invoke the choice of our para-Killing vector field $X$ in Lemma \eqref{l:X} for the first term
(which we have not used so far), and the choice of $\tq$ in \eqref{choose-d} for the second term,
thereby completing the proof  of \eqref{en0-int}.

\bigskip

\subsection{The  $H^{s+1} \times H^s$ bound for the linear paradifferential flow} Here we prove Theorem~\ref{t:para-wp} in the general case,
where $s \neq 1$. The argument will be a more complex variation of the 
argument in the case $s=1$, where paraproduct based multipliers 
have to be replaced by paradifferential multipliers.

\subsubsection{ The conjugated equation} For simplicity in notations we will consider the linear paradifferential 
equation in $\H^{s+1}$ with $s \neq 0$.
We begin by setting $w = \bD^s v$, which
solves a perturbed linear paradifferential equation of the form
\begin{equation}\label{paralin-inhom+}
( \D_{\alpha} T_{\tg^{\alpha \beta}} \D_{\beta} - 2 T_{\tA^\gamma}\D_\gamma)w  = \bD^s f + \tB w ,
\end{equation} 
where the conjugation error $\tB$ in the new source term is given by 
\begin{equation}\label{def-R}
\tB = \bD^s [ \D_{\alpha} T_{\tg^{\alpha \beta}} \D_{\beta} -  T_{\tA^\gamma}\D_\gamma, \bD^{-s}] .
\end{equation}
Then we need to construct an $H^1 \times L^2$ balanced energy for the solution $w$ to \eqref{paralin-inhom+}. 
\medskip

We note that $\tilde B$ is a paradifferential operator, whose principal symbol $\tb_0$ is homogeneous  of order one and a first degree polynomial in the time variable $\xi_0$, and is given by 
\[
  \tb_0(x,\xi)= - i |\xi'|^s \{ \tg^{\alpha \beta} \xi_{\alpha} \xi_{\beta},|\xi'|^{-s}\}.
\]
Using the expression \eqref{d-gab} for the derivatives of the metric $g$, this can be further written in the form
\begin{equation}\label{b0}
\begin{aligned}
\tb_0(x,\xi) = & \  2is (\partial^\beta u \tg^{\alpha \nu}  \partial_j \partial_\nu u\  -  \partial^0 u \tg^{0\nu} \partial_j \partial_\nu u \tg^{\alpha\beta}) \xi_\alpha \xi_\beta \xi_j |\xi'|^{-2}
\\
:= & \ 2 i s \tb_0^\alpha \xi_\alpha ,
\end{aligned}
\end{equation}
where
\begin{equation} \label{tb0-alpha}
 \tb_0^\alpha=   (\partial^\beta u \, \tg^{\alpha \nu} 
 - \partial^0 u \tg^{0\nu} \partial_j \partial_\nu u \tg^{\alpha\beta})
 \partial_j \partial_\nu u\ \xi_\beta \xi_j |\xi'|^{-2}  .
\end{equation}
Here the unbalanced part of the coefficients corresponds to the case when the factor $\partial^2 u$
is higher frequency compared to the $\partial^\beta u$ and $\tg^{\alpha \nu}$ factors. The important feature is that, at the operator level, $T_{\tb_0^\gamma} \partial_\gamma w$ presents a null form structure of the type $Q_0(\partial u, w)$, with added more regular paradifferential coefficients in $\PP$. 

We switch the term $2sT_{\tb_0^\gamma} \partial_\gamma$ to the left hand side of the equation; there it will play a role similar to the 
gradient term $\tA^\gamma \partial_\gamma$. The reminder $\tB- 2sT_{\tb_0^\gamma} \partial_\gamma$ will play a secondary role; one should think of it as renormalizable, though we will achieve this at the level of the energy, via an energy correction, rather than through an actual normal form transformation. Our equation \eqref{paralin-inhom+} becomes
\begin{equation}\label{paralin-inhom++}
(\D_{\alpha} T_{\tg^{\alpha \beta}} \D_{\beta} -  T_{\tA^\gamma}\D_\gamma - 2s T_{\tb_0^\gamma}\D_\gamma)w  = \bD^s f + (\tB-2s T_{\tb_0^\gamma}\partial_\gamma) w ,
\end{equation} 
where the leading operator is denoted by 
\begin{equation}
T_{\tP_{B}} =  \D_{\alpha} T_{\tg^{\alpha \beta}} \D_{\beta} -  T_{\tA^\gamma}\D_\gamma - 2s T_{\tb_0^\gamma}\D_\gamma.    
\end{equation}

As in the previous case of the $H^1 \times L^2$ bounds, our strategy will be to construct a suitable vector field, or multiplier, denoted $\tX_s$, which depends only on the principal symbol $b_0$ above, and which formally generates a balanced energy estimate at the leading order. Then, reinterpreting all the analysis at the paradifferential level, we will rigorously prove that the  generated energy satisfies favourable, balanced bounds.

\bigskip

\subsubsection{The multiplier $\tX_s$.}
In the previous section, the multiplier $\tX \in \PP$ was a well-chosen vector field which belongs to our space $\PP$ of paracontrolled distributions.
Here, this can no longer work due to the the presence of the operator $T_{\tb_0}$, which is a pseudodifferential rather than a differential operator. For this reason we will instead use a pseudodifferential ``vector field" $i \tX_s$, where $\tX_s$ has a real symbol of the form
\[
\tX_s(x,\xi) = \tX_{s1}(x,\xi') + \tX_{s0}(x,\xi') \xi_0, \qquad   
X_{sj} \in \PP S^j,
\]
which will be homogeneous away from frequency zero.
We carefully note that we want the symbol $\tX_s$ to be a first order polynomial in 
$\xi_0$; this is important so that  we can still do integration by parts in time and have a well defined fixed time energy. The symbol $\tX_s$ may be interpreted as a pseudodifferential operator using  the Weyl paradifferential quantization, 
\begin{equation}
T_{\tX_s} =    i T_{\tX_{s1}}  + T_{\tX_{s0}} \partial_0  +\frac12 T_{\partial_0 \tX_{s0}}.
\end{equation}
However, as in the $s=0$ case, we will allow a more general choice for the 
zero order component, and work instead with the modified multiplier
\begin{equation}\label{hatXs}
T_{i\tM_s} :=   i T_{\tX_{s1}} + T_{\tX_{s0}} \partial_0  +\frac12 T_{\tY_0}, 
\end{equation}
where the zero order symbol $\tY_0 \in \partial_x\PP S^0$ will be carefully chosen later on
in order to provide an appropriate Lagrangian correction in our energy estimates.

Repeating the heuristic computation in the previous subsection, in the absence of time boundaries we have  an identity of the form
\begin{equation}\label{IBP-nobdr-A+}
2\iint T_{\tP_{A,B}} v \cdot  (i T_{\tX_{s1}} + T_{\tX_{s0}} \partial_0) v \, dx dt = \iint c_{X_s} (v,v) \, dx dt,
\end{equation}
where $c_{\tX_s,B}(v,v)$ is a bilinear form whose
principal symbol $c_{\tX_s,B}$ is of order two, 
\begin{equation}\label{tcXA-s}
c_{\tX_s,B}(x,\xi) =  \{ \tp,\tX_s\}(x,\xi) +   2\tX_s(x,\xi)( \tA^\gamma + 2s \tb^\gamma_0(x,\xi))  \xi_\gamma - \partial_0 \tX_{s0} \tp(x,\xi).
\end{equation}

The objective would now be to choose the symbols $\tX_{sj} \in \PP S^j$
 so that we cancel the unbalanced part of the symbol $\tc_{X,B}$.
However, it is immediately clear that this may be a bit too much to ask, as it conflicts with the requirement  that $\tX_s$ is a first degree polynomial in $\xi_0$. Hence, as a substitute,  we will seek to  achieve this cancellation on the characteristic set $p(x,\xi) = 0$. Then, instead of asking for
\[
c_{\tX_s,B} \bapprox 0,
\]
we will settle for the slightly weaker property
\[
c_{\tX_s,B}(x,\xi) \bapprox  \tY(x,\xi') \cdot  \tp(x,\xi),
\]
where $\tY \in \DPP S^0$ is a purely spatial zero homogeneous symbol, with the spatial dependence at the level of $\partial^2 u$.  This term will be harmless, as we will also be able to remove it in our energy energy estimates
with a Lagrangian correction, by making a good choice for $\tY_0$.

An additional requirement on our paradifferential ``vector field" $\tX_s$ will be that, in the energy estimate generated by $\tX_s$, the associated energy functional $E_{\tX_s}$ should be positive definite at the level of it principal part. Earlier, in the case when $X$ was a vector field, this requirement was identified, via the energy momentum tensor, with the property that $X$ is forward time-like. Here we will generalize this notion to symbols:

\begin{definition}
We say that the (real) symbol $X= \xi_0 X_{0}+ X_{1} \in C^0 S^1$ is \emph{forward time-like} if the following two properties hold:

a) $X_{0}(x,\xi') > 0$.

b) $X(x,\xi_0^1,\xi') X(x,\xi_0^2, \xi') < 0$, where $\xi_0^1(x,\xi') < \xi_0^2(x,\xi')$ are the two real zeros of $p(x,\xi)$ as a polynomial of $\xi_0$.
\end{definition}

We remark that, using $X$ as a multiplier, relative to the metric $g$, we will obtain an energy  functional which at leading order can be described via the symbol 
\begin{equation}\label{eX-s}
\begin{aligned}
e_{X}(x,\xi) = & \ g^{\alpha \beta} \xi_\alpha \xi_\beta X_0 - 2 g^{0\alpha} \xi_\alpha
(X_1 + X_0 \xi_0)
\\
:= & \ e_{X_s}^0(x,\xi') \xi_0^2 + e_{X_s}^1(x,\xi') \xi_0 + e_{X_s}^2(x,\xi'),
\end{aligned}
\end{equation}
which should be compared with the expression \eqref{eX} defined earlier
in terms of the energy momentum tensor in the case when $X$ is a vector field. Correspondingly, we define the energy functional
\begin{equation}\label{EX-s}
E_{X}[w] =   \int - T_{e_{X}^0} \partial_t w \cdot \partial_t w
+ T_{ie_{X}^1}w \cdot \partial_t w
+ T_{e_{X}^2} w \cdot w \, dx.
\end{equation}
The main property of forward time-like symbols is as follows:

\begin{lemma}
The symbol $e_{X_s}$ is positive definite iff $X$ is forward time-like.
\end{lemma}
\begin{proof}
Assuming $X_0$ is nonzero, we represent $X$ in the form 
\[
X = X_0 ( a_1 (\xi_0 -\xi_0^1) + a_2(\xi_0-\xi_0^2))  ,
\]
where $a_1+a_2 = 1$. Then $e_X$ has the form
\[
\begin{aligned}
e_X = & \  g^{00}(\xi_0 -\xi_0^1)(\xi_0-\xi_0^2)  X_0 
-g^{00}(2 \xi_0 - \xi_0^1-\xi_0^2)   X_0 ( a_1 (\xi_0 -\xi_0^1) + a_2(\xi_0-\xi_0^2))
\\
= & - g^{00} X_0 [a_1(\xi_0 - \xi_0^1)^2 +  a_2(\xi_0 - \xi_0^2)^2].
\end{aligned}
\]
Here $g^{00} = -1$ and at least one of $a_1$ and $a_2$ are positive.
Then $e_X$ is positive definite iff $X_0 > 0$ and $a_1, a_2 > 0$.
This is easily seen to be equivalent with the forward time-like
condition in the above definition. 

\end{proof}

\subsubsection{ The construction of $X^s$} Here we return to the matter of choosing $\tX_s$, whose properties almost exactly mirror  those of the vector field $\tX$ in the previous subsection:

\begin{proposition}\label{p:Xs}
There exists a real homogeneous symbol of order one $\tX_s \in \PP S^1$, which is a first degree polynomial in $\xi_0$, so that:

i) $\tX_s$ is forward time-like.

ii) The principal symbol $c_{X_s,B}$ of the $\tX_s$ energy flux admits a representation
of the form
\begin{equation}\label{cX2-rep}
c_{X_s,B}(x,\xi) = \tq_2(x,\xi) + \tq_0(x,\xi')\tp(x,\xi),
\end{equation}
where $\tq_2$ is balanced,
\begin{equation}\label{q2}
\|\tq_2\|_{L^\infty S^2} \lesssim \BB^2  ,
\end{equation}
and $\tq_0$ has $\DPP$ type regularity,
\begin{equation}\label{q0}
\| \tq_0\|_{ \DPP S^0} \lesssim A   .
\end{equation}

iii) The symbol $\tX_s$ admits the $\PP S^1$ representation
\begin{equation}\label{Xs-PP-re}
 X_s = T_{a^\gamma} \partial_\gamma u +  r_s,
\end{equation} 
where  the para-coefficients $a^\gamma(x,\xi) = a^\gamma_1(x,\xi') + a^\gamma_0(x,\xi')$ with $a^\gamma_j
\in \PP S^j$  have the form
 \begin{equation}\label{a-gamma}
 a^\gamma = - \xi_\delta \partial^\delta u  
\partial_{\xi_\gamma} \tX_s   -  \partial^\gamma u \tX_s -2 \tX^s \partial^0 u \tg^{0\gamma}  -   
s \xi_\delta \partial_{\xi_\gamma} \log |\xi'|^2
  \partial^\delta u \tX_s + p q_0^\gamma
\end{equation}
with $q_0^\gamma  \in \PP S^0$, independent of $\xi_0$.
\end{proposition}
From the perspective of energy estimates, it might seem that parts (i) and (ii) are 
the important ones. However, part (ii) will be seen as an immediate consequence 
of the representation in part (iii), which thus can be thought of as the more fundamental property. 
Also, in the proof of the energy estimates it will  on occasion be more convenient to 
directly use \eqref{a-gamma}.  In the sequel we will refer to $a^\gamma$ as the para-coefficients of $\tX_s$. We note that the choice of $q_0^\gamma$
is uniquely determined by the requirement that $a^\gamma$ are first degree polynomials in $\xi_0$.

\begin{proof}
It will be somewhat easier to construct the corresponding symbol $X_s$ as associated to  $p$, rather than to $\tp$; this avoids the slight symmetry breaking in the transition from $p$ to $\tp$. Precisely, we will choose $\tX_s$ of the form
\[
\tX_s = g^{00} X_s,
\]
and then express $c_{\tX_s,B}$ in terms of $X_s$ as follows:
\[
c_{\tX_s,B} = \{ p, X_s\} + 2 X_s (A^\gamma + 2sb^\gamma_0)\xi_\gamma + q_{00} p,
\]
where $q_{00}$ is given by
\[
q_{00} = \{ X_{1s}, \log g^{00}\} + \xi_0 \{X_{0s},\log g^{00}\} -  \partial_0 X_{0s}
- 2 s X_s \partial^0 u \, \tg^{0\nu} \partial_j \partial_\nu u 
   \xi_j |\xi'|^{-2}  ,
\]
and $b_0^\gamma$ has the form
\begin{equation}\label{b0-alpha}
b^\gamma_0 =  \partial^\beta u \, g^{\alpha \nu} 
 \partial_j \partial_\nu u\ \xi_\beta \xi_j |\xi'|^{-2}.  
\end{equation}
Here we have separated the two terms in $\tb_0^\gamma$; the first has contributed to $b_0^\gamma$, while the second has contributed the last term in the Lagrangian coefficient $q_{00}$. 

For clarity, we note that the exact expression of $q_{00}$
is not important, we will only use the fact that $q_{00} \in \DPP S^{0}$.
On the other hand, for $b_0^\gamma$ we will need the fact that it has a null structure.

Now we restate the proposition in terms of the new symbol $X_s$. Our goal will be 
to find $X_s$ in the same class as $\tX_s$, so that the reduced symbol
\begin{equation}\label{cred-def}
c^{red}_{X_s,B} = \{ p, X_s\} + 2X_s (A^\gamma + 2s b^\gamma_0)\xi_\gamma
\end{equation}
can be represented in the form 
\begin{equation}\label{cred-rep}
c^{red}_{X_s,B} =q_2(x,\xi) + q_0(x,\xi)p(x,\xi).
\end{equation}

Here there is a small twist in the argument. While $c_{X,B}$ is a second degree 
polynomial in $\xi_0$, this is no longer the case for $c^{red}_{X,B}$, 
which contains the term $\xi_0^3 \{ g^{00},X_{s0}\}$. For this reason, in \eqref{cred-rep} we apriori have to allow for symbols $q_2$, respectively $q_0$ which are third, respectively first degree polynomials in $\xi_0$. However, we can eliminate the $\xi_0^3$  term in $q_2$ with a $\xi_0$ correction in $q_0$. Then, returning to $c_{X,B}$, we obtain the representation \eqref{cX2-rep} with $\tq_2$ of second degree and $\tq_0$ of first degree. But $c_{X,B}$ is a second degree 
polynomial in $\xi_0$, so we finally conclude that $\tq_0$ must be independent of $\xi_0$.

We now proceed to construct the symbol $X_s$. 
As a first step in the proof, we seek to obtain a variant $\zX_s$ of the symbol $X_s$
where we drop the requirement that $\zX_s$ is a first order polynomial 
in $\xi_0$ but we ask for the stronger property that the associated symbol $c^{red}_{\zX_s,B}$ is fully balanced, which corresponds to $q_0=0$. Then, at the end, we 
choose $X_s$ to be first degree polynomial in $\xi_0$ which matches $\zX_s$ at the two roots of $p(x,\xi) = 0$ viewed as a polynomial in $\xi_0$.

The relation we seek for $\zX_s$ to satisfy on the characteristic set of $p$ is
\begin{equation}\label{tX-a}
  c^{red}_{\zX_s,B} \bapprox 0,  
\end{equation}
where, using the expressions \eqref{def-A} and \eqref{b0-alpha} for $A^\gamma$ and $b^\gamma$,
\[
c^{red}_{\zX_s,B} = \{ g^{\alpha \beta} \xi_\alpha \xi_\beta, \zX_s\} + 
2\partial^\alpha u \xi_\delta g^{\beta \delta} \partial_\alpha \partial_\beta u \cdot \zX_s 
+4s \zX_s \partial^\beta u g^{\alpha \nu}  \partial_j \partial_\nu u\ \xi_\alpha \xi_\beta \xi_j |\xi'|^{-2}.
\]

Here we recall the expression for the derivatives of $g$, see \eqref{d-gab},
\begin{equation}\label{dg}
\partial_\gamma g^{\alpha \beta} \xi_\alpha \xi_\beta
= -2 \partial^\beta u  g^{\alpha \delta} \partial_\delta \partial_\gamma u\,  \xi_\beta \xi_\alpha.
\end{equation}
Substituting this in the previous expression for $c^{red}_{\zX,B}$, we need the following relation to hold modulo balanced terms:
\[
\begin{aligned}
2 \xi_\alpha g^{\alpha \beta} \partial_\beta \zX_s
\bapprox & \ -2 \partial^\delta u \xi_\delta 
\partial_{\xi_\gamma} \zX_s
\cdot \xi_\alpha g^{\alpha \beta} \partial_\beta \partial_\gamma u  - 2\zX_s \partial^\gamma u \cdot 
\xi_\alpha g^{\beta \alpha} \partial_\gamma \partial_\beta u 
\\ & 
- 4s \zX_s  \partial^\delta u
\xi_\delta \xi_j |\xi'|^{-2}
  \cdot \xi_\alpha g^{\alpha \beta}  \partial_j \partial_\beta u .\ 
\end{aligned}
\]
  We can rewrite this using  the following operator
\[
L = \xi_\alpha g^{\alpha \beta} \partial_\beta
\]
in the form
\begin{equation}\label{tX-aa}
L \zX_s \bapprox  - \xi_\delta \partial^\delta u  
\partial_{\xi_\gamma} \zX_s
\cdot L \partial_\gamma u  -  \partial^\gamma u \zX_s \cdot 
L \partial_\gamma  u -  s  \partial^\delta u \zX_s
\xi_\delta \partial_{\xi_j} \log |\xi'|^2
  \cdot L \partial_j  u .
\end{equation}

By Lemma~\ref{l:X}, we already have a solution $X$ for $s=0$. 
Thinking of this multiplicatively, it is then natural  to look for $\zX_s$ of the form  
\[
\zX_s = Z_s X,
\]
where $Z_s \llcurly \partial u$ should be zero homogeneous in $\xi$ and must satisfy
\begin{equation}\label{LY}
L  Z_s \bapprox - \xi_\delta  \partial^\delta u( 
\partial_{\xi_\gamma}Z_s
 L \partial_\gamma u  + s   Z_s \partial_{\xi_j} \log |\xi|^2   
 L \partial_j  u ).
\end{equation}
We will also assume that $Z_s$ is a positive symbol; this will help later with the time-like condition. Then we can rewrite the above relation as a condition for $\log Z_s$, namely
\begin{equation}\label{LlogY}
L \log Z_s \bapprox - \xi_\delta  \partial^\delta u( 
\partial_{\xi_\gamma}( \log Z_s )
 L \partial_\gamma u  + s   \partial_{\xi_j} \log |\xi|^2   L \partial_j  u ).
\end{equation}
Here the inhomogeneous term is linear in $s$, so we will also look for a solution $\log Z_s$ which is linear in $s$.

There is one last algebraic simplification, which is to replace $Z_s$ by 
$\tilde Z_s = Z_s |\xi'|^{2s}$, which is $2s$-homogeneous and inherits the property that 
$\log \tilde Z_s$ is linear in $s$. Then $\log \tilde Z_s$ must solve
\begin{equation}\label{LlogtY}
 L \log \tilde Z_s \bapprox  - \xi_\delta \partial^\delta u 
\partial_{\xi_\gamma} \log \tilde Z_s  L \partial_\gamma u  .
\end{equation}
Dispensing with the log, we replace this by
\begin{equation}\label{LlogtYa}
 L \tilde Z_s \bapprox - \xi_\delta \partial^\delta u 
\partial_{\xi_\gamma}\tilde Z_s  L \partial_\gamma u  .
\end{equation}

Now we interpret the last relation paradifferentially, formally cancelling the $L$'s.
This suggests the following scheme to construct the dyadic parts of $\tilde Z_s$ inductively by setting
\begin{equation}\label{tZs-def}
\begin{aligned}
&\tilde Z_{s<0} =  \  |\xi|^{2s}
\\
&\tilde Z_{sk} =  \   \xi_\delta  T_{\partial^\delta u} 
T_{\partial_{\xi_\gamma} \tilde Z_{s}}
 \partial_\gamma u_k, \qquad  k \geq 1.
\end{aligned} 
\end{equation}
Since $\log \tilde  Z_s$ is linear in $s$, it suffices to solve this for some nonzero $s$.
The advantage here is that, if $s$ is a positive integer (say $s=1$) then all our iterates are polynomials of degree  $2s$ in $\xi$.  Hence the convergence issue disappears, due to our smallness condition for $u$, $\AA \ll 1$;  this is exactly as in the construction of $X$ in Section~\ref{s:def-X}.
 This defines $\tilde Z_1$ as a positive definite polynomial in $\xi$ of degree $2$, so that $\tilde Z_1 = \xi^2 (1+O(\AA))$.
Further, by the same argument as in the proof of Lemma~\ref{l:X}, it follows 
that the coefficients of $Z_1$ are paracontrolled by $\partial u$; 
in other words, $Z_1 \in \PP S^2$. In addition, by \eqref{tZs-def}, it also follows that (a choice for) the para-coefficients of $\tilde Z_1$,
as in Definition~\ref{d:paracontrol} is given by
\begin{equation}\label{tZ1-rep}
 \tilde Z_{1} =    T_{ \xi_\delta \partial^\delta u \partial_{\xi_\gamma} \tilde Z_{1}} \partial_\gamma u + r.
\end{equation}

\begin{remark}
We remark on the symbol $\tilde Z_1$, which is  quadratic in $\xi$
and para-commutes with $p$, in the sense that their Lie bracket is balanced and thus bounded by $\BB^2$. This symbol plays a role that is similar to that of the first order symbol $X$ constructed earlier. 
\end{remark}

Now that we have $\tilde Z_1$, for all real $s$ we may define 
\[
\tilde Z_s = \tilde Z_1^s,\qquad 
\zX_s = X (\tilde Z_1 /|\xi'|^2)^s.
\]
By Lemma~\ref{l:X} in the previous subsection we have $X \llcurly \partial u$.
Combining this with the similar property of $\tilde Z_1$, by the algebra and Moser properties of the space $\PP$ of paracontrolled distributions it follows that $\tX_s \llcurly \partial u$. Finally, combining the representations of $X$ and of $\tilde Z_1$ as paracontrolled distributions, as in \eqref{X-paracontrol}
and \eqref{tZ1-rep}, we obtain the corresponding $\PP$ representation for $\zX_s$ as in Definition~\ref{d:paracontrol} (see the relation \eqref{tX-aa})
\begin{equation}\label{tXs-PP}
 \zX_s =  - \xi_\delta T_{\partial^\delta u  
\partial_{\xi_\gamma} \zX_s}  \partial_\gamma u  -  T_{\partial^\gamma u \zX_s} \partial_\gamma  u -   
s \xi_\delta \partial_{\xi_j} \log |\xi'|^2
  T_{\partial^\delta u \zX_s} \partial_j  u + r_s.
\end{equation}
This in turn yields
the desired conclusion that $c^{red}_{\zX_s,B}$ is balanced, 
\begin{equation}\label{c-for-zX}
\|c_{\zX_s,B}\|_{L^\infty S^2} \lesssim \BB^2. 
\end{equation}
Indeed, the equivalent form \eqref{tX-aa} can be obtained by directly applying the operator $L$ in the relation \eqref{tXs-PP}; this is because the terms where the paracoefficients get differentiated are balanced, so we
are left with the terms where $L$ is applied to the main factors $\partial u$.

Now we carry out the last step of the proof, and define the symbol $X_s$ as the unique
first degree polynomial in $\xi_0$ with the property that 
\[
X_s(x,\xi) = \zX_s(x,\xi) \qquad \text{on} \quad \{g^{\alpha\beta}\xi_\alpha \xi_\beta = 0\}.
\]
We now show that this choice for $X_s$ has the desired properties. 

Recall that $\xi_0^1(x,\xi') < \xi_0^2(x,\xi')$ are the two real zeros of $p(x,\xi)$ as a polynomial of $\xi_0$, which are $1$-homogeneous and smooth in $\xi'$ and are also smooth functions of $\partial u$. Thus, 
\[
\xi_0^1, \xi_0^2 \in \PP S^1.
\]
The coefficients $X_{s0}$ and $X_{s1}$ in $X_s$ are obtained by solving a linear system,
\[
X_{s0} = \frac{ \zX_s(x,\xi_0^2,\xi') - \zX_s(x,\xi_0^1,\xi')}{ \xi_0^2 - \xi_0^1},
\qquad X_{s1} = \frac{ \zX_s(x,\xi_0^2,\xi')\xi_0^1 - \zX_s(x,\xi_0^1,\xi')\xi_0^2}{ \xi_0^2 - \xi_0^1}.
\]
By the algebra and Moser properties of the space $\PP$ of paracontrolled distributions,
it immediately follows that we have the symbol regularity properties $X_0 \in \PP S^0$ and $X_1 \in \PP S^1$.
By construction we also have a smooth division,
\begin{equation}\label{division}
X_s =  \zX_s + d p, 
\end{equation}
where we easily see that the quotient $d$ has regularity $d \in \PP S^{-1}$ by computing directly
\[
\begin{aligned}
d(x,\xi) = & \ \frac{X_s(x,\xi) - \zX_s(x,\xi)}{ p(x,\xi)}
\\
= & \ \frac{1}{g^{00}(\xi_0^1-\xi_0^2)}\left( \frac{ 
\zX_s(x,\xi) - \zX_s(x,\xi_0^1,\xi')}{\xi_0-\xi_0^1} -
\frac{\zX_s(x,\xi) - \zX_s(x,\xi_0^2,\xi')}{\xi_0-\xi_0^2}
\right).
\end{aligned}
\]
One may also interpret this as a form of the Malgrange preparation theorem in an easier case where the roots are separated.

We can now use \eqref{division} to relate $c^{red}_{X_s,B}$ with $c^{red}_{\zX_s,B}$:
\[
c^{red}_{X_s,B}= c^{red}_{\zX_s,B} + p(\{ p,d\} + d(A^\gamma + s b^\gamma_0) \xi_\gamma),
\]
which is exactly the desired representation \eqref{cred-rep}.

We can also use the relation \eqref{tXs-PP} for part (iii) of the proposition.
For this we first transition from $\zX_s$ to $X_s$. Using  \eqref{division}
and peeling off balanced terms, this gives the $\PP$ representation
\[
\begin{aligned}
X_s = & \  - \xi_\delta T_{\partial^\delta u  
\partial_{\xi_\gamma} X_s}  \partial_\gamma u  -  T_{\partial^\gamma u X_s} \partial_\gamma  u - s
\xi_\delta \partial_{\xi_j} \log |\xi'|^2
  T_{\partial^\delta u X_s} \partial_j  u 
\\
& \  \, - T_p \left(  \xi_\delta T_{\partial^\delta u  
\partial_{\xi_\gamma} d}  \partial_\gamma u  +   T_{\partial^\gamma u d} \partial_\gamma  u + s  
\xi_\delta \partial_{\xi_j} \log |\xi'|^2
  T_{\partial^\delta u d} \partial_j  u + d\right)
\\ 
&  \ \, + \left(T_d p + \xi_\delta T_{\partial^\delta u d \partial_{\xi_\gamma} p } \partial_\gamma u\right)+ r_s.
\end{aligned}
\]
In view of the paradifferential expansion  \eqref{est-R-for-du-inf}
for $g^{\alpha\beta}$,
 in the last bracket 
there is a leading order cancellation,
\[
T_d p + \xi_\delta T_{\partial^\delta u d \partial_{\xi_\gamma} p } \partial_\gamma u = r_s.
\]
This implies that $X_s$ admits a representation of the form
\begin{equation}\label{Xs-PP}
 X_s =  - \xi_\delta T_{\partial^\delta u  
\partial_{\xi_\gamma} X_s}  \partial_\gamma u  -  T_{\partial^\gamma u X_s} \partial_\gamma  u - \frac{s}2  
\xi_\delta \partial_{\xi_j} \log |\xi'|^2
  T_{\partial^\delta u X_s} \partial_j  u + 
 T_{p} z+
  r_s.
\end{equation}
where $z \in \PP S^{-1}$. 
%\red{ is given by [not needed, not used]
%\[
%z =  - \xi_\delta T_{\partial^\delta u  
%\partial_{\xi_\gamma} d}  \partial_\gamma u  -  %T_{\partial^\gamma u d} \partial_\gamma  u - \frac{s}2  
%\xi_\delta \partial_{\xi_j} \log |\xi'|^2
%  T_{\partial^\delta u d} \partial_j  u + d
%\]}
At this stage we only know that $z$ and $r_s$ are smooth as functions of $\xi_0$.
On the other hand, the remaining terms are at most second degree polynomials in $\xi_0$.
We claim that, without any restriction in generality, we may take $z$ independent of $\xi_0$ and then $r_s$ has to be at most second degree polynomial in $\xi_0$.

Subtracting a multiple of $p$ from all the paracoefficients above and discarding balanced contributions, we may reduce to the case of a first degree polynomial, i.e. to a relation of the form
\[
T_p z(x,\xi) + r_s(x,\xi) = Z_1(x,\xi') + Z_0(x,\xi') \xi_0,
\]
where $z \in \PP S^{-1}$ and $Z_j \in \PP S^j$, while $\partial r_s = O(\BB^2)$,
with full symbol regularity in $\xi$. We will show that in this case we must have
$\partial Z_j= O(\BB^2)$, again with full symbol regularity. This would imply
that we may take $z=0$ in the last relation.

We begin by differentiating this relation in $x$ and $t$, noting that
$T_{\partial p} z$ may be placed in $\partial r_s$:
\[
T_p \partial z(x,\xi) + \partial r_s(x,\xi) = \partial Z_1(x,\xi') + \partial Z_0(x,\xi') \xi_0.
\]
We may also perturbatively replace $T_p$ with $p$, arriving at
\begin{equation}
p \partial z(x,\xi) + r^1_s(x,\xi) = \partial Z_1(x,\xi') + \partial Z_0(x,\xi') \xi_0,
\end{equation}
where $r^1$ has size $\BB^2$ and symbol regularity, 
\[
| \partial_\xi^\alpha r_1| \lesssim \BB^2 |\xi|^{1-|\alpha|}.
\]
For fixed $x$, we examine this relation on the characteristic cone $C = \{p(x,\xi)= 0\}$.
There we have 
\[
| \partial Z_1(x,\xi') + \partial Z_0(x,\xi') \xi_0| \lesssim \BB^2 |\xi| ,
\]
so we may directly conclude that $\partial Z_1(x,\xi'), \partial Z_0(x,\xi') = O(\BB^2)$.
Next we need a similar bound for their derivatives 
$ \partial_{\xi'}^\alpha \partial Z_1(x,\xi'), \partial_{\xi'}^\alpha \partial Z_0(x,\xi')$ with respect to $\xi'$. We fix $x$ and argue by induction in $|\alpha|$.
Then it suffices to use derivatives which are tangent to the cone at that $x$,
which on one hand kill $p$ but on the other hand give a full range of $\xi'$ derivatives
for $Z_j$.
 
Lastly we switch from $X_s$ to $\tX_s$. Again peeling off $r_s$ type contributions, we have
\[
\begin{aligned}
\tX_s = & \ g^{00} X_s
\\
= & \ T_{g^{00}} X_s + T_{X_s} g^{00} + r_s
\\
= & \ - \xi_\delta T_{\partial^\delta u  
\partial_{\xi_\gamma} \tX_s}  \partial_\gamma u  -  T_{\partial^\gamma u \tX_s} \partial_\gamma  u - s  
\xi_\delta \partial_{\xi_j} \log |\xi'|^2
  T_{\partial^\delta u \tX_s} \partial_j  u +  T_{\tX_s} \log g^{00} +
 T_{p} T_{g^{00}} z +
  r_s.
\end{aligned}
\]
It remains to expand the fourth term, using Lemma~\ref{l:g-para}:
\[
T_{\tX_s} \log g^{00} = -2 T_{\tX^s \partial^0 u \tg^{0\alpha}} \partial_\alpha u + r_s .
\]
This finally yields the representation \eqref{Xs-PP-re}  with the paracoefficients in \eqref{a-gamma}, thereby concluding the proof of part (iii) of Proposition~\ref{p:Xs}.

The final property of $X_s$ to be verified is that $X_s$ is time-like. This property is easily seen to depend only on the sign of the symbol $X_s$ on the characteristic set $\{p=0\}$. But 
by construction, $X_s$ has the same sign as $\tX_s$ there, which in turn has the same sign as the vector field $X$ in Section~\ref{s:def-X}. Then the time-like property for $X_s$ follows
from the similar property of $X$.
\end{proof}

\bigskip

While it is more streamlined to state Proposition~\ref{p:Xs} and its proof directly in terms 
of the symbol $c_{\tX_s,B}$, in order to prove energy estimates it is more efficient
to peel off balanced components of  $c_{\tX_s,B}$, so that we are left 
with less debris to contend with.

To start with, let us assume that $\tX_s \in \PP S^1$ admits the representation \eqref{Xs-PP-re} with $a^\gamma \in \PP$ but without requiring that $a^\gamma$ satisfy the relation \eqref{a-gamma}. For such $X_s$, we peel off balanced components of $c_{\tX_s,B}$ following the two steps in Lemma~\ref{l:circle-app}.
These steps are briefly reviewed in the sequel.

In a first stage, we note that all expressions in $c_{\tX_s,B}$ can be seen as linear combinations on the form $\PP \cdot \partial \PP$, where the output is balanced unless the second factor has higher frequency, 
see Lemma~\ref{l:Moser-control}(a). This allows us to replace such products by paraproducts of the form $T_{\PP} \partial \PP$.  Further, using the definition of paracontrolled distributions for the second
factor, we can discard the error term as balanced and arrive at  more precise paraproducts
of the form $ T_{\PP} \partial^2 u$, namely 
\begin{equation}\label{tc_para}
c_{\tX_s,B} \bapprox T_{a^{\alpha \beta}} \partial_\alpha \partial_\beta u,
\end{equation}
where the coefficients $a^{\alpha \beta}$ and $\ta^{\alpha\beta}$ are explicitly computable as algebraic expressions in terms of $u,X$ and the paracoefficients $a^\gamma$ of $X$. Precisely,
\[
\begin{aligned}
c_{\tX_s,B}(x,\xi) \bapprox & \   T_{\partial_{\xi_\gamma} \tp(x,\xi)} \partial_\gamma \tX_s
-  T_{\partial_{\xi_\gamma} \tX_s} \partial_\gamma \tp(x,\xi) 
+   2T_{\tX_s}( \tA^\gamma + 2s\tb^\gamma_0(x,\xi))   \xi_\gamma
+ T_p \partial_0 \tX_{s0}
\\
\bapprox & \ 2 T_{\xi_\alpha \tg^{\alpha\beta} a^\gamma} \partial_\beta \partial_\gamma u  
+ 2 T_{\xi_\alpha \xi_\beta \partial_{\xi_\gamma} \tX_s \partial^\beta u \tg^{\alpha \delta}} \partial_\delta \partial_\gamma u + 2 T_{\tX_s (\partial^\beta u\xi_\gamma \tg^{\gamma \alpha}
+  \partial^0 u \tg^{0\beta} \xi_\gamma \tg^{\gamma \alpha})}
\partial_{\alpha} \partial_{\beta} u
\\ & \ + 2 s T_{ \tX_s\partial^\beta u \xi_\gamma g^{\gamma \alpha} \xi_\beta \partial_{\xi_\delta} (\log |\xi'|^2)} \partial_\alpha \partial_\delta u
-  2s T_{ \tX_s\partial^0 u \tg^{0\nu} \xi_\alpha \tg^{\alpha\beta} \xi_\beta \partial_{\xi_\delta} (\log |\xi'|^2)} \partial_\nu \partial_\delta u
\\ & 
\ +  T_{\tp a^\gamma_0} \partial_0 \partial_\gamma u
\end{aligned}
\]
so we obtain, in unsymmetrized form, the relation \eqref{tc_para} with
\begin{equation}\label{a-ag}
\begin{aligned}
a^{\alpha \gamma} = & \  2\xi_\beta \tg^{\alpha\beta} 
( a^\gamma + \xi_\delta \partial^\delta u \partial_{\xi_\gamma} \tX_s + \tX_s \partial^\gamma u
+ \tX_s  \partial^0 u \tg^{0\beta}
+ \frac{s}2 \tX_s \xi_\delta \partial^\delta u \partial_{\xi_\gamma}(\log|\xi'|^2) )
\\ & \ - \tp(2s \tX_s\partial^0 u \tg^{0\alpha} \partial_{\xi_\gamma} (\log |\xi'|^2)- \delta^{\alpha}_0  a^\gamma_0).
\end{aligned}
\end{equation}

Finally, the last difficulty we face is that  we do not have good enough estimates for $\partial_t^2 u$.  This is rectified by using instead  the corrected expression
$\hat \partial_t^2 u $ introduced in \eqref{hat-dt2}.
This yields a corresponding correction of $c$, namely
\begin{equation}\label{hc-def}
\hc_{\tX_s,B} =   T_{a^{\alpha \beta}} \widehat{\partial_\alpha \partial_\beta} u.
\end{equation}

With these notations, we can now state a more refined version of Proposition~\ref{p:Xs}:

\begin{proposition}\label{p:Xs+}
Let $\tX_s$ be the symbol constructed in Proposition~\ref{p:Xs}.
Then the conclusion of Proposition~\ref{p:Xs} holds equally 
for $\hc_{\tX_s,B}$, with the corresponding expressions 
$\hq_2$ and $\hq_0$ satisfying a
stronger version of 
\eqref{q2}, 
\begin{equation}\label{c-q2}
\|\hq_2\|_{L^\infty S^2} \lesssim \BB^2, \qquad \|P_{<k} \partial_0 \hq_2\|_{ L^\infty S^2} \lesssim 2^k \BB^2,
\end{equation}
and with $\hq_0$ having $\partial_x \PP$ type regularity,
\begin{equation}\label{c-q0}
\| \hq_0\|_{\partial_x \PP S^0} \lesssim \AA    .
\end{equation}
\end{proposition}

\begin{proof}
A direct computation using \eqref{a-ag} and \eqref{a-gamma} shows that the coefficients $a^{\alpha \beta}$ have the form 
\[
a^{\alpha \beta} = \tp q^{\alpha\beta}, \qquad q^{\alpha\beta} \in \PP S^0,
\]
and thus
\[
\hc_{X,B} =  T_{ \tp q^{\alpha \beta}}
\widehat{\partial_\alpha \partial_\beta} u,
\]
 We need to express this in the form $T_{\tp} \partial_x \BB$ plus a balanced component.
For this we consider two cases:

\bigskip

a) If $(\alpha,\beta) \neq (0,0)$ then the above component of $\hc$ has the form
\[
T_{\tp q} \partial_x \partial u = T_{\tp } \partial_x (T_q \partial u) - T_{\tp} T_{\partial_x q} \partial u 
+ (T_{\tp q} - T_{\tp} T_{q}) \partial_x \partial u. 
\]
Here the first term on the right is as needed, so we set 
\[
\hq_2 = - T_{\tp} T_{\partial_x q} \partial u 
+ (T_{\tp q} - T_{\tp} T_{q}) \partial_x \partial u. 
\]
The first term is balanced by Lemma~\ref{l:Moser-control}(a) and the second 
is balanced by Lemma~\ref{l:para-p+}.
We still need to estimate $\partial_0 \hq_2$ as in \eqref{c-q2},
which is immediate using Lemma~\ref{l:utt}, Lemma~\ref{l:Moser-control}(a) and Lemma~\ref{l:para-p+}.

\medskip
b) If $(\alpha,\beta) = (0,0)$ then the above component has the form
\[
T_{\tp q} \hat\partial_t^2 u = \sum_{(\alpha,\beta) \neq (0,0)} T_{\tp q} (T_{\tg^{\alpha\beta}} \partial_\alpha \partial_\beta u
+ T_{\partial_\alpha \partial_\beta u} \tg^{\alpha \beta}).
\]
The first term on the right is treated exactly as in case (a), by pulling out one spatial derivative, while the second is directly placed in $\hq_2$ 
using again  Lemma~\ref{l:utt} and (a minor variation of) Lemma~\ref{l:Moser-control}(a).
\end{proof}

\subsubsection{Paradifferential energy estimates associated to $\tX_s$}
We now use the symbol $\tX_s$ given by Proposition~\ref{p:Xs} in order to construct an $H^1 \times L^2$ balanced energy functional for the conjugated problem \eqref{paralin-inhom+}. This in turn gives an
$H^{s+1} \times H^s$ balanced energy functional for the original linear paradifferential flow \eqref{paralin-inhom}, thus completing the proof of Theorem~\ref{t:para-wp}.

Broadly speaking, we will be following the analysis in the $s=0$ case, but with 
more care since we are replacing the vector field $X$ with the pseudodifferential multiplier $\tX_s$. In particular, here, instead of paraproducts we we will have to commute 
paraproducts with paradifferential operators. The difficulty  is that we 
will no longer be able to estimate the commutator contributions in a direct, perturbative fashion; instead, we will need to take into account unbalanced subprincipal commutator terms, and devise an additional zero order correction to $\tX_s$ in order to deal with them.

We begin by considering the conjugation operator $B$, for which we provide a favourable decomposition:

\begin{lemma}
The operator $B$ given by \eqref{def-R} admits a decomposition
\begin{equation}
B = B_0 + B_1 + B_2,
\end{equation}
where the three components are as follows:

(i) $B_0 = T_{b_0^\gamma} \partial_\gamma$ is the leading part, with symbol
\begin{equation}
b_0^\gamma(x,\xi)= i |\xi|^s \{ \tg^{\alpha \beta} \xi_{\alpha} \xi_{\beta},|\xi|^{-s}\}.
\end{equation}

(ii) $B_1$ is unbalanced but with a favourable null structure, 
\begin{equation}
B_1 w = T_{h(\partial u)} T_{\tg^{\alpha \beta}} L_{lh}(\partial_\alpha \partial u, \partial_\beta|D_x|^{-1}w)
\end{equation}
with $h$ depending smoothly on $\partial u$.

(iii) $B_2$ is balanced,
\begin{equation}
\| B_2 w \|_{L^2} \lesssim \BB^2 \|\partial w\|_{L^2}   .  
\end{equation}
\end{lemma}

This result is a direct consequence of Proposition~\ref{p:com-pdo}; we have stated it here separately only for quick reference in this section.

At this point, we can repeat the multiplier computation in the previous section,
using as multiplier the operator $T_{\tM_s}$ defined in \eqref{hatXs}.
Here $\tX_s$ will be the symbol constructed in the previous subsection, so it remains 
to consider the choice of $\tY_{0s}$, which will be chosen as $\tY_{0s} = -\hq_0$
with $\hq_0$ as in Proposition~\ref{p:Xs+}.

Using the $T_{\tM_s}$ operator as a multiplier, we seek to derive 
an associated energy identity. Here, at leading order, we would like the 
energy functional to be described by the symbol $e_X$ defined as in \eqref{EX-s} by the symbol $E_{\tX_s}$ in  
\eqref{eX-s}. On the other hand the energy flux is to be described at leading order by the symbol $\hc_{X,B}$ in \eqref{hc-def} where we add the contribution of $\tY_0$.

To have a modular argument, at first we simply assume that 
\begin{itemize}
\item $\tX_s \in \PP S^1$, 
with the representation \eqref{Xs-PP-re} with $a^\gamma \in \PP S^1$, but without assuming 
that $a^\gamma$ are given by \eqref{a-gamma}.

\item $\tY_{0s} \in \partial_x \PP$, but  without assuming 
that $\tY_{0s}$ ias as in Proposition~\ref{p:Xs+}.
\end{itemize}

Given such $\tX_s$ and $\tY_{0s}$, we will describe the leading part of the energy flux
using the symbol
\begin{equation}\label{tc-s}
\tc_s = \hc_{X_s,B} + T_p \tY_{0s}.
\end{equation}
This is a second degree polynomial in $\xi_0$, which 
 we expand  as 
\begin{equation}\label{cXA2-exp}
\tc_s(x,\xi) = \tc_s^0(x,\xi') \xi_0^2 + \tc_s^1(x,\xi') \xi_0
+ \tc_{s}^2(x,\xi').
\end{equation}
To this expansion we associate the bilinear form 
\begin{equation}\label{CXAB}
C_{s}(w,w) =  \int - T_{\tc_{s}^0} \partial_t w \cdot \partial_t w + \frac12 T_{\tc_s^0} \partial w \cdot \partial_t w
+ T_{i\tc_{s}^1}w \cdot \partial_t w
+ T_{\tc_{s}^2} w \cdot w \, dx,
\end{equation}
which, integrated also over time, would yield exactly the quadratic form 
generated by the symbol $\tc_s$ in Weyl calculus.

Now we can state our main multiplier energy identity, which  is as follows:

\begin{proposition}\label{p:IBP-s}
Let $\tX_s \in \PP S^1$ and $\tY_{0s} \in \partial_x \PP$ be as above, and the multiplier $T_{\tM}$ be as in \eqref{hatXs}. Then there exists  an energy function $E_{X,B}$ with the following properties:

i) Leading order expression:
\begin{equation}\label{pos:IBP-s}
E_{\tX_s,B}[w] = E_{\tX_s}[w]+ Err(\AA) .
\end{equation}

ii) Energy identity:
\begin{equation}\label{e:IBP-s}
\frac{d}{dt} E_{\tX_s,B}[w] =  \int T_{\tP_{B}} w \cdot T_{\tM} w \, dx +  C_{s}(w,w) + Err(\BB^2).   
\end{equation}
\end{proposition}

We recall again that here we do not assume neither that $\tX_s$ is the "vector field" constructed  in the previous subsection nor that $\tX_s$ is forward time-like.
Instead we will add these two assumptions later on when we apply the Proposition, in order to guarantee that $C_{s}(w,w)$ is controlled by $Err(\BB^2)$, respectively that $E_{\tX_s}$ is positive definite.

\begin{proof}
As stated, the result in the Proposition is linear with respect to both 
$\tX_s$ and $\tY_{0s}$, and also separately in $\tA$ and $\tb_0$. This allows us to divide the proof into several cases, which turn out to be easier to manage separately.

\bigskip

\emph{ I. The contribution of $X_{1s}$ with $\tA=0$ and $\tb_0=0$.}
Our starting point here is the integral
\[
I_{X}^1 = 2\iintT T_{\tP} w \cdot T_{i X_{1s}} w \, dx dt .
\]
The operator $T_{i X_{1s}}$ is purely spatial and antisymmetric, so we can 
integrate by parts three times in $[0,T] \times \R^n$ to rewrite $I_X^1$ in the form
\[
\begin{aligned}
I_X^1 = & \ 2\iintT T_{\tg^{\alpha \beta}} T_{i\partial_\beta X_{1s}} w \cdot \partial_\alpha w
\, dx dt + \iintT [T_{\tg^{\alpha\beta}}, T_{iX_{1s}}] \partial_\alpha w \cdot \partial_\beta w \, dx dt
\\
& \ + \left. 2 \int T_{\tg^{\alpha 0}} T_{i X_{1s}} w \cdot \partial_\alpha w \, dx
\right|_0^T.
\end{aligned}
\]
Here the expression on the second line should be thought of as the energy 
and the expression on the first line represents the energy flux. We remark that if there were no boundaries at times $t=0,T$ then this would be akin to computing the commutator
of $T_{\tP}$ and $T_{iX_{1s}}$. 

The above expression needs some further processing to put it in the desired form. 
We begin with the energy component, where we need to compound the paraproducts 
and separate the cases $\alpha=0$ and $\alpha \neq 0$. This is done using Lemma~\ref{l:para-pdo-A},
\[
\int T_{\tg^{\alpha 0}} T_{i X_{1s}} w \cdot \partial_\alpha w \, dx
= \int T_{\tg^{00} X_{1s}} w \cdot \partial_0 w \, dx + \int T_{\tg^{j 0} X_{1s}  } w \cdot \partial_\alpha w \, dx + O(A) \| w[t]\|_{\H}^2,
\]
as needed.

We now successively consider the space-time integrals on the first line in $I_X^1$.
In the first integral, the components where the $\tg^{\alpha\beta}$ frequency is at least comparable to the $X_{1s}$ frequency are balanced, and we can use Lemma~\ref{l:para-pdo-BDB} to compose the paraproducts  as 
\[
\iintT T_{\tg^{\alpha \beta}} T_{i\partial_\beta X_{1s}} w \cdot \partial_\alpha w
\, dx = \iintT T_{i T_{\tg^{\alpha \beta}} \partial_\beta X_{1s}} w \cdot \partial_\alpha w
\, dx + Err(\BB^2) ,
\]
where the integral on the right can be freely switched to the Weyl calculus if $\alpha \neq 0$, and represents one of the desired components of our energy flux.

For the second space-time integral in $I_X^1$ we use the commutator expansion in Proposition~\ref{p:com-pdo} to get a principal part, an unbalanced subprincipal part and a balanced term, 
\[
\iintT [T_{\tg^{\alpha\beta}}, T_{iX_{1s}}] \partial_\alpha w \cdot \partial_\beta w \, dx dt  = 
\iintT T_{ \{ \tg^{\alpha\beta},X_{1s}\}_p} \partial_\alpha w \cdot \partial_\beta w \, dx dt + I_{X,sub}^1 + Err(\BB^2) ,
\]
where the unbalanced subprincipal part $I_X^{sub}$ has the form 
\begin{equation}\label{IX-sub}
I_{X,sub}^1 = \iintT T_{\PP S^{-1}} T_{\tg^{\alpha\gamma}} L_{lh}(\partial_\gamma \partial_x^2 u,\  \partial_\alpha w) \cdot \partial_\beta w \, dx dt .
\end{equation}
We postpone the analysis of $I_{X,sub}^1$ for later, and focus now on the principal part,
which has symbol 
\[
 \partial_{\xi_j} X_{1s} \partial_j \tg^{\alpha \beta}.
\]
As in Lemma~\ref{l:ppxDPP}, we may perturbatively (with $O(\BB^2 L^\infty S^0)$ errors)
replace this by 
\[
h^{\alpha \beta} =  T_{\partial_{\xi_l} X_{1s}} \partial_l \tg^{\alpha \beta}.
\]
This is almost in the desired form, except that we need to switch it to Weyl 
calculus. We observe that we have no contribution if both $\alpha$ and $\beta$
are zero. We separate the remaining cases, where switching to the Weyl calculus yields errors as follows,
\[ 
\begin{aligned}
Err = & \  \iintT T_{\partial_j h^{j0}} w \cdot \partial_0 w \, dxdt 
+ \frac12 \iintT T_{\partial_j h^{jm}}  w \cdot \partial_m w \, dx dt
\\ = & \ 
 \frac14 \iintT T_{\partial_\alpha \partial_\beta h^{\alpha \beta}}  w \cdot w \, dx dt + \left. \int T_{  \partial_j h^{j0}} w \cdot  w    \,dx \right|_0^T.
\end{aligned}
\]

The last integral is an acceptable energy correction. For the first integral to be an acceptable  energy flux error, it suffices
to show that 
\[
\| P_{<k} \partial_\alpha \partial_\beta h^{\alpha \beta} \|_{L^\infty S^0} \lesssim 2^{2k} \BB^2.
\]
It is easily seen that this is indeed the case if any of the derivatives apply to $X_{1s}$, by using the time derivative component of the $\CC$ bound for either $X_{1s}$ or $\tg^{\alpha\beta}$ to bound time derivatives (of which we can have at most one). So we are left with showing that 
\[
\| P_{<k} \partial_\alpha \partial_\beta \tg^{\alpha \beta} \|_{L^\infty S^0} \lesssim 2^{k} \BB^2.
\]
But for this we use Lemma~\ref{l:ddiv-g}.

\bigskip

\emph{ I. The contribution of $X_{0s}$ with $A=0$.} Here we will follow the same road map 
as in the case of $X_{1s}$, but additional care will be needed in order to handle the additional time derivatives. The integral we need to consider is 
\[
I_{X}^0 = 2\iintT T_{\tP} w \cdot T_{X_{0s}} \partial_0 w \, dx dt .
\]
We can integrate by parts once in $[0,T] \times \R^n$ to rewrite $I_X^0$ in the form
\[
\begin{aligned}
I_X^0 = & 
\ - 2\iintT T_{\tg^{\alpha \beta}} T_{\partial_\beta X_{0s}} \partial_0
  w \cdot \partial_\alpha w
\, dx dt  
 - 2 \iintT T_{\tg^{\alpha\beta}} T_{X_{s0}} \partial_0 \partial_\beta w \cdot \partial_\alpha w \, dx dt \\ & \ \ +  \left. 2 \int T_{\tg^{\alpha 0}} T_{X_{0s}} \partial_0  w \cdot \partial_\alpha w
\, dx
\right|_0^T.
\end{aligned}
\]
In the middle term we switch the operator $T_{\tg^{\alpha\beta}} T_{X_{s0}} \partial_0$
to the right, while integrating by parts once in time, in order to put it in the more symmetric form
\[
\begin{aligned}
I^0_X = & 
\ - 2\iintT T_{\tg^{\alpha \beta}} T_{\partial_\beta X_{0s}} \partial_0
  w \cdot \partial_\alpha w
\, dx dt  
 +  \iintT ( \partial_0  T_{X_{s0}}  T_{\tg^{\alpha\beta}} - T_{\tg^{\alpha\beta}} T_{X_{s0}} \partial_0) \partial_\beta w \cdot \partial_\alpha w \, dx dt
 \\ & +  \left.  \int 2 T_{\tg^{\alpha 0}} T_{X_{0s}} \partial_0  w \cdot \partial_\alpha w
 -  T_{X_{0s}} T_{\tg^{\alpha \beta}} \partial_\alpha w \cdot \partial_\beta w
\, dx
\right|_0^T.
\end{aligned}
\]
For the energy term there is nothing new,  we use as before paraproduct rules to rewrite it as 
the desired leading part plus an acceptable error. We now consider the second space-time integral,
where more care is needed. The operator 
\[
C = \partial_0  T_{X_{s0}}  T_{\tg^{\alpha\beta}} - T_{\tg^{\alpha\beta}} T_{X_{s0}} \partial_0
\]
has a commutator structure, which is good. However we have to carefully decide on the order
in which we commute, because, depending on whether $\alpha=0$ or $\beta = 0$, we might
carelessly end up with a double time derivative. The positive feature, arising from the fact that 
we work with the metric $\tg$ rather than $g$, is that if $(\alpha,\beta)= (0,0)$
then there is a single commutator which does not involve time derivatives. For clarity we  consider the four cases separately:

\begin{enumerate} [label=\roman*)]
\item The case $\alpha \neq 0$, $\beta \neq 0$. This is the simplest case, where,
commuting and peeling off operators of size $O_{L^2}(\BB^2)$,
we write
\[
\begin{aligned}
C = & \   T_{\partial_0 X_{s0}}  T_{\tg^{\alpha\beta}}+ 
 T_{X_{s0}}  T_{\partial_0 \tg^{\alpha\beta}}
- [T_{\tg^{\alpha\beta}}, T_{X_{s0}}] \partial_0.
\end{aligned}
\]

\item The case $\alpha = 0$, $\beta \neq 0$.
Here we use the same order as before.

\item The case $\alpha \neq 0$, $\beta = 0$. Here we reverse the order,
to write
\[
\begin{aligned}
C = & \ T_{\tg^{\alpha\beta}}  T_{\partial_0 X_{s0}}  + 
  T_{\partial_0 \tg^{\alpha\beta}} T_{X_{s0}}  
-\partial_0 [T_{\tg^{\alpha\beta}}, T_{X_{s0}}] ,
\end{aligned}
\]
where the middle term is integrated by parts once more to move $\partial_0$ together 
with $\partial_\alpha$, at the expense of another negligible energy correction
\[
- \left.  \int   [T_{\tg^{\alpha 0}}, T_{X_{s0}}] \partial_{0} w \cdot \partial_\alpha w \, dx
\right|_0^T.
\]
\item The case $\alpha = 0$, $\beta = 0$. Here we simply have 
\[
C =  T_{\partial_0 X_{s0}}.
\]
\end{enumerate}

Now we put together the terms in the four cases.

a) In the $\partial_0 X_0$ term the multiplication order does not matter, and we can further 
replace it by $T_{ T_{\tg^{\alpha\beta}} \partial_0 X_{s0}}$ modulo $O_{L^2}(\BB^2)$ errors.
Thus we retain the integral
\[
\iintT T_{ T_{\tg^{\alpha\beta}} \partial_0 X_{s0}} \partial_\beta w \cdot \partial_\alpha w \, dx dt.
\]

b) In the $\partial_0 \tg^{\alpha \beta}$ term, however, the commutator is not negligible,
so in addition to $T_{X_{s0}}  T_{\partial_0 \tg^{\alpha\beta}}$ we also need the commutator 
$[T_{\partial_0 \tg^{\alpha 0}}, T_{X_{s0}}]$. Hence we get two contributions, 
\[
\iintT T_{X_{s0}}  T_{\partial_0 \tg^{\alpha\beta}} \partial_\beta u  \cdot \partial_\alpha u \, dx dt
+ \iintT [ T_{\partial_0 \tg^{\alpha 0}},T_{X_{s0}}]  \partial_0 u  \cdot \partial_\alpha u \, dx dt.
\]

c) In the $[T_{\tg^{\alpha\beta}}, T_{X_{s0}}]$ term where, distinguishing between $\beta = j \neq 0$
and $\beta = 0$, we get
\[
\iintT - [T_{\tg^{\alpha j}}, T_{X_{s0}}] \partial_0 \partial_j w \cdot \partial_\alpha w \, dx dt
- \iintT \partial_0 w \cdot [T_{\tg^{\alpha 0}},T_{X_{s0}}] \partial_0 \partial_\alpha w \, dx dt.
\]
In the first integral we move $\partial_j$ and the commutator term to the right, also commuting them,
so the above expression is rewritten as 
\[
\iintT \partial_0 w \cdot [T_{X_{s0}},T_{\tg^{\alpha \beta}}] \partial_\beta \partial_\alpha w \, dx dt
+ \iintT [T_{\partial_j \tg^{\alpha j}}, T_{X_{s0}}] \partial_0  w \cdot \partial_\alpha w
+  [T_{\tg^{\alpha j}}, \partial_j T_{X_{s0}}] \partial_0  w \cdot \partial_\alpha w \, dx dt.
\]
We retain the first term as it is, combine the second one with the second term in part (b) 
and discard the last one as perturbative, $Err(\BB^2)$.

Putting all terms together, we have rewritten $I^0_X$, modulo perturbative terms, as
\[
\begin{aligned}
I_X^0 = & \ - 2\iintT T_{T_{\tg^{\alpha \beta}} \partial_\beta X_{0s}} \partial_0
  w \cdot \partial_\alpha w
\, dx dt  + \iintT T_{ T_{\tg^{\alpha\beta}} \partial_0 X_{s0}} \partial_\beta w \cdot \partial_\alpha w \, dx dt
\\ & \ 
+ \iintT T_{X_{s0}}  T_{\partial_0 \tg^{\alpha\beta}} \partial_\beta u  \cdot \partial_\alpha u \, dx dt + \iintT \partial_0 w \cdot \partial_\beta [T_{X_{s0}},T_{\tg^{\alpha \beta}}]  \partial_\alpha w \, dx dt
\\ & \   - \iintT  \partial_0  w \cdot [\partial_\beta T_{X_{s0}}, T_{ \tg^{\alpha \beta}}] \partial_\alpha w
\, dx dt + Err(\BB^2) 
\\ 
:= & \ I_X^{01} + I_X^{02} +  I_X^{03} + I_X^{04} + I_X^{05} + Err(\BB^2) .
\end{aligned}
\]
This can be simplified further by observing that, in view of Lemma~\ref{l:para-com-pk}, the term $I_X^{05}$ is also perturbative.
Thus we arrive at 
\[
I_X^0 =  I_X^{01} + I_X^{02} +  I_X^{03} + I_X^{04} + Err(\BB^2). 
\]

We successively consider these terms:

\medskip

\emph{I. The contribution of $I_X^{01}$.} This corresponds to the symbol
\[
2 T_{T_{\tg^{\alpha \beta}} \partial_\beta X_{0s}} \xi_0 \xi_\alpha,
\]
which is akin to one of the components of $c_{X_s,B}$ in \eqref{tcXA-s}.
We can turn this into the corresponding component of $\hc_{X,B}$.
Precisely, given $X_{0s}$ as in the representation \eqref{Xs-PP-re}, 
that component is 
\[
2 T_{T_{\tg^{\alpha \beta} a^\gamma_0} \widehat {\partial_\beta \partial_\gamma} u} \xi_0 \xi_\alpha,
\]
where we recall that the hat above is understood as nonexistent 
unless $\beta = \gamma = 0$, in which case it is interpreted as the corrected expression \eqref{hat-dt2}. The difference between the two coefficients is easily seen to have size $\BB^2$, so it is perturbative, as in Lemma~\ref{l:circle-app}.
It remains to switch this modification of $I_X^{01}$ to the
 Weyl calculus, which requires estimating the integral
\[
\iintT T_{\partial_\alpha T_{\tg^{\alpha \beta} a^\gamma} \widehat{\partial_\beta  \partial_\gamma} u} \partial_0
  w \cdot  w
\, dx dt.
\]
This follows from the bound
\[
\| P_{<k} \partial_\alpha T_{\tg^{\alpha \beta} a^\gamma} \widehat{\partial_\beta  \partial_\gamma} u\|_{L^\infty} \lesssim 2^{k} \BB^2,
\]
which in turn follows from Lemma~\ref{l:tp-du}(b) after commuting $a^\gamma_0$ out.

\medskip

\emph{II. The contribution of $I_X^{02}$}. Exactly as above, this integral also corresponds 
to a term in $c_{X_s,B}$. Again, after a perturbative $Err(\BB^2)$ correction we can turn this 
into the corresponding term in $\hc_{XB}$, which has the para-coefficient 
\[
T_{\tg^{\alpha\beta} a^\gamma} \widehat{\partial_0 \partial_\gamma} u.
\]
The associated integral is 
\[
 \iintT T_{ T_{\tg^{\alpha\beta} a^\gamma} \widehat{\partial_0 \partial_\gamma} u} \partial_\beta w \cdot \partial_\alpha w \, dx dt.
\]
We would like to switch this to Weyl calculus, but we 
need to be careful here because the convention for the Weyl form  differs depending on whether $\beta$ is zero or not. 

If $\beta \neq 0$ then the error corresponds 
to switching the operator on the left to Weyl calculus,
and has the form
\begin{equation}\label{Weyl-cor}
\frac12 \iintT T_{ \partial_\beta T_{\tg^{\alpha\beta} a^\gamma} \widehat{\partial_0 \partial_\gamma} u} w \cdot \partial_\alpha w \, dx dt.
\end{equation}
The same applies if $\alpha = \beta = 0$.
But if $\alpha \neq 0$ and $\beta = 0$ then we 
have to switch the paraproduct to the right, and then the Weyl correction is 
\[
\frac12 \iintT \partial_\beta w \cdot T_{ \partial_\alpha T_{\tg^{\alpha\beta} a^\gamma} \widehat{\partial_0 \partial_\gamma} u} w \, dx dt.
\]
We can rectify this discrepancy and switch this 
correction to the form in \eqref{Weyl-cor} by integrating twice by parts, first in $x_\beta$ and then in $x_\alpha$. Since we are in the case $\beta = 0$, the first step yields an energy correction, namely
\[
\frac14 \left. \int   w \cdot T_{ \partial_\alpha T_{\tg^{\alpha0} a^\gamma} \widehat{\partial_0 \partial_\gamma} u} w \, dx  \right|_0^T .
\]
As $\alpha \neq 0$, for this to be an acceptable $Err(\AA)$ error
we need the bound
\[
\| P_{<k} \widehat{\partial_0 \partial_\gamma} u \|_{L^\infty} \lesssim 2^k \AA.
\]
This is obvious if $\gamma \neq 0$, and follows from \eqref{fe-dt2u} otherwise.

Thus we are left with considering the correction 
in \eqref{Weyl-cor} summed over all $\alpha$ and $\beta$,
and which we would like to estimate perturbatively.

Here there is no structure in the $\gamma$
summation, so we can fix $\gamma$. The easier case is when $\gamma \neq 0$. Then we can commute 
$\partial_\gamma$ out, as well as $a^\gamma$, and $\partial_\beta$ in, writing
\[
\partial_\beta T_{\tg^{\alpha\beta} a^\gamma} {\partial_0 \partial_\gamma} u
= T_{a^\gamma} \partial_\gamma T_{\tg^{\alpha\beta}} \partial_\beta \partial_0 u + f ,
\]
where the error term $f$ satisfies 
\begin{equation}\label{f-is-good}
\| P_{<k} f\|_{L^\infty} \lesssim 2^k \BB^2.
\end{equation}
We may also correct the second order time derivative, arriving at
\[
\partial_\beta T_{\tg^{\alpha\beta} a^\gamma} {\partial_0 \partial_\gamma} u
= T_{a^\gamma} \partial_\gamma T_{\tg^{\alpha\beta}} \widehat{\partial_\beta \partial_0} u + f.
\]
The remaining term is no longer perturbative, but its contribution may be instead estimated
integrating by parts,
\[
\begin{aligned}
 \iintT T_{ T_{a^\gamma} \partial_\gamma T_{\tg^{\alpha\beta}} \widehat{\partial_\beta \partial_0} u}
  w \cdot \partial_\alpha w \, dx dt =  & \ 
- \frac12   \iintT T_{ \partial_\alpha T_{a^\gamma} \partial_\gamma T_{\tg^{\alpha\beta}} \widehat{\partial_\beta \partial_0} u}  w \cdot  w \, dx dt
\\ & 
+ \left.  \int   T_{ T_{a^\gamma} \partial_\gamma T_{\tg^{0\beta}} \widehat{\partial_\beta \partial_0} u}  w \cdot w \, dx
\right|_0^T.
\end{aligned}
\]
The last term is a bounded energy correction, as 
\[
\| P_{k}  T_{a^\gamma} \partial_\gamma T_{\tg^{0\beta}} \widehat{\partial_\beta \partial_0} u\|_{L^\infty} \lesssim 2^{2k} \AA.
\]
It remains to show that the first term is also perturbative,
\[
\| P_{<k} T_{a^\gamma} \partial_\gamma T_{\tg^{\alpha\beta}} \widehat{\partial_\beta \partial_0} u\|_{L^\infty} \lesssim 2^{2k} \BB^2.
\]
Commuting $\partial_{\alpha} $ inside and discarding $a^\gamma \partial_\gamma$, this reduces to 
\[
\| P_{<k}  \partial_\alpha T_{\tg^{\alpha\beta}} \widehat{\partial_\beta \partial_0} u\|_{L^\infty} \lesssim 2^{k} \BB^2,
\]
which is again a consequence of Lemma~\ref{l:tp-du}(b).

It remains to consider the case $\gamma=0$, where 
we take advantage of the hat correction.
Precisely, using the $u$ equation, we write
\[
\widehat{\partial_0 \partial_0} u 
= - \sum_{(\mu,nu) \neq (0,0)} T_{\tg^{\mu \nu}}
\partial_\mu \partial_\nu u + T_{\partial_\mu \partial_\nu u} \tg^{\mu \nu}.
\]
We substitute this into the paracoefficient in \eqref{Weyl-cor}, peeling off perturbative contributions.
Fixing $\mu$ and $\nu$ we may assume $\mu \neq 0$ and arrive at
\[
\partial_\beta T_{\tg^{\alpha\beta} a^\gamma} 
\widehat{\partial_0 \partial_0} u
= - \sum_{\mu \neq 0}  T_{a^\gamma \tg^{\mu \nu}} \partial_\mu T_{\tg^{\alpha\beta}} \partial_\beta \partial_\nu u + f ,
\]
with $f$ as in \eqref{f-is-good}.
At this point we can repeat the argument in the case 
$\gamma \neq 0$.

\medskip

\emph{III. The contribution of $I_X^{03}$.} We recall that this is
\[
I_X^{03} = \iintT T_{X_{s0}}  T_{\partial_0 \tg^{\alpha\beta}} \partial_\beta u  \cdot \partial_\alpha u \, dx dt.
\]
This term is easily seen to be perturbative unless the spatial frequency of $X_{0s}$ is smaller than that of $\partial_0 \tg^{\alpha\beta}$, see Lemma~\ref{l:para-pdo-BDB}.  Thus we can think of the principal symbol of the 
product $T_{X_{s0}}  T_{\partial_0 \tg^{\alpha\beta}}$ as being
$T_{T_{X_{s0}} \partial_0 \tg^{\alpha\beta}}$. However some care
is needed here with the error, which is lower order but not necessarily balanced. Precisely, using Proposition~\ref{p:prod-pdo},
we can expand this product into a leading part, an unbalanced 
subprincipal part and a perturbative term,
\begin{equation}\label{Xdg-exp}
T_{X_{s0}}  T_{\partial_0 \tg^{\alpha\beta}} = T_{T_{X_{s0}} \partial_0 \tg^{\alpha\beta}} + T_{lh} (\partial_\xi X_{s0}, \partial_x \partial_0 \tg^{\alpha\beta}) + O_{L^2}(\BB^2).
\end{equation}
This yields a corresponding decomposition of $I_X^{03}$ into
\[
I_X^{03} = I_{X,main}^{03}+ I_{X,sub}^{03} + Err(\BB^2).
\]

To better describe the first two terms we 
take a closer look at the coefficient $\partial_0 \tg^{\alpha \beta}$, for which we compute
\begin{equation}\label{dtg}
\partial_0 \tg^{\alpha \beta} = 
-  \left(\partial^{\beta} u \tg^{\alpha \delta } \partial_\delta \partial_0 u +  \partial^{\alpha} u \tg^{\beta \delta } \partial_\delta \partial_0 u \right)
+ 2  \tg^{\alpha \beta} \partial^{0} u \tg^{0 \delta } \partial_\delta \partial_0 u.
\end{equation}
Here we have a double time derivative $\partial_t^2 u$ when $\delta = 0$, which we replace as before by $\hat \partial_t^2 u$ with perturbative errors.  Once this is done, we may also replace all products by paraproducts, arriving at the modified expression
\[
\begin{aligned}
\mathring \partial_0 \tg^{\alpha \beta} := & \ 
- \left( T_{\partial^{\beta} u \tg^{\alpha \delta }} \widehat{\partial_\delta \partial_0} u +  T_{\partial^{\alpha} u
\tg^{\beta \delta }} \widehat{\partial_\delta \partial_0} u \right)
+ 2 T_{\partial^{0} u \tg^{0 \delta }\tg^{\alpha \beta}} \widehat{\partial_\delta \partial_0} u 
\end{aligned}
\]
so that the difference is perturbative in the sense that
\[
\partial_0 \tg^{\alpha \beta} = \mathring \partial_0 \tg^{\alpha \beta} + O(\BB^2).
\]
Finally we return to the operator setting, where we make the above substitution. In the principal part can compound the 
outer paracoefficients at the expense of more negligible errors, 
writing it in a paradifferential form 
\[
T_{T_{X_{s0}} \partial_0 \tg^{\alpha\beta}} = 
T_{q^{\alpha\beta}}+ O_{L^2}(\BB^2), \qquad 
I_{X,main}^{03} = \iintT T_{q^{\alpha\beta}} \partial_\alpha w \cdot \partial_\beta w \, dx dt+ Err(\BB^2),
\]
where the order zero symbols $q^{\alpha\beta}$
are given by
\[
q^{\alpha\beta} = - \left( T_{X_{s0} \partial^{\beta} u \tg^{\alpha \delta }} \widehat{\partial_\delta \partial_0} u +  T_{X_{s0} \partial^{\alpha} u
\tg^{\beta \delta }} \widehat{\partial_\delta \partial_0} u \right)
+ 2 T_{X_{s0} \partial^{0} u \tg^{0 \delta }\tg^{\alpha \beta}} \widehat{\partial_\delta \partial_0} u.
\]
Here the symbol $q^{\alpha\beta} \xi_\alpha \xi_\beta$ is a component of $\hc_{X,B}$, as desired. All we need now is to 
convert the last expression for $I_{X,main}^{03}$ to Weyl form. 
This conversion yields the additional error
\[
\frac12 \iintT T_{\partial_\alpha q^{\alpha\beta}}  w \cdot \partial_\beta w \, dx dt,
\]
which we need to estimate. Here we separate the three terms in $q^{\alpha\beta}$.  For the first term, after one commutation it remains 
to show that 
\[
\| P_{<k} \partial_\alpha T_{\tg^\alpha\delta} \widehat{\partial_\delta \partial_0} u \|_{L^\infty} \lesssim 2^k \BB^2 ,
\]
which we get from Lemma~\ref{l:tp-du}. The second term is similar if we integrate 
by parts to switch $\alpha$ and $\beta$, at the expense of a bounded energy correction. Finally, the third term is exactly as in the case of $I_X^{02}$.

Similarly, in the subprincipal term in \eqref{Xdg-exp}   we may peel off perturbative errors to write it as a linear combination of expressions of the form
\[
T_{\PP S^{-1}} \tilde T_{T_{\tg^{\alpha \delta }} \partial_x \widehat{\partial_\delta \partial_0} u}  +T_{\PP S^{-1}} \tilde T_{T_{\tg^{\beta \delta }} \partial_x \widehat{\partial_\delta \partial_0} u} + T_{\PP S^{-1}}\tilde T_{T_{\tg^{\alpha \beta }}\partial_x  \widehat{\partial_\delta \partial_0} u}.
\]
We postpone their analysis for later, for now we simply list the two
types of contributions:
\begin{equation}\label{IXsub-031}
I_{X,sub}^{031} = \iintT  T_{\PP S^{-1}}  T_{\tg^{\alpha \delta }} L_{lh}(\partial_x \widehat{\partial_\delta \partial_0} u, \partial_\alpha w) \cdot  \partial w \, dx dt  . 
\end{equation}
\begin{equation}\label{IXsub-032}
I_{X,sub}^{032} = \iintT  T_{\PP S^{-1}}  T_{\tg^{\alpha \beta }} L_{lh}(\partial_x \widehat{\partial_\delta \partial_0} u, \partial_\alpha w) \cdot \partial_\beta w \, dx dt   . 
\end{equation}.

\medskip

\emph{III. The contribution of $I_X^{04}$.} We recall that this is
\[
I_X^{04} = \iintT \partial_0 w \cdot \partial_\alpha [T_{X_{s0}}, T_{\tg^{\alpha \beta}}] \partial_\beta  w \, dx dt.
\]
This has a similar treatment to $I_X^{03}$. For the commutator above we must have again a smaller frequency on $X_{0s}$, else this yields a perturbative contribution.
Using Proposition~\ref{p:com-pdo} we expand the commutator 
into a leading part, an unbalanced subprincipal part and a perturbative term,
\begin{equation}\label{Xdg-com}
[T_{X_{s0}},T_{\tg^{\alpha \beta}}] = T_{T_{\partial_{\xi_j}X_{s0}} \partial_{j} \tg^{\alpha \beta}} + T_{lh} (\partial_\xi^2 X_{s0}, \partial_{x}^2 \tg^{\alpha \beta})+ R^{\alpha\beta},
\end{equation}
where the remainder $R$ satisfies perturbative bounds of the form
\[
\| R\|_{L^2 \to H^1} \lesssim \BB^2, \qquad \| R\|_{L^2 \to H^1} \lesssim \BB^2,
\qquad \| \partial_0 R\|_{L^2 \to L^2} \lesssim \BB^2.
\]
We first consider the contribution of the leading part $I_{X,main}^{04}$.
For $\partial_{j}\tg^{\alpha\beta}$ we use the expansion in \eqref{dtg} with the subscript $0$ replaced by $j \neq 0$, and then correct the double time derivative of $u$ as before, arriving at 
\[
  T_{T_{\partial_{\xi_j}X_{s0}} \partial_{j} \tg^{\alpha \beta}} = 
T_{q^{\alpha\beta}}+ R^{\alpha\beta},
\]
where the order zero symbols $q^{\alpha\beta}$ are given by
\[
q^{\alpha\beta} = - \left( T_{\partial_{\xi_j} X_{s0} \partial^{\beta} u \tg^{\alpha \delta }} {\partial_\delta \partial_j} u +  T_{\partial_{\xi_j}X_{s0} \partial^{\alpha} u
\tg^{\beta \delta }} {\partial_\delta \partial_j} u \right)
+ 2 T_{\partial_{\xi_j} X_{s0} \partial^{0} u \tg^{0 \delta }\tg^{\alpha \beta}} {\partial_\delta \partial_j} u,
\]
and the remainder $R$ is as above. Then the leading part can be written as
\[
I_{X,main}^{04} = \iintT T_{q^{\alpha\beta}} \partial_\alpha \partial_\beta w \cdot \partial_0 w \, dx dt+ Err(\BB^2).
\]
Now the symbol $q^{\alpha\beta} \xi_0\xi_\alpha \xi_\beta$ is a component of $\htc_{X,B}$, as desired. It remains to  convert the last expression for $I_{X,main}^{04}$ to Weyl form. The error in doing that is
\[
\frac14 \iintT T_{\partial_\alpha \partial_\beta q^{\alpha\beta}}  w \cdot \partial_0 w \, dx dt.
\]
Estimating this expression requires the bound
\[
\| P_{<k} \partial_\alpha \partial_\beta q^{\alpha\beta} \|_{L^\infty}
\lesssim 2^{k} \BB^2.
\]
Here $q^{00}=0$ so we avoid the case of two time derivatives. This allows us 
to commute $\partial_\alpha \partial_\beta$ inside and take $\partial_j$ outside
modulo perturbative terms. Then $\partial_j$ yields the $2^k$ factor, and we have reduced the problem to proving that
\[
 \| P_k \left( T_{\partial_{\xi_j} X_{s0} \partial^{\beta} u \tg^{\alpha \delta }}  + T_{\partial_{\xi_j}X_{s0} \partial^{\alpha} u
\tg^{\beta \delta }} 
- 2 T_{\partial_{\xi_j} X_{s0} \partial^{0} u \tg^{0 \delta }\tg^{\alpha \beta}}\right) {\partial_\delta \partial_\alpha \partial_\beta} u\|_{L^\infty}
\lesssim \BB^2.
\]
Re-labeling this becomes
\[
 \| P_k T_{\partial_{\xi_j} X_{s0}( \partial^{\delta} u - \partial^{0} u \tg^{0 \delta })\tg^{\alpha \beta}} {\partial_\delta \partial_\alpha \partial_\beta} u\|_{L^\infty}
\lesssim \BB^2.
\]
The expression on the left vanishes if $\delta =0$. This allows us to break 
the para-coefficient in two using Lemma~\ref{l:para-com-pk} and replace this by 
\[
\| P_k  T_{\partial_{\xi_j} X_{s0}( \partial^{\delta} u - \partial^{0} u \tg^{0 \delta })} T_{\tg^{\alpha \beta}} \partial_\delta \partial_\alpha \partial_\beta u\|_{L^\infty}
\lesssim \BB^2,
\]
which is finally a consequence of Lemma~\ref{l:tp-du}.

Next we consider the subprincipal term. Here we use again 
the expansion in \eqref{dtg} and recombine paracoefficients to rewrite it 
as a linear combination of terms of the form
\begin{equation}\label{IXsub-041}
I_{X,sub}^{041} = \iintT \partial_0 w \cdot T_h \partial_\alpha T_{\tg^{\alpha\delta}} L_{lh}
(\partial_x^2 \partial_\delta u, \partial_\beta w) \,dx dt ,
\end{equation}
\begin{equation}\label{IXsub-042}
I_{X,sub}^{042} = \iintT \partial_0 w \cdot T_{\PP S^{-2}} \partial_\alpha T_{\tg^{\beta\delta}} L_{lh}
(\partial_x^2 \partial_\delta u, \partial_\beta w) \,dx dt ,
\end{equation}
respectively 
\begin{equation}\label{IXsub-043}
I_{X,sub}^{042} = \iintT \partial_0 w \cdot T_{\PP S^{-2}} \partial_\alpha T_{\tg^{\alpha \beta}} L_{lh}
(\partial_x^2 \partial_\delta u, \partial_\beta w) \,dx dt,
\end{equation}
where $h \in \PP S^{-2}$ roughly corresponds to $\partial_\xi^2 X_{0s}$.
Here we can freely separate variables and reduce to the case when $h$ is a function, 
including the multiplier part in $L_{lh}$. 

We remark that until now we were able to exclude the case when $\alpha = \beta = 0$.
However, at this point we need to separate the three types of contributions
in order to take advantage of their structures. Because of this, from here 
on we have to also allow for the case $\alpha=\beta = 0$. We postpone the estimate 
for the subprincipal terms for the end of the proof.

\bigskip

\emph{III. The contribution of $Y_0$ with $\tA=0$ and $\tb=0$.}
Here we consider the integral 
\[
I_Y = \iintT T_{\tP} w \cdot T_{Y_{0s}} w \, dx dt,
\]
where we recall that $Y_{0s} \in \partial_x \PP$. We integrate once by parts to write
\[
I_Y = \iintT T_{\tg^{\alpha\beta}} \partial_\alpha w \cdot T_{Y_{0s}}\partial_\beta w \, dx dt -  \iintT T_{\tg^{\alpha\beta}} \partial_\alpha w \cdot T_{\partial_\beta Y_{0s}} w \, dx dt 
+ \left. \int T_{\tg^\alpha 0} \partial_\alpha w \cdot T_{Y_{0s}} w \, dx
\right|_0^T.
\]
The last integral is an admissible energy correction.
In both space-time integrals we move $T_{Y_{0s}}$ to the left, and combine the two 
paraproducts as in Lemma~\ref{l:para-prod}, peeling off perturbative contributions,
to get
\[
I_Y = \iintT T_{T_{\tg^{\alpha\beta}}Y_{0s}} \partial_\alpha w \cdot \partial_\beta w \, dx dt -  \iintT T_{\partial_\beta T_{\tg^{\alpha\beta}} Y_{0s}} \partial_\alpha w \cdot w \, dx dt .
+ Err(\BB^2)+ Err(\AA)|_0^T.
\]
The symbol of the bilinear form in the first integral is the desired component of $\hc_s$, but we need to convert it to Weyl calculus. This yields an error which is half of the second integral, which in turn needs to be estimated perturbatively.
Commuting $\partial_\beta$ inside, 
we are left with 
\[
\iint T_{T_{\tg^{\alpha\beta}}\partial_\beta Y_0 } w \cdot \partial_\alpha w \, dx dt .
\]
Here $Y_0$ is of the form $Y_0= \partial_x h$, with $h \in \PP S^0$. We can harmlessly commute $\partial_x$ out, to arrive at 
\begin{equation}\label{def-J-Y}
J = \iint T_{\partial_x T_{\tg^{\alpha\beta}}\partial_\beta h } w \cdot \partial_\alpha w \, dx dt.
\end{equation}
In the absence of boundaries at $t = 0,T$ here we could integrate by parts once more to rewrite this as 
\[
-\frac12 \iint T_{\partial_x \partial_\alpha T_{\tg^{\alpha\beta}}\partial_\beta h } w \cdot \partial_\alpha w \, dx dt,
\]
and then use Lemma~\ref{l:pcbounds-extra}. The same argument applies if we add in the boundaries, by carefully tracking the boundary contributions.
Precisely, we use the lemma to rewrite the expression $J$ in \eqref{def-J-Y} as follows:
\[
\begin{aligned}
J = & \ \iint T_{\partial_x (T_{\tg^{\alpha\beta}}\partial_\beta h - \delta_0^\alpha f^0) } w \cdot \partial_\alpha w \, dx dt
+ \iint T_{\partial_x f_0} w \cdot \partial_0 w \, dx dt
\\
= & \ -\frac12 \iint T_{\partial_x (\partial_\alpha T_{\tg^{\alpha\beta}}\partial_\beta h - \partial_0 f^0) } w \cdot w \, dx dt
+ \iint T_{\partial_x f_0} w \cdot \partial_0 w \, dx dt 
\\ & \ + \frac12 \left. \int T_{\partial_x (T_{\tg^{0\beta}}\partial_\beta h - f^0) }w \cdot w \, dx \right|_0^T
\\
= & \ -\frac12 \iint T_{\partial_x \partial_j f^j} w \cdot w \, dx dt
+ \iint T_{\partial_x f_0} w \cdot \partial_0 w \, dx dt  + \frac12 \left. \int T_{\partial_x (T_{\tg^{0\beta}}\partial_\beta h - f^0) }w \cdot w \, dx \right|_0^T.
\end{aligned}
\]
Now, in view of Lemma~\ref{l:pcbounds-extra}, both the energy and the flux terms are perturbative. 

\bigskip

\emph{IV. The contribution of the gradient potential and of $b_0$.}
We discuss the two together, as their contributions are similar. This has the form
\[
I_X^2 = \iint T_{\tM_s} w \cdot (T_{\tA^\gamma}  + T_{b_0^\gamma})\partial_\gamma w\, dx dt,
\]
which we need to shift to Weyl calculus after peeling off a perturbative contribution. For instance the contribution of $Y_0$
is directly perturbative. On the other hand, $\tA^\gamma$
contains $\partial_0^2 u$ terms which need to be corrected, 
while $b_0^\gamma$ does not. In any case, the correction can be freely added as its
contribution has size $Err(\BB^2)$.

Next we consider the contribution of $X_{1s}$, where we need 
to shift the operator product $T_{X_{s1}} T_{\tA^\gamma}$ to the 
Weyl calculus via Lemma~\ref{l:para-prod2}:
\[
T_{X_{s1}} T_{\tA^\gamma} = T_{T_{X_{s1}} \hat{\tA}^\gamma} + 
T_{\PP S^{0}} T_{\tg^{\alpha\gamma}} L_{lh}(\partial_x \widehat{\partial_\alpha \partial} u, \cdot) + O_{H^1 \to L^2}(\BB^2),
\]
and similarly for $b_0$,
i.e. the desired term plus a null unbalanced lower order term 
plus a perturbative contribution. We note here that the 
contribution of the null unbalanced lower order term has the 
form 
\[
I_{X,sub}^{31} = \iintT  T_{\PP S^{0}} T_{\tg^{\alpha\gamma}} L_{lh}(\partial_x \widehat{\partial_\alpha \partial} u, \partial_{\gamma} w) \cdot w \, dx dt.
\]

Finally we consider the contribution of $X_{0s}$,
\[
I_{X}^{30}=\iint  T_{X_{0s}}\partial_0 w \cdot (T_{\hat{\tA}^\gamma}  + T_{b_0^\gamma})\partial_\gamma w\, dx dt = 
\iint  \partial_0 w \cdot T_{X_{0s}} (T_{\hat{\tA}^\gamma}  + T_{b_0^\gamma})\partial_\gamma w\, dx dt.
\]
We use again the product formula for paraproducts to 
write
\[
T_{X_{0s}} (T_{\hat{\tA}^\gamma}  + T_{b_0^\gamma}) = 
T_{T_{X_{0s}}(\tA^\gamma  + b_0^\gamma)} + T_{\PP S^{-1}} 
T_{\tg^{\alpha\gamma}} L_{lh}(\partial_x \widehat{\partial \partial_\alpha} u,\cdot)
+O_{L^2}(\BB^2),
\]
which generates a leading term and a subprincipal term.

The leading term is
\[
I_{X,main}^{30}= 
\iintT  \partial_0 w \cdot T_{X_{0s}} (T_{\hat{\tA}^\gamma}  + T_{b_0^\gamma})\partial_\gamma w\, dx dt.
\]
Its symbol is as needed, but we still have to switch it to Weyl calculus.
This switch introduces an error 
\[
\frac12 \iintT T_{\partial_\gamma [T_{X_{0s}}(\hat{\tA}^\gamma  + b_0^\gamma)]} w 
\cdot \partial_0 w \, dx dt.
\]
To bound its contribution, we would like to have the symbol bound
\begin{equation}
|P_{<k} \partial_\gamma T_{X_{0s}}(\hat{\tA}^\gamma  + b_0^\gamma)|
\lesssim 2^k \BB^2 .
\end{equation}
Here we use the expressions for $\tA^\gamma$ and $b_0^\gamma$, take out 
bounded paracoefficients, and we are left with
\[
\|P_{<k} \partial_\gamma T_{\tg^{\gamma \alpha}} \widehat{\partial_\alpha \partial_\beta} u\|_{L^\infty}
\lesssim 2^k \BB^2 .
\]
But this is in turn a consequence of Lemma~\ref{l:tp-du}.

To conclude, we record the form of the subprincipal term,
\begin{equation}\label{IXsub-30}
I_{X,sub}^{30}  = \iintT  T_{\PP S^{-1}} T_{\tg^{\alpha\gamma}} L_{lh}(\partial_x \widehat{\partial_\alpha \partial} u, \partial_{\gamma} w) \cdot \partial_0 w \, dx dt.
\end{equation}

\bigskip

\emph{ V. The unbalanced lower order terms.} These are the 
expressions identified earlier, which we recall here:
\begin{equation}\label{IX-sub-re}
I_{X,sub}^1 = \iint T_{\PP S^{-1}} T_{\tg^{\alpha\gamma}} L_{lh}(\partial_\gamma \partial_x^2 u,\  \partial_\alpha w) \cdot \partial_\beta w \, dx dt .
\end{equation}

\begin{equation}\label{IXsub-031-re}
I_{X,sub}^{031} = \iintT  T_{\PP S^{-1}}  T_{\tg^{\alpha \delta }} L_{lh}(\partial_x \widehat{\partial_\delta \partial_0} u, \partial_\alpha w) \cdot  \partial w \, dx dt   .
\end{equation}
\begin{equation}\label{IXsub-032-re}
I_{X,sub}^{032} = \iintT  T_{\PP S^{-1}}  T_{\tg^{\alpha \beta }} L_{lh}(\partial_x \widehat{\partial_\delta \partial_0} u, \partial_\alpha w) \cdot \partial_\beta w \, dx dt   .
\end{equation}
\begin{equation}\label{IXsub-041-re}
I_{X,sub}^{041} = \iintT \partial_0 w \cdot T_{\PP S^{-2}} \partial_\alpha T_{\tg^{\alpha\delta}} L_{lh}
(\partial_x^2 \partial_\delta u, \partial_\beta w) \,dx dt .
\end{equation}
\begin{equation}\label{IXsub-042-re}
I_{X,sub}^{042} = \iintT \partial_0 w \cdot T_{\PP S^{-2}} \partial_\alpha T_{\tg^{\beta\delta}} L_{lh}
(\partial_x^2 \partial_\delta u, \partial_\beta w) \,dx dt .
\end{equation}
\begin{equation}\label{IXsub-043-re}
I_{X,sub}^{043} = \iintT \partial_0 w \cdot T_{\PP S^{-2}} \partial_\alpha T_{\tg^{\alpha \beta}} L_{lh}
(\partial_x^2 \partial_\delta u, \partial_\beta w) \,dx dt.
\end{equation}
\begin{equation}\label{IXsub-30-re}
I_{X,sub}^{30}  = \iintT  T_{\PP S^{-1}} T_{\tg^{\alpha\gamma}} L_{lh}(\partial_x \widehat{\partial_\alpha \partial} u, \partial_{\gamma} w) \cdot \partial_0 w \, dx dt.
\end{equation}
All of these exhibit a null structure.

We compress four of these into the expression
\begin{equation}\label{IXsub-main}
I_{sub,\gamma\delta} = \iint T_{h} T_{\tg^{\alpha\beta}} L_{lh}(\partial_x \widehat{\partial_\alpha \partial_\gamma} u,\  \partial_\beta w) \cdot \partial_\delta w \, dx dt, \qquad h \in \PP S^1,
\end{equation}
where the analysis will be slightly different depending on whether 
$\gamma$ and $\delta$ are zero or not.

In $I_{X,sub}^{042}$ the case $\alpha \neq 0$ is included above. If instead 
$\alpha = 0$ then we integrate by parts $\partial_\alpha$ to the left, so that, 
after a perturbative energy correction,
we arrive at
\[
\iintT \partial_0^2 w \cdot T_{\PP S^{-2}}  T_{\tg^{\beta\delta}} L_{lh}
(\partial_x^2 \partial_\delta u, \partial_\beta w) \,dx dt .
\]
Now we use the paradifferential equation $T_{\tP}$ equation for $w$,
which after more perturbative errors allows us to replace the leading 
$\partial_0^2$ operator by $\partial \partial_x$, with a $\PP$ paracoefficient.
Then $\partial_x$ combines with $T_{\PP S^{-2}}$ to give $T_{\PP S^{-2}}$, 
thereby reducing the problem to the case of \eqref{IXsub-main}.

Finally in $I_{X,sub}^{041}$
we commute inside and distribute the $\partial_\alpha$ derivative,
peeling off perturbative errors. We arrive at
\[
\begin{aligned}
I_{X,sub}^{041} = & \  \iintT \partial_0 w \cdot T_{\PP S^{-2}} T_{\tg^{\alpha\delta}} L_{lh}
(\partial_x^2 \partial_\alpha \partial_\delta u, \partial_\beta w) \,dx dt 
\\ & +  \iintT \partial_0 w \cdot T_{\PP S^{-2}} T_{\tg^{\alpha\delta}} L_{lh}
(\partial_x^2  \partial_\delta u, \partial_\alpha \partial_\beta w) \,dx dt .
\end{aligned}
\]
The first term is estimated by commuting $T_{\tg^{\alpha\delta}}$ inside $L_{lh}$
and onto the first argument, after which we use Lemma~\ref{l:tp-du}. In the second term we pull $\partial_\beta$ out, reducing the problem either to 
$I_{X,sub}^{042}$, which was discussed earlier, or to
\[
 \iintT \partial_0 w \cdot T_{\PP S^{-2}} T_{\tg^{\alpha\delta}} L_{lh}
(\partial_x^2  \partial_\delta \partial_\beta u, \partial_\alpha  w) \,dx dt .
\]
But here we can pull out one of the $\partial_x$ operators to reduce to 
the case of \eqref{IXsub-main}.

After this discussion we have reduced the problem to the estimate for
$I_{sub,\gamma,\delta}$. Here, from easiest to hardest, we need to consider 
the case when neither of $\gamma$ or $\delta$ is zero, then when one of them is zero, and finally when none of them is zero. We will first illustrate the principle in the easiest case, and then describe the additional complications for the 
most difficult case. We leave the intermediate case for the reader.

\bigskip

\emph{ A. The case $\gamma,\delta \neq 0$.}
The argument here consists of three integrations by parts in a circular manner. 
Here we have $h \in \PP S^{-1}$. We may include  $\partial_\delta$ in $h$
in which case $h \in \PP S^0$. Separating variables and 
the problem can be further reduced to $h \in \PP$. In the computations below we omit $h$ altogether, as it does not play any role. Then it remains to bound the integral
\begin{equation}\label{IXsub-nz}
I_{sub} = \iint T_{\PP S^{0}} T_{\tg^{\alpha\beta}} L_{lh}(\partial_x^2 {\partial_\alpha} u,\  \partial_\beta w) \cdot  w \, dx dt .
\end{equation}
Similarly, derivatives applied to $g$ or  yield perturbative contributions,
of size $O(\BB^2)$, and will not be explicitly written in order
to avoid cluttering the formulas. In the absence of boundary terms, we compute as follows, integrating by parts in order to convert the null form into three $T_{\tP}$
operators modulo admissible errors: 
\[
\begin{aligned}
 I_{sub} = & \ \iintT
 w \cdot   T_{\tg^{\alpha\beta}} L_{lh}(\partial_\alpha \partial_x^2 
  u, \partial_\beta  w) \, dx dt
\\
= & \ - \iint 
\partial_\beta  w \cdot  T_{\tg^{\alpha\beta}} L_{lh}(\partial_\alpha \partial_x^2 u,  w) \, dx dt
- \iint 
 w \cdot   T_{\tg^{\alpha\beta}} L_{lh}(\partial_\beta \partial_\alpha \partial_x^2 u, \partial_0 w) \, dx dt + Err(\BB^2)
\\
= & \ \iint 
\partial_\alpha \partial_\beta  w \cdot   T_{\tg^{\alpha\beta}} L_{lh}( \partial_x^2 u,  w) \, dx dt
 \ +\iint 
\partial_\beta  w \cdot   T_{\tg^{\alpha\beta}} L_{lh}( \partial_x^2 u, \partial_\alpha  w) \, dx dt
\\
&  - \iint 
 w \cdot   T_{\tg^{\alpha\beta}} L_{lh}(\partial_\beta \partial_\alpha \partial_x^2 u, w) \, dx dt + Err(\BB^2)
\\
= & \ \iint 
\partial_\alpha \partial_\beta  w \cdot   T_{\tg^{\alpha\beta}} L_{lh}( \partial_x^2 u,  w) \, dx dt
 \ -  I_{sub} - \iint 
\ w \cdot   T_{\tg^{\alpha\beta}} L_{lh}( \partial_x^2 u, \partial_\alpha \partial_\beta  w) \, dx dt
\\
&  - \iint 
 w \cdot   T_{\tg^{\alpha\beta}} L_{lh}(\partial_\beta \partial_\alpha \partial_x^2 u, w) \, dx dt + Err(\BB^2).
\end{aligned}
\]
We now distribute $T_{\tg^{\alpha\beta}}$, noting that any commutator errors involve derivatives of $\tg$ and thus are perturbative.
We arrive at
\begin{equation}\label{C0-est}
\begin{aligned}
2I_{sub} = &  \  \iint 
T_{\tg^{\alpha\beta}} \partial_\alpha \partial_\beta  w \cdot    L_{lh}( \partial_x^2 u,  w) \, dx dt
 - \iint 
\  w \cdot   L_{lh}( \partial_x^2 u,  T_{\tg^{\alpha\beta}} \partial_\alpha \partial_\beta  w) \, dx dt
\\
&  - \iint 
 w \cdot  L_{lh}( T_{\tg^{\alpha\beta}} \partial_\beta \partial_\alpha \partial_x^2 u,  w) \, dx dt + Err(\BB^2).
\end{aligned}
\end{equation}
It remains to add the boundary terms at times $t=0,T$ into the above computation. Such boundary terms arise from the integration by parts with respect to $x_0$. We obtain the following enhanced version of 
\eqref{C0-est}:
\begin{equation}\label{C0-est+}
\begin{aligned}
2I_{sub} = &  \  \iintT 
T_{\tg^{\alpha\beta}} \partial_\alpha \partial_\beta  w \cdot    L_{lh}( \partial_x^2 u,  w) \, dx dt
 - \iintT 
\  w \cdot   L_{lh}( \partial_x^2 u,  T_{\tg^{\alpha\beta}} \partial_\alpha \partial_\beta  w) \, dx dt
\\
&  - \iintT 
 w \cdot  L_{lh}( T_{\tg^{\alpha\beta}} \partial_\beta \partial_\alpha \partial_x^2 u,  w) \, dx dt + Err(\BB^2)
\\ & + \left. \!\!\!
\int \! w \cdot   T_{\tg^{\alpha 0}} L_{lh}(\partial_\alpha \partial_x^2 u,  w) -\partial_\beta  w \cdot  T_{\tg^{0\beta}} L_{lh}( \partial_x^2 u,  w) +   w \cdot   T_{\tg^{\alpha 0}} L_{lh}( \partial_x^2 u, \partial_\alpha  w) \, dx
\right|_0^T. \!\!\!\!\!\!\!\!
\end{aligned}
\end{equation}
The boundary terms are easily seen as lower order energy corrections,
so it remains to estimate the interior contributions. For the first one
we can use the $w$ equation to get the fixed time bounds
\begin{equation}\label{Boxp-v}
\| T_{\tg^{\alpha\beta}} \partial_\alpha \partial_\beta  w\|_{L^2}
\lesssim \| \tP_{B} w\|_{L^2} + (\BB^2 + 2^{\frac{k}2} \BB c_k) \| \partial w\|_{L^2} ,
\end{equation}
which suffices\footnote{Here the $\tP_{B} w$ term may be interpreted as arising 
from a lower order correction to our multiplier $\tM$.}
by combining  the two components of the last term with either the $\AA$ or the $\BB$ bound for $u$ in $L_{lh}$.
The other two interior contributions reduce to the bound 
\begin{equation}\label{Boxp-u}
\|P_{<k}  T_{\tg^{\alpha\beta}} \partial_\alpha \partial_\beta \partial_x^2 u\|_{L^\infty}
\lesssim 2^{-k} \BB^2  ,
\end{equation}
which is a consequence of Lemma~\ref{l:tp-dxu} and which suffices to estimate the expressions in \eqref{C0-est+}.

\medskip

\emph{ B. The case $\gamma =\delta = 0$.}
Here we seek to estimate the integral
\begin{equation}\label{IXsub-00}
I_{sub,00} = \iint T_{\PP S^{-1}} T_{\tg^{\alpha\beta}} L_{lh}(\partial_x \widehat{\partial_\alpha \partial_0} u,\  \partial_\beta w) \cdot \partial_0 w \, dx dt .
\end{equation}
Here the hat correction plays a perturbative role and could be omitted.
However, in the computations below we need to keep it in order to be able to estimate
energy corrections. Our computations emulate the simpler case considered 
above, but with some care in order to avoid iterated time derivatives.
Integrating by parts in $\beta$ we get 
\[
\begin{aligned}
I_{sub,00} = & \  - \iint T_{\PP S^{-1}} T_{\tg^{\alpha\beta}} L_{lh}(\partial_x \partial_\beta \widehat{\partial_\alpha \partial_0} u,\  w) \cdot \partial_0 w \, dx dt  
\\ & - \iint T_{\PP S^{-1}} T_{\tg^{\alpha\beta}} L_{lh}(\partial_x  \widehat{\partial_\alpha \partial_0} u,\  w) \cdot \partial_0 \partial_\beta w \, dx
+ Err(\BB^2) + Err(\AA)|_0^T,
\end{aligned}
\]
where the last term accounts for the boundary contributions obtained when $\beta=0$.
In the first integral we perturbatively move $ T_{\tg^{\alpha\beta}}$
on the first $L_{lh}$ argument and then use Lemma~\ref{l:tp-dxu}; this allows us to move the entire first integral into the error, leading us to
\begin{equation}\label{Isub001}
   I_{sub,00} = - \iintT T_{\PP S^{-1}} T_{\tg^{\alpha\beta}} L_{lh}(\partial_x  \widehat{\partial_\alpha \partial_0} u,\  \partial_\beta w) \cdot \partial_0 \partial_\beta w \, dx
+ Err(\BB^2) + Err(\AA)|_0^T.
\end{equation}
On the other hand we can perturbatively drop the hat and integrate by parts in $\alpha$. This gives
\[
\begin{aligned}
I_{sub,00} = & \  - \iintT T_{\PP S^{-1}} T_{\tg^{\alpha\beta}} L_{lh}(\partial_x   \partial_0 u,\ \partial_\alpha \partial_\beta w) \cdot \partial_0 w \, dx dt  
\\ & - \iintT T_{\PP S^{-1}} T_{\tg^{\alpha\beta}} L_{lh}(\partial_x  \partial_0 u,\  w) \cdot \partial_0 \partial_\alpha w \, dx
+ Err(\BB^2) + Err(\AA)|_0^T.
\end{aligned}
\]
In the first integral we perturbatively move $ T_{\tg^{\alpha\beta}}$
on the second $L_{lh}$ argument; using the paradifferential $w$ equation where
the $\tA$ and $\tb$ terms are perturbative, this yields a lower order correction 
to our multiplier. Switching the $\alpha$ and $\beta$ indices we obtain 
\begin{equation}\label{Isub002}
\begin{aligned}
   I_{sub,00} = & \  - \iintT T_{\PP S^{-1}} T_{\tg^{\alpha\beta}} L_{lh}(\partial_x  \partial_0  u,\  \partial_\alpha w) \cdot \partial_0 \partial_\beta w \, dx
   + \iintT T_{\tP} w \cdot Q \partial_0 w \, dx dt
\\ & + Err(\BB^2) + Err(\AA)|_0^T,
\end{aligned}
\end{equation}
where $\|Q\|_{L^2 \to L^2} \lesssim A$.

The next step is to add the relations \eqref{Isub001} and \eqref{Isub002}.
Here we separate the cases $\alpha = j \neq 0$ and $\alpha=0$, where 
in the first case we can drop the hat and pull out the $\partial_\alpha$ in $L_{lh}$,
\[
\begin{aligned}
2 I_{sub,00} = & \  - \iintT T_{\PP S^{-1}} T_{\tg^{j\beta}} \partial_j L_{lh}(\partial_x  \partial_0  u,\   w) \cdot \partial_0 \partial_\beta w \, dx
\\ & \ - \iintT T_{\PP S^{-1}} T_{\tg^{0\beta}} [ L_{lh}(\partial_x  \widehat{\partial_0 \partial_0} u,\   w) + L_{hl}(\partial_x \partial_0 u, \partial_0 w)] \cdot \partial_0 \partial_\beta w \, dx
\\ & 
   + \iintT T_{\tP} w \cdot Q \partial_0 w \, dx dt + Err(\BB^2) + Err(\AA)|_0^T.
\end{aligned}
\]
In the first integral we integrate by parts to switch $\partial_0$ to the left
and then $\partial_j$ to the right. Then we distribute the $\partial_0$ on the left. This yields
\[
\begin{aligned}
2 I_{sub,00} = & \  - \iintT T_{\PP S^{-1}} T_{\tg^{j\beta}} \partial_j [ L_{lh}(\partial_x  \partial_0 \partial_0  u,\   w)
+ L_{lh}(\partial_x  \partial_0  u,\  \partial_0 w)] \cdot \partial_j \partial_\beta w \, dx
\\ & \ - \iintT T_{\PP S^{-1}} T_{\tg^{0\beta}} [ L_{lh}(\partial_x  \widehat{\partial_0 \partial_0} u,\   w) + L_{hl}(\partial_x \partial_0 u, \partial_0 w)] \cdot \partial_0 \partial_\beta w \, dx
\\ & 
   + \iintT T_{\tP} w \cdot Q \partial_0 w \, dx dt + Err(\BB^2) + Err(\AA)|_0^T.
\end{aligned}
\]
In the first integral we may perturbatively correct $\partial_0^2 u$; 
this allows us to put back together the cases when $\alpha$ is zero and nonzero,
\[
\begin{aligned}
2 I_{sub,00} = & \  - \iintT T_{\PP S^{-1}} T_{\tg^{\alpha\beta}} \partial_j [ L_{lh}(\partial_x  \widehat{\partial_0 \partial_0}  u,\   w)
+ L_{lh}(\partial_x  \partial_0  u,\  \partial_0 w)] \cdot \partial_\alpha \partial_\beta w \, dx
\\ & 
   + \iintT T_{\tP} w \cdot Q \partial_0 w \, dx dt + Err(\BB^2) + Err(\AA)|_0^T.
\end{aligned}
\]
Finally, we commute $T_{\tg^{\alpha\beta}} $ to the right factor, and use the 
$w$ equation to add another perturbative factor to our multiplier. This gives
\[
2 I_{sub,00} =  \iintT T_{\tP} w \cdot Q \partial_0 w \, dx dt + Err(\BB^2) + Err(\AA)|_0^T,
\]
with a modified $Q$, as desired.

\bigskip

The proof of the Proposition is now concluded.
\end{proof}

We now  conclude the proof of Theorem~\ref{t:para-wp}
using Proposition~\ref{p:IBP-s}, with the vector field $X_s$ and $Y_s$
chosen as in Proposition~\ref{p:Xs}. For these we have at our disposal not 
only the conclusion of Proposition~\ref{p:Xs}, but also the refined
version in  Proposition~\ref{p:Xs+}. This guarantees that the symbol $\tc_s$
in \eqref{tc-s} has size $\BB^2$, in the sense that its coefficients in 
\eqref{cXA2-exp} satisfy 
\[
\tc_s^j \in \BB^2 L^\infty S^j, \qquad 2^{-k} P_{<k} \partial_0 \tc_s^j \in \BB^2 L^\infty S^j.
\]
These conditions, in turn, guarantee that the  flux term $C_s$ in our energy estimate
\eqref{e:IBP-s} satisfies
\[
|C_s(w,w)| \lesssim \BB^2 \| \partial w\|_{L^2}^2,
\]
and thus the conclusion of Theorem~\ref{t:para-wp} follows.
\bigskip

\section{Energy estimates for the full equation} 
\label{s:ee-full-eqn}

Our objective here is to prove 
energy estimates for the solution $u$
to the minimal surface equation \eqref{msf} in $\H^s = H^s \times H^{s-1}$,
in terms of our control parameters $\AA$ and $\BB$.

\begin{theorem} \label{t:ee}
For each $s \geq 1$ there exists an energy functional $E^s_{NL}$ for the minimal surface  equation \eqref{msf-short} in $H^s \times H^{s-1}$ with the property that for all $\H^s$ solutions $u $ to \eqref{msf} with $\AAs \ll 1$  and $\BB \in L^2$ we have:

a) coercivity,
\begin{equation}
E^s_{NL} (u[t]) \approx \| u[t]\|_{\H^s}^2 
\end{equation}

b) energy bound, 
\begin{equation}
\frac{d}{dt} E^{s}_{NL}(u[t]) \lesssim \BB^2 E^{s}_{NL}(u[t]).
\end{equation}
\end{theorem}

The rest of this section is devoted to the 
proof of the theorem. This has two main ingredients:
\begin{enumerate}
    \item Reduction to the paradifferential equation, using normal form analysis.
    
    \item Energy estimates for the paradifferential equation, which have already been proved in Theorem~\ref{t:para-wp}.
\end{enumerate}
    
Hence, our primary objective here will be to carry 
out the above reduction. We recall the minimal surface equation,
\[
g^{\alpha \beta}  \D_{\alpha} \D_{\beta} u = 0.
\]
In order to use the energy estimates obtained in the previous section, we write this in  paradifferential form: 
\begin{equation}
\label{para-org-eq}
(T_{g^{\alpha \beta}} \D_{\alpha} \D_{\beta} - 2 T_{A^\gamma}\D_\gamma) u  = N(u),
\end{equation}
where the source term $N(u)$ is given by
\begin{equation}
\label{para-nonlinear}
N(u) = -\Pi (\D_{\alpha} \D_{\beta}u, g^{\alpha \beta})-T_{\D_{\alpha} \D_{\beta}u}g^{\alpha \beta}-2 T_{A^\gamma}\D_\gamma u.
\end{equation}

Here we cannot treat $N$ perturbatively; precisely, we do not have an estimate of the form
\[
\| N(u)(t) \|_{H^{s-1}} \lesssim_{\AA} \BB^2 \|u[t]\|_{\H^s},
\]
even though $N(u)$ is cubic in $u$, and the above inequality is dimensionally correct. This is because $N$ contains  some unbalanced contributions.

To address this issue, our strategy will be to correct $u$ via a well chosen normal form transformation, in order to eliminate the unbalanced part of $N(u)$.
But in order to do this, we have to first identify the unbalanced part of 
$N(u)$, and reveal its null structure. A first step in this direction is to 
better describe the contributions of the metric coefficients $g^{\alpha \beta}$
in $N$; explicitly we want to extract the renormalizable terms (i.e. the terms to which we can apply a normal form correction). For this we express $g^{\alpha \beta}$ paradifferentially as follows:

\begin{lemma}\label{l:g-para} The  metric $g^{\alpha \beta}$ can be expressed paradifferentially as follows
\begin{equation}
\label{better g-para-rep}
    g^{\alpha \beta}(u) =-T_{g^{\alpha \gamma}\partial^\beta u } \partial_{\gamma}u-T_{g^{ \gamma \beta}\partial^\alpha u } \partial_{\gamma}u +R(u),
\end{equation}
where $R (u)$ satisfies the following balanced bounds:
\begin{equation}
\label{Rb}
    \Vert R(u)\Vert_{H^{s-\frac{1}{2}}}\lesssim_\AA \BB \| \partial u\|_{H^{s-1}},
\end{equation}
as well as 
\begin{equation}
\label{Reb}
 \Vert R(u)\Vert_{H^{s-1}}\lesssim_\AA \| \partial u\|_{H^{s-1}}.
\end{equation}

\end{lemma}

\begin{proof} The representation in \eqref{better g-para-rep} and the bound \eqref{Rb} for $R$ follow from \eqref{dg-dp} and Lemma~\ref{l:R}. To get \eqref{Reb} one estimates each term in $R$ separately, using  no cancellations. 
\end{proof}

This suggests that the nonlinear contribution $N(u)$ should be seen as the sum of two terms
\[
N(u) =N_1(u)+N_2(u),
\]
where $N_1$ has null structure and $N_2$ is balanced,
\[
\begin{aligned}
N_1 (u) &= -2 \Pi (\D_{\alpha} \D_{\beta}u, T_{g^{\alpha \gamma}\partial^\beta u } \partial_{\gamma}u), \\
N_2 (u)&= - 2\left( T_{\partial_{\alpha}\partial_{\beta} u}T_{g^{\alpha \gamma}\partial^{\beta}u}\partial_{\gamma}u   -T_{\partial_{\alpha}\partial_{\beta} ug^{\alpha \gamma}\partial^{\beta}u}\partial_{\gamma}u \right)+T_{\partial_{\alpha}\partial_{\beta} u}R(u)+ \Pi (\D_{\alpha} \D_{\beta}u, R(u)).
\end{aligned}
\]

We will first prove that  $N_2(u)$ is a perturbative term:
\begin{lemma}
\label{l:n2}
The expression given by $N_2(u)$ satisfies the bound
\begin{equation}
    \Vert N_2(u)\Vert _{H^{s-1}}\lesssim \BB^2 \Vert \partial  u\Vert_{H^{s-1}}.
\end{equation}
\end{lemma}

\begin{proof} We begin with the first difference in $N_2$, and look separately at each $\alpha, \beta $ and $\gamma$. If $(\alpha, \beta)\neq (0,0)$ then we apply Lemma~\ref{l:para-prod} to obtain
\[
\begin{aligned}
\Vert\! \left( T_{\partial_{\alpha}\partial_{\beta} u}T_{g^{\alpha \gamma}\partial^{\beta}u}  -T_{\partial_{\alpha}\partial_{\beta} ug^{\alpha \gamma}\partial^{\beta}u} \right)\! \partial_{\gamma}u \Vert_{H^{s-1}} &\lesssim \Vert  \vert D\vert ^{-\frac{1}{2}} \partial_\alpha \partial_{\beta} u\Vert _{BMO} \Vert  \vert D\vert ^{\frac{1}{2}}\!\left( g^{\alpha \gamma}\partial^{\beta}u\right)\!\Vert _{BMO}\Vert \partial u\Vert_{H^{s-1}}\\
&\lesssim \BB ^2 \Vert \partial u\Vert_{H^{s-1}}.
\end{aligned}
\]
If  $(\alpha, \beta)= (0,0)$ then we use the wave equation for $u$ \eqref{msf-short} with the $\tilde{g}$ metric to write
\begin{equation}\label{dt2u-dec}
\partial_t^2 u=\left( T_{\tilde{g}}\partial_x\partial u+ T_{\partial_x\partial u}\tilde{g}\right) +\Pi (\tilde{g},\partial_x\partial u)=:\hat \partial_0^2 u +\pi_2(u),
\end{equation}
exactly as in Lemma~\ref{l:utt}. 
Then for the first term  we have the estimate 
\begin{equation}
\label{f1}
\Vert \hat \partial_0^2 u
\Vert_{BMO^{-\frac12}} \lesssim_{\AA} \BB
\end{equation}
which suffices in order to apply Lemma~\ref{l:para-prod} as above in order to estimate its contribution. 

On the other hand, the bound for the contribution of $\pi_2(u)$ is easier because by Lemma~\ref{l:utt} we have the direct uniform bound
\begin{equation}
\label{f2}
\Vert \pi_2(u)\Vert_{L^{\infty}}\lesssim_{\AA} \BB^2.
\end{equation}

Now, we turn our attention to the second term in $N_2(u)$, where we again discuss separately the  $(\alpha, \beta)\neq (0,0)$ and  $(\alpha, \beta) = (0,0)$ cases. 

For the $(\alpha, \beta)\neq (0,0)$ case we use the bound in \eqref{Rb} to obtain
\[
\Vert T_{\partial_{\alpha}\partial_{\beta}u}R(u)\Vert_{H^{s-1}}\lesssim 
\Vert\partial_{\alpha}\partial_{\beta}u \Vert_{BMO^{-\frac12}} \Vert R(u)\Vert_{H^{s-\frac12}} \lesssim_{\AA}
\BB^2 \Vert \partial u\Vert_{H^{s-1}}.
\]

Next we consider the case $(\alpha, \beta)= (0,0)$, and observe that we again need to use the 
decomposition \eqref{dt2u-dec}.
The  contribution of $\hat \partial_0^2 u$ is estimated using \eqref{f1} and the bound \eqref{Rb} for $R$, exactly as above:
\[
\Vert T_{\hat \partial_0^2 u}R(u)\Vert_{H^{s-1}}\lesssim \Vert \vert D\vert^{-\frac12}\hat \partial_0^2 u \Vert_{BMO} \Vert R(u)\Vert_{H^{s-\frac12}}\lesssim _{\AA}\BB ^2\Vert \partial u\Vert_{H^{s-1}}.
\]
 For the $\pi_2(u)$ contribution we use the pointwise bound \eqref{f2} and the $H^{s-1}$ bound \eqref{Reb} for $R$,
 \[
\Vert T_{\pi_2(u)}R(u)\Vert_{H^{s-1}}\lesssim \Vert \vert \pi_2(u)\Vert_{L^{\infty}} \Vert R(u)\Vert_{H^{s-1}}\lesssim _{\AA}\BB ^2\Vert \partial u\Vert_{H^{s-1}}.
\]
Finally, a similar analysis  leads to the  bound for  the balanced term $\Pi (\D_{\alpha} \D_{\beta}u, R(u))$.

\end{proof}

To account for the unbalanced part $N_1 (u)$ of $N$ we introduce a normal form correction 
\begin{equation}
\label{u-nft}
\tu =u- \Pi (\D_{\beta} u,  T_{\partial^{\beta}u } u):=u-u_2.
\end{equation}
Our goal will be to show that  the normal form variable solves a linear inhomogeneous paradifferential equation with a balanced source term. 

\begin{lemma}
\label{l:nfu2}The normal form correction above has the following properties:
\begin{itemize}
\item[a)] It is bounded,
\[
\Vert \tilde{u}[t] \Vert_{\H^s}\approx  \Vert u[t] \Vert_{\H^s}.
\]
\item[b)] It solves the an equation of the form
\begin{equation}
\label{important}
( \D_{\alpha} T_{g^{\alpha \beta}}\D_{\beta} -  T_{A^\gamma}\D_\gamma) \tu =N_2(u) + \partial_t R_1(u) +R_2(u),
\end{equation}
where 
\begin{equation}
\label{R1R2}
\Vert R_{1}(u)\Vert_{H^{s}}\lesssim \BB^2 \Vert u\Vert_{\H^{s}}, \qquad \Vert R_{2}(u)\Vert_{H^{s-1}}\lesssim \BB^2 \Vert  u\Vert_{\H^{s}},
\end{equation}
and 
\begin{equation}
\label{R1}
\Vert R_1 (u)\Vert _{H^{s-1}}\lesssim \AAs^2 \Vert u\Vert_{\H^{s}}.
\end{equation}
\end{itemize}
\end{lemma}
We remark that here we expand the meaning of ``balanced source terms'' to include
expressions of the form $\partial_t R_1$ with $R_1$ as above. This is 
required due to the fact that time derivatives are often more difficult to estimate in our context, and are allowed in view of the result in Theorem~\ref{t:lin-wp-inhom+}.

\begin{proof}
a) In view of the smallness of $\AAs$, for the boundedness of the normal form it
suffices to  show that 
\begin{equation}
\label{u2}
\begin{aligned}
 \qquad \Vert u_2\Vert_{H^s} \lesssim \AAs^2 \Vert u\Vert_{_{\H^s}},\\
\end{aligned}
\end{equation}
as well as
\begin{equation}
\label{dtu2}
\begin{aligned}
\qquad \Vert \partial_t u_2\Vert_{H^{s-1}} \lesssim \AAs^2 \Vert u\Vert_{_{\H^{s}}} .
\end{aligned}
\end{equation}
For the first bound we directly have 
\begin{equation}
    \Vert \Pi (\partial_{\beta}u, T_{\partial^{\beta}u}u)\Vert_{H^s} \lesssim \Vert \partial_{\beta}u\Vert _{H^{s-1}} \Vert  T_{\partial^{\beta}u}u\Vert_{BMO^1}\lesssim \AA^2 \Vert u\Vert _{\H^s} .
    \end{equation}
We prove the second bound in a similar manner, but we first apply the time derivative  and analyze each term separately: 
\[
\partial_t\Pi (\partial_{\beta}u, T_{\partial^{\beta}u}u)=  \Pi (\partial_t\partial_{\beta}u, T_{\partial^{\beta}u}u) +  \Pi (\partial_{\beta}u, T_{\partial^{\beta}u}\partial_t u)+ \Pi (\partial_{\beta}u, T_{\partial_t\partial^{\beta}u}u):=r_1+r_2+r_3.
\]
There are multiple cases arising from the strategy we will implement for terms involving two times derivatives, as well as from the particular structure of each of the terms.

We begin with $r_1$, where we need to separate the  $\beta \neq 0$ and $\beta = 0$ cases. The easiest case is when $\beta \neq 0$, where we have
\begin{equation}
    \label{pi1}
  \Vert \Pi (\partial_t\partial_{\beta}u, T_{\partial^{\beta}u}u)\Vert_{H^{s-1}}  \lesssim  \Vert  T_{\partial^{\beta}u}u\Vert_{H^s}
  \sup_k 2^k \Vert P_{<k} \partial_t\partial_{\beta}u\Vert_{L^\infty} \lesssim \AA^2 \Vert u\Vert_{_{\H^{s}}} .
\end{equation}
Here we have used the energy control we have for $\partial_t u$, which in turn gives control of all spatial derivatives of $\partial_t u$. 
For the case $\beta=0$ we use the decomposition for $\partial_t^2 u$ as in Lemma~\ref{l:utt}. For the first component we use the second bound in \eqref{fe-dt2u},
\[
 \Vert \Pi (\hat \partial_t^2 u, T_{\partial^{\beta}u}u)\Vert_{H^{s-1}}  \lesssim  \Vert  T_{\partial^{\beta}u}u\Vert_{H^s}
 \sup_{k} 2^{-k} \Vert P_{<k} \hat \partial_t^2 u\Vert_{L^\infty} \lesssim \AA^2 \Vert u\Vert_{_{\H^{s}}}.
\]
For the second component we argue similarly but using  the second bound in \eqref{fe-dt2u-err} together with Bernstein's inequality,
\[
 \Vert \Pi (\pi_2(u), T_{\partial^{\beta}u}u)\Vert_{H^{s-1}}  \lesssim  \Vert  T_{\partial^{\beta}u}u\Vert_{H^s}
 \sup_{k} 2^{-k} \Vert P_{<k} \pi_2(u) u\Vert_{L^\infty} \lesssim \AAs^2 \Vert u\Vert_{_{\H^{s}}}.
\]

\bigskip

We continue with the bound for $r_2$, where we do not need to distinguish between the time and space derivatives, 
\[
\Vert \Pi (\partial_{\beta}u, T_{\partial^{\beta}u}\partial_t u)\Vert_{H^{s-1}}\lesssim \Vert T_{\partial^{\beta}u}\partial_t u\Vert_{H^{s-1}} \Vert \partial_{\beta} u\Vert_{L^\infty} \lesssim  \AA \Vert \partial^{\beta} u\Vert _{L^{\infty}}  \Vert \partial_t u\Vert_{H^{s-1}}\lesssim \AA^2 \Vert u[t]\Vert_{\H^s}.
\]

%where we used the $BMO$ bound in  Proposition 2.2 (B.14) from \cite{HIT}.

Lastly, we need to bound $r_3$. Here we argue as in \eqref{pi1},  
\[
\Vert \pi_3\Vert_{H^{s-1}} \lesssim \Vert \partial_\beta u \Vert_{L^{\infty}} \Vert  T_{\partial_t \partial^{\beta}u} u\Vert_{H^{s-1}}
\lesssim \AA \Vert u\Vert_{H^s}
  \sup_k 2^{-k} \Vert P_{<k} \partial_t\partial^{\beta}u\Vert_{L^\infty},
\]
so it remains to prove the following bound for $\partial_t \partial^\beta u$,
\begin{equation}\label{pi3}
\Vert P_{<k} \partial_t\partial^{\beta}u\Vert_{L^\infty} \lesssim 2^{k} \AAs.
\end{equation}
Here we use the equation \eqref{msf-short} and the chain rule for the paracoefficient to write $\partial_t \partial^\beta u$
as a linear combination of $\partial^2_t u$ and $\partial_t\partial_x u$,
\[
\partial_t\partial^{\beta} u :=   \partial_t h( \partial u ) = h( \partial u ) \partial_t^2 u+ h( \partial u ) \partial_t\partial_x u 
\]
where $h:= h(\partial u)$. For the $\partial^2_t$ term we use again the equation  and schematically obtain 
\[
\partial_t\partial^{\beta} u :=   \partial_t h( \partial u )= \tilde{h} (\partial u)  \partial_x\partial u ,
\]
where $\tilde{h}$ incorporates the corresponding metric coefficients. As before, we need to use a Littlewood-Paley decomposition 
\[
\tilde{h} (\partial u)  \partial_x\partial u =T_{\tilde{h} (\partial u) } \partial_x\partial u  +T_{ \partial_x\partial u }\tilde{h} (\partial u) + \Pi (\tilde{h} (\partial u) , \partial_x\partial u ).
\]
The first two terms are easy to estimate using only  $L^\infty$ bounds for 
$\partial u$ and $h(\partial u)$,
\[
\Vert P_{<k} T_{\tilde{h} (\partial u) } \partial_x\partial u\Vert_{L^\infty}+ \Vert P_{<k}T_{ \partial_x\partial u }\tilde{h} (\partial u) \Vert_{L^\infty} \lesssim 2^{k} \AA.
\]
Finally for the third we use instead the $\AAs$ component of the $\CC_0$ norm for both $\partial u$ and $h(\partial u)$,
\[
\Vert P_{<k}\Pi(\tilde{h} (\partial u),  \partial_x\partial u)\Vert_{L^\infty} \lesssim 2^{k} \Vert \Pi(\tilde{h} (\partial u),  \partial_x\partial u)\Vert_{L^n}
\lesssim 2^k \AAs^2.
\]
Hence \eqref{pi3} follows. This finishes the proof of \eqref{dtu2}, and thus of the  boundedness from above and below of the normal form transformation in our desired Sobolev space $\H^s$.

\bigskip

b) We begin with the supposedly easier contribution, meaning with  the term $T_{A^\gamma}\D_\gamma u_2$.   To bound this term we would like to commute the $\partial_{\gamma}$ and place it in front of the product, 
\[
T_{A^\gamma}\D_\gamma u_2=\D_\gamma T_{A^\gamma} u_2 - T_{\partial_{\gamma}A^{\gamma}}u_2.
\]
This would look good for the first term on the RHS. However the last term would be problematic, as it may contain three derivatives with respect to time.  To avoid this issue we first substitute $A^{\gamma}$, which by \eqref{def-A} is given by
\[
{A}^{\gamma}:= \partial^{\alpha}u g^{\alpha \beta} \partial_{\alpha}\partial_{\beta}u,
\]
with the more manageable leading part  $\mathring{A}^{\gamma}$ given by
\[
\mathring{A}^{\gamma}:= T_{\partial^{\alpha}u}T_{g^{\alpha \beta}}\widehat{\partial_{\alpha}\partial_{\beta}u}.
\]
 Here the hat correction is from the Definition~\ref{d:dt2u}.  Then
\[
\begin{aligned}
T_{A^\gamma}\D_\gamma u_2&=\left( T_{A^\gamma}-T_{\rA^\gamma} \right) \D_\gamma u_2+T_{\rA^\gamma}  \D_\gamma u_2\\
&=\left( T_{A^\gamma}-T_{\rA^\gamma} \right) \D_\gamma u_2+ \D_\gamma T_{\rA^\gamma} u_2 -T_{\partial_{\gamma}\rA^{\gamma}}u_2.
\end{aligned}
\]
We will successively place each of these terms in $\partial_t R_1$ and $R_2$. We place the first term in $R_2$ using the bounds \eqref{u2} and \eqref{dtu2} for $u_2$. Then it remains  to bound the coefficient
\begin{equation}
\label{diffA}
 \Vert A^\gamma-\rA^\gamma \Vert_{L^{\infty}}\lesssim \BB^2.
\end{equation}
This is a direct consequence of Lemma~\ref{l:circle-app}.
We will place the second term in $\partial_t R_1$ if $\gamma =0$ and in $R_2$ if $\gamma \neq 0$. For this we measure $u_2$ in $H^{s+\frac{1}{2}}$,
\[
\Vert u_2\Vert _{H^{s+\frac{1}{2}}} \lesssim \AA \BB \Vert u\Vert_{\H^s}.
\]
This implies that for \eqref{R1} it is sufficient to bound the coefficient  
\[
 \Vert \rA^{\gamma} \Vert_{BMO^{-\frac{1}{2}}}\lesssim  \BB.
\]
This is also a consequence of Lemma ~\ref{l:circle-app}, see \eqref{ring-bd}. On the other hand for $R_1$ we need 
\[
 \Vert \rA^{\gamma} \Vert_{BMO^{-1}}\lesssim  \AA,
\]
which is similar.

The last term is placed in $R_2$ using again the bounds \eqref{u2} and \eqref{dtu2} for $u_2$ on one hand and Lemma~\ref{l:tp-du} on the other hand, to obtain 
\[
 \Vert P_{<k} \partial_{\gamma} \rA^{\gamma} \Vert_{L^{\infty}}\lesssim 2^k\BB^2.
\]

Now we consider the main term $\partial_{\alpha} T_{g^{\alpha \beta}}\partial_{\beta}u_2$, which  can be written in divergence form as
\[
\begin{aligned}
\partial_{\alpha} T_{g^{\alpha \beta}}\partial_{\beta}u_2 &= \partial_{\alpha} T_{g^{\alpha \beta}} \partial_{\beta}  \left[  \Pi (\D_{\gamma} u,  T_{\partial^{\gamma}u } u)\right]  \\
&= \partial_{\alpha} T_{g^{\alpha \beta}}   \left[  \Pi (\partial_{\beta}\D_{\gamma} u,  T_{\partial^{\gamma}u } u) + \Pi (\D_{\gamma} u,  T_{\partial_{\beta}\partial^{\gamma}u } u) + \Pi (\D_{\gamma} u,  T_{\partial^{\gamma}u } \partial_{\beta} u) \right].
\end{aligned}
\]
Depending on whether $\alpha =0$ or $\alpha\neq 0$, we place the middle term into $\partial_tR_1$  or $R_2$, respectively:
\[
\Vert \Pi (\D_{\gamma} u,   T_{\partial_{\beta}\partial^{\gamma}u } u)\Vert_{H^s}\lesssim \BB^2 \Vert u\Vert_{\H^s},
\]
\[
\Vert \Pi (\D_{\gamma} u,   T_{\partial_{\beta}\partial^{\gamma}u } u)\Vert_{H^{s-1}}\lesssim \AAs^2 \Vert u\Vert_{\H^s}.
\]
Here we use the property $\partial_{\beta}\partial^{\gamma}u \in \DCC$ to handle the case when $\beta=0$ for the first bound, and \eqref{pi3} for the second.

The first term can be rewritten in the form
\begin{equation}
\label{scary1}
\partial_{\alpha} T_{g^{\alpha \beta}}     \Pi (\partial_{\beta}\D_{\gamma} u,  T_{\partial^{\gamma}u } u) =\partial_{\alpha}    \Pi ( T_{g^{\alpha \beta}}  \widehat{\partial_{\beta}\D_{\gamma}} u,  T_{\partial^{\gamma}u } u) +\partial_tR_1+R_2,
\end{equation}
by using Lemma~\ref{l:para-assoc}, as well as Lemma~\ref{l:utt} for the case $(\beta, \gamma)=(0,0)$. Similarly the last term can be rewritten in the analogous form
\begin{equation}
\label{scary3}
\partial_{\alpha} T_{g^{\alpha \beta}}     \Pi (\D_{\gamma} u,  T_{\partial^{\gamma}u }\partial_{\beta} u) =\partial_{\alpha}    \Pi (   \D_{\gamma} u ,  T_{\partial^{\gamma}u} T_{g^{\alpha \beta}}  \partial_{\beta} u) +\partial_tR_1+R_2.
\end{equation}
Finally we distribute the $\alpha$ derivative in both \eqref{scary1} and \eqref{scary3}. For the first term on the right in \eqref{scary1} we get
\[
\begin{aligned}
\partial_{\alpha}    \Pi ( T_{g^{\alpha \beta}}  \widehat{\partial_{\beta}\D_{\gamma}} u,  T_{\partial^{\gamma}u } u) &= 
\Pi (\partial_\alpha  T_{g^{\alpha \beta}}  \widehat{\partial_{\beta}\D_{\gamma}} u,  T_{\partial^{\gamma}u } u)  + \Pi ( T_{g^{\alpha \beta}}  \widehat{\partial_{\beta}\D_{\gamma}} u,  T_{\partial_{\alpha}\partial^{\gamma}u } u) \\
& \quad + \Pi ( T_{g^{\alpha \beta}}  \widehat{\partial_{\beta}\D_{\gamma}} u,  T_{\partial^{\gamma}u } \partial_{\alpha}u) \\
&:=s_1+s_2+s_3.
\end{aligned}
\]
We place $s_1$ in $R_2$ using Lemma~\ref{l:tp-du},
\[
\Vert s_1\Vert_{H^{s-1}} \lesssim  \sup_k 2^{-k} \Vert P_{<k} \partial_\alpha  T_{g^{\alpha \beta}}  \widehat{\partial_{\beta}\D_{\gamma}} u\Vert_{L^{\infty}} \Vert u\Vert_{\H^s}\lesssim \BB^2 \Vert u\Vert_{\H^s}.
\]
The term $s_2$  is also estimated perturbatively using 
the fact that $\partial_{\alpha}\partial^{\gamma}u \in \DCC$,
which allows us to decompose it as a sum $f_1+f_2$ as in \eqref{DC-decomp}. Then we estimate
\[
\begin{aligned}
\Vert s_2\Vert_{H^{s-1}} \lesssim & \  \sup_k 2^{-\frac{k}{2}} \Vert P_{<k} T_{g^{\alpha \beta}}  \widehat{\partial_{\beta}\D_{\gamma}} u\Vert_{L^{\infty}}\sup_{j} 2^{-\frac{j}{2}}\Vert P_{<j} f_1\Vert_{L^{\infty}} \Vert u\Vert_{\H^s}
\\ & \ +  \sup_k 2^{-k} \Vert P_{<k} T_{g^{\alpha \beta}}  \widehat{\partial_{\beta}\D_{\gamma}} u\Vert_{L^{\infty}} \Vert  f_2\Vert_{L^{\infty}} \Vert u\Vert_{\H^s}
\\
\lesssim & \  \BB^2 \Vert u\Vert_{\H^s},
\end{aligned}
\]
using Lemma~\ref{l:utt} for $\widehat{\partial_{\beta}\D_{\gamma}} u$. In $s_3$ we can switch $T_g$ onto the other argument of $\Pi$ using Lemma~\ref{l:para-assoc} and remove the hat correction, so that it becomes half of $N_1$. 

The last remaining term to bound is the one on the RHS of \eqref{scary3}. Here we distribute again the $\alpha$-derivative
\[
\begin{aligned}
\partial_{\alpha}    \Pi (   \D_{\gamma} u ,  T_{\partial^{\gamma}u} T_{g^{\alpha \beta}}  \partial_{\beta} u)&=  \Pi (  \partial_{\alpha} \D_{\gamma} u ,  T_{\partial^{\gamma}u} T_{g^{\alpha \beta}}  \partial_{\beta} u)
+  \Pi (   \D_{\gamma} u ,  T_{\partial_{\alpha}\partial^{\gamma}u} T_{g^{\alpha \beta}}  \partial_{\beta} u)\\
&\quad + \Pi (   \D_{\gamma} u ,  T_{\partial^{\gamma}u} T_{\partial_{\alpha}g^{\alpha \beta}}  \partial_{\beta} u) + 
\Pi (   \D_{\gamma} u ,  T_{\partial^{\gamma}u} T_{g^{\alpha \beta}}  \partial_{\beta} \partial_{\alpha}u).
\end{aligned}
\]
By inspection we observe that the first term in the equality above is the second  half of $N_1$. The remaining three terms can be estimated perturbatively using exactly the  same approach as in the case of \eqref{scary1}.
\end{proof}

In view of Lemma~\ref{l:n2} we can include $N_2(u)$ into $R_2(u)$ in \eqref{important}, obtaining the shorter representation of the source term
\begin{equation}
\label{fimportant}
( \D_{\alpha} T_{g^{\alpha \beta}}\D_{\beta} -  T_{A^\gamma}\D_\gamma) \tu = \partial_t R_1(u) +R_2(u),
\end{equation}
where $R_1$ and $R_2$ satisfy the bounds \eqref{R1R2}
and \eqref{R1}. 

For the homogeneous paradifferential problem we have the $\H^s$ energy $E^s$ given by Theorem~\ref{t:para-wp}. We will use this to construct our desired  nonlinear energy $E^s_{NL}$ in Theorem~\ref{t:ee}. Because we have the source term $\partial_tR_1$, the associated nonlinear energy will not be simply given by $E^{s}(\tilde{u}[t])$. Instead, the correct energy is the one provided by Theorem~\ref{t:lin-wp-inhom+}, namely
\begin{equation}
E^s_{NL}(u[t]):=E^s(\tilde{u}[t]-r[t]),
\end{equation}
where the correction $r[t]$ is given by 
\begin{equation}
r[t]=\myvec{0\\ (T_{g^{00}})^{-1}R_1(u)}.
\end{equation}
Then by the estimate in \eqref{ee-corect} we obtain
\begin{equation}
    \dfrac{d}{dt}E^s_{NL} (u[t])\lesssim E^s(\tilde{u}[t]-r[t],r_1 )+ \BB^2E^s_{NL}(u[t]),
\end{equation}
where $r_1[t]$ is as in \eqref{new-source},
\[
r_1[t]:= \ \myvec{ (T_{g^{00}})^{-1} R_1 \\ (T_{g^{00}})^{-1}( R_2 - \partial_k T_{g^{k0}} (T_{g^{00}})^{-1} R_1 -  T_{g^{0k}} \partial_k (T_{g^{00}})^{-1} R_1) }.
\]
Our nonlinear energy $E^s_{NL}$ is coercive because $r[t]$ is small, 
\[
\Vert r\Vert_{\H^s} \lesssim \AA^2 \Vert u\Vert_{\H^s},
\]
due to the bound \eqref{R1}. Finally, we control the time derivative of the energy, because 
\[
\Vert r_1\Vert_{\H^s}\lesssim \BB^2 \Vert u\Vert_{\H^s}.
\]
This is due to the bound in \eqref{R1R2}.

%%%%%%%%%%%%%%%%%%%%%%%%%%%%%%%%%%%%%%%%%%%
%%%%%%%%%%%%%%%%%%%%%%%%%%%%%%%%%%%%%%%%%%%
%%%%%%%%%%%%%%%%%%%%%%%%%%%%%%%%%%%%%%%%%%%
\section{Energy and Strichartz estimates for the linearized equation} \label{s:linearized}
%%%%%%%%%%%%%%%%%%%%%%%%%%%%%%%%%%%%%%%%%%%
%%%%%%%%%%%%%%%%%%%%%%%%%%%%%%%%%%%%%%%%%%%
%%%%%%%%%%%%%%%%%%%%%%%%%%%%%%%%%%%%%%%%%%%

Our objective here is to prove that the homogeneous linearized equation  is well-posed in the  space $\H^\frac12$ (respectively $\H^\frac58$ in two dimensions)
and satisfies Strichartz estimates with an appropriate loss of derivatives, namely 
\eqref{Str-hom} with $S=S_{AIT}$,
under the assumption that the associated 
linear paradifferential equation has similar
properties. The main result of this section is as follows:

\begin{theorem} \label{t:linearized}
Let $s$ be as in \eqref{s-AIT3} (respectively \eqref{s-AIT2} in dimension $n = 2$). Let $u$ be a smooth solution for the minimal surface equation in a unit time interval, and which satisfies the energy and Strichartz bounds 
\begin{equation}\label{u-have-now}
\| u \|_{S^s_{AIT}} + \| \partial_t u\|_{S^{s-1}_{AIT}} \ll 1.     
\end{equation}
Assume also that the associated linear paradifferential equation 
\begin{equation}\label{paralin-inhom-re+}
\begin{aligned}
\D_{\alpha} T_{\hg^{\alpha \beta}} \D_{\beta} v  = f
\end{aligned}
\end{equation}
is well-posed in $\H^\frac12$ (respectively $\H^\frac58$ in dimension $n = 2$) in a unit time interval, and satisfies the full Strichartz estimates \eqref{Str-full+}, with $S= S_{AIT}$, in the same interval.

Then the homogeneous linearized equation 
\begin{equation}\label{lin-inhom}
\begin{aligned}
\D_{\alpha} \hg^{\alpha \beta} \D_{\beta} v  = 0
\end{aligned}
\end{equation}
is also well-posed in $\H^\frac12$ (respectively $\H^\frac58$ in dimension $n = 2$), and satisfies the homogeneous form of the  Strichartz estimates in \eqref{Str-hom} with $S = S_{AIT}$.
\end{theorem}

We continue with several comments on the result in the theorem, in order to better place it into context.
\begin{itemize}
\item Up to this point  we only know that both the full equation and the associated  linear paradifferential equation
satisfy good energy estimates, but we do not yet know that they also 
satisfy the corresponding Strichartz estimates. This is however not a problem, as the main result of this section, namely
Theorem~\ref{t:linearized} below, will only be used as a module within our main bootstrap argument in the last section of the paper,
by which time we will have already had established the energy and Strichartz estimates for both the full equation and for the
linear paradifferential equation.

\item The exponent $s$ in the above result need not be the same as the one in our main result in Theorem~\ref{t:main}; it can be taken to be smaller,
as long as it still satisfies the constraints in \eqref{s-AIT2}, \eqref{s-AIT3}. 

\item While we can no longer control the linearized evolution 
purely in terms of the control parameters $\AA$, $\AAs$, and $\BB$,
these still play role in the analysis. The hypothesis
of the theorem guarantees that
\[
 \AAs \ll 1, \qquad \| \BB \|_{L^2} \ll 1.
\]
\item The exponent $\delta$ in the \eqref{Str-full+} with $S= S_{AIT}$
should be thought of as being sufficiently small, compared with the distance between $s$  and the threshold in \eqref{s-AIT2}, \eqref{s-AIT3}. 

\item The smoothness assumption on $u$ is not used in any quantitative way. Its role 
is only to ensure that we already have solutions for the linearized problem, 
so we can skip an existence argument. Thus, by a density argument, the result of the theorem may be seen as an a-priori estimate for smooth solutions $v$ to the linearized equation.
As our rough solutions will be obtained in the last section as limits of smooth solutions, this assumption may be discarded a posteriori.

\item The reason we only consider the homogeneous case in the linearized equation 
\eqref{lin-inhom} is to shorten the proof, as this is all that is used later in the last section. However, the result also extends to the inhomogeneous case. In particular in dimension $n \geq 3$ this is an immediate consequence of Theorem~\ref{t:Str-move-around}, but in dimension $n=2$ an extra argument would be needed. 
\end{itemize}

One major simplification in this section, compared with the previous two sections,
is that  we no longer have the earlier difficulties in estimating the
second order time derivatives and even some third order time derivatives of $u$. In particular, we have  the following relatively straightforward lemma:

\begin{lemma}\label{l:all-u}
For solutions $u$ to the minimal surface equation as in
\eqref{u-have-now} we have the bounds
\begin{equation}\label{utt-easy}
\| \partial^2 u\|_{S^{s-2}_{AIT}}+ \| \partial \hg\|_{S^{s-2}_{AIT}} \ll 1,
\end{equation}
as well as 
\begin{equation}\label{hp-du}
\|\bD^{\sigma+\delta_0} \partial_\alpha T_{\hg^{\alpha \beta}} \partial_\beta \partial_\gamma  u\|_{L^p L^q} \ll 1  , 
\end{equation}
where $\delta_0 >0$ depends only on $s$ and
\[
\frac{1}{p} + \frac{1}q = 1, \qquad  \frac{2}{n-\frac12} \leq 
\frac{1}{q} \leq \frac12 +  \frac{2}{n-\frac12},
\qquad \sigma = \left(n-\frac12\right)\left(\frac{1}q - \frac{2}{n-\frac12}\right) \geq 0,
\qquad n \geq 3.
\]
respectively 
\[
\frac{2}{p} + \frac{1}q = 1, \qquad  \frac47 \leq 
\frac{1}{q} \leq \frac{11}{14},
\qquad \sigma = \frac{7}4 \left(\frac{1}q - \frac{4}7\right) \geq 0,
\qquad n =2.
\]
\end{lemma}

\begin{proof}
For \eqref{utt-easy} we only need to consider the second order 
time derivatives, which we can write using the minimal surface equation as
\[
\partial_t^2 u= h(\partial u) \partial_x \partial u.
\]
By Moser inequalities we have $\|h(\partial u)\|_{H^{s-1}} \ll 1$.
Since $s -1 > n/2$, it is easily verified that the space $S^{s-2}_{AIT}$
is left unchanged by multiplication by $h(\partial u)$.
The same argument applies to derivatives of the metric $\partial \hg$.

For \eqref{hp-du} we can use again the minimal surface equation to obtain the representation
\[
\D_\alpha T_{\hg^{\alpha \beta}}  \D_\beta \D_\gamma u
= T_{\partial_x \partial \hg} \partial u + T_{\partial \hg} \partial^2 u + \Pi(\partial \hg, \partial^2 u) + \Pi(\hg,\partial^3 u),
\]
where we can further write 
\[
\partial^3 u = \partial_x (\hg \partial^2 u) + \partial \hg \partial^2 u.
\]
Hence we need to multiply two functions in $S^{s-2}_{AIT}$, which contains a range of mixed $L^p$ norms at varying spatial Sobolev regularities. We can do this optimally if both mixed norms can be chosen to have non-negative Sobolev index. In order to avoid using Sobolev embeddings we further limit the 
range of exponents to the case when one of the Sobolev indices may be taken to be zero. This gives the range of exponents in the lemma.
\end{proof}

Next we discuss the strategy of the proof. The first potential strategy here would 
be to try to view the equation \eqref{lin-inhom}
as a perturbation of \eqref{paralin-inhom-re+}.
Unfortunately such a strategy does not seem to 
work in our context, because this would require 
a balanced estimate for the difference 
between the two operators, whereas this difference
contains some terms which are clearly unbalanced.

To address the above difficulty, the key observation is that the aforementioned difference
exhibits a null structure, at least in its unbalanced part. This opens the door to 
a partial normal form analysis, in order 
to develop a better reduction of \eqref{lin-inhom}
to \eqref{paralin-inhom-re+}.
Because of this, the proof of the theorem will be done in two steps:

\begin{enumerate}
    \item The normal form analysis, where a suitable normal form transformation is constructed.
    
    \item Reduction to the paradifferential equation, using the above normal form.
    
\end{enumerate}

\subsection{Preliminary bounds for the linearized variable}
The starting point for the proof of the theorem is to rewrite the divergence form of the linearized equation \eqref{lin-inhom}
as an inhomogeneous paradifferential evolution \eqref{paralin-inhom-re+} with a perturbative source $f$, as follows:
\begin{equation}\label{paralin}
\begin{aligned}
T_\ag v = -\D_{\alpha} T_{\D_{\beta} v} \hg^{\alpha \beta} - \D_{\alpha} \Pi(\D_{\beta} v, \hg^{\alpha \beta}) =: f.
\end{aligned}
\end{equation}
While we cannot directly prove a balanced cubic estimate for $f$,
a useful initial step is to establish a quadratic estimate for it.
The expression for $f$ involves $v$ and $\partial_t v$, which we already control,
but also $\partial_t^2 v$, which we do not. So we estimate it first:

\begin{lemma}\label{l:vtt}
For solutions $v$ to \eqref{lin-inhom} we have 
\begin{equation}\label{vtt-est}
\begin{aligned}
\| \partial_t^2 v(t) \|_{H^{-\frac32}} &\lesssim \|v[t]\|_{H^\frac12} \qquad n \geq 3, \\
\| \partial_t^2 v(t) \|_{H^{-\frac{11}{8}}} &\lesssim \|v[t]\|_{H^\frac58} \qquad n = 2.
\end{aligned} 
\end{equation}
\end{lemma}
\begin{proof}
We consider the case $n \geq 3$, and comment on the case $n=2$ at the end. Using the equation \eqref{lin-inhom} for $v$, we may write 
\[
\partial_t^2 v = h(\partial u) \partial_x \partial v 
+  h(\partial u) \partial^2 u \partial v .
\]
Here by Moser inequalities we have
\[
\|h(\partial u)\|_{H^{s-1}} \ll 1.
\]
Then, using also \eqref{utt-easy}  the conclusion of the Lemma follows from the straightforward  multiplicative estimates
\[
H^{s-1} \cdot H^{-\frac32} \to H^{-\frac32},
\qquad H^{s-1} \cdot H^{s-2} \cdot H^{-\frac12} \to H^{-\frac32}.
\]
where it is important that $s > \frac{n}2+1$ and $s > \frac52$.
This last condition is not valid in dimension $n=2$, where we only ask that $s > 2+\frac38$. This is why the Sobolev exponents in this case need to be increased by $\frac18$.

\end{proof}

We now return to the quadratic estimate for the source term in \eqref{paralin}:

\begin{lemma}\label{l:pre-paradiff-eqn}
Let $v \in S^\frac12_{AIT}$ satisfy \eqref{lin-inhom}. Then $v$ also solves the inhomogeneous paradifferential equation
\begin{equation}
\begin{aligned}
T_{\ag} v &= f,
\end{aligned}
\end{equation}
with source term $f$  satisfies the following bounds:

a) For $n \geq 3$ we have the uniform bound
\begin{equation}\label{easy3}
\| f\|_{H^{-\frac54}} \ll  \|v\|_{L^\infty \H^\frac12},  
\end{equation}
and the space-time bound
\begin{equation}\label{strong-F-est3}
\|f\|_{L^p L^q}  \ll %\|\D u\|_{L^p W^{\frac32, n - \frac12}}
\|v\|_{L^\infty \H^\frac12}
\end{equation}
with
\[
\frac{1}{q} = \frac{1}{n - \frac12} + \frac12, \qquad \frac{1}{p} + \frac{1}{q} 
 = 1.
\]
a) For $n =2 $ we have the uniform bound
\begin{equation}\label{easy2}
\| f\|_{H^{-1}} \ll  \|v\|_{L^\infty \H^\frac58},  
\end{equation}
and the space-time bound
\begin{equation}\label{strong-F-est2}
\begin{aligned}
\|f\|_{L^4 L^2}  \ll % \|\D u\|_{L^\infty H^{2+\frac38}}
\|v\|_{L^\infty \H^\frac58}.
\end{aligned}
\end{equation}
\end{lemma}

\begin{proof}
To avoid cluttering the notations, we prove the result in the case $n \geq 3$. The two dimensional case is identical up to appropriate
adjustments of $L^p$ exponents.

We write
\[
-f = T_{ \D_{\alpha}\D_{\beta} v} \hg^{\alpha \beta} + T_{\D_{\beta} v} \D_{\alpha} \hg^{\alpha \beta} + \Pi(\D_{\alpha}\D_{\beta} v, \hg^{\alpha \beta}) +  \Pi(\D_{\beta} v, \D_{\alpha}\hg^{\alpha \beta}).
\]
For the two terms where $\D_\alpha$ has fallen on $g$, we have 
\[
\|T_{\D_{\beta} v} \D_{\alpha} \hg^{\alpha \beta} + \Pi(\D_{\beta} v, \D_{\alpha}\hg^{\alpha \beta}) \|_{L^q} \lesssim \|v\|_{\H^\frac12} \| \bD^\half \D_{\alpha}\hg^{\alpha \beta}\|_{L^{n - \frac12}} \lesssim_\AA \|v\|_{\H^\frac12}\|\D u\|_{W^{\frac32, n - \frac12}}.
\]
Finally, in the cases where the $\D_\alpha$ has fallen on $v$, we easily obtain the same estimate due to a good balance of derivatives. Here we use 
Lemma~\ref{l:vtt} to bound second derivatives of $v$.

\end{proof}

\subsection{The normal form analysis}
The estimate in Lemma~\ref{l:pre-paradiff-eqn} is suboptimal as it does not recognize the cubic structure of the source. This is due to components of the source in which the linearized variable $v$ is the second highest frequency, and which are not efficiently balanced with respect to derivatives. In fact, these cubic terms may heuristically be viewed as quadratic with a low frequency coefficient.

To better understand the source terms, we begin with a better 
description of the metric coefficients.
By applying Lemma~\ref{l:R} to $g^{-\half}g^{\alpha \beta}$ and rearranging, we may write 
\begin{equation}\label{g-paradiff}
\hg^{\alpha \beta} = - T_{\hg^{\alpha \beta} \partial^{\gamma} u + \hg^{\alpha \gamma} \partial^{\beta} u + \hg^{\beta \gamma}  \partial^{\alpha} u}\partial_{\gamma}u + R(\D u) ,
\end{equation}
where $R$ satisfies favourable balanced bounds,
\begin{equation}\label{good-R}
\| \partial R \|_{L^n} \lesssim \AAs^2, \qquad \| \partial R \|_{L^\infty} \lesssim \BB^2.
\end{equation}

To obtain a cubic estimate for \eqref{paralin}, we substitute \eqref{g-paradiff} in \eqref{paralin} and write
\begin{equation}\label{paralin2}
\begin{aligned}
T_{\ag} v &= N_{1}(u) + N_{2}(u),
\end{aligned}
\end{equation}
where
\begin{equation*}
\begin{aligned}
N_{1}(u) &= \D_\alpha T_{\D_\beta v} T_{\hg^{\alpha \beta} \partial^{\gamma} u + \hg^{\alpha \gamma} \partial^{\beta} u + \hg^{\beta \gamma}  \partial^{\alpha} u} \D_\gamma u + \D_\alpha \Pi(\D_\beta v, T_{\hg^{\alpha \beta} \partial^{\gamma} u + \hg^{\alpha \gamma} \partial^{\beta} u + \hg^{\beta \gamma}  \partial^{\alpha} u} \D_\gamma u)
\end{aligned}
\end{equation*}
consists of the essentially quadratic, nonperturbative components, while 
\begin{equation*}
\begin{aligned}
N_{2}(u) &= - \D_\alpha T_{\D_{\beta}v}R(\D u) -\D_\alpha \Pi(\D_{\beta}v, R(\D u))
\end{aligned}
\end{equation*}
consists of the balanced, directly perturbative components. We address the essentially quadratic components in $N_1(u)$ by passing to a renormalization $\tilde v$ of $v$, 
\begin{equation}\label{lin-nf}
\begin{aligned}
\tilde v &= v - T_{T_{\D^\gamma u}  \D_\gamma v} u - T_{T_{\D^\gamma u} v}\D_\gamma u -  \Pi(T_{\D^\gamma u} \D_\gamma v, u) - \Pi(T_{\D^\gamma u} v, \D_\gamma u) := v+ v_2.
\end{aligned}
\end{equation}

This renormalization eliminates the components of the source where the linearized variable $v$ is the second highest frequency. We thus replace $N_1(u)$ with a source consisting of terms with $v$ only at the third highest frequency, and hence may be viewed as authentically cubic.

\begin{proposition}\label{p:pre-paradiff-eqn}
Let $v \in S^\frac12_{AIT}$ be a solution for \eqref{lin-inhom}. 
Then the following two properties hold:

(i) Equivalent norms:
\begin{equation}\label{bdd-nf}
\| \tilde v[t] \|_{\H^\frac12} \approx \| v[t] \|_{\H^\frac12}    .
\end{equation}

(ii) $\tilde v$ solves a good linear paradifferential 
equation of the form
\begin{equation}\label{nf-paradiff-eqn}
\begin{aligned}
T_{\ag}\tilde v &= \D_t f_1 + f_2
\end{aligned}
\end{equation}
where the source terms are perturbative: 
\begin{equation}\label{strong-F-est}
\|f_2\|_{(S^\frac12_{AIT})'} \ll \|v\|_{S^\frac12_{AIT}}, \qquad \|f_1\|_{(S^{-\frac12}_{AIT})'} \ll \|v\|_{L^\infty \H^\frac12},
\end{equation}
as well as
\begin{equation}\label{extra-F-est}
\|f_1 (t)\|_{S^{-\frac12}_{AIT}} \lesssim \AA \|v[t]\|_{\H^\frac12}   . 
\end{equation}

The same result holds in two space dimensions at the level of $v \in \H^{\frac58}$.
\end{proposition}

\begin{proof}
(i) For the bound \eqref{bdd-nf}, it suffices to  estimate $v_2$ as follows:
\begin{equation}\label{v2-nf}
\| v_2(t) \|_{H^\frac12} \lesssim \AAs \| v[t]\|_{\H^\frac12},    
\end{equation}
\begin{equation}\label{dtv2-nf}
\| \partial_t v_2(t) \|_{H^{-\frac12}} \lesssim \AAs \| v[t]\|_{\H^\frac12}.
\end{equation}

The first term of $v_2$ can be directly estimated using scale invariant $\AA$ bounds,
\[
\|  T_{T_{\D^\gamma u}  \D_\gamma v} u\|_{H^\frac12} \lesssim 
\| T_{\D^\gamma u}  \D_\gamma v \|_{H^{-\frac12}} \| u \|_{Lip}
\lesssim \| \partial u \|_{L^\infty} \|  \D_\gamma v \|_{H^{-\frac12}} \| u \|_{Lip} \lesssim \AA^2\|  \D_\gamma v \|_{H^{-\frac12}}.
\]
The third and the fourth terms are similar. However, for the second
term we need to use the $\AAs$ control norm combined with Bernstein's inequality:
\[
\|  T_{T_{\D^\gamma u} v} \D_\gamma u\|_{H^\frac12} \lesssim 
\| T_{\D^\gamma u}  v \|_{H^{\frac12}} \| \D u \|_{W^{\frac12,2n}}
\lesssim \| \partial u \|_{L^\infty} \|  v \|_{H^{\frac12}} \| \D u \|_{W^{\frac12,2n}} \lesssim \AA \AAs \|   v \|_{H^{\frac12}}.
\]

We next consider \eqref{dtv2-nf}, where we distribute the time derivative,
obtaining several types of terms:
\smallskip

a) Terms with distributed derivatives, namely  
$T_{T_{\partial u} \partial v} \partial u$ and $\Pi(T_{\partial u} \partial v , \partial u)$. We estimate the first
by 
\[
\| T_{T_{\partial u} \partial v} \partial u\|_{H^{-\frac12}}
\lesssim \|T_{\partial u} \partial v\|_{H^{-\frac12}}
\| \partial u\|_{L^\infty} \lesssim \AA^2 \| \partial v\|_{H^{-\frac12}},
\]
and the second, using Sobolev embeddings, by 
\[
\| \Pi(T_{\partial u} \partial v , \partial u)\|_{H^{-\frac12}}
\lesssim \| \Pi(T_{\partial u} \partial v , \partial u)\|_{L^\frac{2n}{n-1}} \lesssim \| T_{\partial u} \partial v \|_{H^{-\frac12}} \| \partial u\|_{W^{\frac12,2n}} \lesssim \AA\AAs 
\| \partial v \|_{H^{-\frac12}}.
\]
\smallskip

b) Terms with two derivatives on the high frequency $u$,
namely $T_{T_{\partial u}  v} \partial^2 u$ and $\Pi(T_{\partial u} v , \partial^2 u)$. In view of the bound \eqref{utt-easy}, the corresponding estimate is nearly identical to case (a) above.

\smallskip

c) Terms with $\partial_t \partial^\gamma u$. Here we know that 
$\partial_t \partial^\gamma u \in H^{s-2}$, so we arrive at estimates which are also similar to case (a).

\smallskip

d) Terms with two derivatives on $v$. If one of them is spatial (i.e. $\gamma \neq 0$) then this is similar to or better than case (a). So we are left with the expressions $T_{T_{\partial u} \partial_t^2 v}  u$ and $\Pi(T_{\partial u} \partial_t^2 v ,  u)$. But there we can use the bound
\eqref{vtt-est} and complete the analysis again as in case (a).

\bigskip

(ii) The proof of \eqref{nf-paradiff-eqn} along with the estimates \eqref{strong-F-est}, \eqref{extra-F-est} will occur in four steps. 

\

1)  We first estimate the balanced source term component $N_2(u)$ from \eqref{paralin2}. We consider below the paraproduct $T$ term, but the $\Pi$ term is similar. We first consider the cases where the outer derivative $\D_\alpha = \D_i$ is a spatial derivative, which we will place in $f_2$. We have by Lemma~\ref{l:Moser-control} (see \eqref{good-R} above)
\[
\| \D_i T_{\D_{\beta}v} R(\D u)\|_{H^{-1/2}} \lesssim \|v\|_{\H^{1/2}}\|\partial R(\D u)\|_{L^\infty} \lesssim_{\AA} \BB^2 \|v\|_{\H^{1/2}},
\]
where the $\BB^2$ factor is integrable in time.
We place the case where $\D_\alpha = \D_0$ in $\D_t f_1$, estimating 
\[
\| T_{\D_{\beta}v} R(\D u)\|_{H^{1/2}} \lesssim \|v\|_{\H^{1/2}}\|\partial R(\D u)\|_{L^\infty} \lesssim_{\AA} \BB^2 \|v\|_{\H^{1/2}},
\]
and
\begin{equation*}
\begin{aligned}
\|T_{\D_{\beta}v} R(\D u)\|_{H^{-1/2}} \lesssim \|T_{\D_{\beta}v} R(\D u)\|_{L^{\frac{2n}{1 + n}}}\lesssim \|v\|_{\H^{1/2}}\|\partial R(\D u)\|_{L^{n}} \lesssim \AAs^2 \|v\|_{\H^{1/2}},
\end{aligned}
\end{equation*}
as well as 
\begin{equation*}
\begin{aligned}
\|\bD^{-\frac{n}2-\frac14}  T_{\D_{\beta}v} R(\D u)\|_{L^\infty} \lesssim \|T_{\D_{\beta}v} R(\D u)\|_{L^{2}}\lesssim \|\partial v\|_{\H^{-1/2}}\|\partial R(\D u)\|_{L^{2n}} \lesssim \AA\BB \|v\|_{\H^{1/2}}.
\end{aligned}
\end{equation*}
Here we have a single $\BB$ factor, which is $L^2$ in time, as needed 
for the $L^2 L^\infty$ Strichartz norm in \eqref{extra-F-est}.

\bigskip

2) Next, we apply product and commutator lemmas to exchange $N_1(u)$ for an equivalent expression up to perturbative errors, in preparation for comparison with the contribution from the normal form corrections. Here, we discuss the first term of $N_1(u)$,
\begin{equation}\label{sample-n1}
\D_\alpha T_{\D_\beta v} T_{\hg^{\alpha \beta} \partial^{\gamma} u} \D_\gamma u,
\end{equation}
but the remaining terms, including the balanced $\Pi$ terms, are similar, using the analogous product and commutator lemmas. 
We first consider the cases where the outer derivative $\D_\alpha = \D_i$ is a spatial derivative, and place all perturbative errors in $f_2$. By an application of product and commutator Lemmas~\ref{l:para-prod} and \ref{l:para-com}, we may replace \eqref{sample-n1} with
\[
\D_\alpha T_{\partial^{\gamma} u} T_{\D_\beta v} T_{\hg^{\alpha \beta}} \D_\gamma u.
\]
Then applying Lemmas~\ref{l:para-com} and \ref{l:para-prod2}, it suffices to consider
\[
\D_\alpha T_{\hg^{\alpha \beta}} T_{T_{\D^\gamma u}\D_\beta v} \D_\gamma u.
\]

In the case where $\D_\alpha = \D_0$, we place all perturbative errors in $\D_t f_1$. The bound for $f_1$ in \eqref{strong-F-est} is similar to the one for $f_2$, but there is a price to pay, namely that we also need to 
prove \eqref{extra-F-est}. Fortunately for \eqref{extra-F-est}
we may disregard all commutator structure and discard all the para-coefficients, as they are bounded and gain an $\AA$ factor, so we are left with proving a bound of the form
\[
\| T_{\partial v} \partial u\|_{S_{AIT}^{-\frac12}} \lesssim \|\partial v\|_{L^\infty H^{-\frac12}}.
\]
Here for the uniform bound we simply write at fixed time
\[
\|  T_{\partial v} \partial u\|_{H^{-\frac12}} \lesssim 
 \|\partial v\|_{ H^{-\frac12} } \| \partial u\|_{L^\infty}
 \lesssim \AA  \|\partial v\|_{ H^{-\frac12} },
\]
and for the $L^2 L^\infty$ bound we have
\[
\|  T_{\partial v} \partial u\|_{L^2} \lesssim 
 \BB \|\partial v\|_{ H^{-\frac12} } 
\]
using $\BB$ for the square integrability in time and then applying Bernstein's inequality in space to convert the  $L^2$ bound into $L^\infty$.

Applying the same analysis to the other terms of $N_1(u)$, we have reduced to
\begin{equation*}
\begin{aligned}
N_1'(u) &= \D_\alpha (T_{\hg^{\alpha \beta}}  T_{T_{\D^\gamma u}\D_\beta v}\D_\gamma u + T_{\hg^{\alpha \gamma}}  T_{T_{\D^\beta u}\D_\beta v}\D_\gamma u + T_{\hg^{\beta \gamma}}  T_{T_{\D^\alpha u}\D_\beta v}\D_\gamma u) \\
&\quad + \D_\alpha (\Pi(T_{\D^\gamma u} \D_\beta v, T_{\hg^{\alpha \beta}}\D_\gamma u) + \Pi( T_{\D^\beta u}\D_\beta v, T_{\hg^{\alpha \gamma}}\D_\gamma u) + \Pi(T_{\D^\alpha u}\D_\beta v, T_{\hg^{\beta \gamma}}\D_\gamma u)).
\end{aligned}
\end{equation*}

\

3) We next establish the cancellation between the normal form correction and $N_1'(u)$. In this step, we discuss only the low-high $T$ paraproduct contributions, and return to the $\Pi$ contributions in Step 4. Applying $T_{\ag}$ to the $T$ term of $v_2$ in \eqref{lin-nf}, we have the contribution
\begin{equation}\label{div-para-nf}
\begin{aligned}
- \D_{\alpha}T_{\hg^{\alpha \beta}} \D_{\beta} (T_{T_{\D^\gamma u}\D_\gamma v} u + T_{T_{\D^\gamma u} v}\D_\gamma u  ).
\end{aligned}
\end{equation}

a) We first would like to observe that the cases where the derivatives $\D_\beta$ and $\D_\gamma$ are split, between $v$ and the high frequency $u$, cancel with the first two terms of $N_1'(u)$. The main task to verify before doing so is that the cases where the $\D_\beta$ falls on the lowest frequency para-coefficient $\D^\gamma u$ are perturbative due to an efficient balance of derivatives, and may be absorbed into $f_2$ or $f_1$. To see this, we analyze separately cases involving spatial versus time derivatives. In the case of spatial derivatives $\D_\alpha = \D_i$ and $\D_\beta = \D_j$, we directly estimate 
\[
\|\D_{i}T_{\hg^{ij}} T_{T_{\D_{j} \D^\gamma u}\D_\gamma v} u\|_{H^{-1/2}}  \lesssim_{\AA} \BB^2 \|v\|_{\H^{1/2}}.
\]
In the case where $\D_\beta = \D_0$, we obtain the same estimate in the same manner, except when $\D_\gamma = \D_0$. In this case, we may use Lemma~\ref{l:utt} to estimate the lowest frequency $\D_0^2 u$.

It remains to consider the case $\D_\alpha = \D_0$, which we put in $\D_t f_1$. We have
\[
\|T_{\hg^{0\beta}} T_{T_{\D_{\beta} \D^\gamma u}\D_\gamma v} u\|_{H^{1/2}}  \lesssim_{\AA} \BB^2 \|v\|_{\H^{1/2}}
\]
as before. For $f_1$ however, we also require an  estimate 
for the full Strichartz norm in \eqref{extra-F-est}.
We separate $\D_\beta$ again into spatial and time derivatives. For the spatial case, we have by Sobolev embeddings,
\begin{equation*}
\begin{aligned}
\|T_{\hg^{0j}} T_{T_{\D_{j} \D^\gamma u}\D_\gamma v} u\|_{H^{-1/2}} \lesssim \| T_{T_{\D_{j} \D^\gamma u}\D_\gamma v} u\|_{L^{\frac{2n}{1 + n}}} 
\lesssim \|\bD^{1/2}\D^\gamma u\|_{L^{2n}} \| v\|_{\H^{1/2}}\|\D u\|_{L^{\infty}} 
\lesssim \AA \AAs \| v\|_{\H^{1/2}}
\end{aligned}
\end{equation*}
for the uniform bound, as well as 
\begin{equation*}
\begin{aligned}
\|T_{\hg^{0j}} T_{T_{\D_{j} \D^\gamma u}\D_\gamma v} u\|_{L^2} &\lesssim  \|\bD^{1/2}\D^\gamma u\|_{L^{2n}} \| v\|_{\H^{1/2}}\|\D u\|_{BMO^\frac12} 
\lesssim \BB \AAs \| v\|_{\H^{1/2}}
\end{aligned}
\end{equation*}
for the $L^2 L^\infty$ bound.

For the case $\D_\beta = \D_0$, the lowest frequency includes an instance of $\D_0^2 u$, where we apply Lemma~\ref{l:utt}. This contributes a spatial component $\hat \partial^2_t u$ which is estimated as before, as well as a balanced $\Pi$ interaction, namely $\pi_2(u)$. This case is estimated by 
\begin{equation*}
\begin{aligned}
\| T_{\hg^{00}} T_{T_{\pi_2(u)}\D_0 v} u\|_{H^{-1/2}} &\lesssim \AA \| T_{T_{\pi_2(u)}\D_0 v} u\|_{L^{\frac{2n}{1 + n}}} 
\lesssim \AA \|\pi_2(u)\|_{L^n} \| v\|_{\H^{1/2}}\|\D u\|_{L^\infty} 
\lesssim \AA^2\AAs^2 \| v\|_{\H^{1/2}}
\end{aligned}
\end{equation*}
for the energy norm respectively 
\begin{equation*}
\begin{aligned}
\| T_{\hg^{00}} T_{T_{\pi_2(u)}\D_0 v} u\|_{L^2} &\lesssim \AA \| T_{T_{\pi_2(u)}\D_0 v} u\|_{L^2} \lesssim \AA \|\pi_2(u)\|_{L^{2n}} \| v\|_{\H^{1/2}}\|\D u\|_{L^\infty} \lesssim \AA^2\AAs \BB \| v\|_{\H^{1/2}}
\end{aligned}
\end{equation*}
for the $L^2 L^\infty$ bound.

Having dismissed the perturbative cases via this analysis, we observe an exact cancellation with the first two terms of $N_1'(u)$. Collecting the remaining paraproduct terms from $N_1'(u)$ and \eqref{div-para-nf}, we have
\begin{equation}\label{Tsubstep1}
\begin{aligned}
\D_\alpha T_{\hg^{\beta \gamma}}  T_{T_{\D^\alpha u}\D_\beta v}\D_\gamma u - \D_{\alpha}T_{\hg^{\alpha \beta}} T_{T_{\D^\gamma u} \D_\beta \D_\gamma v} u 
- \D_{\alpha}T_{\hg^{\alpha \beta}} T_{T_{\D^\gamma u}v} \D_\beta \D_\gamma  u.
\end{aligned}
\end{equation}

\medskip

b) Before proceeding, we rewrite the first term in \eqref{Tsubstep1} as follows, with the key step being an integration by parts which reveals an instance of $T_\ag v$. Reindexing, we have
\[
\D_\gamma T_{\hg^{\beta \alpha}}  T_{T_{\D^\gamma u}\D_\beta v}\D_\alpha u.
\]
Then applying Lemmas~\ref{l:para-prod2} and \ref{l:para-com} to commute $T_{\hg^{\alpha \beta}}$, similar to step 2), we replace this by
\[
\D_\gamma T_{T_{\D^\gamma u} T_{\hg^{\beta \alpha}} \D_\beta v}\D_\alpha u.
\]

Simulating an integration by parts with respect to $\D_\alpha$, we write this as 
\[
\D_\alpha\D_\gamma T_{T_{\D^\gamma u} T_{\hg^{\beta \alpha}} \D_\beta v} u - \D_\gamma T_{\D_\alpha T_{\D^\gamma u} T_{\hg^{\beta \alpha}} \D_\beta v} u.
\]
We will carry the first of these terms forward to 3c), while the latter term is perturbative. To see this, we observe that $\D_\alpha$ may commute through $T_{\D^\gamma u}$, similar to the analysis in 3a). Thus we arrive at the expression
\[
\D_\gamma T_{ T_{\D^\gamma u} \D_\alpha T_{\hg^{\beta \alpha}} \D_\beta v} u = \D_\gamma T_{ T_{\D^\gamma u} T_\ag v} u.
\]
We consider separately via $f_2$ and $\D_t f_1$ the contributions 
corresponding to $\D_\gamma = \D_i$ and $\D_\gamma = \D_0$ respectively. For the bound \eqref{strong-F-est}, using Lemma~\ref{l:pre-paradiff-eqn} and the Strichartz exponents $(p_1,q_1)$
given by
\begin{equation}\label{p1q1}
\frac{1}{p_1} = \frac{1}{n - \frac12}, \qquad \frac{1}{p_1} + \frac{1}{q_1} 
 = \frac12,
\end{equation}
we estimate, in both cases,
\begin{equation*}
\begin{aligned}
\|T_{ T_{\D^\gamma u} T_\ag v} u\|_{(S_{AIT}^{-\frac12})'} &
\lesssim \|\bD^{\frac32+\delta} T_{ T_{\D^\gamma u} T_\ag v} u\|_{L^{p'_1} L^{q'_1}} \\
&\lesssim \AA \|T_\ag v\|_{L^{p}L^q} \|u\|_{L^2 W^{\frac32+\delta, \infty}}. \end{aligned}
\end{equation*}
It remains to prove the bound \eqref{extra-F-est}, but this is again a simpler bound where we have a considerable gain. Indeed, using only 
$H^s$ Sobolev bounds but including \eqref{utt-easy} and \eqref{vtt-est}
we obtain at fixed time
\[
\| T_{ T_{\D^\gamma u} T_\ag v} u\|_{L^2} \lesssim \AA
\| T_\ag v\|_{H^{-\frac54}} \| u \|_{H^s} ,
\]
which suffices for all the Strichartz bounds.

\medskip

c) Returning to \eqref{Tsubstep1} and replacing the first term via the analysis in 3b), we have 
\begin{equation*}
\begin{aligned}
\D_\alpha(\D_\gamma T_{T_{\D^\gamma u} T_{\hg^{\beta \alpha}} \D_\beta v} u - T_{\hg^{\alpha \beta}} T_{T_{\D^\gamma u} \D_\beta \D_\gamma v} u 
- T_{\hg^{\alpha \beta}} T_{T_{\D^\gamma u}v} \D_\beta \D_\gamma  u).
\end{aligned}
\end{equation*}
We observe a cancellation between the first two terms. To see this, we apply the Leibniz rule for the $\D_\gamma$ derivative on the first term. Similar to 3a), cases where the derivative falls on the lowest frequency $\D^\gamma u$ or $\hg^{\beta \alpha}$ are perturbative. We also have a term which cancels the second term, leaving us with 
\begin{equation*}
\begin{aligned}
\D_\alpha ( T_{T_{\D^\gamma u} T_{\hg^{\beta \alpha}} \D_\beta v}\D_\gamma u
- T_{\hg^{\alpha \beta}} T_{T_{\D^\gamma u}v} \D_\beta \D_\gamma  u).
\end{aligned}
\end{equation*}
Applying also product and commutator Lemmas~\ref{l:para-com} and \ref{l:para-prod2} as in 2), we rewrite this as
\begin{equation*}
\begin{aligned}
\D_\alpha T_{\D^\gamma u} (T_{T_{\hg^{\beta \alpha}} \D_\beta v}\D_\gamma u
- T_{v}T_{\hg^{\alpha \beta}} {\D_\beta \D_\gamma} u).
\end{aligned}
\end{equation*}

\medskip

d) We apply the Leibniz rule with respect to $\D_\alpha$. Here we observe that cases where $\D_\alpha$ falls on lower frequency instances $u$ or $g$ are perturbative. Note that in contrast to the previous substeps, we no longer have the $\D_\alpha$ divergence and so must put all terms in $f_2$.

We consider for instance the term 
\[
T_{\D_\alpha\D^\gamma u} T_v T_{\hg^{\alpha \beta}} {\D_\beta \D_\gamma} u.
\]
Excluding the case of two time derivatives in $\D_\alpha\D^\gamma u$, this is easily estimated due to a favorable balance of derivatives. In the case of two time derivatives, we have $\D_\alpha\D^\gamma u \in \DCC$
so we can use the decomposition in Definition~\ref{d:CC}, say 
$\D_\alpha\D^\gamma u = h_1+h_2$. The first component can be thought of as 
a spatial derivative and is again easily estimated. It remains to consider 
the contribution of the second term $h_2 \in \BB^2 L^\infty$:
\begin{equation*}
\begin{aligned}
\| T_{h_2} T_v T_{\hg^{\alpha \beta}} {\D_\beta \D_\gamma} u \|_{H^{-1/2}} &\lesssim \AA \| h_2\|_{L^\infty} \| v\|_{H^\frac12} \|  \D^2  u \|_{H^{s-2}} \\
&\lesssim_\AA \BB^2 \|v\|_{\H^{1/2}}.
\end{aligned}
\end{equation*}
A similar analysis applies in the cases where $\D^\alpha$ falls on a low frequency metric coefficient $g$.

\medskip

e) We record the remaining terms after applying the Leibniz rule, and will observe instances of $T_{\ag}$ for which we use the equation \eqref{paralin}, as well as a cancellation. 
We arrive at
\begin{equation*}
\begin{aligned}
 T_{\D^\gamma u} (T_{\D_\alpha T_{\hg^{\beta \alpha}} \D_\beta v}\D_\gamma u + T_{T_{\hg^{\beta \alpha}}\D_\beta v} {\D_\alpha\D_\gamma} u
- T_{ \D_\alpha v}T_{\hg^{\alpha \beta}}  {\D_\beta \D_\gamma} u - T_{v}\D_\alpha T_{\hg^{\alpha \beta}}  {\D_\beta \D_\gamma} u)
\end{aligned}
\end{equation*}
which, reindexing the second term, may be written in the form
\begin{equation}\label{step-f}
\begin{aligned}
 T_{\D^\gamma u} (T_{T_\ag v} \D_\gamma u + T_{T_{\hg^{\alpha \beta}} \D_\alpha v} {\D_\beta\D_\gamma} u
- T_{ \D_\alpha v}T_{\hg^{\alpha \beta}}  {\D_\beta \D_\gamma} u - T_{v} \D_\alpha T_{\hg^{\alpha \beta}}  {\D_\beta \D_\gamma} u).
\end{aligned}
\end{equation}

We apply \eqref{strong-F-est3} to the first term, and estimate
in a dual Strichartz norm with $(p_1,q_1)$ as in \eqref{p1q1},
\begin{equation*}
\begin{aligned}
\|  T_{T_\ag v} \D_\gamma u \|_{(S_{AIT}^{\frac12})'} &
\lesssim \|\bD^{\frac12+\delta} T_{T_\ag v} \D_\gamma u\|_{L^{p'_1} L^{q'_1}} 
\lesssim \AA \|T_\ag v\|_{L^{p}L^q} \|\partial u\|_{L^2 W^{\frac12+\delta, \infty}} .
\end{aligned}
\end{equation*}
For the last term, on the other hand, we use a Strichartz bound for $v$
and match it with the bound \eqref{hp-du} in Lemma~\ref{l:all-u},
\[
\| \bD^{\delta} T_{v} \D_\alpha T_{\hg^{\alpha \beta}}  {\D_\beta \D_\gamma} u\|_{L^{p'_3}L^{q'_3}} 
\lesssim 
\| \bD^{-\delta}  v  \|_{L^{p'_3}L^{q'_3}}
\| \bD^{\delta_0} \D_\alpha T_{\hg^{\alpha \beta}}  {\D_\beta \D_\gamma} u\|_{L^{p_4} L^{q_4}} \ll \| v\|_{S^{\frac12}_{AIT}},
\]
where
\[
\frac{1}{p_3} = \frac{1}{n-\frac12}, \qquad
\frac{1}{p_4} = 1 - \frac{2}{p_3},
\qquad \frac{1}{p_3} + \frac{1}{q_3}
= \frac12, \qquad \frac{1}{p_4}+ \frac{1}{q_4}=1.
\]
Here the Strichartz exponents $p_3$ and $q_3$ are chosen so that the first factor
on the right is controlled by $\| v\|_{S^{\frac12}_{AIT}}$ and $\delta$
is arbitrarily small. On the other hand $\delta_0$ is a fixed positive 
parameter which depends only on the distance between $s$ and its lower bound.

Lastly, we perturbatively estimate the remaining pair of terms in \eqref{step-f},
\[
 \text{Diff} = T_{T_{\hg^{\alpha \beta}} \D_\alpha v}\D_\beta\D_\gamma u
- T_{ \D_\alpha v}T_{\hg^{\alpha \beta}} \D_\beta \D_\gamma u .
\]
This is a minor variation of the para-associativity Lemma~\ref{l:para-assoc}. The difference is nonzero only if the 
frequencies of the three factors $\partial_\alpha v$, $\hg^{\alpha \beta}$
and $\D_\beta \D_\gamma u$ are ordered in a nondecreasing manner. 
Thus in view of our assumption \eqref{u-have-now} we obtain
\[
\|\text{Diff}\|_{H^{-\frac12}} \lesssim \sum_{i \leq j < k}
2^{\frac{i-j}2} \|\partial v_i\|_{H^{-\frac12}} \| \bD^\frac12 \hg\|_{L^\infty} \| \bD^\frac12 \partial u\|_{L^\infty} \ll \|\partial v\|_{H^{-\frac12}} ,
\]
as needed.

%%%%%%%%%%%%%%%%%%%%%

\

4) It remains to consider the cancellation between the balanced $\Pi$ terms in the normal form correction and in $N_1'(u)$. Here the analysis is identical to the analysis for the low-high $T$ contributions in Step 3, due to the analogous structure for the $T$ and $\Pi$ terms in both $v_2$ and $N_1'(u)$. The main care that is needed is to observe that all negative Sobolev exponent norms have been addressed in Step 3 by either using a divergence structure, or by Sobolev embeddings, which apply equally well to the balanced $\Pi$ case.

\end{proof}

\subsection{Reduction to the paradifferential equation} 
Here we first use the well-posedness result for the linear paradifferential
equation in order to obtain a good bound for $\tv$. The source terms are perturbative by \eqref{strong-F-est} and Theorem~\ref{t:Str-move-around},
so the solution $\tv$ must satisfy the bound
\begin{equation}\label{tva}
\| \tv \|_{S_{AIT}^\frac12} + \| \partial_t \tv \|_{S_{AIT}^{-\frac12}}
\lesssim \| \tv[0]\|_{\H^\frac12} + c \| v \|_{S_{AIT}^\frac12} + \| \partial_t v \|_{S_{AIT}^{-\frac12}}, \qquad c \ll 1.
\end{equation}
It remains to show that the Strichartz estimates carry over to $v$.
For this, it suffices to show that 
\begin{equation}\label{tvb}
    \| v_2 \|_{S_{AIT}^\frac12} + \| \partial_t v_2 \|_{S_{AIT}^{-\frac12}}
    \ll \| v \|_{L^\infty \H^\frac12}.
%    {S_{AIT}^\frac12} + \| \partial_t v \|_{S_{AIT}^{-\frac12}}
\end{equation}
If this is true, then combining the last two bounds with the norm equivalence \eqref{bdd-nf} we obtain the desired bound for the linearized evolution 
\eqref{lin-inhom}, namely 
\begin{equation}\label{v-Str}
    \| v \|_{S_{AIT}^\frac12} + \| \partial_t v \|_{S_{AIT}^{-\frac12}}
    \ll \| v[0] \|_{\H^\frac12}
\end{equation}    
with a universal implicit constant. This concludes the proof of Theorem~\ref{t:linearized} in dimension $n \geq 3$. The case $n=2$
is virtually identical.

It remains to prove \eqref{tvb}. The energy norm for $v_2$ has already been estimated in part (i) of Proposition~\ref{p:pre-paradiff-eqn}, so it remains to consider the $L^2 L^\infty$
norm in three and higher dimensions. This is a soft bound, where we only need to use 
the energy bound for $v$ on the right, and not the full Strichartz norm, as we would also have been allowed. There are eight norms to estimate; most of them are similar, so we consider a representative sample, leaving the rest for the reader.

For a streamlined unbalanced bound we consider the term 
\[
\| \bD^{-\frac{n-2}2-\frac14 - \delta} T_{T_{\partial^\gamma u} v} \partial_\gamma u \|_{L^2 L^\infty} \lesssim \| T_{\partial^\gamma u} v\|_{L^\infty L^{\frac{2n}{n-1}}}
\| \bD^{\frac14 -\delta} \partial u \|_{L^\infty} \lesssim \AA \| v\|_{L^\infty \H^\frac12},
\]
where we have used Bernstein's inequality twice and the Strichartz bound for $u$.
This pattern is followed for all unbalanced terms.

For the worst balanced case, we apply the time derivative to $v$ in the next to last 
term in $v_2$. Then we have to estimate
\[
\begin{aligned}
\| \bD^{-\frac{n}2-\frac14 - \delta} \Pi(T_{\partial^\gamma u} \partial^2 v, u) \|_{L^2 L^\infty}&  \lesssim \| \Pi(T_{\partial^\gamma u} \partial^2 v, u) \|_{L^2 L^{\frac{4n}{2n-1}} }\\
&\lesssim \AA \| \partial^2 v\|_{L^\infty H^{-\frac32}} 
\| \bD^{\frac32+\frac14} u\|_{L^2 L^\infty} \lesssim \AA \| v\|_{L^\infty \H^{\frac12}},
\end{aligned}
\]
where we have used Bernstein's inequality twice, Lemma~\ref{l:vtt} and the Strichartz bound for $u$.

%%%%%%%%%%%%%%%%%%%%%%%%%%%%%%%%%%%%%%%%%%%%%%%%%

\section{Short time Strichartz estimates}
\label{s:ST}

The aim of this section is to provide a more detailed 
overview of the local well-posedness result in \cite{ST},
and at the same time to provide a formulation of this result 
which applies in a large data setting, but for a short time.
Instead of working with the equation \eqref{NLW-gen}, here 
it is easier to work with the problem
\begin{equation}\label{ST-eqn}
g^{\alpha\beta}(\bu) \partial_\alpha \partial_\beta \bu  = h^{\alpha\beta}(\bu)
\partial_\alpha \bu \, \partial_\beta \bu
\end{equation}
for a possibly vector valued function $\bu$. This is exactly the set-up
of \cite{ST}, and has the advantage that it is scale invariant. 
We recall that the scaling exponent for this problem is $s_c = \frac{n}2$.
In our problem, we will apply the results in this section to the 
function $\bu = \partial u$.

We begin with a review of the local well-posedness result in \cite{ST}, but where we describe also the structure of the Strichartz estimates:

\begin{theorem}[Smith-Tataru \cite{ST}]\label{t:ST}
Consider the problem \eqref{ST-eqn} with initial data satisfying 
\begin{equation}\label{reg-data}
\| \bu[0]\|_{\H^{s_1}} \ll 1.   
\end{equation}
where \begin{equation}\label{reg-s2}
s_1> s_c+\frac34,   \qquad n = 2,  
\end{equation}
respectively
\begin{equation}\label{reg-s3}
s_1> s_c+\frac12,   \qquad n \geq 3.
\end{equation}
Then the solution exists on the time interval $[0,1]$, and satisfies the following 
Strichartz estimates
\begin{equation}\label{se2}
\| \bD^{\delta_0}\partial \bu\|_{L^4 L^\infty} \lesssim 1 , \qquad n=2 ,  
\end{equation}
respectively
\begin{equation}\label{se3}
\| \bD^{\delta_0}\partial \bu\|_{L^2 L^\infty} \lesssim 1 , \qquad n \geq 3.
\end{equation}
with a small $\delta_0 > 0$.
\end{theorem}

In addition, another conclusion of the work in \cite{ST}, which is used as an intermediate step 
in the proof of the theorem above, is that the linearized problem around the solutions in Theorem~\ref{t:ST} is well-posed in a range of Sobolev spaces,
and almost lossless Strichartz estimates hold for them.
Precisely, we have the following:

\begin{theorem}[\cite{ST}]\label{t:ST-lin}
Let $\bu$ be a solution for \eqref{ST-eqn} in the time interval $[0,1]$ as in Theorem~\ref{t:ST}. Then the linear equation 
\begin{equation}\label{box-g-gen}
\left\{
 \begin{aligned}
 g^{\alpha \beta}(\bu) \partial_\alpha \partial_\beta v = 0 \\
 v[0] = (v_0,v_1) 
 \end{aligned}
 \right.
 \end{equation}
is well-posed in $\H^r$ the same time interval  for $1 \leq r \leq s_1+1$, and the solutions satisfy the uniform and Strichartz estimates \eqref{Str-hom} for the same range of $r$. 
\end{theorem}
 We note that in \cite{ST} it is also assumed that $g^{00}= -1$, akin to our metric $\tg$; but it is clear that such an assumption is not needed in the above theorems, as one can simply divide the equation by $g^{00}$.

We also remark that the equation \eqref{box-g-gen} is not the same as the linearized equation. The reason \eqref{box-g-gen} is preferred in \cite{ST} is 
the extended upper bound for $r$. It is also noted in \cite{ST} that for a range of $r$ with a lower upper bound, the conclusion of the last theorem is also valid 
for the full linearized equation; this is a straightforward perturbative argument. From below, the Sobolev exponent $r = 1$ suffices in dimension $n\geq 3$ in \cite{ST}, though 
it is also clear that this is not optimal.
Indeed, for dimension $n = 2$ the above result is extended in \cite{ST} to the range $\frac34 \leq r \leq s_1+1$, and the linearized equation is  shown to be well-posed in $\H^\frac34$; see \cite[Lemma~A4]{ST}; the same method 
also works in higher dimension.

We also remark that if the linearized equation is in divergence form, (which can be arranged in the present paper, see \eqref{hg-lin}), then, by duality,
(forward/backward) well-posedness in $\H^r$ implies (backward/forward) well-posedness in $\H^{1-r}$, with the center point at $r = \frac12$.
This motivates why, in the context of the present paper,  it is easiest to study the linearized equation exactly in $\H^\frac12$. Unfortunately our argument runs 
into a technical obstruction in dimension $n=2$, which is why we make a slight adjustment there and work instead in $\H^\frac58$.

To summarize, in the present paper we will not need directly the conclusion of Theorem~\ref{t:ST-lin}, but rather a minor variation of it where we also consider
the divergence form equation and its associated paradifferential flow,
and we lower the range for $r$ in order to include the space $\H^\frac12$ ($\H^{\frac58}$ in dimension two).

In the proof of the main result of this paper, we will need to use this result for solutions which are not small in $\H^s$, so we cannot apply it directly.
Instead, we will seek to rephrase it and use it in a large data setting via 
a scaling argument.

The difficulty we face is that rescaling  keeps homogeneous Sobolev norms unchanged, rather than the inhomogeneous ones.  A first step in this direction is to consider smooth solutions, but which may be large at low frequency:

\begin{theorem}\label{t:smooth}
Consider the problem \eqref{ST-eqn} with initial data satisfying 
\begin{equation}\label{reg-data-hom}
\| \bu[0]\|_{\dot \H^N \cap \dot \H^{s_c}} + \| \bu(0)\|_{L^\infty}  \ll 1.   
\end{equation}
Then the solution exists up to time $1$, and satisfies the uniform bound
\begin{equation}\label{reg-point}
\| \bu\|_{L^\infty}  \ll 1.   
\end{equation}
and the Sobolev bound
\begin{equation}\label{reg-sobolev}
\| \bu\|_{L^\infty(0,1;\dot \H^N \cap \dot \H^{s_c})} \lesssim   \| \bu[0]\|_{\dot \H^N + \dot \H^{s_c}}.
\end{equation}

In addition,
\begin{equation}
\| \bu[\cdot]] \|_{L^\infty(0,1; \H)} \lesssim \|\bu[0]\|_{\H}    
\end{equation}
whenever the right hand side is finite.
\end{theorem}

\begin{proof}
Locally, after subtracting a constant, the data is small in $\H^N$ so the existence of 
regular solutions is classical. It remains to establish energy estimates in homogeneous 
Sobolev norms. The problem reduces to the case of the paradifferential flow, and, by conjugation with a power of $\bD$, to bounds in $\H$ which are straightforward. 
\end{proof}

A second step is the following variation of Theorem~\ref{t:ST},
where we consider a small $\H^{s_1}$ perturbation of a small and smooth data:

\begin{theorem}\label{t:ST+}
Consider the problem \eqref{ST-eqn} with initial data $\bu[0]$ of the form
\begin{equation}
\bu[0] = \bu^{lo}[0] + \bu^{hi}[0]    ,
\end{equation}
where  the two components satisfy
\begin{equation}\label{reg-data+}
\| \bu^{lo}[0]\|_{\H^{N}} \ll 1, \qquad    \| \bu^{hi}[0]\|_{\H^{s_1}} \leq \epsilon \ll 1.
\end{equation}
Then the solution $u$ exists on the time interval $[0,1]$, and satisfies the following Strichartz estimates
\begin{equation}\label{se2+}
\| \bD^{\delta_0}\partial (\bu-\bu^{lo})\|_{L^4 L^\infty} \lesssim \epsilon , \qquad n=2   
\end{equation}
respectively
\begin{equation}\label{se3+}
\| \bD^{\delta_0}\partial (\bu-\bu^{lo})\|_{L^2 L^\infty} \lesssim \epsilon , \qquad n \geq 3.
\end{equation}
with a small $\delta_0 > 0$.
\end{theorem}
We remark that the solutions in this second theorem are still covered by Theorem~\ref{t:ST}.
The only difference is that the constant in the Strichartz bound depends only on 
the $u^{hi}[0]$ bound.

\begin{proof}
This follows by a direct application of the results in Theorem~\ref{t:ST} and 
Theorem~\ref{t:ST-lin}. We write an equation for $u^{hi}=u-u^{lo}$,
\[
g^{\alpha \beta}(u) \partial_\alpha \partial_\beta u^{hi} = 
- (g^{\alpha\beta}(u) - g^{\alpha\beta}(u^{lo})) \partial_\alpha \partial_\beta u_{lo} := f^{hi},
\]
where the source term $f^{hi}$ can be estimated at fixed time by 
\[
\| f^{hi}\|_{H^{s_1-1}} \lesssim \|u^{hi} \|_{H^{s_1}},
\]
and thus it is perturbative. Then we apply the Strichartz estimates in Theorem~\ref{t:ST-lin} to $u^{hi}$, and the desired conclusion follows.

\end{proof}

Now we consider the large data problem, where we show local well-posedness by a scaling argument.  The price to pay will be that the time interval for which we have the solutions will be  shorter. Precisely, we will show that

\begin{theorem}% [Smith-Tataru \cite{ST}] 
\label{t:ST-large}
For any $s_1$  as in \eqref{reg-s2}, \eqref{reg-s3} 
there exists $\delta_0 > 0$ so that the following holds: 
For any $M > 0$ and any solution  $\bu$ to the problem \eqref{NLW-gen} with initial data satisfying 
\begin{equation}\label{Mdata}
\| \bu[0]\|_{\dot \H^{s_1}} \ll M, \qquad \|\bu[0]\|_{\dot \H^{s_c}} \ll 1.   
\end{equation}
we have:

a) The solution exists up to time $T_M$ given by
\begin{equation}\label{TM}
T_M^\sigma = M^{-1}, \qquad  \sigma = s_1-s_c,
\end{equation}
with uniform bounds 
\begin{equation}\label{Msoln}
\| \bu[\cdot]\|_{C(0,T_M;\H^{s_1})} \lesssim M, \qquad 
\|\bu[\cdot]\|_{C(0,T_M;\dot \H^{s_c})} \lesssim 1,
\end{equation}
as well as 
\begin{equation}\label{Mdata-low}
\| \bu[\cdot]\|_{C(0,T_M;\H)} \lesssim 
\|\bu[0]\|_{\H},
\end{equation}
whenever the right hand side is finite.

b) The solution $\bu$ satisfies the following 
Strichartz estimates in $[0,T_M]$:
\begin{equation}\label{Str-M2}
\|\langle T_M D'\rangle^{\delta_0}\partial \bu\|_{L^4 L^\infty} \lesssim T_M^{-\frac34}, \qquad n=2  , 
\end{equation}
respectively
\begin{equation}\label{str-M3}
\| \langle T_M D'\rangle^{\delta_0}\partial \bu\|_{L^2 L^\infty} \lesssim T_M^{-\frac12}, \qquad n \geq 3.
\end{equation}

c) Furthermore, the homogeneous Strichartz estimates \eqref{Str-hom} also hold in $\H^r$ for 
the associated linear equations \eqref{box-g-gen}, on the same time intervals for $r \in [1,s_1]$.

\end{theorem}

\begin{proof}
As stated, the result is invariant with respect to scaling. Precisely $M$
plays the role of a scaling parameter, and by scaling we can set it to $1$.

It remains to prove the result for $M=1$ in which case $T_M=1$. In a nutshell, the idea of the proof is to use the finite proof of propagation to  localize the problem and, by scaling, to reduce it to the case when Theorems~\ref{t:smooth},\ref{t:ST+} can be applied.  To fix the notations, we will consider the case $n \geq 3$ in what follows; the two dimensional case is identical after obvious changes in notations.

 On the Fourier side we split the initial data into two components, 
\[
\bu[0] = \bu^{lo}[0] + \bu^{hi}[0], \qquad \bu^{lo}[0] = P_{<0} \bu[0],
\]
and we denote by $\bu$ and $\bu^{lo}$ the corresponding solutions.

On the other hand on the physical side we partition the initial time slice $t=0$ into cubes $Q$ of size $1$, and consider a partition of unity associated to the covering by $8Q$,
\[
1 = \sum \chi_Q,
\]
and define the localized initial data
\[
\bu_{Q}[0] = (\chi_Q(\bu_0 - \bar \bu^{lo}_{0,Q}), \chi_Q u_1), \qquad \bar \bu^{lo}_{0,Q} = \fint_Q \bu^{lo}_0\, dx,
\]
which agrees with $\bu[0]$ in $6Q$ up to a constant. The speed of propagation for solutions 
$\bu$ with $|\bu| \ll 1$ is close to $1$, therefore the corresponding solutions $w_Q$
agree with $\bu$ in $4Q$ (again, up to a constant) in $[0,1]$, assuming both exist up to this time.

Next we consider the existence and properties of the solutions $\bu_Q$ in the time interval
$[0,T_M]$. 
For $\bu_{Q}[0]$ we have a low-high decomposition, 
\[
\bu_{Q}[0] =
(\chi_{T_M^{-1} Q}(\bu^{lo}_0 - \bar \bu^{lo}_{0,Q}), \chi_{T_M^{-1}} Q \bu^{lo}_1)
+ \chi_{T_M^{-1} Q}(\bu^{hi}_0, Q \bu^{hi}_1)
:=\bu^{lo}_{Q}[0] + \bu^{hi}_{Q}[0] 
\]

Now we consider energy bounds for the initial data.
For $\bu^{lo}$ we have
\begin{equation}\label{tu-lo}
\| \bu^{lo} [0]\|_{\dot H^{s_c} \cap \dot \H^{N}}  \ll 1.
\end{equation}
Since $s_1-1 < n/2$, after localization this also implies that the low frequency components
satisfy
\begin{equation}
\| \bu^{lo}_Q [0]\|_{\H^{N}}  \ll 1.
\end{equation}
which is exactly as in Theorem~\ref{t:smooth}, respectively Theorem~\ref{t:ST+}.

On the other hand, for the high frequency bounds we have the almost orthogonality
relation
\begin{equation}
\sum_Q \|\bu_{Q}[0]\|_{\H^{s_1}}^2 \ll 1.  
\end{equation}
By Theorem~\ref{t:ST}, it follows that the solutions $\bu_Q$ exist up to time $1$,
and satisfy the Strichartz bounds
\begin{equation}
\| \bD^{\delta_0}\partial \bu_Q\|_{L^2 L^\infty} \lesssim 1.    
\end{equation}
Theorem~\ref{t:ST+} allows us to improve this to
\begin{equation}
\| \bu_Q - \bu_Q^{lo} \|_{L^\infty \H^{s_1}}+
\| \bD^{\delta_0}\partial (\bu_Q-\bu_Q^{lo})\|_{L^2 L^\infty} \lesssim  \|\bu^{hi}_{Q}[0]\|_{\H^{s_1}}.
\end{equation}

The solutions $\bu_Q$, respectively $\bu^{lo}_Q$ agree with $\bu$, $\bu^{lo}$ in $[0,1]\times 4\tilde Q$. Then we can recombine the $\bu_Q$ bounds using a partition of unity on the unit spatial scale. We obtain a $\bu$ bound,
namely
\begin{equation}
    \| \bu - \bu^{lo} \|_{L^\infty \H^{s_1}}+
\| \bD^{\delta_0}\partial (\bu-\bu^{lo})\|_{L^2 L^\infty} \lesssim \| \bu^{hi}_{Q}[0]\|_{\H^{s_1}}
\ll 1.
\end{equation}
On the other hand for $\bu_{lo}$ we have the bounds given by Theorem~\ref{t:smooth}. 

The energy bounds for $\bu-\bu^{lo}$ and $\bu^{lo}$ combined yield the desired energy bound \eqref{Msoln} in the theorem. In terms of the Strichartz bounds
\eqref{Str-M2}, we already have them for $\bu - \bu^{lo}$ so it remains to 
prove them for $\bu^{lo}$. But there we trivially use Sobolev embeddings
and Holder's inequality in time.

It remains to consider the Strichartz estimates for $\H^1$ solutions to the linearized equation. 
By the same finite speed of propagation argument as above, it suffices to prove them for the linearization around the localized solutions $\bu_Q$. But this follows by  Theorem~\ref{t:ST-lin}.
\end{proof}

To conclude this section we reinterpret the above result in the context of the minimal surface equation, exactly in the for it will be used in the last section.
We keep the same notations, with the only change that now $s_c = \frac{n}2+1$:
 
\begin{theorem}% [Smith-Tataru \cite{ST}] 
\label{t:ST-large-ms}
For any $s_1$  as in \eqref{reg-s2}, \eqref{reg-s3} 
there exists $\delta_0 > 0$ so that the following holds: 
For any $M > 0$ and any solution  $u$ to the problem \eqref{msf} with initial data satisfying 
\begin{equation}\label{Mdata-ms}
\| u[0]\|_{\dot \H^{s_1}} \ll M, \qquad \|u[0]\|_{\dot \H^{s_c}} \ll 1.   
\end{equation}
We have:

a) The solution exists up to time $T_M$ given by
\begin{equation}\label{TM-ms}
T_M^\sigma = M^{-1}, \qquad  \sigma = s_1-s_c,
\end{equation}
with uniform bounds 
\begin{equation}\label{Msoln-ms}
\| u[\cdot]\|_{C(0,T_M;\H^{s_1})} \lesssim M, \qquad 
\|u[\cdot]\|_{C(0,T_M;\dot \H^{s_c})} \lesssim 1.
\end{equation}
as well as 
\begin{equation}\label{Mdata-low-ms}
\| u[\cdot]\|_{C(0,T_M;\H)} \lesssim 
\|u[0]\|_{\H},
\end{equation}
whenever the right hand side is finite.

b) The solution $\bu$ satisfies the following 
Strichartz estimates in $[0,T_M]$:
\begin{equation}\label{Str-M2-ms}
\|\langle T_M D'\rangle^{\delta_0}\partial^2 u\|_{L^4 L^\infty} \lesssim T_M^{-\frac34}, \qquad n=2  , 
\end{equation}
respectively
\begin{equation}\label{str-M3-ms}
\| \langle T_M D'\rangle^{\delta_0}\partial^2u\|_{L^2 L^\infty} \lesssim T_M^{-\frac12}, \qquad n \geq 3.
\end{equation}

c) Furthermore, the homogeneous Strichartz estimates \eqref{Str-hom} also hold in $\H^r$ for 
the associated linear equations \eqref{box-g-gen}, on the same time intervals for $r \in [1,s_1]$.  Also the the full Strichartz estimates 
\eqref{Str-full} with $S= S_{ST}$ hold for the linear paradifferential equation hold in $\H^r$ on the same time intervals for all real $r$.

\end{theorem}

The theorem is obtained by applying the previous theorem to $\bu = \partial u$.
For the Strichartz estimates for the linear paradifferential equation we observe in addition that we have the bound
\[
\| \partial^2 g\|_{L^1(0,T_M; L^\infty)} \lesssim 1.
\]
Then the $r=1$ case of the Strichartz estimates for the linear equations \eqref{box-g-gen} together with Proposition~\ref{p:Str-allr}
imply the desired conclusion.

\section{Conclusion: proof of the main result}\label{s:final}

After using the finite speed of propagation to reduce to the small data problem, 
here we combine our balanced energy estimates with  the short time Strichartz bounds in order to complete the proof of our main result in Theorem~\ref{t:main}. 
Our rough solutions are constructed as limits of smooth solutions obtained 
by regularizing the initial data, so the emphasis is on obtaining favourable estimates for these smooth solutions.

\subsection{Reduction to small data.} 
 By Sobolev embeddings, the initial data satisfies 
\[
\| u_0\|_{C^{1,\sigma}} + \| u_1\|_{C^{\sigma}} \lesssim 1, \qquad \sigma = s -\frac{n}2-1.
\]
 Then given $x_0 \in \R^n$, within a small ball $B(x_0,4r)$ 
we have 
\[
|u_0(x) - (u_0(x_0)+(x-x_0)\partial u(x_0))| + | u_1 - u_1(x_0)| \lesssim r^{\sigma}.
\]
This allows us to truncate the above differences near $x_0$ to obtain the localized data
\[
\begin{aligned}
u_{0}^{r,x_0}(x) = & \  (u_0(x_0)+(x-x_0)\partial u(x_0)) + \chi(r^{-1}(x-x_0))
u_0(x) - (u_0(x_0)+(x-x_0)\partial u(x_0)),
\\
u_{1}^{r,x_0}(x) = & \ u_1(x_0) + \chi(r^{-1}(x-x_0)) (u_1 - u_1(x_0)) ,
\end{aligned}
\]
where $\chi \in \mathcal D(\R^n)$ is equal to $1$ in $B(0,2)$ and $0$ outside $B(0,4)$.

Let $\epsilon > 0$. Then for small enough $r$, depending on $\epsilon$, these initial data are close to the initial  data for the linear solution to the minimal surface equation given by
\[
\tu^{x_0,r}(t,x) = (u_0(x_0)+(x-x_0)\partial u(x_0)) + t u_1(x_0), 
\]
in the sense that
\begin{equation}\label{small-data}
\| u^{x_0,r}[0] - \tu^{x_0,r}[0] \|_{\H^s} \leq \epsilon \ll 1.
\end{equation}
This will be our smallness condition for the initial data, with 
$\partial \bu^{x_0,r}$ in a compact subset of the set described in \eqref{time-like-data}.

To reduce the problem to the case when the initial data satisfies instead the simpler smallness condition 
\begin{equation}\label{small-data+}
\| u^{x_0,r}[0]  \|_{\H^s} \leq \epsilon \ll 1
\end{equation}
it suffices to apply a linear transformation in the Minkowski space $\R^{n+2}$
which preserves the time slices but maps our linear solution $\tu^{x_0,r}$ to the zero solution. The price we pay for this is that the background Minkowski metric 
is then changed to another Lorentzian metric.  But the new metric belongs to a compact set in the space of flat Lorenzian metrics for which the time slices 
are uniformly space-like and the graph of the zero function is uniformly time-like.
Hence our small data result applies uniformly to these localized solutions, see Remark~\ref{r:all-Lorentz}.
Then, due to the finite speed of propagation, we also obtain solutions up to time $O(r)$ for the original problem.

\subsection{Uniform bounds for regularized solutions}
Let $s$ be as in Theorem~\ref{t:main}. Given an initial data $u[0] \in \H^s$ which is small,
\begin{equation}\label{small-data-last}
  \| u[0] \|_{\H^s} \leq \epsilon \ll 1,    
\end{equation}
we consider a continuous family of frequency localizations 
\[
u^h[0] = P_{<h}u[0]
\]
to frequencies $\leq 2^h$. For fixed $h$ and a short time which may depend on $h$, these solutions exist by Theorem~\ref{t:ST-large}. Further, they are smooth and also depend smoothly on $h$. Finally, we consider the functions
\[
v^h = \frac{d}{dh} u^h.
\]
These functions solve the linearized equation
around $u^h$, with initial data 
\begin{equation}
v^h[0]= P_h u[0],   
\end{equation}    
which is localized at frequency $2^h$.
The functions $v^h$ will be measured in $\H^\frac12$ in dimension $n\geq 3$ and in $\H^\frac58$ in dimension $n=2$.
% For exposition purposes we will fix the dimension to $n \geq 3$ in what follows, though the same reasoning  will also apply in $n=2$ with the above adjustment. 
Thus the initial data for $v^h$ satisfies the bound
\begin{equation}\label{vh-data}
\begin{aligned}
\|v^h[0]\|_{\H^\frac12} \lesssim 2^{-(s-\frac12)h} \epsilon    \qquad n \geq 3,
\\
\|v^h[0]\|_{\H^\frac58} \lesssim 2^{-(s-\frac58)h} \epsilon    \qquad n = 2.
\end{aligned}
\end{equation}

Our first objective will be to show that these  solutions exist on a time interval which does not depend on $h$, and satisfy uniform bounds:

\begin{theorem} \label{t:uh}
The above solutions $u^h$ have the following properties:
\begin{description}
\item[a) Uniform lifespan and uniform bounds]
The solutions $u^h$ exist up to time $1$, with uniform bounds
\begin{equation}\label{uh-s}
\| u^{h}[\cdot] \|_{C(0,1;\H^s)} \lesssim \epsilon,   
\end{equation}
and higher regularity bounds
\begin{equation}\label{uh-s+1}
\| u^{h}[\cdot] \|_{C(0,1;\H^{s+1})} \lesssim 2^h \epsilon.    
\end{equation}

\item[b) Bounds for the linearized flow] The linearized equation around $u^h$ is well-posed in $\H^\frac12$,
with uniform estimates in $[0,1]$, uniformly in $h$,
\begin{equation}
\begin{aligned}
\| v\|_{L^\infty(0,1;\H^\frac12)} \lesssim \|v[0]\|_{\H^\frac12}, 
\qquad n \geq 3,
\\
\| v\|_{L^\infty(0,1;\H^\frac58)} \lesssim \|v[0]\|_{\H^\frac58}, 
\qquad n =2,
\end{aligned}
\end{equation}
and uniform Strichartz estimates with loss of derivatives,
\begin{equation}
\begin{aligned}
\| \bD^{-\frac{n}2-\frac14-\delta} \partial v\|_{L^2 L^\infty} \lesssim  \|v[0]\|_{\H^\frac12}, \qquad n \geq 3,
\\
\| \bD^{-\frac{n}2-\frac14-\delta} \partial v\|_{L^4 L^\infty} \lesssim  \|v[0]\|_{\H^\frac58}, \qquad n = 2,
\end{aligned}
\end{equation}
for any $\delta > 0$.
\end{description}
\end{theorem}

The exponent $s+1$ in \eqref{uh-s+1} is chosen so that it falls into the range of existing theory, where we already have well-posedness and continuous dependence. We remark that,
as a corollary of part (b), we also obtain uniform bounds for the functions $v_h$,
namely 
\begin{equation}\label{vh-uniform}
\begin{aligned}
\| v^h\|_{L^\infty(0,1;\H^\frac12)} \lesssim \epsilon  2^{(s-\frac12)h}, 
\qquad n \geq 3,  
\\
\| v^h\|_{L^\infty(0,1;\H^\frac58)} \lesssim \epsilon  2^{(s-\frac58)h},
\qquad n = 2.
\end{aligned}
\end{equation}

\subsection{ The bootstrap assumptions}
Our proof of the main result in Theorem~\ref{t:uh} will be formulated as a bootstrap argument. Then the question is what is a good bootstrap assumption. Having the bounds
for the linearized equation as part of the bootstrap assumption would be technically
complicated. On the other hand, not having any assumptions at all related to the linearized equation would introduce too many difficulties in getting the argument started.
As it turns out, there is a good middle ground, which is to have the uniform energy bounds on both $u^h$ and $v^h$ as part of the bootstrap assumptions, which are then set as follows:
\begin{enumerate}[label=\roman*)]
\item Uniform $\H^s$ bounds:
\begin{equation}\label{boot-uh}
\| u^{h}[\cdot] \|_{C(0,1;\H^s)} \leq 1,    
\end{equation}
\item Higher regularity bounds:
\begin{equation}\label{boot-uh-hi}
\| u^{h}[\cdot] \|_{C(0,1;\H^{s+1})} \leq 2^h,     
\end{equation}
\item Difference bounds,
\begin{equation}\label{boot-vh}
\begin{aligned}
\| v_h\|_{L^\infty(0,1;\H^\frac12)} \leq  2^{-(s-\frac12)h}, \qquad n=3,
\\
\| v_h\|_{L^\infty(0,1;\H^\frac58)} \leq  2^{-(s-\frac58)h}, \qquad n=2.
\end{aligned}
\end{equation}
\end{enumerate}

The $v_h$ bootstrap bound will be useful in particular in order 
to obtain good low frequency bounds for differences of the $u^h$
functions,
\begin{equation}
\| u^h - u^k\|_{L^\infty(0,1;\H^\frac12)} \lesssim   2^{-(s-\frac12)h}, \qquad h \leq k,   
\end{equation}
with the obvious change in two dimensions.

To avoid having a bootstrap assumption on a noncompact set of functions, we may freely
restrict the range of $h$.
Precisely, given an arbitrary threshold $h_0$, we assume the bootstrap assumption to hold for all $h \leq h_0$ and show that the desired bounds hold in the same range. Since $h_0$
plays no role in the analysis, we will simply drop it in the proofs.

\subsection{ Short time Strichartz estimates for $u^h$ and $v^h$}
Our goal here is to use the results in Theorem~\ref{t:ST-large-ms} together with our 
bootstrap assumption in order to obtain short time Strichartz estimates for both 
$u^h$ and $v^h$.  

By the bootstrap assumptions \eqref{boot-uh} and \eqref{boot-uh-hi}, we may bound 
the local well-posedness norm $\H^{s_1}$ of the solution $u^h$ by 
\begin{equation}
\| u^h[\cdot]\|_{L^\infty \H^{s_1}} \lesssim M_h:= 2^{h(s_1-s)}.
\end{equation}
Then the result of Theorem~\ref{t:ST-large} is valid on time intervals $I_h$ of length
\[
|I_h| = T_h := M_h^{-\frac{1}{\sigma}} = 2^{-\frac{s_1- s}{s_1-s_c}h}.
\]
 In practice, $s_1$ will be chosen as close as possible to the threshold in \eqref{reg-s2},\eqref{reg-s3}. This will insure that in all dimensions we have
\[
\frac{s_1-s}{s_1-s_c} < \frac12.
\]

In particular, by Theorem~\ref{t:ST-large} it follows that the solution $u^h$ satisfies full Strichartz estimates 
on such intervals, 
\begin{equation} \label{uh-short3}
\| \bD^{1+\delta_0} \partial u^h\|_{L^2(I_h; L^\infty)} \lesssim T_h^{-\frac12}, \qquad n \geq 3,  \end{equation}
respectively 
\begin{equation}\label{uh-short2}
\| \bD^{1+\delta_0} \partial u^h\|_{L^4(I_h; L^\infty)} \lesssim T_h^{-\frac34}, \qquad n =2.  \end{equation}
Also the linearized problem and the linear paradifferential flow will be well-posed 
in $\H^\frac12$ and will satisfy Strichartz estimates on similar time intervals,
\begin{equation}\label{v-short3}
\| \bD^{-\frac{n}{2}-\delta} \partial v \|_{L^2(I_h; L^\infty)}  \lesssim    \| v\|_{L^\infty(I_h;\H^\frac12)}, \qquad n \geq 3,
\end{equation}
respectively
\begin{equation}\label{v-short2}
\| \bD^{-\frac{n}2-\frac18-\delta} \partial v \|_{L^4(I_h; L^\infty)}  \lesssim    \| v\|_{L^\infty(I_h;\H^\frac58)}, \qquad n =2,
\end{equation}
where the $L^\infty$ norm on the right may be replaced by the same $\H^\frac12$ norm evaluated at some fixed time within $I_h$. The last set of bounds may be in particular applied to $v^h$, which, in view of our bootstrap assumption, yields
\begin{equation}\label{vh-short3}
\| \bD^{-\frac{n}{2}-\delta} \partial v^h \|_{L^2(I_h; L^\infty)}  \lesssim  
2^{-(s-\frac12) h}, \qquad n \geq 3,
\end{equation}
respectively
\begin{equation}\label{vh-short2}
\| \bD^{-\frac{n}2 -\frac18-\delta} \partial v^h \|_{L^4(I_h; L^\infty)}  \lesssim 2^{-(s-\frac58) h}  , \qquad n =2.
\end{equation}

\subsection{ Long time Strichartz estimates for $u^h$ and $v^h$}
Our objective now is to obtain long time Strichartz bounds by simply adding up the short time bounds. Some care is needed when using \eqref{vh-short2} and \eqref{vh-short3} because,
as $h$ increases, we gain on one hand in the bound on the right, but we loose 
in the size of the interval $I_h$. However, the gain overrides the loss, so integrating in $h$ we arrive at the difference bound
\begin{equation}\label{du-short3}
\| \bD^{-\frac{n}{2}-\delta} \partial (u^h-u^k) \|_{L^2(I_h; L^\infty)}  \lesssim  
2^{-(s-\frac12) h}, \qquad n \geq 3, \quad h < k,
\end{equation}
respectively
\begin{equation}\label{du-short2}
\| \bD^{-\frac{n}{2}-\frac18-\delta} \partial (u^h-u^k) \|_{L^4(I_h; L^\infty)}  \lesssim 2^{-(s-\frac58) h}  , \qquad n =2, \quad h < k.
\end{equation}

Now we are able to obtain Strichartz bounds for $u^h$ on the full time interval $[0,1]$, simply by adding the short time bounds. Precisely, we claim that
for some small universal  $\delta_1 > 0$ we have 
\begin{equation}\label{Bdelta3}
\| \bD^{\frac12+\delta_1} P_k \partial u^h\|_{L^2(0,1; L^\infty)} \lesssim 1, \qquad n \geq 3,  \end{equation}
respectively 
\begin{equation}\label{Bdelta2}
\| \bD^{\frac12+\delta_1} P_k \partial u^h\|_{L^4(0,1; L^\infty)} \lesssim 1 \qquad n =2.  
\end{equation}
To see this, we differentiate cases depending on how $k$ and $h$ compare.
We fix the dimension to $n \geq 3$ for clarity.

\medskip

a) If $k \geq h$, then we simply apply \eqref{uh-short3} or \eqref{uh-short2}, taking the loss from the number of intervals. For instance in three and higher dimensions we get for $\delta_1 \leq \delta_0$
\[
\begin{aligned}
\| \bD^{\frac12+\delta_1} P_k \partial u^h\|_{L^2(0,1; L^\infty)} 
\lesssim & \  T_h^{-\frac12} \sup_{I_h} \| \bD^{\frac12+\delta_1} P_k \partial u^h\|_{L^2(I_h; L^\infty)} 
\\ \lesssim & \  T_h^{-\frac12} 2^{-\frac{k}2} \sup_{I_h} \| \bD^{1+\delta_0} P_k \partial u^h\|_{L^2(I_h; L^\infty)}
\\ \lesssim & \  T_h^{-\frac12} 2^{-\frac{h}2} T_h^{-\frac12} = 2^{(\frac{s_1-s}{s_1-s_c}   -\frac12 )h} \leq 1,
\end{aligned}
\]
for a favourable choice of $s_1$; for instance $s_1 = s+\frac14$ suffices,
as then $\frac{s_1-s}{s_1-s_c} < \frac12$. 
The two dimensional argument is similar. 

\medskip

b) if $k<h$ instead, then we first write 
\[
P_k u^h = P_k u^k + P_k(u^h-u^k). 
\]
Here the first term was already estimated before, while for the second we use
\eqref{du-short3} or \eqref{du-short2}, where the loss from the interval size is 
only in terms of $k$ and not $h$. In dimension three and higher this yields
\[
\begin{aligned}
\|\bD^{\frac12+\delta_1} P_k\partial (u^h-u^k)\|_{L^2(0,1; L^\infty)} \lesssim & \ T_k^{-\frac12}
\sup_{I_h} \| P_k\partial (u^h-u^k)\|_{L^2(0,1; L^\infty)} 
\\
\lesssim & \  T_k^{-\frac12} 2^{-(s-\frac12) k} 2^{(\frac{n+1}2+\delta)k}
= 2^{( \frac{s_1-s_c}{2\sigma} -(s-\frac12)+\frac{n+1}2+\delta)k}
\\
\lesssim & \ 2^{ (\frac{s_1-s_c}{2\sigma} - (s-s_c)+ \delta)k} \leq 1,
\end{aligned}
\]
again for a good choice of $s_1$ (same as above) and a small enough $\delta$.

In particular, the estimates \eqref{Bdelta3}, respectively \eqref{Bdelta2}
allow us to estimate our control parameter $\BB$ as follows:
\begin{equation}\label{BB3}
\| \BB\|_{L^2[0,1]} \lesssim 1, \qquad n \geq 3,    
\end{equation}
respectively 
\begin{equation}\label{BB2}
\| \BB\|_{L^4[0,1]} \lesssim 1, \qquad n =2    .
\end{equation}
This in turn allows us to use Theorem~\ref{t:ee} to control the energy growth 
for the full equation, and in particular to prove the bounds \eqref{uh-s} and \eqref{uh-s+1}, thus closing part of the bootstrap loop, namely for the bounds 
\eqref{boot-uh} and for \eqref{boot-uh-hi}.

\subsection{Strichartz estimates for the paradifferential flow}

Our objective here is to establish Strichartz estimates with loss of derivatives
for the linear paradifferential flow around $u^h$. Thus, we consider an $\H^\frac12$ solution $v$ for the paradifferential flow around $u^h$,
and we seek to estimate it dyadic pieces in the Strichartz norm, with frequency losses:

\begin{proposition}\label{p:Str-v}
Under the bootstrap assumptions \eqref{boot-uh}, \eqref{boot-uh-hi} and \eqref{boot-vh},
$\H^r$ solutions $v$ for the linear paradifferential equation 
\begin{equation}
\partial_\alpha T_{g^{\alpha\beta}(\partial u^h)} \partial_\beta v = f   
\end{equation}
satisfy the Strichartz estimates \eqref{Str-full+} with $S = S_{AIT}$ for all $r \in \R$.
\end{proposition}
Compared with the full Strichartz bounds, here we have a loss of $1/4$ derivative in dimension $3$ and higher, respectively $1/8$ derivative in dimension $2$.

\begin{proof}
Our starting point is Theorem~\ref{t:Str-move-around}, which allows us to reduce the problem 
to proving the homogeneous Strichartz estimates \eqref{Str-hom} for the corresponding homogeneous equation, again for all real $r$. To prove the proposition in this case, we have two tools at our disposal:

\begin{enumerate}[label=(\roman*)]
\item The energy estimates of Theorem~\ref{t:para-wp}. In view of the bounds \eqref{BB2}
and \eqref{BB3}, these give
uniform $\H^r$ bounds for $v$,
\[
\| v \|_{L^\infty \H^r} \lesssim \|v[0]\|_{\H^r}.
\]

\item The short time Strichartz estimates \eqref{Str-hom} with $S=S_{ST}$ on the $T_h$ time scale,
provided by Theorem~\ref{t:ST-large-ms}. Adding these with respect to the time intervals, 
we arrive at
\begin{equation}\label{v-long3-}
\| |D|^{-\frac{d}{2}-\delta} \partial v \|_{L^2(0,1; L^\infty)}  \lesssim  T_h^{-\frac12}  \| v[0]\|_{\H^\frac12}, \qquad n \geq 3,
\end{equation}
respectively
\begin{equation}\label{v-long2-}
\| |D|^{-\frac98-\delta} \partial v \|_{L^4(I_h; L^\infty)}  \lesssim   T_h^{-\frac14} \| v[0]\|_{\H^\frac58}, \qquad n =2.
\end{equation}
\end{enumerate}
\bigskip
Now we want to use these tools in order to prove the long term bounds \eqref{Str-hom} with $S=S_{AIT}$ on the unit time scale.
Given the expression for $T_h$, our first observation 
is that the estimates \eqref{v-long3-}, respectively \eqref{v-long2-}
suffice for our bounds at frequencies $ \geq h$, but not below that.

Thus, consider a lower frequency $k < h$, and seek to estimate $P_k v$. At this frequency, we have the correct estimate for the solution $\tv$ to the linear
paradifferential equation around $u_k$. It remains to compare $v$ and $\tv$. For this we use the $T_{\tP(u^k)}$ flow, and we think of $P_k v$
as an approximate solution for this flow,
\[
T_{P(u^k)} P_k v = [T_{P(u^k)},P_k] v
+ P_k
\partial_\alpha T_{g^{\alpha\beta}_h - g^{\alpha \beta}_k}
\partial_\beta v.
\]
We can bound the source terms  as follows, fixing the dimension to $n \geq 3$:
\[
\| T_{P(u^k)} P_k v \|_{L^1(I_k,H^{-\frac12})}
\lesssim \left(\| \partial^2 u^k \|_{L^1(I_k, L^\infty)}
+ 2^k \| P_{<k} (g(\partial u^h) - g(\partial u^k)) \|_{L^1(I_k,L^\infty)}\right)
\| v\|_{L^\infty \H^\frac12}.
\]
To conclude it suffices to estimate 
\begin{equation}\label{str-uk}
\|  \partial^2 u^k \|_{L^1(I_k, L^\infty)} \lesssim 1 ,  
\end{equation}
\begin{equation}\label{str-duk}
\|  P_{<k} (g(\partial u^h) - g(\partial u^k)) \|_{L^1(I_k,L^\infty)}
\lesssim 2^{-k}.
\end{equation}
The first bound follows from our earlier Strichartz estimates for $u^k$, see \eqref{uh-short2}, \eqref{uh-short3}. For the second bound,
we expand and then it suffices to have 
\begin{equation}\label{want-vj}
\|  P_{<k} (g'(\partial u^j) \partial v^j) \|_{L^1(I_k,L^\infty)} \lesssim 2^{-k} 2^{-c(j-k)}, \qquad j > k,
\end{equation}
with a positive constant $c$ in order to allow for integration in $j$. 
We expand paradifferentially, depending on the frequencies of the two factors above.
It suffices to consider the following two cases:

\medskip

a) $v^j$ has the frequency below $2^k$. Then we use the Strichartz bounds for $v_j$ over intervals $I_j$, and then sum over such intervals. For instance in dimension $n \geq 3$ we get
\[
\begin{aligned}
\| P_{<k} \partial v_j\|_{L^1(I_k;L^\infty)}
\lesssim & \ |I_k| |I_j|^{-\frac12} \sup_{I_j} \| P_{<k} \partial v_j\|_{L^2(I_j;L^\infty)}
\\
\lesssim & \ 2^{(\frac{n}2+\delta) k} 2^{-(s-\frac12) j}
\frac{|I_k|}{|I_j|} |I_j|^\frac12
\\
= & \ 2^{[\frac{n}2+\delta - \frac{s_1-s}{s_1-s_c}]  k} 
2^{-(s-\frac12 - \frac{s_1-s}{2(s_1-s_c)}) j} 
\\
= & \  2^{-k} 2^{[(s_c-s)(1-  \frac{1}{2(s_1-s_c)}) +\delta ]k }
2^{-(s-\frac12 - \frac{s_1-s}{2(s_1-s_c)}) (j-k)}.
\end{aligned}
\]
Here the coefficient of $j-k$ is negative by a large margin, 
while the coefficient of $k$ in the middle factor is also negative 
since $s_1 - s_c > \frac12$ and $\delta$ is arbitrarily small.
Hence we obtain a bound as desired in \eqref{want-vj}.
\medskip

b) The balanced case, where both frequencies have size $2^l$ with $l \geq k$.
This is easier, as we have a better energy bound for the first factor.
Hence in this case it is more efficient to estimate the output by applying Bernstein's inequality first,
\[
\begin{aligned}
\| P_{<k} [P_{l}g'(\partial u^j) P_l \partial v^j] \|_{L^1(I_k,L^\infty)} \lesssim & \  
|I_k|^\frac12 2^{\frac{nk}2}  \| P_{l}g'(\partial u^j) \|_{L^\infty(I_k,L^2)}
\|P_l \partial v^j \|_{L^2(I_k,L^\infty)} 
\\ 
\lesssim & \  
|I_k| |I_j|^{-\frac12} 2^{\frac{nk}2}  \| P_{l}g'(\partial u^j) \|_{L^\infty L^2}
\sup_{I_j}\|P_l \partial v^j \|_{L^2(I_j,L^\infty)} 
\\ 
\lesssim & \ |I_k| |I_j|^{-\frac12} 
2^{\frac{nk}2} 2^{-(s-1) l} 2^{(\frac{n}2+\delta) l} 2^{-(s-\frac12) j}
\\
= & \  2^{[\frac{n}2 - \frac{s_1-s}{s_1-s_c}]  k}   2^{(\frac{n}2-s+1+\delta) l} 
2^{-(s-\frac12 - \frac{s_1-s}{2(s_1-s_c)}) j}
\\
= & \  2^{-k} 2^{[(s_c-s)(2 -  \frac{1}{2(s_1-s_c)})+\delta]  k}   2^{(\frac{n}2-s+1+\delta) (l-k)} 
2^{-(s-\frac12 - \frac{s_1-s}{2(s_1-s_c)}) (j-k)},
\end{aligned}
\]
which is better than in case (a).
\end{proof}

\subsection{Strichartz estimates for the linearized flow}

Our aim here is to show that that we have $\H^\frac12$  well-posedness and Strichartz estimates with loss of derivatives for the linearized flow around $u^h$ in $\H^\frac12$ (respectively $\H^\frac58$ if $n=2$):

\begin{proposition}\label{p:Str-v-lin}
Under the bootstrap assumptions \eqref{boot-uh}, \eqref{boot-uh-hi} and \eqref{boot-vh}, the linearized equation around $u^h$ is well-posed in 
$\H^\frac12$ (respectively $\H^\frac58$ if $n=2$), and its solutions
satisfy the full Strichartz estimates \eqref{Str-full+} with $S = S_{AIT}$.
\end{proposition}

Here we use the analysis in Section~\ref{s:linearized}. Precisely Theorem~\ref{t:linearized} there shows that the above proposition follows 
directly from the similar result in Proposition~\ref{p:Str-v} for the 
linear paradifferential equation.

\subsection{ Closing the bootstrap argument} 
Combining the Strichartz estimates for the linear paradifferential equation in 
Proposition~\ref{p:Str-v} with the result of Theorem~\ref{t:linearized}, 
it follows that the linearized flow around $u^h$ is wellposed in $\H^\frac12$ (respectively $\H^\frac58$ if $n=2$), with 
the same Strichartz estimates as in Proposition~\ref{p:Str-v}, which is exactly part (b) of Theorem~\ref{t:uh}. As a consequence,
the initial data bound \eqref{vh-data} for $v$ implies the uniform bound \eqref{vh-uniform}, which in turn closes the bootstrap assumption \eqref{boot-vh}.

\subsection{The well-posedness result}

In order to be able to obtain a complete well-posedness argument, we follow the outline in \cite{IT-primer}, and  measure the size of the functions $u^h$ and $v^h$ in terms of frequency envelopes. Precisely, we consider a normalized frequency envelope $\epsilon c_h$ for $u[0]$ in $\H^s$.
Then for the localized initial data we have the bounds
\begin{equation}
\| u^h[0]\|_{\H^s} \lesssim \epsilon,    
\end{equation}
\begin{equation}
\| u^h[0]\|_{\H^{s+1}} \lesssim 2^h \epsilon c_h.    
\end{equation}
On the other hand, fixing the dimension to $n \geq 3$, we will measure $v^h$ in $\H^{\frac12}$, where for the 
initial data we have 
\begin{equation}
\begin{aligned}
\|v^h[0]\|_{\H^\frac12} \lesssim 2^{-(s-\frac12)h} c_h, \qquad n\geq 3.
\end{aligned}
\end{equation}

Then by Theorem~\ref{t:uh}, we obtain corresponding uniform bounds for the solutions on the time interval $[0,1]$,

\begin{equation}\label{uh-uniform}
\| u^h[\cdot]\|_{L^\infty(0.1;\H^s)} \lesssim \epsilon,    
\end{equation}
\begin{equation}\label{uh-unif+1}
\| u^h[\cdot]\|_{L^\infty(0,1;\H^{s+1})} \lesssim 2^h \epsilon c_h.    
\end{equation}
Similarly, the linearized increments $v_h$ satisfy the uniform bounds 
\begin{equation}\label{vh-unif}
\| v_h\|_{L^\infty(0,1;\H^\frac12)} \lesssim \epsilon  2^{(s-\frac12)s} c_h.  
\end{equation}

Integrating the last bound with respect to $h$, we obtain the difference bounds
\begin{equation}\label{uh-diff}
\| u^h - u^k\|_{L^\infty(0,1;\H^\frac12)} \lesssim \epsilon  2^{(s-\frac12)s} c_h, \qquad h <k.  
\end{equation}
This implies that the limit 
\[
u = \lim_{h \to \infty} u^h
\]
exists in $C(0,1;\H^\frac12)$. In view of \eqref{uh-uniform},
the limit $u$ will also satisfy 
\[
\| u\|_{L^\infty(0,1;\H^s)} \lesssim \epsilon.
\]
We can also prove that we have the previous convergence in this stronger topology. To see this, we consider unit increments in $h$, and compare $u_h$ with $u_{h+1}$, using \eqref{uh-unif+1}
on one hand, and \eqref{uh-diff} on the other hand.
This yields
\begin{equation}\label{uh-diff-high}
\| u^h - u^{h+1}\|_{C(0,1;\H^{s+1})} \lesssim 2^h \epsilon c_h,    
\end{equation}
respectively 
\begin{equation}\label{uh-diff-lo}
\| u^h - u^{h+1}\|_{C(0,1;\H^\frac12)} \lesssim \epsilon  2^{(s-\frac12)s} c_h, \qquad h <k.  
\end{equation}
These two bounds balance exactly at frequency $2^h$,
and measure the $\H^s$ norm but with decay away from frequency $2^k$. Hence the differences are almost orthogonal in $\H^s$,
and, summing them up, we obtain
\begin{equation}\label{uh-diff=}
\| u^h - u^{k}\|_{C(0,1;\H^{s})} \lesssim  \epsilon c_{[h,k]}.    \end{equation}
This implies uniform convergence in $\H^s$. Thus our solution 
$u$ is uniquely identified as the strong $\H^s$ uniform limit of $u^h$.

The continuous dependence and the weak Lipschitz dependence follow exactly as in \cite{IT-primer}.

%\bibliography{nlw}{}
\bibliographystyle{plain}

\end{document}